\documentclass[11pt]{amsart}
\usepackage{amsmath}
\usepackage{amsfonts}
\usepackage{amsxtra}
\usepackage{amscd}
\usepackage{amsthm}
\usepackage{amssymb}
\usepackage{eucal}
\usepackage[all]{xy}
\usepackage{graphicx}

\usepackage{enumitem}

\makeatletter
\def\amsbb{\use@mathgroup \M@U \symAMSb}
\makeatother

\usepackage{bbold}

\renewcommand{\mathbb}{\amsbb}

\usepackage{xcolor}
\usepackage[hidelinks=true]{hyperref}
\usepackage{tikz-cd}
\usepackage{mathtools}

\newtheorem{cor}[subsubsection]{Corollary}
\newtheorem{lem}[subsubsection]{Lemma}
\newtheorem{prop}[subsubsection]{Proposition}

\newtheorem{thm}[subsubsection]{Theorem}

\theoremstyle{plain}
\newtheorem{Thm}{Theorem}[section]
\newtheorem{Thm*}{Theorem}[section]
\newtheorem{Thm'}[Thm]{"Theorem"}

\newtheorem*{thm*}{Theorem}
\newtheorem*{cor*}{Corollary}

\theoremstyle{definition}
\newtheorem{Def}[Thm]{Definition}

\newtheorem{Q}[Thm]{Question}

\numberwithin{equation}{section}
\newcommand{\nc}{\newcommand}
\nc{\lm}{\lambda}

\newcommand{\surj}{\twoheadrightarrow}

\newcommand{\wt}{\widetilde}

\renewcommand{\P}{P}

\newcommand{\Spec}{\operatorname{Spec}}

\newcommand{\Gr}{\operatorname{Gr}}

\newcommand{\Sch}{\operatorname{Sch}}

\newcommand{\Lie}{\operatorname{Lie}}

\newcommand{\Fl}{\operatorname{Fl}}

\newcommand{\Ind}{\operatorname{Ind}}

\newcommand{\Hom}{\operatorname{Hom}}

\newcommand{\cD}{\mathcal{D}}

\newcommand{\Res}{\operatorname{Res}}

\theoremstyle{definition}

\theoremstyle{remark}
\newtheorem{rem}[subsubsection]{Remark}

\newcommand{\thmref}[1]{Theorem~\ref{#1}}

\newcommand{\secref}[1]{Sect.~\ref{#1}}
\newcommand{\lemref}[1]{Lemma~\ref{#1}}
\newcommand{\propref}[1]{Proposition~\ref{#1}}
\newcommand{\corref}[1]{Corollary~\ref{#1}}

\numberwithin{equation}{section}

\nc{\renc}{\renewcommand}
\nc{\ssec}{\subsection}
\nc{\sssec}{\subsubsection}
\nc{\on}{\operatorname}

\nc\ol{\overline}
\nc\tboxtimes{\wt{\boxtimes}}
\nc\tstar{\wt{\star}}
\nc{\alp}{a}

\nc{\ZZ}{{\mathbb Z}}
\nc{\NN}{{\mathbb N}}
\nc{\OO}{{\mathbb O}}
\renc{\SS}{{\mathbb S}}
\nc{\DD}{{\mathbb D}}
\nc{\GG}{{\mathbb G}}

\nc{\Fq}{{\mathbb F}_q}
\nc{\Fqb}{\ol{{\mathbb F}_q}}
\nc{\Ql}{\ol{{\mathbb Q}_\ell}}
\nc{\id}{\text{id}}
\nc\X{\mathcal X}

\nc{\Loc}{\on{Loc}}
\nc{\Pic}{\on{Pic}}
\nc{\Bun}{\on{Bun}}
\nc{\IC}{\on{IC}}
\nc{\Fls}{\on{Fl}^{\frac{\infty}{2}}}
\nc{\ICs}{\on{IC}^{\frac{\infty}{2}}}
\nc{\ICsl}{\on{IC}^{\lambda+\frac{\infty}{2}}}
\nc{\ICslm}{\on{IC}^{\lambda+\frac{\infty}{2},-}}
\nc{\ICsm}{\on{IC}^{\frac{\infty}{2},-}}
\nc{\Sh}{\on{Sh}}
\nc{\pos}{{\on{pos}}}
\nc{\Conv}{\on{Conv}}
\nc{\Sph}{\on{Sph}}
\nc{\Sat}{\on{Sat}}
\nc{\Sym}{\on{Sym}}
\nc{\BunBb}{\overline{\Bun}_B}
\nc{\BunNb}{\overline{\Bun}_N}
\nc{\BunTb}{\overline{\Bun}_T}
\nc{\BunBbm}{\overline{\Bun}_{B^-}}
\nc{\BunBbel}{\overline{\Bun}_{B,el}}
\nc{\BunBbmel}{\overline{\Bun}_{B^-,el}}
\nc{\Buno}{\overset{o}{\Bun}}
\nc{\BunPb}{{\overline{\Bun}_P}}
\nc{\BunBM}{\Bun_{B(M)}}
\nc{\BunBMb}{\overline{\Bun}_{B(M)}}
\nc{\BunPbw}{{\widetilde{\Bun}_P}}
\nc{\BunBP}{\widetilde{\Bun}_{B,P}}
\nc{\GUb}{\overline{G/U}}
\nc{\GUPb}{\overline{G/U(P)}}

\nc{\Hhom}{\underline{\on{Hom}}}
\nc\syminfty{\on{Sym}^{\infty}}
\nc\lal{\ol{\kappa_x}}
\nc\xl{\ol{x}}
\nc\thl{\ol{\theta}}
\nc\nul{\ol{\nu}}
\nc\mul{\ol{\mu}}
\nc\Sum\Sigma
\nc{\oX}{\overset{\circ}{X}{}}
\nc{\hl}{\overset{\leftarrow}h{}}
\nc{\hr}{\overset{\rightarrow}h{}}
\nc{\M}{{\mathcal M}}
\nc{\N}{{\mathcal N}}
\nc{\F}{{\mathcal F}}
\nc{\D}{{\mathcal D}}

\nc{\Y}{{\mathcal Y}}
\nc{\G}{{\mathcal G}}
\nc{\E}{{\mathcal E}}
\nc{\CalC}{{\mathcal C}}
\nc\Dh{\widehat{\D}}

\nc{\C}{{\mathcal C}}
\nc{\K}{{\mathcal K}}
\renewcommand{\H}{{\mathcal H}}

\nc{\T}{{\mathcal T}}
\nc{\V}{{\mathcal V}}
\renc{\P}{{\mathcal P}}
\nc{\A}{{\mathcal A}}
\nc{\B}{{\mathcal B}}
\nc{\U}{{\mathcal U}}

\nc{\frn}{{\check{\mathfrak u}(P)}}

\nc{\fC}{\mathfrak C}
\nc{\p}{\mathfrak p}
\nc{\q}{\mathfrak q}
\nc\f{{\mathfrak f}}

\nc{\qo}{{\mathfrak q}}
\nc{\po}{{\mathfrak p}}
\nc{\s}{{\mathfrak s}}
\nc\w{\text{w}}

\renewcommand{\mod}{{\on{-mod}}}

\nc\mathi\iota
\nc\Mod{\on{Mod}}

\nc{\pw}{\widetilde{\mathfrak p}}
\nc{\qw}{\widetilde{\mathfrak q}}
\nc{\jw}{\widetilde j}

\nc{\grb}{\overline{\Gr}}
\nc{\I}{\mathcal I}

\nc{\kappach}{{\check\kappa_x}}
\nc{\Lambdach}{{\check\Lambda}{}}
\nc{\much}{{\check\mu}}
\nc{\omegach}{{\check\omega}}
\nc{\nuch}{{\check\nu}}
\nc{\etach}{{\check\eta}}
\nc{\alphach}{{\checka}}
\nc{\oblvtach}{{\check\oblvta}}
\nc{\pich}{{\check\pi}}
\nc{\ch}{{\check h}}

\nc{\Hb}{\overline{\H}}

\nc{\BA}{{\mathbb{A}}}
\nc{\BC}{{\mathbb{C}}}
\nc{\BE}{{\mathbb{E}}}
\nc{\BF}{{\mathbb{F}}}
\nc{\BG}{{\mathbb{G}}}
\nc{\BM}{{\mathbb{M}}}
\nc{\BO}{{\mathbb{O}}}
\nc{\BD}{{\mathbb{D}}}
\nc{\BL}{{\mathbb{L}}}
\nc{\BN}{{\mathbb{N}}}
\nc{\BP}{{\mathbb{P}}}
\nc{\BQ}{{\mathbb{Q}}}
\nc{\BR}{{\mathbb{R}}}
\nc{\BV}{{\mathbb{V}}}
\nc{\BW}{{\mathbb{W}}}
\nc{\BZ}{{\mathbb{Z}}}
\nc{\BS}{{\mathbb{S}}}

\nc{\CA}{{\mathcal{A}}}
\nc{\CB}{{\mathcal{B}}}

\nc{\CE}{{\mathcal{E}}}
\nc{\CF}{{\mathcal{F}}}
\nc{\CG}{{\mathcal{G}}}
\nc{\CH}{{\mathcal{H}}}

\nc{\CL}{{\mathcal{L}}}
\nc{\CC}{{\mathcal{C}}}
\nc{\CM}{{\mathcal{M}}}
\nc{\CN}{{\mathcal{N}}}
\nc{\cCN}{\check{{\mathcal{N}}}}
\nc{\CK}{{\mathcal{K}}}
\nc{\CO}{{\mathcal{O}}}
\nc{\CP}{{\mathcal{P}}}
\nc{\CQ}{{\mathcal{Q}}}
\nc{\CR}{{\mathcal{R}}}
\nc{\CS}{{\mathcal{S}}}
\nc{\CT}{{\mathcal{T}}}
\nc{\CU}{{\mathcal{U}}}
\nc{\CV}{{\mathcal{V}}}
\nc{\CW}{{\mathcal{W}}}
\nc{\CX}{{\mathcal{X}}}
\nc{\CY}{{\mathcal{Y}}}
\nc{\CZ}{{\mathcal{Z}}}
\nc{\CI}{{\mathcal{I}}}
\nc{\CJ}{{\mathcal{J}}}

\nc{\csM}{{\check{\mathcal A}}{}}
\nc{\oM}{{\overset{\circ}{\mathcal M}}{}}
\nc{\obM}{{\overset{\circ}{\mathbf M}}{}}
\nc{\oCA}{{\overset{\circ}{\mathcal A}}{}}
\nc{\obA}{{\overset{\circ}{\mathbf A}}{}}
\nc{\ooM}{{\overset{\circ}{M}}{}}
\nc{\osM}{{\overset{\circ}{\mathsf M}}{}}
\nc{\vM}{{\overset{\bullet}{\mathcal M}}{}}
\nc{\nM}{{\underset{\bullet}{\mathcal M}}{}}
\nc{\oD}{{\overset{\circ}{\mathcal D}}{}}
\nc{\obD}{{\overset{\circ}{\mathbf D}}{}}
\nc{\oA}{{\overset{\circ}{\mathbb A}}{}}
\nc{\op}{{\overset{\bullet}{\mathbf p}}{}}
\nc{\cp}{{\overset{\circ}{\mathbf p}}{}}
\nc{\oU}{{\overset{\bullet}{\mathcal U}}{}}
\nc{\oZ}{{\overset{\circ}{\mathcal Z}}{}}
\nc{\ofZ}{{\overset{\circ}{\mathfrak Z}}{}}
\nc{\oF}{{\overset{\circ}{\fF}}}

\nc{\fa}{{\mathfrak{a}}}
\nc{\fb}{{\mathfrak{b}}}
\nc{\fd}{{\mathfrak{d}}}
\nc{\ff}{{\mathfrak{f}}}
\nc{\fg}{{\mathfrak{g}}}
\nc{\fgl}{{\mathfrak{gl}}}
\nc{\fh}{{\mathfrak{h}}}
\nc{\fj}{{\mathfrak{j}}}
\nc{\fl}{{\mathfrak{l}}}
\nc{\fm}{{\mathfrak{m}}}
\nc{\fn}{{\mathfrak{n}}}
\nc{\fu}{{\mathfrak{u}}}
\nc{\fp}{{\mathfrak{p}}}
\nc{\fr}{{\mathfrak{r}}}
\nc{\fs}{{\mathfrak{s}}}
\nc{\ft}{{\mathfrak{t}}}
\nc{\fz}{{\mathfrak{z}}}
\nc{\fsl}{{\mathfrak{sl}}}
\nc{\hsl}{{\widehat{\mathfrak{sl}}}}
\nc{\hgl}{{\widehat{\mathfrak{gl}}}}
\nc{\hg}{{\widehat{\mathfrak{g}}}}
\nc{\htt}{{\widehat{\mathfrak{t}}}}
\nc{\chg}{{\widehat{\mathfrak{g}}}{}^\vee}
\nc{\hn}{{\widehat{\mathfrak{n}}}}
\nc{\chn}{{\widehat{\mathfrak{n}}}{}^\vee}

\nc{\fA}{{\mathfrak{A}}}
\nc{\fB}{{\mathfrak{B}}}
\nc{\fD}{{\mathfrak{D}}}
\nc{\fE}{{\mathfrak{E}}}
\nc{\fF}{{\mathfrak{F}}}
\nc{\fG}{{\mathfrak{G}}}
\nc{\fK}{{\mathfrak{K}}}
\nc{\fL}{{\mathfrak{L}}}
\nc{\fM}{{\mathfrak{M}}}
\nc{\fN}{{\mathfrak{N}}}
\nc{\fP}{{\mathfrak{P}}}
\nc{\fR}{{\mathfrak{R}}}
\nc{\fU}{{\mathfrak{U}}}
\nc{\fV}{{\mathfrak{V}}}
\nc{\fX}{{\mathfrak{X}}}
\nc{\fZ}{{\mathfrak{Z}}}

\nc{\ba}{{\mathbf{a}}}

\nc{\bc}{{\mathbf{c}}}
\nc{\bd}{{\mathbf{d}}}
\nc{\bbf}{{\mathbf{f}}}
\nc{\be}{{\mathbf{e}}}
\nc{\bi}{{\mathbf{i}}}
\nc{\bj}{{\mathbf{j}}}
\nc{\bn}{{\mathbf{n}}}
\nc{\bo}{{\mathbf{o}}}
\nc{\bp}{{\mathbf{p}}}
\nc{\bq}{{\mathbf{q}}}
\nc{\bu}{{\mathbf{u}}}
\nc{\bv}{{\mathbf{v}}}
\nc{\bx}{{\mathbf{x}}}
\nc{\bs}{{\mathbf{s}}}
\nc{\by}{{\mathbf{y}}}
\nc{\bw}{{\mathbf{w}}}
\nc{\bA}{{\mathbf{A}}}
\nc{\bK}{{\mathbf{K}}}
\nc{\bB}{{\mathbf{B}}}
\nc{\bF}{{\mathbf{F}}}
\nc{\bC}{{\mathbf{C}}}
\nc{\bG}{{\mathbf{G}}}
\nc{\bD}{{\mathbf{D}}}
\nc{\bE}{{\mathbf{E}}}
\nc{\bH}{{\mathbf{H}}}
\nc{\bI}{{\mathbf{I}}}
\nc{\bL}{{\mathbf{L}}}
\nc{\bM}{{\mathbf{M}}}
\nc{\bN}{{\mathbf{N}}}
\nc{\bO}{{\mathbf{O}}}
\nc{\bV}{{\mathbf{V}}}
\nc{\bW}{{\mathbf{W}}}
\nc{\bX}{{\mathbf{X}}}
\nc{\bZ}{{\mathbf{Z}}}
\nc{\bS}{{\mathbf{S}}}

\nc{\sA}{{\mathsf{A}}}
\nc{\sB}{{\mathsf{B}}}
\nc{\sC}{{\mathsf{C}}}
\nc{\sD}{{\mathsf{D}}}
\nc{\sF}{{\mathsf{F}}}
\nc{\sG}{{\mathsf{G}}}
\nc{\sH}{{\mathsf{H}}}
\nc{\sK}{{\mathsf{K}}}
\nc{\sM}{{\mathsf{M}}}
\nc{\sO}{{\mathsf{O}}}
\nc{\sW}{{\mathsf{W}}}
\nc{\sQ}{{\mathsf{Q}}}
\nc{\sP}{{\mathsf{P}}}
\nc{\sZ}{{\mathsf{Z}}}
\nc{\sr}{{\mathsf{r}}}
\nc{\sh}{{\mathsf{h}}}
\nc{\bk}{{\mathsf{k}}}
\nc{\sg}{{\mathsf{g}}}
\nc{\sff}{{\mathsf{f}}}
\nc{\sfe}{{\mathsf{e}}}
\nc{\sfj}{{\mathsf{j}}}
\nc{\sfb}{{\mathsf{b}}}
\nc{\sfc}{{\mathsf{c}}}
\nc{\sd}{{\mathsf{d}}}
\nc{\sv}{{\mathsf{v}}}
\nc{\sfX}{{\mathsf{X}}}

\nc{\BK}{{\bar{K}}}

\nc{\tA}{{\widetilde{\mathbf{A}}}}
\nc{\tB}{{\widetilde{\mathcal{B}}}}
\nc{\tg}{{\widetilde{\mathfrak{g}}}}
\nc{\tG}{{\widetilde{G}}}
\nc{\TM}{{\widetilde{\mathbb{M}}}{}}
\nc{\tO}{{\widetilde{\mathsf{O}}}{}}
\nc{\tU}{{\widetilde{\mathfrak{U}}}{}}
\nc{\TZ}{{\tilde{Z}}}
\nc{\tx}{{\tilde{x}}}
\nc{\tbv}{{\tilde{\bv}}}
\nc{\tfP}{{\widetilde{\mathfrak{P}}}{}}
\nc{\tz}{{\tilde{\zeta}}}
\nc{\tmu}{{\tilde{\mu}}}

\nc{\urho}{\underline{\pi}}
\nc{\uB}{\underline{B}}
\nc{\uC}{{\underline{\mathbb{C}}}}
\nc{\ui}{\underline{i}}
\nc{\uj}{\underline{j}}
\nc{\ofP}{{\overline{\mathfrak{P}}}}
\nc{\oB}{{\overline{\mathcal{B}}}}
\nc{\og}{{\overline{\mathfrak{g}}}}
\nc{\oI}{{\overline{I}}}

\nc{\eps}{\varepsilon}
\nc{\hrho}{{\hat{\pi}}}

\nc{\one}{{\mathbf{1}}}
\nc{\two}{{\mathbf{t}}}

\nc{\Hilb}{{\mathop{\operatorname{\rm Hilb}}}}
\nc{\CHom}{{\mathop{\operatorname{{\mathcal{H}}\it om}}}}
\nc{\de}{{\mathop{\operatorname{\rm def}}}}
\nc{\length}{{\mathop{\operatorname{\rm length}}}}
\nc{\supp}{{\mathop{\operatorname{\rm supp}}}}

\nc{\Cliff}{{\mathsf{Cliff}}}
\nc{\Fib}{{\mathsf{Fib}}}
\nc{\Coh}{{\on{Coh}}}
\nc{\QCoh}{{\on{QCoh}}}
\nc{\IndCoh}{{\on{IndCoh}}}
\nc{\FCoh}{{\mathsf{FCoh}}}

\nc{\reg}{{\text{\rm reg}}}

\nc{\cplus}{{\mathbf{C}_+}}
\nc{\cminus}{{\mathbf{C}_-}}
\nc{\cthree}{{\mathbf{C}_*}}
\nc{\Qbar}{{\bar{Q}}}
\nc\Eis{\on{Eis}}
\nc\Eisb{\ol\Eis{}}
\nc\Eisr{\on{Eis}^{rat}{}}
\nc\wh{\widehat}

\nc{\barZ}{\overline{Z}{}}
\nc{\barbarZ}{\overline{\barZ}{}}
\nc{\barpi}{\overline\iota}
\nc{\barbarpi}{\overline\barpi}
\nc{\barpip}{\overline\iota{}^+}
\nc{\barpim}{\overline\iota{}^-}

\nc{\fq}{\mathfrak q}

\nc{\fqb}{\ol{\fq}{}}
\nc{\fpb}{\ol{\fp}{}}
\nc{\fpr}{{\fp^{rat}}{}}
\nc{\fqr}{{\fq^{rat}}{}}

\nc{\hattimes}{\wh\otimes}

\nc{\bh}{{\bar{h}}}
\nc{\bOmega}{{\overline{\Omega(\check \fn)}}}

\nc{\seq}[1]{\stackrel{#1}{\sim}}

\nc{\cT}{{\check{T}}}
\nc{\cG}{{\check{G}}}
\nc{\cM}{{\check{M}}}
\nc{\cB}{{\check{B}}}
\nc{\cN}{{\check{N}}}

\nc{\ct}{{\check{\mathfrak t}}}
\nc{\cg}{{\check{\fg}}}
\nc{\hcg}{{\widehat{\check{\fg}}}}
\nc{\cb}{{\check{\fb}}}
\nc{\cn}{{\check{\fn}}}

\nc{\cLambda}{{\check\Lambda}}

\nc{\cla}{{\check\kappa_x}}
\nc{\cmu}{{\check\mu}}
\nc{\clambda}{{\check\lambda}}
\nc{\cnu}{{\check\nu}}
\nc{\ceta}{{\check\eta}}

\nc{\DefbE}{{\on{Def}_{\cB}(E_\cT)}}

\nc{\imathb}{{\ol{\imath}}}
\nc{\rlr}{\overset{\longrightarrow}{\underset{\longrightarrow}\longleftarrow}}

\nc{\KG}{K\backslash G}
\nc{\comult}{{co\text{-}mult}}
\nc{\counit}{{co\text{-}unit}}
\nc{\uHom}{{\underline{\Maps}}}
\nc{\dgSch}{\on{Sch}}
\nc{\affdgSch}{\on{Sch}^{\on{aff}}}
\nc{\affSch}{\on{Sch}^{\on{aff}}}
\nc{\Groupoids}{\on{Grpd}}
\nc{\inftygroup}{\on{Spc}}
\nc{\inftyCat}{\infty\on{-Cat}}
\nc{\StinftyCat}{\inftyCat^{\on{St}}}
\nc{\MoninftyCat}{\infty\on{-Cat}^{\on{Mon}}}
\nc{\SymMoninftyCat}{\infty\on{-Cat}^{\on{SymMon}}}
\nc{\SymMonStinftyCat}{\on{DGCat}^{\on{SymMon}}}
\nc{\MonStinftyCat}{\on{DGCat}^{\on{Mon}}}
\nc{\inftystack}{\on{Stk}}
\nc{\inftystackalg}{Stk^{1\text{-}alg}}
\nc{\inftyprestack}{\on{PreStk}}
\nc{\inftydgnearstack}{\on{NearStk}}
\nc{\inftydgstack}{\on{Stk}}
\nc{\inftydgstackalg}{DGStk^{1\text{-}alg}}
\nc{\inftydgprestack}{\on{PreStk}}
\nc{\dgindSch}{\on{indSch}}
\nc{\indSch}{{}^{\on{cl}}\!\on{indSch}}
\nc{\infSch}{\on{infSch}}
\nc{\dr}{{\on{dR}}}

\nc{\mmod}{{\on{-}\!{\mathbf{mod}}}}

\nc{\starr}{\text{\dh}}
\nc{\Spectra}{\on{Spectra}}
\nc{\Crys}{\on{Crys}}
\nc{\oblv}{{\mathbf{oblv}}}
\nc{\ind}{{\mathbf{ind}}}
\nc{\coind}{{\mathbf{coind}}}
\nc{\inv}{{\mathbf{inv}}}
\nc{\triv}{{\mathbf{triv}}}
\nc{\CMaps}{{\mathcal Maps}}
\nc{\Maps}{\on{Maps}}
\nc{\bMaps}{\mathbf{Maps}}
\nc{\BMaps}{\ul{\on{Maps}}}
\nc{\Grid}{\on{Grid}}
\nc{\hGrid}{\on{Grid}^{\geq\,\on{dgnl}}}
\nc{\Diag}{\on{Diag}}
\nc{\bDelta}{\mathbf{\Delta}}
\nc{\tCateg}{(\infty\on{-2)-Cat}}
\nc{\ul}{\underline}
\nc{\Seg}{\on{Seq}}
\nc{\biSeg}{\on{bi-Seq}}
\nc{\triSeg}{\on{tri-Seq}}
\nc{\quadSeg}{\on{quad-Seq}}
\nc{\nSeg}{\on{n-Seq}}
\nc{\Segm}{\on{Seg}^{\on{mkd}}}
\nc{\fLm}{\fL^{\on{mkd}}}
\nc{\inftyCatm}{\inftyCat^{\on{mkd}}}
\nc{\Blocks}{\mathbf{Blocks}}
\nc{\Snakes}{\mathbf{Snakes}}
\nc{\bifL}{\on{bi-}\!\fL}
\nc{\Sets}{\on{Sets}}
\nc{\Ran}{{\on{Ran}}}
\nc{\Vect}{\on{Vect}}
\nc{\Shv}{\on{Shv}}
\nc{\unn}{\mathbf{union}}
\nc{\Spc}{\on{Spc}}
\nc{\ppart}{(\!(t)\!)}
\nc{\qqart}{[\![t]\!]}
\nc{\Dmod}{\on{D-mod}}
\nc{\ocD}{\overset{\circ}{\cD}}
\nc{\sfp}{\mathsf{p}}
\nc{\sfq}{\mathsf{q}}
\nc{\DGCat}{\on{DGCat}}
\renc{\det}{\on{det}}
\nc{\Conf}{\on{Conf}}
\nc{\Whit}{\on{Whit}}
\nc{\Reg}{\on{Reg}}
\nc{\BunNbx}{(\BunNb)_{\infty\cdot x}}
\nc{\bHecke}{\overset{\bullet}{\on{Hecke}}}
\nc{\Hecke}{\on{Hecke}}
\nc{\bCZ}{\ol\CZ}
\nc{\oCZ}{\overset{\circ}\CZ}
\nc{\boCZ}{\ol{\oCZ}}
\nc{\sotimes}{\overset{!}\otimes}
\nc{\semiinf}{\frac{\infty}{2}}
\nc{\coInd}{\on{coInd}}
\nc{\bCM}{\overset{\bullet}\CM{}}
\nc{\bCF}{\overset{\bullet}\CF{}}
\nc{\SI}{\on{SI}}
\nc{\KL}{\on{KL}}
\nc{\Av}{\on{Av}}
\nc{\Inv}{\on{Inv}}
\nc{\KM}{\on{KM}}
\nc{\LocSys}{\on{LocSys}}
\nc{\ofX}{\overset{\circ}{\fX}}

\renc{\sG}{G}

\begin{document}

\title{Connections on moduli spaces and infinitesimal Hecke modifications}
\author[Nick Rozenblyum]{Nick Rozenblyum \\
\MakeLowercase{with an appendix by} Dennis Gaitsgory}

\date{\today}

\maketitle

\begin{abstract}
Let $X$ be a proper scheme and $\CZ$ a prestack over $X$ equipped with a flat connection.  We give a local-to-global description of D-modules on the prestack $\CS(\CZ)$ of flat sections of $\CZ$.  Examples of $\CS(\CZ)$
include the moduli stacks of principal $G$-bundles and de Rham local systems on $X$.  We show that the category of D-modules is equivalent to the category of ind-coherent sheaves which are equivariant with respect to infinitesimal Hecke groupoids parametrized by finite subsets of $X$. We describe a number of applications to geometric representation theory and conformal field theory,
including a derived enhancement of the Verlinde formula: the derived space of conformal blocks (a.k.a. chiral
homology) of the WZW model is isomorphic to the cohomology of the corresponding line bundle on
$\on{Bun}_G$, the moduli stack of $G$-bundles.
\end{abstract}

\tableofcontents

\section*{Introduction}

Let $X$ be a proper scheme over an algebraically closed field $k$ of characteristic 0.  In many situations, one is
interested in studying various moduli spaces associated to $X$; for instance, the moduli stack of principal $G$-bundles, for an algebraic group $G$.  Many such moduli spaces admit the following common origin (see \secref{ss:intro applications}).

\medskip

Let $\CZ_{\on{crys}} \to X$ be a prestack, equipped with a flat connection along $X$.  
A convenient way to package the data of such a flat connection is as a prestack $\CZ$ 
over the de Rham space $X_\dr$ of $X$ (see \eqref{e:dR}) such that
$$\CZ_{\on{crys}}:=X\underset{X_\dr}\times \CZ.$$
We can then form the corresponding moduli space $\CS(\CZ)$, which is the prestack of sections of
$\CZ \to X_\dr$, i.e., the prestack of flat sections of $\CZ_{\on{crys}} \to X$.
In this paper, we study the differential algebraic geometry of the prestacks $\CS(\CZ)$.

\medskip

Suppose that the prestack $\CS(\CZ)$ is \emph{laft:=locally almost of finite type}\footnote{Note that we are not assuming that $\CZ$
itself is laft, and it is indeed not such in a large class of examples of interest, see \secref{ss:jets}.}, 
so that one can consider the categories 
$$\IndCoh(\CS(\CZ)) \text{ and } \Dmod(\CS(\CZ)),$$
as developed in \cite[Chapter 4]{Vol2}.

\medskip

Furthermore, suppose that the prestack $\CZ$ \emph{admits deformation theory}
(we call such prestacks \emph{laft-def}). In this case, 
by \cite[Chapter 4, Sect. 3.1]{Vol2}, the forgetful functor
$$\oblv_{\CS(\CZ)}:\Dmod(\CS(\CZ))\to \IndCoh(\CS(\CZ))$$
admits a left adjoint, denoted 
$$\ind_{\CS(\CZ)}:\IndCoh(\CS(\CZ)) \to \Dmod(\CS(\CZ)),$$
and the category $\Dmod(\CS(\CZ))$ identifies with the category of modules
for the monad (called \emph{the infinitesimal monad})
$$\on{Diff}_{\CS(\CZ)}:=\oblv_{\CS(\CZ)}\circ \ind_{\CS(\CZ)}$$ 
acting on $\IndCoh(\CS(\CZ))$. 

\medskip

The monad $\on{Diff}$ should be thought of as tensoring with the
algebra of differential operators.
This is literally the case if $\CS(\CZ)$ were a smooth
scheme. However, one can make sense of this even for a general laft-def prestacks,
see \cite[Chapter 8, Sect. 4]{Vol2}. 

\medskip

The goal of this paper is to describe the monad $\on{Diff}_{\CS(\CZ)}$ acting on 
$\IndCoh(\CS(\CZ))$ in terms that are \emph{local} with respect to $X$ (even though
$\CS(\CZ)$ is itself a global object).  We also establish a number of applications of this description
to geometric representation theory and conformal field theory.

\ssec{Description of the main result}

\sssec{}

Let $\ul{x}$ be a finite collection of points on $X$, and let $\oX$ denote the corresponding punctured curve. 
Let $\overset{\circ}\CS(\CZ)$ denote the prestack of sections of $\CZ\to X_{\dr}$ \emph{over} $\oX_\dr$. 
Let $\on{Hecke}(\CZ)_{\ul{x}}$ denote the fiber product
$$\CS(\CZ)\underset{\overset{\circ}\CS(\CZ)}\times \CS(\CZ).$$
%

Finally, let
$$\widehat{\on{Hecke}}(\CZ)_{\ul{x}}$$ 
denote the formal completion of $\on{Hecke}(\CZ)_{\ul{x}}$ along the unit section
$$\CS(\CZ)\to \on{Hecke}(\CZ)_{\ul{x}}.$$ 
Let
$$\partial_s,\partial_t:\widehat{\on{Hecke}}(\CZ)_{\ul{x}}\to \CS(\CZ)$$
denote the two projections. 

\sssec{}

In \secref{sss:defn sect laft} we will introduce a condition on $\CZ$ 
(called \emph{sectional left-ness}) that guarantees that the prestack
$\widehat{\on{Hecke}}(\CZ)_{\ul{x}}$ is laft-def. (However, this condition
does \emph{not} necessarily imply that the entire $\on{Hecke}(\CZ)_{\ul{x}}$ is laft!).

\medskip

Then we can consider the endofunctor of $\IndCoh(\CS(\CZ))$ given by 
$$\CH^{\on{inf}}(\CZ)_{\ul{x}}:=(\partial_t)^\IndCoh_*\circ \partial_s^!:\IndCoh(\CS(\CZ))\to \IndCoh(\CS(\CZ)),$$
which is easily seen to have the structure of a monad. 

\medskip

Because of the presence of the formal completion, we have a (tautologically defined) map of monads
\begin{equation} \label{e:inf Hecke to dr Intro}
\CH^{\on{inf}}(\CZ)_{\ul{x}}\to \on{Diff}_{\CS(\CZ)}.
\end{equation} 

\sssec{} \label{sss:main intro}
Our main result can be informally stated as saying that the map \eqref{e:inf Hecke to dr Intro} becomes
an isomorphism after ``integrating the left-hand side over the space of all possible $\ul{x}$'s''.

\medskip

Specifically, all choices of $\ul{x}$ are parametrized by the $\emph{Ran space}$, denoted $\on{Ran}(X)$ 
(which we describe in \secref{s:Ran space schemes}).  The $\CH^{\on{inf}}(\CZ)_{\ul{x}}$ assemble to a family of monads $\CH^{\on{inf}}(\CZ)_{\Ran}$ 
parametrized by the Ran space.  Applying the functor of compactly supported de Rham cohomology along 
$\Ran(X)$, we obtain the ``integrated'' monad\footnote{The fact that this procedure produces another
monad is not evident from this description; to obtain the structure of a monad, we use the additional 
structure of $\CH^{\on{inf}}(\CZ)_{\Ran}$ that it is defined over the \emph{unital} Ran space (see 
\secref{sss:intro Ran variants}).} $\Gamma_{c, \Ran}(\CH^{\on{inf}}(\CZ)_{\Ran})$.  Our main result is the following ``local-to-global'' description of $\on{Diff}_{\CS(\CZ)}$:

\begin{thm*}[\thmref{t:main}]
The natural map of monads
$$ \Gamma_{c, \Ran}(\CH^{\on{inf}}(\CZ)_{\Ran}) \to \on{Diff}_{\CS(\CZ)} $$ is an isomorphism. 
\end{thm*}

\sssec{}
As variants of \thmref{t:main}, we give descriptions of the category $\Dmod(\CS(\CZ))$ and its non-linear counterpart, the de Rham prestack $\CS(\CZ)_{\dr}$.

\medskip

The first variant describes the category $\Dmod(\CS(\CZ))$.  Roughly, it states that
a D-module on $\CS(\CZ)$ is given by the data of its underlying ind-coherent sheaf, together with
equivariance with respect to the groupoids $\widehat{\on{Hecke}}(\CZ)_{\ul{x}}$
for all finite collections of points $\ul{x}$.

\begin{thm*}[\thmref{t:main non-unital}]
There is a natural isomorphism
$$ \Dmod(\CS(\CZ)) \simeq \IndCoh(\CZ(\CS))\underset{\IndCoh(\CZ(\CS)\times \Ran(X))}\times 
\IndCoh(\CS(\CZ)\times \Ran(X))^{\widehat{\on{Hecke}}(\CZ)_{\Ran(X)}} ,$$
where the second factor is the category of ind-coherent sheaves on $\CS(\CZ)\times \Ran(X)$
equivariant with respect to the groupoid $\widehat{\on{Hecke}}(\CZ)_{\Ran(X)}$.
\end{thm*}

\sssec{}
The second variant gives an analogous description of the entire prestack $\CS(\CZ)_{\dr}$ as a pushout:

\begin{thm*}[\thmref{t:push-out}]
The natural map
$$\CS(\CZ)\underset{\CS(\CZ)\times \Ran(X)}\sqcup \overset{\circ}\CS(\CZ)\strut^\wedge_{\Ran(X)}\to 
\CZ(\CS)_\dr$$
is an isomorphism, where:

\begin{itemize}

\item $\overset{\circ}\CS(\CZ)\strut^\wedge_{\Ran(X)}$ is the formal completion of 
$\overset{\circ}\CS(\CZ)_{\Ran(X)}$ along the restriction map $$\CS(\CZ)\times \Ran(X)\to \overset{\circ}\CS(\CZ)_{\Ran(X)};$$

\item The pushout is taken in the category of laft-def prestacks.

\end{itemize} 

\end{thm*}

\ssec{Applications}\label{ss:intro applications}

\sssec{}\label{sss:when apply}
As already mentioned, in order for the above results to apply to a given $\CZ$, it needs
to be sectionally laft. We will formulate this condition in \secref{sss:explain fake}.
For now, we just remark that it consists of two conditions, which can be checked in practical situations:

\medskip

The first condition is that the classical prestack underlying $\CS(\CZ)$ should be locally of finite type. The second
condition is of \emph{linear} nature: it is equivalent to demanding the existence of an antecedent of 
the tangent complex of $\CS(\CZ)$, denoted $\Theta(\CZ)$, and which is an object of the
category
$$\IndCoh(\CS(\CZ)\times X_\dr)\simeq \IndCoh(\CS(\CZ))\otimes \Dmod(X),$$
see \secref{ss:Theta}. 

\medskip

We can describe two families of examples that are sectionally laft. 

\medskip

One family consists of $\CZ$ that are themselves laft, see \corref{c:laft => sect laft}. Another family is given
by the \emph{jet construction}. Namely, we start with a laft prestack $\CZ'\to X$ and set
$$\CZ:=\Res^{X}_{X_\dr}(\CZ')$$
to be the Weil restriction of $\CZ'$ along the projection $X\to X_\dr$ (see \secref{sss:Weil restr} for a review
of Weil restriction). Such $\CZ$ will typically \emph{not} be laft, but it will be sectionally laft, see \propref{p:jets sect laft}. 

\sssec{The stack of local systems}
Let
$$\CZ = (\on{pt}/\sG) \times X_\dr,$$
where $\sG$ is an algebraic group.

\medskip

In this case, the resulting prestack $\CS(\CZ)$ is the derived Artin stack $\LocSys_\sG(X)$ classifying 
$\sG$-local systems on $X$.

\medskip

It is not difficult to see that the forgetful map
$$\widehat{\on{Hecke}}(\on{pt}/\sG)_{\Ran(X)}\to \on{Hecke}(\on{pt}/\sG)_{\Ran(X)}$$
is an isomorphism.

\medskip

Thus, \thmref{t:main} (along with its relative variant \thmref{t:rel main}) give a local-to-global
description of the infinitesimal monad of $\LocSys_\sG(X)$ in terms of Hecke modifications
of local systems at points of $x$.

\medskip

Such a description of the (relative) infinitesimal monad for $$\LocSys_{\sP}(X)\to \LocSys_{\sG}(X)$$
(here $\sG$ is a reductive group and $\sP\subset \sG$ is a parabolic subgroup)
is a crucial ingredient in the study of the spectral side of geometric Langlands,
see 
\cite[Sect. 6.5]{Outline}. 

\sssec{The stack of $G$-bundles} \label{sss:BunG intro}

Another important class of examples is the following.  Let
$$\CZ:=\Res^{X}_{X_\dr}((\on{pt}/G)\times X),$$
where $G$ is, as before, an algebraic group.

\medskip

In this case, $\CS(\CZ)$ is $\Bun_G(X)$, the moduli stack of $G$-bundles on $X$. 

\sssec{}
We now specialize to the case that $X$ is a smooth curve.

\medskip

\thmref{t:main} can be used to give a very convenient description of the localization functor from the 
Kazhdan-Lusztig category of Kac-Moody representations to D-modules on $\Bun_G(X)$. We will discuss this in a subsequent work.

\sssec{}
The entire Hecke stack (as opposed to the infinitesimal version we've considered so far) of $\Bun_G(X)$ over the Ran space also defines a family
of groupoids on $\Bun_G(X)$.

\medskip

Gaitsgory \cite{Contr} proved that the corresponding category of equivariant D-modules
$$ \Dmod(\Bun_G(X))^{\Hecke}$$
is equivalent to $\Vect$.  Using \thmref{t:main} and Gaitsgory's theorem, we show the analogous statement for \emph{quasi-coherent sheaves} on $\Bun_G(X)$:

\begin{thm*}[\thmref{t:Hecke equiv}]
There is an equivalence of categories
$$ \QCoh(\Bun_G(X))^{\Hecke} \simeq \Vect$$
given by $k \mapsto \mathcal{O}_{\Bun_G(X)}$.
\end{thm*}

\sssec{}
Let $\Gr_{G,\Ran(X)}$ denote the Ran version of the affine Grassmannian.  Using \thmref{t:Hecke equiv}, we prove:

\begin{thm*}[\thmref{t:uniformization}]
The uniformization map
$$\pi:\Gr_{G,\Ran(X)}\to \Bun_G(X)$$
induces an isomorphism
$$\CHom_{\QCoh(\Bun_G(X))}(\CF,\CE)\to
\CHom_{\QCoh(\Gr_{G,\Ran(X)})}(\pi^*(\CF),\pi^*(\CE)),$$
for $\CF\in \QCoh(\Bun_G(X))$ and $\CE\in \on{Perf}(\Bun_G(X))\subset \QCoh(\Bun_G(X))$. 
\end{thm*}

\sssec{WZW conformal blocks}

Now, let $\CL$ be a factorizable line bundle on $\Gr_{G,\Ran(X)}$ (see \secref{sss:fact line bundles}); in the case that
$G$ is semi-simple and simply connected, this is equivalent to a choice of a level $\kappa \in \mathbb{Z}$.

\medskip

To this data,
we can associate a factorization algebra (in the sense of \cite{BD})
$$\CA_\CL:=q^{\IndCoh}_*(\CL\otimes \omega_{\Gr_{G,\Ran}}).$$
In the case that $G$ is semi-simple and simply connected and $\kappa$ is non-positive, this is the factorization algebra corresponding to the
vertex algebra given by the integrable quotient of the Kac-Moody vertex algebra corresponding to $G$ at
level $-\kappa$.  In other words, this is the factorization algebra of the Wess-Zumino-Witten (WZW) model in conformal field theory.

\medskip
Moreover, to the data of $\CL$, one can associate a line bundle $\CL_{glob}$ on $\Bun_G(X)$.  When $G$
is semi-simple and simply connected, this is the corresponding theta line bundle; i.e., the $\kappa$-th power of the determinant line bundle on $\Bun_G(X)$

\medskip

We use \thmref{t:uniformization} to compute $\on{C}^{\on{Fact}}_\cdot(X,\CA_{\CL,X})$, the chiral homology (a.k.a. factorization homology) of $\CA_{\CL}$:

\begin{thm*}[\thmref{t:WZW'}]
There exists a canonical isomorphism\footnote{The functor $\Gamma_c$ of compactly supported cohomology for quasi-coherent sheaves on $\Bun_G(X)$ is defined in \secref{sss:qcoh gamma-c}}
$$\on{C}^{\on{Fact}}_\cdot(X,\CA_{\CL,X})\to \Gamma_c(\Bun_G(X),\CL_{\on{glob}}\otimes \omega_{\Bun_G}).$$
\end{thm*}

\medskip

Recall that for a vertex operator algebra $V$, the space of conformal blocks of $V$ is given by the \emph{dual} of the zero-th factorization homology of the corresponding factorization algebra.  Dualizing the above, we obtain a derived enhancement of the Verlinde formula for the space of conformal blocks:
$$ \on{C}^{\on{Fact}}_\cdot(X,\CA_{\CL,X})^{\vee} \simeq \Gamma(\Bun_G(X),\CL_{\on{glob}}^{-1}).$$
In the simply-connected, semi-simple case, this recovers the usual Verlinde formula by taking the zeroth cohomology $H^0$.  In fact, a theorem of Teleman \cite{Te} asserts that in this case, the higher cohomology of the line bundle $\CL_{\on{glob}}^{-1}$ vanishes.  As a result, we obtain:

\begin{cor*}[\corref{c:WZW pos pos}]
Let $G$ be semi-simple and simply-connected and $\kappa$ a positive level. Then the chiral homology of the
integrable quotient of the Kac-Moody chiral algebra at level $\kappa$ is concentrated in
cohomological degree $0$. 
\end{cor*}

\ssec{Organization of the paper and overview of contents}

\sssec{}\label{sss:intro Ran variants}

As explained above, the statement of our main theorem involves the Ran space.
However, there are two versions of the Ran space, both of which play a role in this paper, denoted
$$\Ran(X) \text{ and } \Ran^{\on{untl}}(X),$$
respectively.

\medskip

The first of the two is a prestack that assigns to an affine test-scheme $S$ the set (=discrete category)
of finite subsets of the set $\Hom(S,X_\dr)$. The second is a \emph{categorical} prestack: it 
assigns to $S$ the \emph{category} of finite subsets of the set $\Hom(S,X_\dr)$, with the morphisms
given by inclusion. 

\medskip

Prestacks such as $\Ran(X)$ are familiar objects in algebraic geometry, but categorical 
prestacks are not. However, one can define and work with categories of sheaves on them. 
The appearance of the categorical prestack $\Ran^{\on{untl}}(X)$ in this paper is due to the fact that
the category of D-modules on it has several remarkable properties that are \emph{not} shared by $\Ran(X)$.

\medskip

One of these properties is that the functor
$$\Gamma_{c,\Ran^{\on{untl}}}:\Dmod(\Ran^{\on{untl}}(X))\to \Vect,$$
left adjoint to the ``constant functor" $k\mapsto \omega_{\Ran^{\on{untl}}(X)}$,
is \emph{strictly} symmetric monoidal, and not lax symmetric monoidal (which would
be the case for essentially any prestack; in particular, $\Ran(X)$). 

\medskip

Another property of $\Ran^{\on{untl}}(X)$ is that the direct image functor along the
tautological map $X\to \Ran^{\on{untl}}(X)$ is fully faithful and identifies 
$\Dmod(X)$ with the full subcategory of $\Dmod(\Ran^{\on{untl}}(X))$ 
that consists of \emph{linear factorization sheaves}. 

\medskip

Sections \ref{s:Ran} and \ref{s:Ran space schemes} are devoted to a discussion of Ran spaces. 
In \secref{s:Ran}, we consider the Ran spaces attached to a \emph{space} (i.e., a homotopy 
type) $\sfX$, and study its basic properties. In \secref{s:Ran space schemes} we apply this
in the context of algebraic geometry (the spaces $\sfX$ in question arise as $\Hom(S,X_\dr)$,
where $S$ is a test affine scheme). 

\sssec{}

In \secref{s:sect laft}, we introduce the key finiteness condition: we define what it means for
a prestack $\CZ\to X_\dr$ to be sectionally laft. 

\medskip

This condition ensures that the prestack $\CS(\CZ)$ is laft-def, as well as the prestack
$$\overset\circ{\CS}(\CZ)\strut^\wedge_{\Ran(X)},$$
obtained by taking the completion of the prestack $\overset\circ{\CS}(\CZ)_{\Ran(X)}$ of sections
over the punctured $X_\dr$, along the restriction map
$$\CS(\CZ)\times \Ran(X)\to \overset\circ{\CS}(\CZ)_{\Ran(X)}.$$

\medskip

In this section we also introduce the object
$$\Theta(\CZ)\in \IndCoh(\CS(\CZ))\otimes \Dmod(X),$$
which governs the tangent complex of $\CS(\CZ)$ and the relative tangent complex
of $$\CS(\CZ)\to \overset\circ{\CS}(\CZ)\strut^\wedge_{\Ran(X)}.$$

\medskip

Namely, the tangent complex of $\CS(\CZ)$ identifies with the de Rham pushforward of
$\Theta(\CZ)$ along $\CS(\CZ)\times X\to \CS(\CZ)$. For a fixed
$$\ul{x}=\{x_1,...,x_n\}\in \Ran(X),$$
the relative tangent complex of 
\begin{equation} \label{e:identify tan intro}
\CS(\CZ)\to \overset\circ{\CS}(\CZ)\strut^\wedge_{\ul{x}}
\end{equation}
is identified with
$$\underset{k}\oplus\, (\on{Id}\otimes i_{x_k}^!)(\Theta(\CZ))\in \IndCoh(\CS(\CZ)),$$
where $i_k$ is the embedding of the point $x_{i_k}$ into $X$. 

\sssec{}

In \secref{s:main} we state our main result, \thmref{t:main}, along with its relative and parameterized
variants, and discuss various reformulations, mentioned above.

\medskip

\secref{s:proof} is devoted to the proof of \thmref{t:main}. The key steps are as follows. 

\medskip

Recall that we have to show that a certain canonically defined map of monads
is an isomorphism. Using the procedure of \emph{deformation to the normal cone} from
\cite[Chapter 9]{Vol2}, we show that both monads are (non-negatively) filtered and 
our map preserves the filtrations. Hence, it suffices to show that the map is an isomorphism
at the associated graded level.

\medskip

Next we observe, following \cite[Chapter 9]{Vol2}, that the associated graded monads
are both given by tensoring with symmetric algebras of certain canonically defined linear objects. 
Hence, it suffices to show that the induced map between these linear objects is an isomorphism.

\medskip

Finally, the required isomorphism follows from the identification \eqref{e:identify tan intro}
and the theory of linear factorization sheaves developed in \secref{s:Ran space schemes}. 

\sssec{}

In \secref{s:examples} we consider the two families of examples of sectionally laft $\CZ$ mentioned
in \secref{sss:when apply}, and we apply our main result to
$$\CZ:=(\on{pt}/\sG) \times X_\dr \text{ and } \CZ:=\Res^X_{X_\dr}((\on{pt}/G) \times X).$$ 

\medskip

In \secref{s:WZW}, we specialize to the case that $X$ is a smooth projective curve. 
We use the results concerning $\CZ:=\Res^X_{X_\dr}((\on{pt}/G) \times X)$ (so that
$\CS(\CZ)=\Bun_G(X)$) to prove the results about the category $\QCoh(\Bun_G(X))$
and the WZW conformal blocks described above.

\sssec{} \label{sss:explain laft}

In Appendix \ref{s:finiteness} we develop the theory needed to introduce the object
$$\Theta(\CZ)\in \IndCoh(\CS(\CZ))\otimes \Dmod(X),$$
mentioned above.

\medskip 

Namely, for an affine scheme $S$ almost of finite type, consider the category
$$\on{Pro}(\QCoh(S\times X_\dr)^-).$$
%
We introduce a full subcategory
$$\on{Pro}(\QCoh(S\times X_\dr)^-)_{\on{laft}}\subset \on{Pro}(\QCoh(S\times X_\dr)^-).$$

We then prove a key finiteness result, \thmref{t:finiteness product}, which ultimately allows to identify the
opposite category of $\on{Pro}(\QCoh(S\times X_\dr)^-)_{\on{laft}}$ with 
$$\IndCoh(S \times X_\dr) \simeq \IndCoh(S)\otimes \Dmod(X),$$
by means of a functor that we call \emph{Serre-Verdier duality}.

\medskip

It is this operation that will ultimately produce the object $\Theta(\CZ)$ from the cotangent complex 
$T^*(\CZ)$ of $\CZ$, using one additional construction, introduced in \secref{s:fake}.

\sssec{} \label{sss:explain fake}

In Appendix \ref{s:fake} we introduce the final ingredient needed in order to define the object $\Theta(\CZ)$.

\medskip

Let $\CY$ be a prestack that admits deformation theory. In particular, for an affine test-scheme $S$
with a map $y: S\to \CY$,
we have a well-defined object
$$T^*_y(\CY)\in \on{Pro}(\QCoh(S)^-).$$

These objects assemble into the cotangent complex of $\CY$, which is an object
$$T^*(\CY)\in \underset{(S,y)}{\on{lim}}\, \on{Pro}(\QCoh(S)^-)=: \on{Pro}(\QCoh(\CY)^-)^{\on{fake}}.$$

The above category $\on{Pro}(\QCoh(\CY)^-)^{\on{fake}}$ admits a natural evaluation functor
$$\oblv^{\on{fake}}:\on{Pro}(\QCoh(\CY)^-)^{\on{fake}}\to \on{Pro}(\QCoh(\CY)^-),$$
but this functor is in general not an equivalence.

\medskip

In \secref{s:fake} we study the category $\on{Pro}(\QCoh(\CY)^-)^{\on{fake}}$, its functoriality with
respect to maps of prestacks, and its interaction with deformation theory. 

\medskip

In particular, for a map $f:\CY\to \CZ$ we introduce a functor
$$f^{\#}:\on{Pro}(\QCoh(\CZ)^-)^{\on{fake}}\to \on{Pro}(\QCoh(\CY)^-)^{\on{fake}}.$$

\medskip

This functor is used in the construction of $\Theta(\CZ)$ as follows.

\medskip

Let $\CZ$ be as in the main body of the paper. The Serre-Verdier predual of $\Theta(\CZ)$ is comprised of 
a compatible collection of objects of 
$$\on{Pro}(\QCoh(S \times X_\dr)^-)$$
for $s:S\to \CS(\CZ)$, where $S\in \affSch_{\on{aft}}$. The sought-for objects are given by
\begin{equation} \label{e:pre Theta intro}
\oblv^{\on{fake}}\circ (\on{ev}_s)^{\#}(T^*(\CZ)),
\end{equation}
where $\on{ev}_s$ denotes the map 
$$S\times X_\dr\to \CZ$$
corresponding to $s$. 

\medskip 

Now, the condition that $\CZ$ is sectionally laft asserts that the objects \eqref{e:pre Theta intro}
actually belong to $\on{Pro}(\QCoh(S\times X_\dr)^-)_{\on{laft}}$ (see \secref{sss:explain laft} above), so that the operation
of Serre-Verdier duality is well-defined on them. 

\sssec{}

Appendix \ref{s:categ prestacks} is a summary of the theory of \emph{categorical prestacks}, which are
(accessible) functors from the category opposite to that of affine schemes to the $\infty$-category
of $\infty$-categories. 

\medskip

As mentioned earlier, the need for categorical prestacks arises in this paper because we want
to use the unital Ran space, $\Ran^{\on{untl}}(X)$. All other categorical prestacks
that will appear in this paper will be ``modeled on $\Ran^{\on{untl}}(X)$" in that they will be either
Cartesian or coCartesian fibrations in groupids over it. 

\sssec{}

Appendix \ref{s:BBW}, written by Dennis Gaitsgory, gives a proof of the affine version of the Borel-Weil-Bott theorem
that we use in this paper.  To the best of our knowledge, the original proof is due to Kumar and uses the
technique of Frobenius splitting.

\medskip

The proof Gaitsgory gives is representation-theoretic, and instead of Frobenius splitting it uses
the Kashiwara-Tanisaki theory of localization at the positive level.

\ssec{Background, conventions and notation}

\sssec{}

This paper heavily uses the language of derived algebraic geometry, and in that relies 
on the machinery developed in the book \cite{Vol1,Vol2}. The notation in this paper 
follows closely that of {\it loc. cit.}. 

\medskip

Below follows a glossary of terms and references to where they are introduced.

\sssec{Higher category theory}

By a ``category" we will always mean an $\infty$-category, and by a ``groupoid" 
an $\infty$-groupoid (a.k.a. space). Categories form an $\infty$-category denoted
$\on{Cat}$, and groupoids form an $\infty$-category denoted $\Spc$. See 
\cite[Chapter 1, Sect. 1]{Vol1}. 

\medskip

Various properties of functors between categories (such as what it means to be 
a (co)-Cartesian fibration, etc) are reviewed in \cite[Chapter 1, Sect. 1.3]{Vol1}. 

\sssec{Higher algebra}

Of fundamental interest are $k$-linear categories (a.k.a. DG categories) 
that arise from algebraic geometry. We follow the conventions of 
\cite[Chapter 1, Sect. 10]{Vol1}. 

\medskip

In particular, we have the notions of
(symmetric) monoidal DG category, an algebra in a monoidal DG 
category, a module for an algebra within a module category. 
These notions are reviewed in \cite[Chapter 1, Sect. 3]{Vol1}. 

\sssec{(Derived) algebraic geometry}

Let $\affSch$ denote the category of affine (derived) schemes over $k$,
see \cite[Chapter 2, Sect. 1.1.3]{Vol1}; we will just call them schemes. 

\medskip

Let $\on{PreStk}$ denote the category of prestacks over $k$, i.e., the
category of accessible functors
$$(\affSch)^{\on{op}}\to \Spc.$$

\medskip

A crucial role in this paper is played by the subcategory of affine schemes
\emph{almost of finite type} $\affSch_{\on{aft}}\subset \affSch$, 
see \cite[Chapter 2, Sect. 1.7.1]{Vol1}. The corresponding notion for
prestacks is being  \emph{laft:=locally almost of finite type},
see \cite[Chapter 2, Sect. 1.7.2]{Vol1}. It is on these algebro-geometric 
objects that we have a well-defined theory of ind-coherent sheaves
and D-modules.  

\medskip

The associated notions (such as convergence, truncations) can also be found 
in \cite[Chapter 2, Sects. 1.2-1.4]{Vol1}. 

\sssec{Ind-coherent sheaves} 

The category of ind-coherent sheaves is initially defined on laft schemes
(see  \cite[Chapter 4]{Vol1}), and is then extended to laft prestacks 
(see  \cite[Chapter 5, Sect. 3.4]{Vol1})). 

\medskip

The assignment $\CX\mapsto \IndCoh(\CX)$ is by construction functorial
with respect to !-pullbacks of prestacks. However, there is one more functoriality
that plays a crucial role in this paper.

\medskip

In \cite[Chapter 2, Sect. 3]{Vol2}, the notion of \emph{inf-scheme} is introduced
(an inf-scheme is a laft prestack that \emph{admits deformation theory} and
whose underlying reduced prestack is a scheme). A morphism between laft prestacks
is \emph{inf-schematic} if its base change by an object of $\affSch_{\on{aft}}$ yields
an inf-scheme.

\medskip

In \cite[Chapter 3, Sect. 3.4]{Vol2}, it is explained how to construct a direct image
functor
$$f^\IndCoh_*:\IndCoh(\CX_1)\to \IndCoh(\CX_2)$$
when $f:\CX_1\to \CX_2$ is an inf-schematic map.

\medskip

The functors $f^\IndCoh_*$ satisfy base-change against !-pullbacks. Moreover, if $f$
is proper (at the reduced level), the functor $f^\IndCoh_*$ is left adjoint to $f^!$.

\sssec{D-modules} 

For a prestack $\CX$ we set $\CX_\dr$ to be the prestack defined by
$\CX_\dr(S):=\CX({}^{\on{red}}\!S)$. One easily shows that if $\CX$ is laft, then 
so is $\CX_\dr$.  

\medskip

We define $\Dmod(\CX):=\IndCoh(\CX_\dr)$. Pullback with respect to the natural 
projection $p_{\CX,\dr}:\CX\to \CX_\dr$ defines a forgetful functor
$$\oblv_\CX:\Dmod(\CX)\to \IndCoh(\CX).$$

If $\CX$ admits deformation theory, then the morphism $p_{\CX,\dr}$ is inf-schematic,
and hence the functor $(p_{\CX,\dr})^\IndCoh_*$ is defined, and provides a left adjoint
to $\oblv_\CX$. This is the functor of induction
$$\ind_\CX:\IndCoh(\CX)\to \Dmod(\CX).$$

This theory is developed in  \cite[Chapter 4, Sects. 1 and 4]{Vol2}. 

\sssec{}

Finally, we make essential use of deformation theory. We refer the reader to
\cite[Chapter 1, Sect. 7]{Vol2} for what we mean for a prestack to admit
deformation theory, to \cite[Chapter 1, Sect. 4]{Vol2} to what we mean by a cotangent complex, 
and to \cite[Chapter 1, Sect. 4.4]{Vol2} to what we mean by the \emph{tangent} complex
(the latter is defined when the prestack in question is laft). 

\ssec{Acknowledgements}
Many of the ideas in this paper arose from discussions with Dennis Gaitsgory.
I would like to express my utmost gratitude to him for generously sharing his ideas, for our continuing collaboration, and for contributing Appendix \ref{s:BBW}.

\section{Ran categories of spaces} \label{s:Ran}

In this section we will introduce the general construction of the Ran category and its associated groupoid,
called the \emph{unital Ran space} and the (usual) \emph{Ran space} and denoted $\on{Ran}^{\on{untl}} (\sfX)$
and $\on{Ran}(\sfX)$, respectively, associated to a space $\sfX$. 

\medskip

This is will be done in a greater generality than is necessary for the main applications in this paper: in our case,
the space $\sfX$ will be discrete, i.e., a plain set. 

\ssec{The Ran category a.k.a. unital Ran space}

\sssec{}

Let $\sfX$ be a space.  Define the \emph{unital Ran space} to be the category
$$ \on{Ran}^{\on{untl}} (\sfX):= \{\text{poset of finite subsets of } \pi_0(\sfX) \} ,$$
where the morphisms are given by inclusion.

\medskip

Similarly, define
$$ \on{Ran}(\sfX) :=  (\on{Ran}^{\on{untl}} (\sfX))^{\on{grpd}} - \{\emptyset\}, $$
the set of non-empty finite subsets of $\pi_0(\sfX)$.

\sssec{}

Since the category $ \on{Ran}(\sfX)^{\on{untl}}$ admits finite coproducts, we obtain: 

\begin{lem}\label{l:Ran sifted}
The category $\on{Ran}^{\on{untl}}(\sfX)$ is sifted.
\end{lem}

\sssec{}

Let $\on{Fin}$ denote the category of finite sets, and let
$$ \on{Fin}^{s}  \subset \on{Fin} $$
denote the subcategory of non-empty finite sets and surjective maps.  For a space $X$, we have the functor
$$ \sfX^{(-)}: \on{Fin}^{\on{op}}  \to \on{Spc} $$
given by $I \mapsto \sfX^I$.
Let $\on{Fin}_{/\sfX} \to \on{Fin}$ denote the Cartesian fibration in groupoids corresponding to this functor.  As the notation suggests, 
it is also the full subcategory of the slice category $\on{Spc}_{/\sfX}$, consisting of finite sets mapping to $\sfX$.
Similarly, let
$$ \on{Fin}^s_{/\sfX} := \on{Fin}_{/\sfX} \underset{\on{Fin}}{\times} \on{Fin}^s $$
denote the Cartesian fibration in groupoids corresponding to the restrction of this functor to $\on{Fin}^s$.

%
%
%
%
%

\sssec{}

We now express the unital Ran space $\on{Ran}^{\on{untl}} (\sfX)$ in terms of $\on{Fin}_{/\sfX}$. 

\begin{prop}\label{p:unital ran pres}
For a space $X$, the unital Ran space is equivalent to the localization of $\on{Fin}_{/\sfX}$ by morphims over surjective maps
of finite sets, i.e.,
$$  \on{Ran}^{\on{untl}} (\sfX) \simeq \on{Fin}_{/\sfX}[(\on{Fin}^s_{/\sfX})^{-1}] .$$
\end{prop}
\begin{proof}
Consider the natural functor
$$\on{im}:\on{Fin}_{/\sfX} \to \on{Ran}^{\on{untl}} (\sfX),$$
which maps a finite set to its image in $\pi_0(\sfX)$.  It is clearly a Cartesian fibration.  For a finite subset $I \subset \pi_0(\sfX)$, 
the preimage $\on{im}^{-1}(I)$ is the full subcategory
$$ \on{im}^{-1}(I) \subset \on{Fin}_{/\sfX} $$
consisting of maps $J \to X$ such that the image of $J$ in $\pi_0(\sfX)$ is $I$.  This category admits finite (non-empty) coproducts and is therefore contractible.  
It follows that $\on{im}$ is a localization.

\medskip

It remains to show that the subcategory $\on{Fin}^s_{/\sfX}$ generates all morphisms that map to isomorphisms under $\on{im}$.  A choice of basepoint in each connected component of $\sfX$ gives a map $s: \pi_0(\sfX) \to \sfX$,
which is an isomorphism on $\pi_0$.  The desired assertion follows from the fact that any map $I \to \sfX$
in $\on{Fin}_{/\sfX}$ factors (non-canonically) as $I \surj \on{im}(I) \overset{s}{\to} \sfX$.

\end{proof}

\begin{cor}\label{c:unital Ran sifted}
The functor $\on{Ran}^{\on{untl}} : \on{Spc} \to \on{Cat}$ commutes with sifted colimits.  
In particular, for any $\sfX\in \on{Spc}$, $$ \on{Ran}^{\on{untl}} (\sfX) \simeq \underset{I \in \on{Fin}_{/\sfX}}{\on{colim}} \on{Ran}^{\on{untl}} (I).$$
\end{cor}
\begin{proof}
Note that the functor
$$ \on{Spc} \to \on{Fun}(\on{Fin}^{\on{op}} , \on{Spc}) $$
given by $\sfX \mapsto \sfX^{(-)}$ commutes with sifted colimits.  Moreover, for any category $\bC$, the Grothendieck construction
$$ \on{Fun}(\bC^{\on{op}} , \on{Cat}) \to \on{Cat}_{/\bC} $$
commutes with colimits.  Thus, the functors
$$ X \mapsto \on{Fin}_{/\sfX} \mbox{ and } X\mapsto \on{Fin}^s_{/\sfX} $$
commute with sifted colimits.  Since localization also commutes with all colimits, we obtain that
$$ X \mapsto \on{Fin}_{/\sfX}[(\on{Fin}^s_{/\sfX})^{-1}] \simeq \on{Ran}^{\on{untl}} (\sfX) $$
commutes with sifted colimits.
\end{proof}

%

\sssec{}\label{sss:maps from Ran}
Let $\bC$ be a category.  By the definition of $\on{Fin}_{/\sfX}$ as a Grothendieck construction, a functor
$\Phi:\on{Fin}_{/\sfX}\to \bC$ is given by a natural transformation between the following two functors
$\on{Fin}^{\on{op}}\to \on{Cat}$:
$$(I\mapsto \sfX^I) \Rightarrow (I\mapsto \bC).$$

In other words, a functor $\Phi: \on{Fin}_{/\sfX} \to \bC$ is given by a collection of
functors
$$ \Phi_I: \sfX^I \to \bC $$
for $I \in \on{Fin}$, natural transformations
$$ \nu_f: \Phi_I \circ \Delta_f \to \Phi_J$$
for every map $f: I \to J$, and higher coherence data for compositions of maps in $\on{Fin}$, where
$$ \Delta_f: \sfX^{J} \to \sfX^I$$
is the diagonal map given by $f$.

\medskip

By \propref{p:unital ran pres}, a functor $\Phi$ as above factors through a functor 
$$\on{Ran}^{\on{untl}} (\sfX) \to \bC$$
if and only if the natural transformation $\nu_f$ is a natural isomorphism whenever $f:I\to J$ is surjective.

\ssec{The usual Ran space}

\sssec{}

We will now establish analogous facts about the usual (non-unital) Ran space.  
To do so, we will need a version of the contractibility of the Ran space.  Recall the following fact (\cite[Lemma 3.4.1]{BD}):

\begin{lem}[Beilinson-Drinfeld]\label{l:contractibility}
Let $Y$ be a connected space together with a map
$$ a: Y\times Y \to Y$$
such that $a$ is symmetric up to homotopy and the composite
$$ Y \overset{\Delta}{\to} Y\times Y \overset{a}{\to} Y $$
is homotopic to the identity.  Then $Y$ is contractible.
\end{lem}

From this we obtain:

\begin{prop}\label{p:contr fin surj}
Let $\sfX$ be a connected space.  Then the category $\on{Fin}^s_{/\sfX}$ is contractible.
\end{prop}
\begin{proof}
Let $\{\star\} \to \sfX$ be a basepoint of $\sfX$.  Since $\sfX$ is connected, the space of maps $I\to X$ is connected.  Therefore,
 every map $I \to \sfX$ admits a factorization as $I \surj \{\star\} \to \sfX$.  It follows that the classifying space of $\on{Fin}^s_{/\sfX}$ is connected.

Now, the coproduct on $\on{Fin}_{/\sfX}$ induces a symmetric functor
$$ \on{Fin}^s_{/\sfX} \times \on{Fin}^s_{/\sfX} \to \on{Fin}^s_{/\sfX}, $$
which in turn gives a symmetric map on classifying spaces.  Moreover, the fold map $I \sqcup I \surj I$ gives a natural transformation between the composite
$$ \on{Fin}^s_{/\sfX} \overset{\Delta}{\to} \on{Fin}^s_{/\sfX} \times \on{Fin}^s_{/\sfX} \to \on{Fin}^s_{/\sfX} $$
and the identity.  Thus, the desired result follows by \lemref{l:contractibility}.
\end{proof}

\begin{prop}\label{p:ran pres}
Let $\sfX$ be a space.  The Ran space $\on{Ran}(\sfX)$ is isomorphic to the classifying space
$$ \on{Ran}(\sfX) \simeq |\on{Fin}^s_{/\sfX}|. $$
In other words, $Ran(\sfX)$ is isomorphic to the colimit
\begin{equation} \label{e:Ran as colim}
\on{Ran}(\sfX) \simeq \underset{I \in (\on{Fin}^s)^{\on{op}} }{\on{colim}}\  \sfX^I.
\end{equation}
\end{prop}

\begin{rem}
Note that \propref{p:ran pres} says something a bit surprising: since $\sfX$ was not assumed discrete,
the terms on the colimit \eqref{e:Ran as colim} are not discrete either. However, the colimit is discrete, and moreover
it only depends on $\pi_0(\sfX)$.
\end{rem}

\begin{proof}[Proof of \propref{p:ran pres}]
Consider the functor
$$ \on{im}: \on{Fin}^s_{/\sfX} \to \on{Ran}(\sfX),$$
which maps a finite set to its image in $\pi_0(\sfX)$.  We need to show that the fibers are contractible categories.  For $I \subset \pi_0(\sfX)$,
the preimage $\on{im}^{-1}(I)$ is isomorphic to the product
$$ \on{im}^{-1}(I) \simeq \prod_{i\in I} \on{im}^{-1}(\{i\}) .$$
Moreover, for each $i \in I$, the category $\on{im}^{-1}(\{i\}) \simeq \on{Fin}^s_{/\sfX_i}$, where $\sfX_i$ is the component of $\sfX$ corresponding to 
$i\in I \subset \pi_0(\sfX)$.  Thus, by \propref{p:contr fin surj}, we obtain that each $\on{im}^{-1}(\{i\})$ is contractible and therefore so is $\on{im}^{-1}(I)$.
\end{proof}

By the same argument as \corref{c:unital Ran sifted}, we obtain:

\begin{cor}\label{c:Ran sifted}
The functor $\on{Ran}: \on{Spc} \to \on{Spc}$ commutes with sifted colimits.  
In particular, for any $X\in \on{Spc}$, $$ \on{Ran}(\sfX) \simeq \underset{I \in \on{Fin}_{/\sfX}}{\on{colim}} \on{Ran}(I).$$
\end{cor}

\ssec{The pointed version}

\sssec{}

In what follows, it will be useful to also consider the pointed Ran category.  For $\sfX\in \on{Spc}$, we have
the functor
$$ \on{Ran}^{\on{untl}} (\sfX) \to \on{Spc} $$
given by $(I \subset \pi_0(\sfX)) \mapsto \sfX_I \subset \sfX$, where $\sfX_I$ is the union of connected components corresponding to $I$.

\medskip

The pointed Ran category $\on{Ran}^{\on{untl}}_{\star}(\sfX)$ is the total category of the coCartesian fibration in groupoids corresponding to this functor.

\sssec{}

We have the natural functor $\on{Ran}^{\on{untl}}_{\star}(\sfX) \to \sfX$ and the fiber over $x\in \sfX$ is the full subcategory
$$ \on{Ran}_x^{\on{untl}} (\sfX) \subset \on{Ran}^{\on{untl}} (\sfX) $$
consisting of subsets $I \in \pi_0(\sfX)$ containing the image of $x$ in $\pi_0(\sfX)$.

\medskip

Set 
$$\on{Ran}_{\star}(\sfX):=(\on{Ran}^{\on{untl}}_{\star}(\sfX))^{\on{grpd}}.$$

\sssec{}
Let $\on{Fin}_{\star}$ and $\on{Fin}_{\star}^s$ denote the categories of pointed finite sets and pointed finite sets and surjections, respectively.  
By the same argument as \propref{p:unital ran pres} and \propref{p:ran pres}, we have:

\begin{prop}\label{p:pointed Ran pres}
Let $X$ be a space.  We have
$$ \on{Ran}_\star(\sfX) \simeq (\on{Fin}_{\star})_{/\sfX}[((\on{Fin}_\star^{s})_{/\sfX})^{-1}] , \mbox{ and} $$
$$ \on{Ran}_\star(\sfX) \simeq \underset{I \in (\on{Fin}_{\star}^s)^{\on{op}} }{\on{colim}}\ \sfX^I. $$
\end{prop}

\section{Ran (categorical) prestacks} \label{s:Ran space schemes}

In this section we will introduce the (several versions of the) Ran space, associated with
a separated scheme. 

\ssec{The Ran prestacks}

\sssec{}

Let $X$ be a (separated) scheme. Define the \emph{Ran space of} $X$ to be the prestack given by 
$$\on{Ran}(X)(S) := \on{Ran}(\Maps(S,X_\dr)),$$
where 
\begin{equation} \label{e:dR}
X_{\dr}(S):=\Maps({}^{\on{red}}\!S,X)
\end{equation} 
is the deRham prestack of $X$.

\medskip

Define the \emph{unital Ran space of} $X$ to be the \emph{categorical} prestack given by 
$$\on{Ran}^{\on{untl}} (X)(S) := \on{Ran}^{\on{untl}}(\Maps(S,X_\dr)).$$

\sssec{}

Note that by definition, $\on{Ran}(X)$ and $\on{Ran}^{\on{untl}}(X)$ depend only on the de Rham
prestack $X_{\dr}$ of $X$.

\medskip

Tautologically, the maps
\begin{equation} \label{e:Ran to dR}
\on{Ran}(X)\to \on{Ran}(X)_\dr \text{ and } \on{Ran}^{\on{untl}}(X)\to \on{Ran}^{\on{untl}}(X)_\dr
\end{equation} 
are isomorphisms. 

\medskip

\begin{lem} 
Assume that $X$ is locally of finite type. Then $\on{Ran}^{\on{untl}} (X)$ and $\on{Ran}(X)$ are laft.
\end{lem} 

\begin{proof}

Convergence is automatic. We only need to check that the functors in question commute with
filtered colimits at the classical level. However, this is guaranteed by 
Corollaries \ref{c:unital Ran sifted} and \ref{c:Ran sifted}.

\end{proof} 

\sssec{}

Thus, since $\on{Ran}^{\on{untl}} (X)$ and $\on{Ran}(X)$ are laft, we have well-defined categories
$$\IndCoh(\on{Ran}(X)) \text{ and } \Dmod(\on{Ran}(X)),$$
and similarly for $\on{Ran}^{\on{untl}}(X)$.

\medskip

Note however, that the forget functors
$$\Dmod(\on{Ran}(X))\to \IndCoh(\on{Ran}(X)) \text{ and }
\Dmod(\on{Ran}^{\on{untl}}(X))\to \IndCoh(\on{Ran}^{\on{untl}}(X))$$
are equivalences, since the maps \eqref{e:Ran to dR} are isomorphisms. 

\ssec{The universal punctured space} \label{ss:punctured fibration}

In this subsection we will make a digression and define a level-wise Cartesian fibration in groupoids
$$\overset{\circ}{\mathfrak{X}}_{\on{Ran}^{\on{untl}} } \to \on{Ran}^{\on{untl}} (X)$$
that will be extensively used later. 

\sssec{} \label{sss:punctured curve}

For each $I \in \on{Fin}$, consider the scheme $X\times X^I$.  Let
$$ x: X\times X^I \to X \mbox{ and } y_i: X\times X^I \to X , i\in I$$
denote the projections to the first component, and the component corresponding to $i\in I$.  Let
$$ \overset{\circ}{\mathfrak{X}}_I \subset X \times X^I_\dr $$
denote the open subfunctor given by the complement of the closed subschemes given by $\{x = y_i\}, i\in I$.

\sssec{}

The assignment $I \mapsto \overset{\circ}{\mathfrak{X}}_I$ defines a functor
$$ \on{Fin}^{\on{op}}  \to \on{PreStk} .$$
Let $\overset{\circ}{\mathfrak{X}}_{\on{Fin}}$ denote the categorical prestack that assigns to $S$ the total 
category of the Cartesian fibration in groupoids given by the composite
$$ \on{Fin}^{\on{op}}  \to \on{PreStk} \overset{\on{Maps}(S,-)}{\to} \on{Spc} .$$

\sssec{}\label{sss:fin over X}

Let $\on{Fin}_{/X_\dr}$ denote the categorical prestack given by $S \mapsto \on{Fin}_{/X_\dr(S)}$. 

\medskip

By construction, we have a map of categorical prestacks
$$ \overset{\circ}{\mathfrak{X}}_{\on{Fin}} \to \on{Fin}_{/X_\dr}, $$
which is a level-wise Cartesian fibration in groupoids.  

\medskip

Note that for a surjective map $I\surj J$, we have that the map
$$ \overset{\circ}{\mathfrak{X}}_I \underset{X^I_\dr}{\times} X^J_\dr \to \overset{\circ}{\mathfrak{X}}_J $$
is an isomorphism.  Hence, it follows from \propref{p:unital ran pres} there exists a canonically 
defined level-wise Cartesian fibration in groupoids
\begin{equation}
\overset{\circ}{\mathfrak{X}}_{\on{Ran}^{\on{untl}} } \to \on{Ran}^{\on{untl}} (X)
\end{equation}
so that
$$ \overset{\circ}{\mathfrak{X}}_{\on{Ran}^{\on{untl}} } \underset{\on{Ran}^{\on{untl}} (X)}{\times} \on{Fin}_{/X_\dr} \simeq \overset{\circ}{\mathfrak{X}}_{\on{Fin}} .$$

\sssec{}

Note that by construction, the fiber of $\overset{\circ}{\mathfrak{X}}_{\on{Ran}^{\on{untl}} }$ over a $k$-point of $\Ran(X)$ 
given by a finite subset of $X(k)$ consists 
of the complement of that subset.  For this reason, we will refer to $\overset{\circ}{\mathfrak{X}}_{\on{Ran}^{\on{untl}} }$ as the \emph{universal punctured space of $X$}.

\medskip

By construction, we have a map
\begin{equation}
\overset{\circ}{\mathfrak{X}}_{\on{Ran}^{\on{untl}} } \to X \times \on{Ran}^{\on{untl}} (X)
\end{equation}
over $\on{Ran}^{\on{untl}} (X)$, which is a level-wise Cartesian fibration in groupoids.

\ssec{Geometric properties of the Ran space}

\sssec{}

We retain the assumptions on $X$ and consider the prestack $\Ran(X)$. 
We will establish the following proposition, which says that many natural maps
associated with $\Ran(X)$ are \emph{pseudo-proper} (see \secref{ss:pseudo-proper} for
a review of the notion of pseudo-properness):

\begin{prop}  \label{p:Ran diag ps-proper} \hfill
\begin{enumerate}[label={(\alph*)}]
\item
For any finite set $I$, the natural map $X^I_\dr \to \on{Ran}(X)$ is pseudo-proper.

\item
The diagonal map $\on{Ran}(X) \to \on{Ran}(X) \times \on{Ran}(X)$ is pseudo-proper.

\item
The union map $ \cup: \on{Ran}(X) \times \on{Ran}(X) \to \on{Ran}(X)$ is pseudo-proper.
\end{enumerate}
\end{prop}

\sssec{Proof of \propref{p:Ran diag ps-proper}(a)}

By \propref{p:ran pres},
$$ \on{Ran}(X) = \underset{I \in (\on{Fin}^s)^{op}}{\on{colim}} X^I_{\dr} .$$
The result now follows by \propref{p:ps-pr colimit}.

\sssec{Proof of \propref{p:Ran diag ps-proper}(b)}

Since pseudo-proper maps are closed under colimits, to show that $\on{Ran}(X) \to \on{Ran}(X) \times \on{Ran}(X)$ is pseudo-proper, it suffices to show that for each $I \in \on{Fin}$, the map
$$ X^I_\dr \to \on{Ran}(X) \times \on{Ran}(X) $$
is pseudo-proper.  This map factors as
$$ X^I_\dr \to X^I_\dr \times X^I_\dr \to \on{Ran}(X) \times \on{Ran}(X),$$
where the first map is proper and the second is pseudo-proper by part (a).  Thus, the composite is pseudo-proper by \lemref{l:pseudo-proper ops}(2).

\sssec{Proof of \propref{p:Ran diag ps-proper}(c)}
For $I, J \in \on{Fin}$, the composite
$$ X^I_\dr \times X^J_\dr \to \on{Ran}(X) \times \on{Ran}(X) \overset{\cup}{\to} \on{Ran}(X) $$
is given by
$$ X^I_\dr \times X^J_\dr \to X^{I\sqcup J}_\dr \to \on{Ran}(X),$$
which is pseudo-proper since the first map is an isomorphism and the second is pseudo-proper.  It follows that $\cup: \on{Ran}(X)\times \on{Ran}(X) \to \on{Ran}(X)$ is pseudo-proper.

\ssec{Sheaves on Ran prestacks}

In this section, we study the basic properties of D-modules on the prestacks $\on{Ran}^{\on{untl}} (X)$ and $\on{Ran}(X)$.
As before, $X$ is a (separated) scheme almost of finite type.

\sssec{}\label{sss:dmod on Ran as finite prods}
By \secref{sss:maps from Ran}, an object $M \in \Dmod( \on{Ran}^{\on{untl}}(X))$ is a collection of D-modules
$$ M_I \in \Dmod(X^I) $$
for each $I\in \on{Fin}$, together with maps
$$ \nu(M)_f: \Delta_f^!(M_I) \to M_J \in \Dmod(X^J)$$
for every $f: I \to J$ in $\on{Fin}$, which are isomorphisms when $f$ is surjective, together with higher coherence
for composition of maps in $\on{Fin}$.

\medskip

By contrast, an object $M \in \Dmod( \on{Ran}(X))$ is a similar collection of D-modules $M_I$, but the maps
$\nu(M)_f$ are defined \emph{only} when $f$ is surjective (and are required to be isomorphisms). 

\sssec{}

Since the prestack $\on{Ran}(X)$ is a colimit of inf-schemes along proper maps, we have:

\begin{lem} \label{l:ran dualizable}
The category $\Dmod(\on{Ran}(X))$ is compactly generated and in particular dualizable.
\end{lem}

Our next goal will be to show that the category $\Dmod(\on{Ran}^{\on{untl}} (X))$ is also dualizable.

\sssec{}\label{sss:ps-pr Ran}
Let $\on{Ran}_\emptyset(X):= (\on{Ran}^{\on{untl}} (X))^{\on{grpd}} \simeq \on{Ran}(X) \sqcup \{\emptyset\}$ be the space of all finite subsets of $X$.  
We have the natural inclusion
$$ \iota: \on{Ran}_{\emptyset}(X) \hookrightarrow \on{Ran}^{\on{untl}} (X).$$
Clearly, the functor
$$ \iota^!: \Dmod(\on{Ran}^{\on{untl}} (X)) \to \Dmod(\on{Ran}_{\emptyset}(X)) $$
is conservative.  We will show that $\iota^!$ admits a left adjoint, following \cite[Sect. 4.3]{AB}.

\medskip
Consider the categorical prestack $\on{Ran}^{\to}(X)$ defined as the fiber product
$$ \xymatrix{
\on{Ran}^{\to}(X) \ar[r]\ar[d]_{\phi} & \on{Ran}_{\emptyset}(X) \times \on{Ran}^{\on{untl}} (X) \ar[d]^{\cup \times p_2} \\
\on{Ran}^{\on{untl}} (X) \ar[r]^-{\Delta} & \on{Ran}^{\on{untl}} (X) \times \on{Ran}^{\on{untl}} (X)
}$$
Explicitly, $\on{Ran}^{\to}(X)(S)$ is the category whose objects are pairs of nested subsets $I \subset J \subset \Hom(S,X_\dr)$,
and whose morphisms
$$(I \subset J \subset \Hom(S,X_\dr)) \to (I' \subset J' \subset \Hom(S,X_\dr))$$ are given by inclusions $J \subset J'$ as subsets of $\Hom(S,X_\dr)$, 
such that the image of $I$ is $I'$.

\medskip

The map $$\phi: \on{Ran}^{\to}(X) \to \on{Ran}^{\on{untl}} (X), \quad \phi(I \subset J)=J$$
is a level-wise coCartesian fibration in groupoids. Moreover, we see using \propref{p:Ran diag ps-proper} and \lemref{l:pseudo-proper ops} that it is pseudo-proper. 
In particular, the functor
$$ \phi^!: \Dmod(\on{Ran}^{\on{untl}} (X)) \to \Dmod(\on{Ran}^{\to}(X)) $$
admits a left adjoint $\phi_!$, by \lemref{l:pseudo-proper pushforward}. 

\medskip

Let $\xi$ denote the map
$$\on{Ran}^{\to}(X) \to \on{Ran}_{\emptyset}(X), \quad \xi(I\subset J)=I.$$

We claim:

\begin{prop} \label{p:pushfrwrd from non-untl}
The restriction map $\iota^!: \on{Ran}^{\on{untl}} (X) \to \on{Ran}_{\emptyset}(X)$ admits a left adjoint given by
$$ \iota_! := \phi_! \circ \xi^!: \Dmod(\on{Ran}_{\emptyset}(X)) \to \Dmod(\on{Ran}^{\on{untl}} (X)) .$$
\end{prop}

\begin{proof} 

The inclusion map
$$\on{Ran}_{\emptyset}(X) \overset{\iota}\to \on{Ran}^{\on{untl}} (X) $$
factors as
$$ \on{Ran}_{\emptyset}(X) \overset{\nu}{\to} \on{Ran}^{\to}(X) \overset{\phi}{\to} \on{Ran}^{\on{untl}} (X) ,$$
where $\nu: \on{Ran}_{\emptyset}(X) \to \on{Ran}^{\to}(X)$ is the map given by $I \mapsto (I \subset I)$.

\medskip

Hence, in order to prove the proposition, it remains to show that the functor 
$$\xi^!:\Dmod(\on{Ran}_{\emptyset}(X))\to \on{Ran}^{\to}(X)$$
is the left adjoint of the functor 
$$\nu^!: \Dmod(\on{Ran}^{\to}(X)) \to \Dmod(\on{Ran}_{\emptyset}(X)).$$

However, this follows from the fact that the maps $(\nu,\xi)$ form an adjoint pair
as maps between categorical prestacks. 

\end{proof}

\begin{cor}\label{c:unital Ran dualizable}
The category $\Dmod(\on{Ran}^{\on{untl}} (X))$ is compactly generated, and in particular dualizable.
\end{cor}

\begin{proof}
Follows from the fact that $\Dmod(\on{Ran}_{\emptyset}(X))$ is compactly generated, while $i^!$ is conservative.
\end{proof}

As another corollary, we obtain: 

\begin{cor}  \label{c:Gamma defined on Ran untl}
Let $X$ be proper. Then the functor 
$$\Gamma_{c,\on{Ran}^{\on{untl}}}:\Dmod(\on{Ran}^{\on{untl}} (X)) \to \on{Vect},$$
left adjoint to 
$$p^!_{\on{Ran}^{\on{untl}}}:\Vect\to \Dmod(\on{Ran}^{\on{untl}} (X)),$$
is well-defined.
\end{cor}

\begin{proof}

It is enough to show that the functor $\Gamma_{c,\on{Ran}^{\on{untl}}}$ is well-defined on the generators of
$\Dmod(\on{Ran}^{\on{untl}} (X))$, i.e., on the essential image of the functor $\iota_!$. Hence, it suffices
to show that the functor
$$\Gamma_{c,\on{Ran}^{\on{untl}}}\circ \iota_!,$$
left adjoint to $\iota^!\circ p^!_{\on{Ran}^{\on{untl}}}\simeq p^!_{\on{Ran}_\emptyset}$.

\medskip

However, this follows from the fact that $\on{Ran}_\emptyset(X)$ is pseudo-proper. 

\end{proof}

Combining Corollaries \ref{c:unital Ran dualizable} and \ref{c:Gamma defined on Ran untl}, and \lemref{l:product dualizable}, 
we obtain;

\begin{cor} \label{c:Gamma defined on Ran untl rel}
For any laft categorical prestack $\CY$, the functor
$$(\on{id}_\CY\times p_{\on{Ran}^{\on{untl}}})_!:\Dmod(\CY\times \on{Ran}^{\on{untl}} (X))\to \Dmod(\CY),$$
left adjoint to $(\on{id}_\CY\times p_{\on{Ran}^{\on{untl}}})^!$ is well-defined, is compatible with base change,
and satisfies the projection formula.
\end{cor} 


\ssec{Contractibility results for the Ran space, the unital version}

\sssec{}

Recall the notions of \emph{universally homologically contractible} and 
\emph{universally homologically cofinal} morphism, see \cite[Sects. 3.4 and 3.5]{AB}. 
We claim:

\begin{prop} \label{p:Ran untl contr} \hfill
\begin{enumerate}[label={(\alph*)}]
\item
The categorical prestack $\on{Ran}^{\on{untl}} (X)$ is universally homologically contractible.

\item
The diagonal map $\on{Ran}^{\on{untl}} (X)\to \on{Ran}^{\on{untl}} (X)\times \on{Ran}^{\on{untl}} (X)$
is universally homologically cofinal.
\end{enumerate}
\end{prop}

\begin{proof}

Both assertions of the proposition are valid level-wise; i.e., for a \emph{space} $\sfX$, the category 
$\on{Ran}^{\on{untl}} (\sfX)$ is contractible, and the diagonal map
$$\on{Ran}^{\on{untl}} (\sfX)\to \on{Ran}^{\on{untl}} (\sfX)\times \on{Ran}^{\on{untl}} (\sfX)$$
is cofinal.

\medskip

The first assertion is evident as $\on{Ran}^{\on{untl}} (\sfX)$ has an initial object (the empty subset).
The second assertion follows from the fact that $\on{Ran}^{\on{untl}} (\sfX)$ admits finite coproducts. 

\end{proof}

Combining with \cite[Proposion 3.4.9 and Corollary 3.5.12]{AB}, we obtain:

\begin{cor} \hfill \label{c:Gamma on Ran untl mon}
\begin{enumerate}[label={(\alph*)}]
\item
The functor 
$p^!_{\on{Ran}^{\on{untl}}}:\Vect\to \Dmod(\on{Ran}^{\on{untl}} (X))$
is fully faithful. 

%

\item
Let $X$ be proper, so that the functor 
$$\Gamma_{c,\on{Ran}^{\on{untl}}}:\Dmod(\on{Ran}^{\on{untl}} (X)) \to \on{Vect}$$
is well-defined. Then this functor is symmetric monoidal. 

%
\end{enumerate}
\end{cor} 

Combining with \corref{c:unital Ran dualizable} and \lemref{l:product dualizable}, we obtain:

\begin{cor} Let $X$ be proper. Let $\CY$ be any laft categorical prestack. 
\begin{enumerate}[label={(\alph*)}]

\item
For any $\CF\in \Dmod(\CY)$, the map
$$(\on{id}_\CY\times p_{\on{Ran}^{\on{untl}}})_!(\CF\boxtimes \omega_{\on{Ran}^{\on{untl}}(X)})\to \CF$$
is an isomorphism. 

\item
For any $\CF\in \Dmod(\CY\times \on{Ran}^{\on{untl}} (X)\times \on{Ran}^{\on{untl}} (X))$
the map
$$(\on{id}_\CY \times p_{\on{Ran}^{\on{untl}}})_! \circ (\on{id}_\CY\times \Delta_{\on{Ran}^{\on{untl}}})^!(\CF)\to
(\on{id}_\CY\times p_{\on{Ran}^{\on{untl}}}\times p_{\on{Ran}^{\on{untl}}})_!(\CF)$$
is an isomorphism.
\end{enumerate}
\end{cor}

\begin{rem}
The key difference between $\Ran(X)$ and $\Ran^{\on{untl}}(X)$ is that the functor 
$\Gamma_{c,\on{Ran}^{\on{untl}}}$ is (strictly) symmetric monoidal, whereas the functor 
$$\Gamma_{c,\on{Ran}}: \Dmod(\on{Ran}(X)) \to \on{Vect}$$
is only left-lax symmetric monoidal. This feature of $\Gamma_{c,\on{Ran}^{\on{untl}}}$ will play an essential role
in the main theorem.
\end{rem} 

\ssec{Contractibility results for the Ran space, the \emph{non}-unital version}

The contractibility results for the usual Ran space $\Ran(X)$ hold under the assumption
that $X$ is \emph{connected}, which will be imposed for the duration of this subsection. 

\sssec{}

We have the following fundamental contractibility statement about the usual Ran space:

\begin{thm} \label{t:Ran contr}
The prestack $\on{Ran}(X)$ is universally homologically contractible.
\end{thm}

For the proof see \cite[Sect. 6]{Contr}.

\begin{rem}
The proof of \thmref{t:Ran contr} is essentially a variation of \lemref{l:contractibility}.
\end{rem}

\sssec{}

Now, we claim:

\begin{thm} \label{t:Ran cofinality new}
The inclusion
$$\iota:\on{Ran}(X) \to \on{Ran}^{\on{untl}}(X)$$
is universally homologically cofinal.
\end{thm}

For the proof, see \cite[Theorem 4.6.2]{AB}. 

\begin{rem}
Obviously, \thmref{t:Ran cofinality new} combined with \propref{p:Ran untl contr} implies
\thmref{t:Ran contr}. However, the proof of \thmref{t:Ran cofinality new} given
in \cite{AB} is an immediate consequence of \thmref{t:Ran contr}.

\end{rem}

\sssec{}

As a first consequence of \thmref{t:Ran cofinality new}, we obtain:

\begin{cor}\label{c:Ran cofinality}
Let $\CY$ be an arbitrary laft prestack, $\CF'\in \Dmod(\CY)$ and 
$$\CF \in \Dmod(\CY)\otimes \Dmod(\on{Ran}^{\on{untl}}(X))\simeq \Dmod(\CY\times \on{Ran}^{\on{untl}}(X)).$$
Then the map
\begin{multline*}
\on{Maps}_{\Dmod(\CY)\otimes \Dmod(\on{Ran}^{\on{untl}}(X))}(\CF, \CF'\boxtimes \omega_{\on{Ran}^{\on{untl}} (X)})) \to \\
\to \on{Maps}_{\Dmod(\CY)\otimes \Dmod(\on{Ran}(X))}(\iota^!(\CF), \iota^!(\CF')\boxtimes \omega_{\on{Ran}(X)})) 
\end{multline*}
is an isomorphism.
\end{cor}

\ssec{Linear factorization sheaves} \label{ss:lin fact}

In this subsection, we will study the left adjoint to the pullback functor along the inclusion $$i:X\to \on{Ran}^{\on{untl}} (X).$$  
In particular, we will show that it is fully faithful and identify the essential image.

\sssec{}

Consider the union map:
$$ \cup: \on{Ran}^{\on{untl}} (X) \times \on{Ran}^{\on{untl}} (X) \to \on{Ran}^{\on{untl}} (X) $$
Evidently, since $I \subset I \cup J$ and $J \subset I\cup J$, we have
morphisms
$$ p_1 \to \cup \mbox{ and } p_2 \to \cup$$
in the category $\on{Maps}(\on{Ran}^{\on{untl}} (X)\times \on{Ran}^{\on{untl}} (X), \on{Ran}^{\on{untl}} (X))$
where $p_1$ and $p_2$ are the two projections.

\medskip

In particular, for $\CF\in \Dmod(\on{Ran}^{\on{untl}} (X))$, we have the canonically defined maps
$$p_i^!(\CF) \to \cup^!(\CF), \quad i=1,2$$
which give a map
\begin{equation} \label{e:lin fact}
p_1^!(\mathcal{F}) \oplus p_2^!(\mathcal{F}) \to \cup^!(\mathcal{F}).
\end{equation} 

\sssec{}

Let
$$(\on{Ran}^{\on{untl}} (X) \times \on{Ran}^{\on{untl}} (X))_{\on{disj}} \subset \on{Ran}^{\on{untl}} (X) \times \on{Ran}^{\on{untl}} (X)$$
denote the full sub-categorical prestack consisting of pairwise disjoint subsets.

\medskip

We define the full subcategory of $\Dmod(\on{Ran}^{\on{untl}} (X))$, to be called the category of \emph{linear factorization sheaves}, to be denoted 
$\on{LFS}(X)$, to consist of those objects $\CF\in \Dmod(\on{Ran}^{\on{untl}} (X))$, for which the map \eqref{e:lin fact} 
restricts to an isomorphism over $(\on{Ran}^{\on{untl}} (X) \times \on{Ran}^{\on{untl}} (X))_{\on{disj}}$.

\medskip

Note that for $\CF \in \on{LFS}(X)$, restricting the isomorphism \eqref{e:lin fact} along the map
$$ \on{Ran}^{\on{untl}}(X) \overset{(\emptyset, \on{id})}{\longrightarrow} (\on{Ran}^{\on{untl}} (X) \times \on{Ran}^{\on{untl}} (X))_{\on{disj}} $$
shows that the restriction of $\CF$ along
$$\on{pt} \overset{\emptyset}\hookrightarrow \on{Ran}^{\on{untl}} (X)$$
vanishes. 

\sssec{} \label{sss:explicit LFS}

We can describe the subcategory 
$$\on{LFS}(X) \subset \Dmod(\on{Ran}^{\on{untl}} (X))$$
in a more hands-on way, in the spirit of \secref{sss:dmod on Ran as finite prods}, as follows.

\medskip

Suppose we have a $D$-module
$$ M \in \Dmod(\on{Ran}^{\on{untl}} (X)) .$$
Then $M$ lies in $\on{LFS}(X)$ iff
for every partition $I = I_1 \sqcup I_2$ of a finite set,
the map
$$ \nu(M)_{f_1} \oplus \nu(M)_{f_2}:
 \Delta_{f_1}^!(M_{I_1}) \oplus \Delta_{f_2}^!(M_{I_2}) \to M_{I} $$
becomes an isomorphism when restricted to $(X^{I_1}\times X^{I_2})_{\on{disj}}$, 
where $$f_i: I_i \to I$$ are the inclusion maps for $i=1,2$, and $f_i:X^I\to X^{I_i}$ are the corresponding projections. 

\begin{rem}

\thmref{t:image of Dmod in Ran} is the second main reason that the unital Ran space appears in this paper:

\medskip

One could define linear factorization sheaves on the non-unital Ran space, but in this case, this will
be extra structure and not a property, hence much more cumbersome to work with.

\medskip

The relevance of linear factorization sheaves is explained by \thmref{t:image of Dmod in Ran} below.

\end{rem} 

\sssec{}

We claim:

\begin{thm}\label{t:image of Dmod in Ran}
Let $X$ be a (separated) laft scheme.  The restriction functor
$$ i^!: \Dmod(\on{Ran}^{\on{untl}} (X)) \to \Dmod(X) $$
admits a fully faithful left adjoint
$$ i_!: \Dmod(X) \to \Dmod(\on{Ran}^{\on{untl}} (X)),$$
whose essential image equals $\on{LFS}(X)$. 
\end{thm}

\ssec{Proof of \thmref{t:image of Dmod in Ran}}

\sssec{}

Consider the pointed unital Ran categorical prestack of $X$:
$$ \on{Ran}^{\on{untl}}_{\star}(X)(S) = \on{Ran}_{\star}(\Maps(S,X_\dr))$$
We have the Cartesian square
$$
\xymatrix{
\on{Ran}^{\on{untl}}_{\star}(X) \ar[r]^-{\on{pr}}\ar[d] & X_\dr \ar[d] \\
\on{Ran}^{\to}(X) \ar[r]\ar[d]^{\phi} & \on{Ran}_\emptyset(X) \\
\on{Ran}^{\on{untl}}(X)},$$
where the composite vertical map, to be denoted $q$, is the natural projection 
$$\on{Ran}^{\on{untl}}_{\star}(X)\to \on{Ran}^{\on{untl}}(X).$$

\medskip

As we saw in \secref{sss:ps-pr Ran}, the map $\phi$ is pseudo-proper.  The map $X_\dr\to \on{Ran}_\emptyset(X)$
is pseudo-proper by \propref{p:Ran diag ps-proper}(a). Hence the map $q$ is also pseudo-proper. 

\medskip

We will show that the functor $q_! \circ \on{pr}^!$ provides a left adjoint to $i^!$. 

\sssec{}

The inclusion $i:X_\dr \to \on{Ran}^{\on{untl}} (X)$ factors as
$$ X_\dr \overset{i_0}{\to} \on{Ran}^{\on{untl}}_{\star}(X) \overset{q}{\to} \on{Ran}^{\on{untl}} (X).$$
The map $\on{pr}$ is the right adjoint to the inclusion map $i_0: X \to \on{Ran}_\star(X)$.  Therefore,
$$ \on{pr}^!: \Dmod(\on{Ran}^{\on{untl}}_{\star}(X)) \to \Dmod(X) $$
is the left adjoint to $i_0^!$.  Moreover, since $q$ is pseudo-proper, $q^!$ admits a left adjoint $q_!$.
Therefore, the left adjoint to $i^!$ exists and is given by $q_!\circ \on{pr}^!$.

\sssec{}
We will now show that this functor is fully faithful.
Consider the fiber product
$$ \on{Ran}^{\on{untl}}_{\star}(X) \underset{\on{Ran}^{\on{untl}} (X)}{\times} X_\dr.$$
It is easy to see, however, that the projection 
$$\on{Ran}^{\on{untl}}_{\star}(X)\underset{\on{Ran}^{\on{untl}} (X)}{\times} X_\dr\to X_\dr$$
is an isomorphism. Hence, so is the map
$$X_\dr\overset{i_0\times \on{id}}\to \on{Ran}^{\on{untl}}_{\star}(X) \underset{\on{Ran}^{\on{untl}} (X)}{\times} X_\dr.$$

\medskip

Hence, the diagram
$$
\CD
X_{\on{dR}} @>{i_0}>>  \on{Ran}^{\on{untl}}_{\star}(X)_{\on{dR}}  \\
@V{\on{id}}VV    @VV{q}V \\
X_{\on{dR}} @>{i}>> \Ran^{\on{untl}}(X) 
\endCD
$$
is a pullback square. 

\medskip

Since $q$ is pseudo-proper, the functor $q_!$ satisfies base change.  Hence, from the above diagram, we obtain: 
$$ i^! \circ q_! \simeq i_0^! ,$$
and thus
$$ i^! \circ i_! \simeq i^! \circ q_! \circ \on{pr}^! \simeq i_0^! \circ \on{pr}^! \simeq \on{id} ,$$
and therefore $i_!$ is fully faithful.

\sssec{}

Finally, let's show that the essential image of $i_!$ equals
$$\on{LFS}(X) \subset \Dmod(\on{Ran}^{\on{untl}} (X)).$$ 

First, note that the restriction functor
$$ i^!: \on{LFS}(X) \to \Dmod(X) $$
is conservative.  

\medskip

Indeed, suppose that $\mathcal{F}$ is an object in $\on{LFS}(X)$ such that
$i^!(\mathcal{F})=0$.  By the conditions defining $\on{LFS}$ (and induction), for a finite set $I$ we obtain that the object
$\mathcal{F}_{X^I}$ restricts to $0$ on each stratum of the diagonal stratification of $X^I$.  Therefore, $\mathcal{F}=0$.

\sssec{}

Thus, it remains to show that the essential image of $i_!$ lies in $\on{LFS}(X)$. Set
\begin{multline*}
(\on{Ran}^{\on{untl}}_{\star}(X) \times \on{Ran}^{\on{untl}} (X))_{\on{disj}}:= \\
\left(\on{Ran}^{\on{untl}}_{\star}(X) \times \on{Ran}^{\on{untl}} (X)\right)
\underset{\on{Ran}^{\on{untl}}_{\star}(X) \times \on{Ran}^{\on{untl}} (X)}\times
(\on{Ran}^{\on{untl}}(X) \times \on{Ran}^{\on{untl}} (X))_{\on{disj}}.
\end{multline*}

Consider the following diagram
$$
\CD
& & X \\
& & @A{\on{pr}}AA \\
(\on{Ran}^{\on{untl}}_{\star}(X) \times \on{Ran}^{\on{untl}} (X))_{\on{disj}} \sqcup (\on{Ran}^{\on{untl}} (X) \times \on{Ran}^{\on{untl}}_{\star}(X))_{\on{disj}} @>{\cup}>>
\on{Ran}^{\on{untl}}_{\star}(X)  \\
@V{(q \times \on{id}) \sqcup (\on{id} \times q)}VV @VV{q}V \\
(\on{Ran}^{\on{untl}} (X) \times \on{Ran}^{\on{untl}} (X))_{\on{disj}} @>{\cup}>> \on{Ran}^{\on{untl}}(X)
\endCD
$$
%
where the square is Cartesian.  

\medskip

By base change for the pseudo-proper map $q$, we have 
$$ \cup^! \circ i_! = \cup^! \circ q_! \circ \on{pr}^! \simeq ((q\times \on{id})_! \boxplus (\on{id}\times q)_!) \circ \cup^! \circ \on{pr}^! .$$
Note that the composite
$$ (\on{Ran}^{\on{untl}}_{\star}(X) \times \on{Ran}^{\on{untl}} (X))_{\on{disj}} \sqcup (\on{Ran}^{\on{untl}} (X) \times \on{Ran}^{\on{untl}}_{\star}(X))_{\on{disj}} 
\overset{\on{pr}\circ \cup}\longrightarrow X $$
is isomorphic to $(\on{pr} \circ p_1) \sqcup (\on{pr}\circ p_2)$.  It follows that
$$ ((q\times \on{id})_! \boxplus (\on{id}\times q)_!) \circ \cup^! \circ \on{pr}^! \simeq (p_1^!\circ q_! \circ \on{pr}^!) \oplus (p_2^! \circ q_!\circ \on{pr}^!)=
(p_1^!\circ i^!) \oplus (p_2^! \circ i^!), $$
which gives the desired isomorphism.

\qed[\thmref{t:image of Dmod in Ran}]

%
%
%
%
%
%
%
%
%
%
%

\section{D-prestacks and prestacks of (flat) sections} \label{s:sect laft}

\ssec{Sectionally finite type D-prestacks} \label{ss:sect laft}

\sssec{} \label{sss:geom setup}

Let $X$ be a (separated) scheme of finite type.  Let $\CZ \to X_{\dr}$ be a prestack over $X_{\dr}$.

\medskip

Let $\CS(\CZ)$ denote the prestack of flat sections of $\CZ \to X_{\dr}$,
$$\mathcal{S}(\CZ)(S) := \on{Maps}_{/X_{\on{dR}}}(S \times X_{\on{dR}}, \CZ).$$

In other words, $\CS(\CZ)$ is the Weil restriction
$$\on{Res}^{X_{\dr}}_{\on{pt}}(\CZ)$$
of $\CZ \in \on{PreStk}_{/X_{\dr}}$ along $p_X:X_{\dr}\to \on{pt}$.

\medskip

In this subsection we will formulate a key finiteness condition on $\CZ$ that will ensure that
$\CS(\CZ)$ (along with its relatives that arise from punctured sections) 
has good infinitesimal properties. 

\sssec{}

First, we assume that $\CZ$ admits deformation theory. 

\medskip

Note that since the deformation
theory of $X_\dr$ is trivial, we could have equivalently demanded that $\CZ$ admit
deformation theory relative to $X_\dr$. 

\medskip

Similarly, since $T^*(X_{\on{dR}})=0$, we have $T^*(\CZ) \simeq T^*(\CZ/X_{\on{dR}})$.

\sssec{} \label{sss:defn sect laft}

We will say that $\CZ$ is \emph{sectionally laft} if:
\begin{enumerate}[label=(\roman*)]
\item
The prestack of sections $\mathcal{S}(\CZ)$ is laft.

\smallskip

\item
For every $(s: S \to \mathcal{S}(\CZ)) \in (\affSch_{\on{aft}})_{/\mathcal{S}(\CZ)}$, we have
$$ \oblv^{\on{fake}}\circ \on{ev}_s^{\#}(T^*(\CZ)) \in \on{Pro}(\QCoh(S \times X_{\on{dR}})^-)_{\on{laft}} ,$$
where 
\begin{itemize}
\item $\on{ev}_s: S \times X_{\on{dR}} \to \CZ$ is the section corresponding to $s$;

\item $\on{Pro}(\QCoh(S \times X_{\on{dR}})^-)_{\on{laft}}  \subset \on{Pro}(\QCoh(S \times X_{\on{dR}})^-)$
is the subcategory defined in \secref{sss:pro laft cat}.

\item $\oblv^{\on{fake}}$ and $f^{\#}$ (for a morphism $f$) are the functors introduced in \secref{ss:fake}. 

\end{itemize}

\end{enumerate}

In \secref{s:examples}, we will see examples of sectionally laft prestacks that arise in practice. 

%
%
%

\ssec{The antecedent of the tangent complex} \label{ss:Theta}

\sssec{} \label{sss:Theta on S}

Recall (see \corref{c:Serre Verdier}) that Serre-Verdier duality gives an equivalence
$$\BD^{\on{SV}}:\on{Pro}(\QCoh(S \times X_{\on{dR}})^-)_{\on{laft}} \simeq \IndCoh(S \times X_{\on{dR}})^{\on{op}}  $$
Moreover, for a map of affine schemes almost of finite type $f: S\to S'$, we have a commutative diagram
$$ \xymatrix{
\IndCoh(S' \times X_{\on{dR}})^{\on{op}}  \ar[r]^{(f^!)^{\on{op}} }\ar[d] & \IndCoh(S \times X_{\on{dR}})^{\on{op}}  \ar[d] \\
\on{Pro}(\QCoh(S' \times X_{\on{dR}})^-)_{\on{laft}} \ar[r]^{f^{\#}} & \on{Pro}(\QCoh(S \times X_{\on{dR}})^-)_{\on{laft}} 
}$$
which is functorial in $f$.

\sssec{}

Hence, we obtain that for every $(s: S \to \mathcal{S}(\CZ)) \in (\affSch_{\on{aft}})_{/\mathcal{S}(\CZ)}$ as above,
the dual 
$$\BD^{\on{SV}}(\oblv^{\on{fake}}\circ \on{ev}_s^{\#}(T^*(\CZ)))$$
is a well-defined object, to be denoted
$$(s\times \on{id})^!(\Theta(\CZ)) \in \IndCoh(S \times X_{\on{dR}}).$$

Moreover, these objects are compatible under pullbacks along maps $f:S\to S'$, thereby giving rise to an object
\begin{equation} \label{e:Theta}
\Theta(\CZ)\in \IndCoh(\mathcal{S}(\CZ) \times X_{\on{dR}})\simeq \IndCoh(\mathcal{S}(\CZ))\otimes \Dmod(X).
\end{equation} 

\begin{rem}

Note that unless $\CZ$ itself is laft, there is no object on $\CZ$
that  $\Theta(\CZ)$ is the pull-back of. If $\CZ$ is laft, $\Theta(\CZ)$ is the pull-back of
the tangent sheaf of $\CZ$, up to a twist, see \propref{p:target ft}. 

\end{rem} 

\sssec{}

The object $\Theta(\CZ)$ will play a similar role to the usual tangent complex.  Note that the usual tangent complex is defined as the 
dual of the cotangent complex, whereas $\Theta(\CZ)$ is defined as the \emph{$D$-module} dual of the cotangent 
complex and these dualities are in general quite different.

%
%
%
%

\medskip

Applying \corref{c:Weil def}, we obtain:

\begin{prop}\label{p:sect tangent}
Let $\CZ \to X_{\on{dR}}$ be a sectionally laft prestack.  Then $\CS(\CZ)$ admits deformation 
theory, and its tangent complex is given by
$$ T(\CS(\CZ)) \simeq  (\on{id}\times p_{X_\dr})^\IndCoh_*(\Theta(\CZ)),$$
where
$$\on{id}\times p_{X_\dr}:S\times X_\dr\to S$$
is the projection map. Equivalently, $T(\CS(\CZ)) \simeq (\on{Id}\otimes p_{\dr,*})(\Theta(\CZ))$, 
where 
$$(p_X)_{\dr,*}:\Dmod(X)\to \Vect$$
is the functor of de Rham cohomology. 
\end{prop}

\begin{rem}
Combining \corref{c:Weil def} and \cite[Chapter 1, Corollary 9.1.3]{Vol2} we obtain that in the definition of
what it means to be sectionally laft in \secref{sss:defn sect laft}, given the second condition, 
we can replace the first condition by the following seemingly weaker one:

\begin{enumerate}
\item[(i')]
The \emph{classical prestack} underlying $\mathcal{S}(\CZ)$ is locally of finite type.
\end{enumerate}

\end{rem}

\ssec{Punctured sections}

\sssec{}
In addition to sections over all of $X_{\on{dR}}$, we will consider sections over complements of points in $X_{\on{dR}}$.  
Specifically, given a set of points $\{x_1, \ldots, x_n\} \subset X$, we can consider the prestack
$$ \overset{\circ}{\CS}(\CZ)_{\{x_1, \ldots, x_n\}}:= \mathcal{S}(\CZ_{| X_{\on{dR}} -\{x_1, \ldots, x_n\}}) $$
of sections over the complement of $\{x_1, \ldots, x_n\}$.  As the set of points vary, these prestacks assemble into a prestack over the (unital) 
Ran space $\on{Ran}^{\on{untl}}(X)$.

\medskip

Specifically, let $\gamma: (\overset{\circ}{\mathfrak{X}}_{\on{Ran}^{\on{untl}} })_{\dr} \to \on{Ran}^{\on{untl}} (X)$ be the 
(deRham space of) the universal punctured space constructed in \secref{ss:punctured fibration}.  The map $\gamma$ is a 
level-wise Cartesian fibration in groupoids.

\medskip

We have the projection map $(\overset{\circ}{\mathfrak{X}}_{\on{Ran}^{\on{untl}} })_{\dr} \to X_{\dr}$, so that 
the map
\begin{equation}\label{e:Z over punctured}
\CZ \underset{X_{\dr}}{\times} (\overset{\circ}{\mathfrak{X}}_{\on{Ran}^{\on{untl}} })_{\dr} \to (\overset{\circ}{\mathfrak{X}}_{\on{Ran}^{\on{untl}} })_{\dr}
\end{equation}
is a level-wise coCartesian fibration in groupoids.  Let
\begin{equation}\label{e:punctured sect}
\overset{\circ}{\CS}(\CZ)_{\on{Ran}^{\on{untl}} (X)} \to \on{Ran}^{\on{untl}}(X)
\end{equation}
denote the Weil restriction of \eqref{e:Z over punctured} along $\gamma$, i.e.,
$$\overset{\circ}{\CS}(\CZ)_{\on{Ran}^{\on{untl}}(X)}:=\on{Res}_\gamma(\overset{\circ}{\CS}(\CZ)_{\on{Ran}^{\on{untl}} (X)}),$$
see \secref{sss:Weil categ}. By construction, 
this is a level-wise coCartesian fibration in groupoids.

\sssec{}

We have a level-wise Cartesian fibration in groupoids
$$ (\overset{\circ}{\mathfrak{X}}_{\on{Ran}^{\on{untl}} })_{\dr} \to X_{\on{dR}} \times \on{Ran}^{\on{untl}}(X) .$$
Therefore, by adjunction, we have the restriction map
\begin{equation}\label{e:ran restr}
\CS(\CZ) \times \on{Ran}^{\on{untl}}(X) \to \overset{\circ}{\CS}(\CZ)_{\on{Ran}^{\on{untl}} (X)}
\end{equation}
of level-wise coCartesian fibrations in groupoids over $\Ran^{\on{untl}}(X)$.

\medskip
 
Explicitly, over a point $(x_1, \ldots, x_n) \in \on{Ran}(X)$, the above restriction map is given by the evident restriction of sections
$$ \CS(\CZ) \to \CS(\CZ_{| X_{\on{dR}} - \{x_1, \ldots, x_n\}}) .$$

\sssec{}\label{sss:punctured sects finite powers}

Here is a concrete description of the categorical prestack $\overset{\circ}{\CS}(\CZ)_{\on{Ran}^{\on{untl}} (X)}$ 
and the restriction map \eqref{e:ran restr}.  

\medskip

By \secref{sss:maps from Ran}, a level-wise coCartesian fibration in 
groupoids $\CY \to \on{Ran}^{\on{untl}}(X)$ can be described as the following data:
\begin{itemize}
\item
for every $I \in \on{Fin}$,
a prestack $\CY_I \to X^I_{\on{dR}}$ over $X^I_{\on{dR}}$;
\item
for every morphism $f: I \to J$ in $\on{Fin}$ a map of prestacks
$$ \CY_I \underset{X^I_{\on{dR}}}{\times} X^J_{\on{dR}} \to \CY_J $$
over $X^I_{\on{dR}}$, which are isomorphisms when $f$ is surjective,
\end{itemize}
equipped with a data of homotopy coherence for compositions of maps in $\on{Fin}$.

\medskip

In these terms, the level-wise coCartesian fibration in groupoids \eqref{e:punctured sect} is given as follows.  For every $I \in \on{Fin}$, the corresponding prestack
$$ \overset{\circ}{\CS}(\CZ)_{I} \to X^I_{\on{dR}} $$
is the Weil restriction of $\CZ \underset{X_{\on{dR}}}{\times} (\overset{\circ}{\mathfrak{X}}_I)_{\dr}$
along the map $\gamma_I: (\overset{\circ}{\mathfrak{X}}_I)_{\dr} \to X^I$, where
$$ \overset{\circ}{\mathfrak{X}}_I \subset X \times X^I $$
is the open subset from \secref{sss:punctured curve}, i.e.,
$$\overset{\circ}{\CS}(\CZ)_{I}=\on{Res}_{\gamma_I}(\overset{\circ}{\mathfrak{X}}_I),$$
see \secref{sss:Weil restr}.

\medskip

By adjunction, we have the canonical evaluation map
\begin{equation}
\on{ev}_{I}: \overset{\circ}{\CS}(\CZ)_{I} \underset{X^I_{\on{dR}}}{\times} (\overset{\circ}{\mathfrak{X}}_I)_{\dr}
\to \CZ
\end{equation}
over $X_{\on{dR}}$.

\medskip

Given a morphism, $f: I \to J$ in $\on{Fin}$, we are supposed to have a map
$$ \overset{\circ}{\CS}(\CZ)_{I} \underset{X^I_{\on{dR}}}{\times} X^J_{\on{dR}} \to \overset{\circ}{\CS}(\CZ)_{J} $$
over $X^J_{\on{dR}}$.
By adjunction, such a map is equivalent to a map 
$$
(\overset{\circ}{\CS}(\CZ)_{I} \underset{X^I_{\on{dR}}}{\times} X^J_{\on{dR}}) \underset{X^J_{\on{dR}}}{\times} (\overset{\circ}{\mathfrak{X}}_J)_{\dr} \simeq \overset{\circ}{\CS}(\CZ)_{I} \underset{X^I_{\on{dR}}}{\times} (\overset{\circ}{\mathfrak{X}}_J)_{\dr} \to \CZ 
$$
over $X_{\on{dR}}$.
This map is given by the composite
\begin{equation}\label{e:monodromy for punctured sects}
 \overset{\circ}{\CS}(\CZ)_{I} \underset{X^I_{\on{dR}}}{\times} (\overset{\circ}{\mathfrak{X}}_J)_{\dr} \to
\overset{\circ}{\CS}(\CZ)_{I} \underset{X^I_{\on{dR}}}{\times} (\overset{\circ}{\mathfrak{X}}_I)_{\dr} \overset{\on{ev}_I}{\to} \CZ
\end{equation}

\medskip
Simiarly, the restriction map \eqref{e:ran restr}
over $X^I_{\on{dR}}$ 
$$ \CS(\CZ) \times X^I_{\on{dR}} \to \overset{\circ}{\CS}(\CZ)_{I} .$$
is, by adjunction, given by the map
\begin{equation}\label{e:restr map on fin}
\CS(\CZ) \times (\overset{\circ}{\mathfrak{X}}_I)_{\dr} \to \CS(\CZ) \times X_{\dr} \overset{\on{ev}}{\to} \CZ ,
\end{equation}
which coincides with \eqref{e:monodromy for punctured sects} for the map $\emptyset \to I$ in $\on{Fin}$.

\ssec{Formal punctured sections}

In this subsection we will define one of the key geometric players for this paper.

\sssec{}

Let $\overset{\circ}{\CS}(\CZ)\strut^{\wedge}_{\on{Ran}^{\on{untl}}}$ denote the formal completion of $\overset{\circ}{\CS}(\CZ)_{\on{Ran^{\on{untl}}}}$ along the map
$$\CS(\CZ) \times \on{Ran}^{\on{untl}}(X)\to \overset{\circ}{\CS}(\CZ)_{\on{Ran^{\on{untl}}}},$$
see \secref{sss:formal compl categ}. 

\medskip

By \lemref{l:form compl fibration}, the maps
$$ \CS(\CZ) \times \on{Ran}^{\on{untl}}(X) \to \overset{\circ}{\CS}(\CZ)\strut^{\wedge}_{\on{Ran}^{\on{untl}} } \to \on{Ran}^{\on{untl}}(X)$$
are level-wise coCartesian fibrations in groupoids.

%

\sssec{}

Assume that $\CZ$ is sectionally laft (see \secref{ss:sect laft}). This assumption, by definition, guarantees that $\CS(\CZ)$ is a laft prestack.
But it does not imply a similar property for $\overset{\circ}{\CS}(\CZ)_{\on{Ran^{\on{untl}}}}$. However, it does imply that 
$\overset{\circ}{\CS}(\CZ)\strut^{\wedge}_{\on{Ran}^{\on{untl}}}$ is laft, see \propref{p:punctured sects laft-def} below. 

\medskip

The laft-ness property of $\overset{\circ}{\CS}(\CZ)\strut^{\wedge}_{\on{Ran}^{\on{untl}}}$
ensures that the category $\IndCoh$, and various monads associated with it, are well-defined on 
$\overset{\circ}{\CS}(\CZ)\strut^{\wedge}_{\on{Ran}^{\on{untl}}}$. 

\begin{prop} \label{p:punctured sects laft-def}
The categorical prestack $\overset{\circ}{\CS}(\CZ)\strut^{\wedge}_{\on{Ran}^{\on{untl}} }$ is laft-def.
\end{prop}

The rest of this subsection is devoted to the proof of \propref{p:punctured sects laft-def}. 

\sssec{}

We first carry out a series of reduction steps. 

\medskip 

The categorical prestack $\on{Ran}^{\on{untl}}(X)$ is laft-def.
Therefore, by \corref{c:fibration laft-def},
it suffices to show that for every $S \in \affSch_{\on{aft}}$ and a map $S \to \on{Ran}^{\on{untl}}(X)$,
the prestack
\begin{equation}\label{e:punct Ran fiber}
S \underset{\on{Ran}^{\on{untl}}(X)}{\times} \overset{\circ}{\CS}(\CZ)\strut^{\wedge}_{\on{Ran}^{\on{untl}} }
\end{equation}
is laft-def. 

\medskip

We have $\on{Ran}^{\on{untl}}(X)^{\on{grpd}} = \{\emptyset\} \sqcup \on{Ran}(X)$.  The fiber of
$\overset{\circ}{\CS}(\CZ)\strut^{\wedge}_{\on{Ran}^{\on{untl}} }$ over $\{\emptyset\}$ is $\CS(\CZ)$, which is
laft-def by the hypothesis on $\CZ$.  Thus, we need to show that for every $S \to \on{Ran}(X)$, the prestack \eqref{e:punct Ran fiber} is laft-def.

\medskip

By \propref{p:ran pres}, every map $S \to \on{Ran}(X)$ factors through $X^I_{\on{dR}} \to \on{Ran}(X)$ for some $I \in \on{Fin}$.  
Thus, it suffices to show that
$$ \overset{\circ}{\CS}(\CZ)\strut^{\wedge}_I := X^I_{\on{dR}} \underset{\on{Ran}^{\on{untl}}(X)}{\times} \overset{\circ}{\CS}(\CZ)\strut^{\wedge}_{\on{Ran}^{\on{untl}} } $$
is laft-def for every $I \in \on{Fin}$.

\sssec{}

By construction, we have a nil-isomorphism
\begin{equation}\label{e:I restr}
\CS(\CZ) \times X^I_{\on{dR}} \to \overset{\circ}{\CS}(\CZ)\strut^{\wedge}_I .
\end{equation}
By the assumption on $\CZ$ and \corref{c:Weil def}, the prestack $\CS(\CZ)$ is laft-def and $\overset{\circ}{\CS}(\CZ)\strut^{\wedge}_I$ admits deformation theory.  
Therefore, by \cite[Chapter 1, Theorem 9.1.2]{Vol2}, it suffices to show that for every $S \in \affSch_{\on{aft}}$ and a map $x:S \to \CS(\CZ) \times X^I_{\on{dR}}$, 
the restriction of the cotangent complex of $\overset{\circ}{\CS}(\CZ)\strut^{\wedge}_I$ to $S$
$$ T^*_{\overset{\circ}{x}}(\overset{\circ}{\CS}(\CZ)\strut^{\wedge}_I) \in \on{Pro}(\QCoh(S)^-) $$
lies in $\on{Pro}(\QCoh(S)^-)_{\on{laft}} \subset \on{Pro}(\QCoh(S)^-)$, where $\overset{\circ}{x}: S \to \overset{\circ}{\CS}(\CZ)\strut^{\wedge}_I$ 
is the composite of $x$ and \eqref{e:I restr}. 

\medskip

At this point we are done with the reduction steps and we proceed to the proof of the stated property of 
$ T^*_{\overset{\circ}{x}}(\overset{\circ}{\CS}(\CZ)\strut^{\wedge}_I)$. 

\sssec{}

We can describe $T^*_{\overset{\circ}{x}}(\overset{\circ}{\CS}(\CZ)\strut^{\wedge}_I)$ as follows.  Consider the commutative diagram
$$\xymatrix{
S \underset{X^I_{\on{dR}}}{\times} {(\overset{\circ}{\mathfrak{X}}_I)_{\on{dR}}} \ar[r]^-j \ar[d]_{\overset{\circ}{p}} & S \times X_{\on{dR}} \ar[r]^-{\on{ev}_S}\ar[d]^{p} & \CZ \\
S \ar@{=}[r] & S,
}$$
where $\on{ev}_S$ is the composite
$$S \times X_{\on{dR}} \to \CS(\CZ) \times X_{\on{dR}} \overset{\on{ev}}\to \CZ.$$

By \eqref{e:restr map on fin}, the top composite map is the adjoint to the composite
$$ S\times X^I_{\on{dR}} \to \CS(\CZ) \times X^I_{\on{dR}} \to \overset{\circ}{\CS}(\CZ)_I $$
By \corref{c:Weil def}, 
$$ T^*_{\overset{\circ}{x}}(\overset{\circ}{\CS}(\CZ)\strut^{\wedge}_I) \simeq (\overset{\circ}{p})_{\#}\circ j^{\#}\circ (\on{ev}_S)^{\#}(T^*(\CZ)) .$$
By \lemref{l:pro base change}, 
$$ (\overset{\circ}{p})_{\#}\circ j^{\#}\circ (\on{ev}_S)^{\#}(T^*(\CZ)) \simeq p_{\#} \circ j_{\#}\circ j^{\#}\circ (\on{ev}_S)^{\#}(T^*(\CZ)). $$

Therefore, since $p$ is eventually coconnective, by \lemref{l:pro oblv functors}(a),
in order to show that $T^*_{\overset{\circ}{x}}(\overset{\circ}{\CS}(\CZ)\strut^{\wedge}_I) \in \on{Pro}(\QCoh(S)^-)_{\on{laft}}$,
it suffices to show that
$$ \oblv^{\on{fake}}(j_{\#}\circ j^{\#}\circ (\on{ev}_S)^{\#}(T^*(\CZ))) \in \on{Pro}(\QCoh(S \times X_{\on{dR}})^-)_{\on{laft}} .$$
By the assumption on $\CZ$, we have 
$$ \oblv^{\on{fake}}((\on{ev}_S)^{\#}(T^*(\CZ))) \in \on{Pro}(\QCoh(S \times X_{\on{dR}})^-)_{\on{laft}} .$$
Thus, it suffices to show that for $\mathcal{F} \in \on{Pro}(\QCoh(S \times X_{\on{dR}})^{-})^{\on{fake}}$ such that
$ \oblv^{\on{fake}}(\mathcal{F}) \in \on{Pro}(\QCoh(S \times X_{\on{dR}})^-)_{\on{laft}}$, we have
$$ \oblv^{\on{fake}} (j_{\#}\circ j^{\#}(\mathcal{F})) \in \on{Pro}(\QCoh(S \times X_{\on{dR}})^-)_{\on{laft}} .$$
Now, since the map $j$ is schematic and quasi-compact quasi-separated, this follows from \lemref{l:pro oblv functors}.

\qed[\propref{p:punctured sects laft-def}]

\section{The main result} \label{s:main}

We let $X$ be a separated scheme of finite type. As $X$ will be fixed, we will abbreviate $\on{Ran}(X)$ 
(resp., $\on{Ran}^{\on{untl}}(X)$) to simply $\on{Ran}$ (resp., $\on{Ran}^{\on{untl}}$). 

\medskip

Throughout this section we let $\CZ$ be a sectionally laft prestack over $X_\dr$. 

\ssec{The infinitesimal Hecke prestack and monad}

\sssec{}

By \propref{p:punctured sects laft-def}, we have a map of laft-def categorical prestacks
\begin{equation}
\fr_{\on{Ran}^{\on{untl}} }: \CS(\CZ) \times \Ran^{\on{untl}} \to \overset{\circ}{\CS}(\CZ)\strut^{\wedge}_{\on{Ran}^{\on{untl}}}
\end{equation}
which is a level-wise coCartesian fibration in groupoids and a nil-isomorphism.  In particular, it is pseudo-proper (see \secref{sss:pseudo proper def}).  

\medskip

Therefore, by \lemref{l:pseudo-proper pushforward},
the functor
$$ \fr_{\on{Ran}^{\on{untl}}}^!: \IndCoh(\overset{\circ}{\CS}(\CZ)\strut^{\wedge}_{\on{Ran}^{\on{untl}} }) \to \IndCoh(\CS(\CZ) \times \Ran^{\on{untl}}) $$
admits a left adjoint $(\fr_{\on{Ran}^{\on{untl}} })^{\IndCoh}_*$, which is a functor of $\IndCoh(\CS(\CZ) \times \Ran^{\on{untl}})$-module categories.  In particular, 
both $\fr_{\on{Ran}^{\on{untl}}}^!$ and $(\fr_{\on{Ran}^{\on{untl}}})^{\IndCoh}_*$ are functors of $\Dmod(\on{Ran}^{\on{untl}} (X))$-module categories.  

\sssec{}

Set
\begin{multline*} 
\CH^{\on{inf}}(\CZ)_{\on{untl}}:=\\
=\fr_{\on{Ran}^{\on{untl}} }^! \circ (\fr_{\on{Ran}^{\on{untl}}})^{\IndCoh}_* \in 
\on{Alg}\left(\on{End}_{\Dmod(\on{Ran}^{\on{untl}} (X))}\left(\IndCoh(\CS(\CZ) \times \Ran^{\on{untl}})\right)\right)
\end{multline*}
denote the corresponding monad (in the 2-category of $\Dmod(\on{Ran}^{\on{untl}} (X))$-linear DG categories).
We will refer to $\CH^{\on{inf}}(\CZ)_{\on{untl}}$ as the \emph{infinitesimal Hecke monad of $\CZ$}.

\sssec{}
Concretely, $\CH^{\on{inf}}(\CZ)_{\on{untl}}$ can be described as follows.  Let
$$ \widehat{\on{Hecke}}(\CZ)_{\on{Ran}^{\on{untl}} }:= (\CS(\CZ) \times \Ran^{\on{untl}}) 
\underset{ \overset{\circ}{\CS}(\CZ)\strut^{\wedge}_{\on{Ran}^{\on{untl}} }}{\times} (\CS(\CZ) \times \Ran^{\on{untl}})$$
denote the groupoid of infinitesimal Hecke modifications. In other words, an $S$-point of $\widehat{\on{Hecke}}(\CZ)_{\on{Ran}^{\on{untl}} }$ consists of the following data:
\begin{itemize}
\item
An $S$-point of $\Ran^{\on{untl}}$; let $\on{Graph} \subset S \times X_{\on{dR}}$ denote the
union of the graphs of the $S$-points of $X$ comprising the given $S$-point of the Ran space.
\item
A pair of sections
$$z_1,z_2: S \times X_{\on{dR}} \to \CZ $$
\item
An isomorphism
$$ \alpha: z_1|_{S \times X_{\on{dR}} - \on{Graph}} \simeq z_2|_{S \times X_{\on{dR}} - \on{Graph}}$$
\item
An isomorphism
$$ \beta: z_1|_{S_{\on{red}} \times X_{\on{dR}}} \simeq z_2|_{S_{\on{red}} \times X_{\on{dR}}} $$
\item
An identification
$$ \alpha_{|S_{\on{red}} \times X_{\on{dR}} - \on{Graph}} \simeq \beta_{|S_{\on{red}} \times X_{\on{dR}} - \on{Graph}}. $$
\end{itemize}

\medskip

Note that since $\fr_{\on{Ran}^{\on{untl}} }$ is pseudo-proper, the two maps
$$\partial_s,\partial_ t: \widehat{\on{Hecke}}(\CZ)_{\on{Ran}^{\on{untl}} } \to \CS(\CZ) \times \Ran^{\on{untl}}$$
are pseudo-proper.
We then have
$$ \CH^{\on{inf}}(\CZ)_{\on{untl}} \simeq (\partial_t)^{\IndCoh}_* \circ (\partial_s)^!,$$
as endofunctors of $\IndCoh(\CS(\CZ) \times \Ran^{\on{untl}})$. 

\sssec{} \label{sss:nil monad}

Recall that if 
$$f:\CY_1\to \CY_2$$
is a level-wise coCartesian fibration in groupoids and a nil-isomorphism between laft-def categorical prestacks, the corresponding adjoint pair
$$f^{\IndCoh}_*:\IndCoh(\CY_1) \leftrightarrows \IndCoh(\CY_2):f^!$$
is monadic, i.e., if we denote by $\sM_f$ the resulting monad on $\IndCoh(\CY_1)$ (so that as a plain endofunctor, $\sM_f\simeq f^!\circ f^{\IndCoh}_*$), 
the functor
$$(f^!)^{\on{enh}}:\IndCoh(\CY_2)\to \sM_f\mod(\IndCoh(\CY_1))$$
is an equivalence. 

\sssec{}

Applying this to the morphism
$$\fr_{\on{Ran}^{\on{untl}}}: \CS(\CZ) \times \Ran^{\on{untl}} \to \overset{\circ}{\CS}(\CZ)\strut^{\wedge}_{\on{Ran}^{\on{untl}}},$$
we obtain that $\fr_{\on{Ran}^{\on{untl}}}^!$ induces an equivalence
$$(\fr_{\on{Ran}^{\on{untl}}}^!)^{\on{enh}}:\IndCoh(\overset{\circ}{\CS}(\CZ)\strut^{\wedge}_{\on{Ran}^{\on{untl}}})\simeq
\CH^{\on{inf}}(\CZ)_{\on{untl}}\mod \left(\IndCoh(\CS(\CZ) \times \Ran^{\on{untl}})\right).$$

\ssec{Creating a monad on the space of sections} \label{ss:pushforward monad}

In this subsection, we will start from the monad $\CH^{\on{inf}}(\CZ)_{\on{untl}}$, which acts on 
$\IndCoh(\CS(\CZ) \times \Ran^{\on{untl}})$, and produce from it a monad that acts on 
$\IndCoh(\CS(\CZ))$; i.e., we will ``integrate out" the Ran space directions.

\medskip

It is for this manipulation that it is crucially important that we are working with the unital version of the Ran space.

\sssec{}

By \corref{c:unital Ran dualizable}, the category $\Dmod(\on{Ran}^{\on{untl}} (X))$ is dualizable.  Therefore, by \lemref{l:product dualizable},
\begin{equation}\label{e:sect Ran product}
\IndCoh(\CS(\CZ) \times \Ran^{\on{untl}}) \simeq \IndCoh(\CS(\CZ)) \otimes \Dmod(\on{Ran}^{\on{untl}} (X)) .
\end{equation}

\medskip

In what follows, we will make use of the following basic fact:

\begin{lem}\label{l:dualizable endomorphisms}
Let $\CA$ be a monoidal DG category such that $\CA$ is dualizable (as a DG category).  Then for any DG category $\bC$, the natural monoidal functor
$$ \on{End}(\bC) \otimes \CA \to \on{End}_{\CA\on{-mod}}(\CA \otimes \bC) $$
is an equivalence.
\end{lem}

\sssec{} \label{sss:End as tensor product}

Applying \lemref{l:dualizable endomorphisms} to $\CA = \Dmod(\on{Ran}^{\on{untl}} (X))$ and $\bC = \IndCoh(\CS(\CZ))$, 
and using \eqref{e:sect Ran product}, we obtain an equivalence of monoidal categories
\begin{multline*}
\on{End}_{\Dmod(\on{Ran}^{\on{untl}} (X))}\left(\IndCoh(\CS(Z) \times \Ran^{\on{untl}})\right) \simeq \\
\simeq \on{End}(\IndCoh(\CS(\CZ))) \otimes \Dmod(\on{Ran}^{\on{untl}} (X)).
\end{multline*}
In particular, we can think of the infinitesimal Hecke monad of $\CZ$ is an algebra object in this monoidal category:
$$\CH^{\on{inf}}(\CZ)_{\on{untl}} \in \on{Alg}(\on{End}(\IndCoh(\CS(\CZ))) \otimes \Dmod(\on{Ran}^{\on{untl}} (X))).$$

\sssec{}

Now assume that $X$ is proper. Recall that in this case, we have a well-defined functor 
$$\Gamma_{c,\on{Ran}^{\on{untl}}}:\Dmod(\on{Ran}^{\on{untl}} (X))\to \Vect,$$
which is symmetric monoidal (see \corref{c:Gamma on Ran untl mon}).

\medskip

Hence, we can apply the functor
$$\on{Id}_{\on{Alg}(\on{End}(\CS(\CZ))}\otimes \Gamma_{c,\on{Ran}^{\on{untl}}}$$
to $\CH^{\on{inf}}(\CZ)_{\on{untl}}$ and obtain a monad, to be denoted
$$\Gamma_{c,\on{Ran}^{\on{untl}}}(\CH^{\on{inf}}(\CZ)_{\on{untl}})\in \on{Alg}(\on{End}(\IndCoh(\CS(\CZ))).$$

\medskip

The underlying endofunctor of $\IndCoh(\CS(\CZ))$ is given by
\begin{equation} \label{e:monad downstairs}
(\on{Id}\otimes \Gamma_{c,\on{Ran}^{\on{untl}}})\circ \CH^{\on{inf}}(\CZ)_{\on{untl}} \circ
(\on{Id}\otimes p^!_{\Ran^{\on{untl}}}),
\end{equation}
and the monad structure on \eqref{e:monad downstairs}
comes from the monad structure on $\CH^{\on{inf}}(\CZ)_{\on{untl}}$ using 
the monoidal structure on the functor $\Gamma_{c,\on{Ran}^{\on{untl}}}$. 

\sssec{}  \label{sss:modules for pushforward monad}

Note that by construction, we have a canonical identification
\begin{multline} \label{e:modules for integrated monad}
\Gamma_{c,\on{Ran}^{\on{untl}}}(\CH^{\on{inf}}(\CZ)_{\on{untl}})\mod\left(\IndCoh(\CS(\CZ))\right) \simeq \\
\simeq 
\IndCoh(\CS(\CZ)) \underset{\IndCoh(\CS(\CZ) \times \Ran^{\on{untl}})}\times 
\CH^{\on{inf}}(\CZ)_{\on{untl}}\mod\left(\IndCoh(\CS(\CZ)\times \Ran^{\on{untl}})\right),
\end{multline}
where
$$\IndCoh(\CS(\CZ))\to \IndCoh(\CS(\CZ) \times \Ran^{\on{untl}})$$
is the functor $(\on{id}_{\CS(\CZ)}\times p_{\on{Ran}^{\on{untl}}})^!$, and the functor
$$\CH^{\on{inf}}(\CZ)_{\on{untl}}\mod\left(\IndCoh(\CS(\CZ)\times \Ran^{\on{untl}})\right)\to 
\IndCoh(\CS(\CZ) \times \Ran^{\on{untl}})$$
is $\oblv_{\CH^{\on{inf}}(\CZ)}$.

\ssec{The main result: first version}

\sssec{}

Recall that if $\CY$ is a laft-def prestack, we can consider the map
$$p_{\CY,\dr}:\CY\to \CY_{\dr},$$
which fits into the formalism of \secref{sss:nil monad}.

\medskip

We will denote the resulting monad on $\IndCoh(\CY)$ by $\on{Diff}_{\CY}$, so that
$$\on{Diff}_{\CY}\mod(\IndCoh(\CY))\simeq \Dmod(\CY).$$

\sssec{}

We apply this to $\CY=\CS(\CZ)$. Multiplying by $\Ran^{\on{untl}}$, we obtain the monad
$$\on{Diff}_{\CS(\CZ)}\otimes \omega_{\Ran^{\on{untl}}}\in 
\on{Alg}\left(\on{End}(\IndCoh(\CS(\CZ)))\otimes \Dmod(\on{Ran}^{\on{untl}} (X))\right),$$
so that
\begin{multline*}
(\on{Diff}_{\CS(\CZ)}\otimes \omega_{\Ran^{\on{untl}}})\mod\left(\IndCoh(\CS(\CZ)\times \Ran^{\on{untl}})\right) \simeq \\
\simeq \IndCoh(\CS(\CZ))\otimes \Dmod(\Ran^{\on{untl}}).
\end{multline*}

We can also think of $\on{Diff}_{\CS(\CZ)}\otimes \omega_{\Ran^{\on{untl}}}$ as obtained from 
$$\on{Diff}_{\CS(\CZ)}\in \on{Alg}(\on{End}(\IndCoh(\CS(\CZ))))$$ by applying the (monoidal) functor
\begin{multline}\label{e:pullback monoidal} 
\on{Id}_{\on{End}(\IndCoh(\CS(\CZ)))}\otimes (p_{\Ran^{\on{untl}}})^!:
\on{End}(\IndCoh(\CS(\CZ)))\to \\
\to \on{End}(\IndCoh(\CS(\CZ)))\otimes \Dmod(\Ran^{\on{untl}}).
\end{multline}

\sssec{}

Note that the map 
$$\CS(\CZ) \times \Ran^{\on{untl}} \to \CS(\CZ)_{\on{dR}} \times \Ran^{\on{untl}}$$
factors via the map 
$$\fr_{\on{Ran}^{\on{untl}}}: \CS(\CZ) \times \Ran^{\on{untl}} \to \overset{\circ}{\CS}(\CZ)\strut^{\wedge}_{\on{Ran}^{\on{untl}}}.$$

\medskip

From here, we obtain a map 
\begin{equation}\label{e:hecke to dmod}
\CH^{\on{inf}}(\CZ)_{\on{untl}} \to \on{Diff}_{\CS(\CZ)} \boxtimes \ \omega_{\Ran^{\on{untl}}}
\end{equation}
in $\on{Alg}\left(\on{End}(\IndCoh(\CS(\CZ)))\otimes \Dmod(\Ran^{\on{untl}})\right)$.

\sssec{}

Now, suppose that $X$ is proper.  In this case, we have the monoidal(!) functor 
\begin{multline}\label{e:pushforward monoidal} 
\on{Id}_{\on{End}(\IndCoh(\CS(\CZ)))}\otimes \Gamma_{c,\on{Ran}^{\on{untl}}}:
\on{End}(\IndCoh(\CS(\CZ)))\otimes \Dmod(\Ran^{\on{untl}})\to \\
\to \on{End}(\IndCoh(\CS(\CZ))),
\end{multline}
left adjoint to \eqref{e:pullback monoidal}. 

\medskip

Hence, by adjunction, from \eqref{e:hecke to dmod} we obtain a map
\begin{equation}\label{e:hecke to dmod downstairs}
\Gamma_{c,\on{Ran}^{\on{untl}}}(\CH^{\on{inf}}(\CZ)_{\on{untl}})\to \on{Diff}_{\CS(\CZ)}
\end{equation}
in $\on{Alg}(\on{End}(\IndCoh(\CS(\CZ)))$.

\sssec{}

The main result of this paper is:

\begin{thm}\label{t:main}
Let $X$ be a proper connected scheme and $\CZ \to X_{\on{dR}}$ a sectionally laft prestack.  Then the map
\eqref{e:hecke to dmod downstairs} is an isomorphism.
\end{thm}

\sssec{}

Combining with \eqref{e:modules for integrated monad}, we obtain:

\begin{cor}\label{c:unital Ran fiber product}
Let $X$ be a proper connected scheme and $\CZ \to X_{\on{dR}}$ a sectionally laft prestack.  Then
then the pullback functor 
$$ \Dmod(\CS(\CZ)) \to \IndCoh(\CS(\CZ)) \underset{\IndCoh(\CS(\CZ)) \otimes \Dmod(\on{Ran}^{\on{untl}} (X))}{\times} \IndCoh(\overset{\circ}{\CS}(\CZ)\strut^{\wedge}_{\on{Ran}^{\on{untl}} }) $$
is an equivalence of categories.
\end{cor}

\ssec{Localization functor} \label{ss:loc}

Here we will discuss a reformulation of \thmref{t:main}. 

\sssec{}

Consider again the map
$$\fr_{\Ran^{\on{untl}}}:\CS(\CZ)\times \Ran^{\on{untl}}\to \overset{\circ}{\CS}(\CZ)\strut^{\wedge}_{\on{Ran}^{\on{untl}}}.$$

We denote by $\on{pre-Loc}$ the functor
$$\IndCoh(\overset{\circ}{\CS}(\CZ)\strut^{\wedge}_{\on{Ran}^{\on{untl}}})\to \IndCoh(\CS(\CZ))$$
equal to
$$(\on{Id}\otimes \Gamma_{c,\Ran^{\on{untl}}})\circ \fr^!_{\Ran^{\on{untl}}}.$$

We will show that this functor factors as
$$p_{\CS(\CZ),\dr}^!\circ \on{Loc}$$
for a canonically defined functor
$$\on{Loc}:\IndCoh(\overset{\circ}{\CS}(\CZ)\strut^{\wedge}_{\on{Ran}^{\on{untl}}})\to
\IndCoh( \CS(\CZ)_\dr)=\Dmod(\CS(\CZ)),$$
where $p_{\CS(\CZ),\dr}$ is the projection
$$\CS(\CZ)\to \CS(\CZ)_\dr,$$
so that $p_{\CS(\CZ),\dr}^!$ is the forgetful functor
$$\oblv:\Dmod(\CS(\CZ))\to \IndCoh(\CS(\CZ)).$$

Moreover, we will describe the above functor $\on{Loc}$ explicitly.

\sssec{}

We denote by $\fq$ the map
$$\overset{\circ}{\CS}(\CZ)\strut^{\wedge}_{\on{Ran}^{\on{untl}}} \to  \CS(\CZ)_\dr\times \Ran^{\on{untl}}.$$

%

\medskip
Consider the commutative triangle
$$\xymatrix{
\CS(\CZ)\times \Ran^{\on{untl}} \ar[r]^{\fr_{\Ran^{\on{untl}}}} 
\ar[dr]_{p_{\CS(\CZ),\dr}\times \on{id}} 
&  \overset{\circ}{\CS}(\CZ)\strut^{\wedge}_{\on{Ran}^{\on{untl}}} \ar[d]^{\fq} \\
& \CS(\CZ)_\dr\times \Ran^{\on{untl}}}.$$

Define the functor 
$$\on{Loc}:\IndCoh(\overset{\circ}{\CS}(\CZ)\strut^{\wedge}_{\on{Ran}^{\on{untl}}})\to
\IndCoh( \CS(\CZ)_\dr)=\Dmod(\CS(\CZ))$$
to be $$(\on{Id}\otimes \Gamma_{c,\Ran^{\on{untl}}})\circ \fq^\IndCoh_*.$$

\medskip

From the above diagram we obtain by adjunction a natural transformation
$$\fr^!_{\Ran^{\on{untl}}}\to (p_{\CS(\CZ),\dr}\times \on{id})^! \circ \fq^{\IndCoh}_*.$$

Composing with the functor $\on{Id}\otimes \Gamma_{c,\Ran^{\on{untl}}}$, we obtain a natural transformation
\begin{multline} \label{e:loc}
(\on{Id}\otimes \Gamma_{c,\Ran^{\on{untl}}})\circ \fr^!_{\Ran^{\on{untl}}} \to \\
\to (\on{Id}\otimes \Gamma_{c,\Ran^{\on{untl}}})\circ (p_{\CS(\CZ),\dr}\times \on{id})^! \circ \fq^{\IndCoh}_*\simeq
(p_{\CS(\CZ),\dr})^! \circ \Loc.
\end{multline}

We claim:

\begin{cor} \label{c:localization}
The above natural transformation
$$(\on{Id}\otimes \Gamma_{c,\Ran^{\on{untl}}})\circ \fr^!_{\Ran^{\on{untl}}} \to (p_{\CS(\CZ),\dr})^! \circ \Loc$$
is an isomorphism.
\end{cor}

From the above corollary we obtain a canonical isomorphism: 
\begin{equation} \label{e:Loc}
\on{pre-Loc}\simeq p_{\CS(\CZ),\dr}^!\circ \on{Loc},
\end{equation}
as desired. 

\sssec{Proof of \corref{c:localization}}

The category $\IndCoh(\overset{\circ}{\CS}(\CZ)\strut^{\wedge}_{\on{Ran}^{\on{untl}}})$ is generated by
the essential image of the functor 
$$(\fr_{\Ran^{\on{untl}}})^{\IndCoh}_*:\IndCoh(\CS(\CZ)\times \Ran^{\on{untl}})\to 
\IndCoh(\overset{\circ}{\CS}(\CZ)\strut^{\wedge}_{\on{Ran}^{\on{untl}}}).$$

Hence, it is enough to show that the natural transformation
\begin{multline} \label{e:loc again}
(\on{Id}\otimes \Gamma_{c,\Ran^{\on{untl}}})\circ \fr^!_{\Ran^{\on{untl}}} \circ (\fr_{\Ran^{\on{untl}}})^{\IndCoh}_* \to \\
\to (\on{Id}\otimes \Gamma_{c,\Ran^{\on{untl}}})\circ (p_{\CS(\CZ),\dr}\times \on{id})^! \circ \fq^{\IndCoh}_* \circ 
(\fr_{\Ran^{\on{untl}}})^{\IndCoh}_*
\end{multline}
is an isomorphism. 

\medskip

By definition, the left-hand side in \eqref{e:loc again} is
$$(\on{Id}\otimes \Gamma_{c,\Ran^{\on{untl}}})\circ \CH^{\on{inf}}(\CZ)_{\on{untl}}.$$

However, since the functor $\Gamma_{c,\Ran^{\on{untl}}}$ is monoidal, the above expression is canonically
isomorphic to 
$$\Gamma_{c,\Ran^{\on{untl}}}(\CH^{\on{inf}}(\CZ)_{\on{untl}}) \circ (\on{Id}\otimes \Gamma_{c,\Ran^{\on{untl}}}).$$

We rewrite the right-hand side in \eqref{e:loc again} as
$$(\on{Id}\otimes \Gamma_{c,\Ran^{\on{untl}}})\circ (p_{\CS(\CZ),\dr}\times \on{id})^! \circ
(p_{\CS(\CZ),\dr}\times \on{id})^{\IndCoh}_*\simeq 
\on{Diff}_{\CS(\CZ)} \circ (\on{Id}\otimes \Gamma_{c,\Ran^{\on{untl}}}),$$
and the map in \eqref{e:loc again} identifies with the map induced by \eqref{e:hecke to dmod downstairs}.

\medskip

Hence, it is an isomorphism, by \thmref{t:main}.

\qed[\corref{c:localization}]

\ssec{Restricting to the non-unital Ran space} \label{ss:non-unital}

\thmref{t:main} used the unital Ran space in that the definition of the left-hand side, namely, 
the object
$$\Gamma_{c,\on{Ran}^{\on{untl}}}(\CH^{\on{inf}}(\CZ)_{\on{untl}}),$$
when viewed as an \emph{algebra} in $\on{End}(\IndCoh(\CS(\CZ)))$,
relied on the  fact that the functor $\Gamma_{c,\on{Ran}^{\on{untl}}}$ is monoidal. 

\medskip

However, one would like to have a version of this theorem, where instead of $\Ran^{\on{untl}}$
we use $\Ran$, which is a prestack, as opposed to categorical prestack, and thus is a geometric
object in the more conventional sense. 

\sssec{}

Recall that
$$\iota: \Ran \to \Ran^{\on{untl}} $$
denotes the inclusion map. 

\medskip

Denote
$$\overset{\circ}{\CS}(\CZ)\strut^{\wedge}_{\on{Ran}}:=\overset{\circ}{\CS}(\CZ)\strut^{\wedge}_{\on{Ran}^{\on{untl}} } 
\underset{\Ran^{\on{untl}}}{\times} \Ran,$$
and by $\fr_{\on{Ran}}$ the resulting map
\begin{equation} \label{e:res Ran non-untl}
\fr_{\on{Ran}}: \CS(\CZ) \times \Ran \to \overset{\circ}{\CS}(\CZ)\strut^{\wedge}_{\on{Ran}},
\end{equation} 
so that we have a Cartesian diagram
$$ \xymatrix{
\CS(\CZ) \times \Ran \ar[r]\ar[d] & \overset{\circ}{\CS}(\CZ)\strut^{\wedge}_{\on{Ran}} \ar[d] \\
\CS(\CZ) \times \Ran^{\on{untl}} \ar[r] & \overset{\circ}{\CS}(\CZ)\strut^{\wedge}_{\on{Ran}^{\on{untl}} }
}$$

\sssec{}

The main result of this subsection is the following assertion. 

\begin{thm}\label{t:main non-unital}
Let $X$ be a proper connected scheme and $\CZ \to X_{\on{dR}}$ a sectionally laft prestack.  Then
then the functor
$$ \Dmod(\CS(\CZ)) \to \IndCoh(\CS(\CZ)) \underset{\IndCoh(\CS(\CZ)) \otimes \Dmod(\Ran)}{\times} \IndCoh(\overset{\circ}{\CS}(\CZ)\strut^{\wedge}_{\on{Ran}}) $$
is an equivalence of categories.
\end{thm}

Note that \thmref{t:main non-unital} is rather a variant of \corref{c:unital Ran fiber product} than of \thmref{t:main}. The rest of this 
subsection is devoted to the proof of \thmref{t:main non-unital}. We will deduce it from \corref{c:unital Ran fiber product}. 

\sssec{} \label{sss:tensored functor}

Let $\CH^{\on{inf}}(\CZ)$ be the monad on $\IndCoh(\CS(\CZ) \times \Ran)$ corresponding
to the map $\fr_{\on{Ran}}$ of \eqref{e:res Ran non-untl} via the paradigm of \secref{sss:nil monad}. 

\medskip

Consider the (symmetric) monoidal functor
$$\iota^!: \Dmod(\on{Ran}^{\on{untl}} (X)) \to \Dmod(\Ran),$$
and the resulting monoidal functor
\begin{multline*}
(\on{Id}\otimes \iota^!):\on{End}(\IndCoh(\CS(\CZ)))\otimes \Dmod(\Ran^{\on{untl}})\to 
\on{End}(\IndCoh(\CS(\CZ)))\otimes \Dmod(\Ran),
\end{multline*}
where we identify the latter category with 
$$\on{End}_{\Dmod(\Ran)}(\IndCoh(\CS(Z) \times \Ran)),$$
as in \secref{sss:End as tensor product}.

\medskip

Consider the object 
$$(\on{Id}\otimes \iota)^!(\CH^{\on{inf}}(\CZ)_{\on{untl}}) \in \on{Alg}(\on{End}(\IndCoh(\CS(\CZ)))\otimes \Dmod(\Ran)).$$

It is easy to see, however, that there is a canonical isomorphism
\begin{equation} \label{e:restr Hecke to Ran}
\CH^{\on{inf}}(\CZ) \simeq (\on{Id}\otimes \iota)^!(\CH^{\on{inf}}(\CZ)_{\on{untl}})
\end{equation}
as objects of $\on{Alg}(\on{End}(\IndCoh(\CS(\CZ)))\otimes \Dmod(\Ran))$.

%
%

\sssec{Proof of \thmref{t:main non-unital}}

Let $\CY$ be a laft prestack and let $\sM$ be an object in 
$$\on{Alg}\left(\on{End}_{\Dmod(\on{Ran}^{\on{untl}})}\left(\IndCoh(\CY) \otimes\Dmod(\Ran^{\on{untl}})\right)\right),$$
and let $\iota^!(\sM)$ be its pullback to an object of
$$\on{Alg}\left(\on{End}_{\Dmod(\on{Ran})}\left(\IndCoh(\CY) \otimes\Dmod(\Ran)\right)\right).$$

It follows from \corref{c:Ran cofinality} that the functor $\on{Id}\otimes \iota^!$
\begin{multline*}
\sM\mod\left(\IndCoh(\CY) \otimes \Dmod(\on{Ran}^{\on{untl}} (X))\right) \to \\ 
\to \iota^!(\sM)\mod\left(\IndCoh(\CY) \otimes \Dmod(\Ran)\right)
\end{multline*} 
restricts to an equivalence
\begin{multline*}
\IndCoh(\CY) \underset{\IndCoh(\CY) \otimes \Dmod(\on{Ran}^{\on{untl}} (X))}{\times} 
\sM\mod\left(\IndCoh(\CY) \otimes \Dmod(\on{Ran}^{\on{untl}} (X))\right) \to \\
\to 
\IndCoh(\CY) \underset{\IndCoh(\CY) \otimes \Dmod(\Ran)}{\times} 
\iota^!(\sM)\mod\left(\IndCoh(\CY) \otimes \Dmod(\Ran)\right).
\end{multline*} 

Applying this to $\CY:=\CS(\CZ)$ and $\sM:=\CH^{\on{inf}}(\CZ)_{\on{untl}}$, 
and using the isomorphism \eqref{e:restr Hecke to Ran}, we obtain the desired result. 

\qed

%
%

\begin{rem} \label{r:non-unital monad}
Recall the monad $\Gamma_{c,\Ran}(\CH^{\on{inf}}(\CZ)_{\on{untl}})$ acting on $\IndCoh(\CS(\CZ))$.  Its underlying endofunctor
is given by \eqref{e:monad downstairs}.

\medskip

By \thmref{t:Ran cofinality new}, we can equivalently rewrite \eqref{e:monad downstairs} as 
\begin{equation} \label{e:monad downstairs non-unital}
(\on{Id}\otimes \Gamma_{c,\on{Ran}})\circ \CH^{\on{inf}}(\CZ) \circ
(\on{Id}\otimes p^!_{\Ran}),
\end{equation}

This is a more economical way to describe the endofunctor of $\IndCoh(\CS(\CZ))$ underlying $\Gamma_{c,\Ran}(\CH^{\on{inf}}(\CZ)_{\on{untl}})$
in that it only involves prestacks (as opposed to categorical prestacks). The disadvantage of \eqref{e:monad downstairs non-unital} is that this expression 
does not immediately describe the monad structure on $\Gamma_{c,\Ran}(\CH^{\on{inf}}(\CZ)_{\on{untl}})$.

\end{rem}

\ssec{A relative version}

For the sequel, we will need two enhancements of \thmref{t:main} having to do with 
the relative and parameterized versions of the set-up of \thmref{t:main}.

\sssec{}

Suppose we have a map
$$ \xymatrix{
\CZ_1 \ar[rr]\ar[rd] && \CZ_2 \ar[dl] \\
& X_{\on{dR}} 
}$$
of sectionally laft prestacks. Denote
$$ \overset{\circ}{\CS}(\CZ_1/\CZ_2)_{\on{Ran}^{\on{untl}} } := \overset{\circ}{\CS}(\CZ_1)_{\on{Ran}^{\on{untl}} } 
\underset{\overset{\circ}{\CS}(\CZ_2)_{\on{Ran}^{\on{untl}} }}{\times} (\CS(\CZ_2)\times \Ran^{\on{untl}}),$$
and let
$$\overset{\circ}{\CS}(\CZ_1/\CZ_2)\strut^{\wedge}_{\on{Ran}^{\on{untl}}}$$
be the formal completion of $ \overset{\circ}{\CS}(\CZ_1/\CZ_2)_{\on{Ran}^{\on{untl}}}$ along the natural map
$$\CS(\CZ_1)\times \Ran^{\on{untl}}\to \overset{\circ}{\CS}(\CZ_1/\CZ_2)_{\on{Ran}^{\on{untl}}}.$$

Equivalently, we have
$$ \overset{\circ}{\CS}(\CZ_1/\CZ_2)\strut^{\wedge}_{\on{Ran}^{\on{untl}} } \simeq \overset{\circ}{\CS}(\CZ_1)\strut^{\wedge}_{\on{Ran}^{\on{untl}} } \underset{\overset{\circ}{\CS}(\CZ_2)\strut^{\wedge}_{\on{Ran}^{\on{untl}} }}{\times} (\CS(\CZ_2)\times \Ran^{\on{untl}}).$$

\sssec{}

It follows from \propref{p:punctured sects laft-def} that $\overset{\circ}{\CS}(\CZ_1/\CZ_2)\strut^{\wedge}_{\on{Ran}^{\on{untl}} }$
is a laft-def categorical prestack and the restriction map
\begin{equation}
\fr_{\on{Ran}^{\on{untl}} }: \CS(\CZ_1) \times \Ran^{\on{untl}} \to \overset{\circ}{\CS}(\CZ_1/\CZ_2)\strut^{\wedge}_{\on{Ran}^{\on{untl}} }
\end{equation}
is a nil-isomorphism and a coCartesian fibration in groupoids.

\medskip

As in the absolute case, we have the \emph{relative infinitesimal Hecke monad}
$$ \CH^{\on{inf}}(\CZ_1/\CZ_2)_{\on{untl}}:= 
\fr_{\on{Ran}^{\on{untl}} }^! \circ (\fr_{\on{Ran}^{\on{untl}} })_! \in \on{Alg}(\on{End}(\IndCoh(\CS(\CZ_1))) \otimes \Dmod(\on{Ran}^{\on{untl}})).$$

\sssec{}

Let
$$ \on{Diff}_{\CS(\CZ_1)/\CS(\CZ_2)} \in \on{Alg}(\on{End}(\IndCoh(\CS(\CZ_1)))) $$
denote the monad corresponding to push-pull along the nil-isomorphism
$$p_{\CS(\CZ_1),\dr/\CS(\CZ_2)}: \CS(\CZ_1) \to (\CS(\CZ_1)/\CS(\CZ_2))_{\on{dR}},$$
where $(\CS(\CZ_1)/\CS(\CZ_2))_{\on{dR}}$ is the relative de Rham prestack, i.e., 
$$(\CS(\CZ_1)/\CS(\CZ_2))_{\on{dR}}:=\CS(\CZ_1)_{\dr}\underset{\CS(\CZ_2)_{\dr}}\times \CS(\CZ_2),$$
which is the same as the formal completion of 
$\CS(\CZ_2)$ along $\CS(\CZ_1)$. 

\sssec{}

We have the commutative diagram
$$
\xymatrix{
\CS(\CZ_1) \times \Ran^{\on{untl}} \ar[r]\ar[d] & 
\overset{\circ}{\CS}(\CZ_1)\strut^{\wedge}_{\on{Ran}^{\on{untl}} } \ar[r]\ar[d] & \CS(\CZ_1)_{\on{dR}} \times \Ran^{\on{untl}} \ar[d] \\
\CS(\CZ_2) \times \Ran^{\on{untl}} \ar[r] & \overset{\circ}{\CS}(\CZ_2)\strut^{\wedge}_{\on{Ran}^{\on{untl}} } \ar[r] & \CS(\CZ_2)_{\on{dR}} \times \Ran^{\on{untl}}
}
$$
and therefore a map
$$ \overset{\circ}{\CS}(\CZ_1/\CZ_2)\strut^{\wedge}_{\on{Ran}^{\on{untl}} } \to (\CS(\CZ_1)/\CS(\CZ_2))_{\on{dR}} \times \Ran^{\on{untl}} $$
under $\CS(\CZ_1) \times \Ran^{\on{untl}}$, which induces a map 
\begin{equation}\label{e:rel hecke to dmod}
\CH^{\on{inf}}(\CZ_1/\CZ_2)_{\on{untl}} \to \on{Diff}_{\CS(\CZ_1)/\CS(\CZ_2)} \boxtimes\  \omega_{\on{Ran}^{\on{untl}} (X)}
\end{equation}
of objects in $\on{Alg}(\on{End}(\IndCoh(\CS(\CZ_1))) \otimes \Dmod(\on{Ran}^{\on{untl}} (X)))$.

\sssec{}

The relative version of \thmref{t:main} is:

\begin{thm}\label{t:rel main}
Let $X$ be a proper connected scheme and $\CZ_1 \to \CZ_2$ a map of sectionally laft prestacks.  Then
the map \eqref{e:rel hecke to dmod} induces an isomorphism of monads
$$ \Gamma_{c,\on{Ran}^{\on{untl}} }(\CH^{\on{inf}}(\CZ_1/\CZ_2)_{\on{untl}}) \simeq \on{Diff}_{\CS(\CZ_1)/\CS(\CZ_2)}.$$
\end{thm}

Note that \thmref{t:main} is an instance of \thmref{t:rel main} when
$$\CZ_2 = X_{\on{dR}}.$$
In this case, $\CS(\CZ_2) = \on{pt}$ and $\overset{\circ}{\CS}(\CZ_2)\strut^{\wedge}_{\on{Ran}^{\on{untl}}} = \Ran^{\on{untl}}$.

\sssec{}

As in \secref{ss:loc}, from \thmref{t:rel main} we obtain:

\begin{cor} \label{c:loc rel}
The functor 
$$\on{pre-Loc}: (\on{Id}\otimes \Gamma_{c,\on{Ran}^{\on{untl}}})\circ \fr_{\on{Ran}^{\on{untl}}}:
\IndCoh(\overset{\circ}{\CS}(\CZ_1/\CZ_2)\strut^{\wedge}_{\on{Ran}^{\on{untl}}})\to \IndCoh(\CS(\CZ_1))$$
factors canonically as
$$\IndCoh(\overset{\circ}{\CS}(\CZ_1/\CZ_2)\strut^{\wedge}_{\on{Ran}^{\on{untl}}}) \overset{\on{Loc}}\longrightarrow
\IndCoh((\CS(\CZ_1)/\CS(\CZ_2))_{\on{dR}}) \overset{p^!_{\CS(\CZ_1),\dr/\CS(\CZ_2)}}\longrightarrow \IndCoh(\CS(\CZ_1)),$$
where $\on{Loc}$ is the functor
$$(\on{Id}\otimes \Gamma_{c,\on{Ran}^{\on{untl}}}) \circ \fq^\IndCoh_*,$$
where
$$\fq: \overset{\circ}{\CS}(\CZ_1/\CZ_2)\strut^{\wedge}_{\on{Ran}^{\on{untl}}} \to 
(\CS(\CZ_1)/\CS(\CZ_2))_{\on{dR}}\times \Ran^{\on{untl}}.$$
\end{cor} 

\sssec{}
As in the absolute case, we obtain a local-to-global description of the category $\IndCoh((\CS(\CZ_1)/\CS(\CZ_2))_{\on{dR}})$:

\begin{cor}\label{c:unital Ran rel fiber product}
Under the assumptions of \thmref{t:rel main}, the functor
\begin{multline*}
\IndCoh((\CS(\CZ_1)/\CS(\CZ_2))_{\on{dR}})\to \\
\to \IndCoh(\CS(\CZ_1)) \underset{\IndCoh(\CS(\CZ_1)) \otimes \Dmod(\on{Ran}^{\on{untl}} (X))}{\times}
\IndCoh(\overset{\circ}{\CS}(\CZ_1/\CZ_2)\strut^{\wedge}_{\on{Ran}^{\on{untl}} }) 
\end{multline*}
is an equivalence of categories.
\end{cor}

Finally, as in \secref{ss:non-unital}, from \corref{c:unital Ran rel fiber product} we obtain:

\begin{thm} \label{t:main non-unital rel}
Under the assumptions of \thmref{t:rel main}, the functor
\begin{multline*}
\IndCoh((\CS(\CZ_1)/\CS(\CZ_2))_{\on{dR}})\to \\
\to \IndCoh(\CS(\CZ_1)) \underset{\IndCoh(\CS(\CZ_1)) \otimes \Dmod(\Ran)}{\times}
\IndCoh(\overset{\circ}{\CS}(\CZ_1/\CZ_2)\strut^{\wedge}_{\on{Ran}}) 
\end{multline*}
is an equivalence of categories.
\end{thm}

\ssec{Parameterized version}

\sssec{}

Let $\CZ_1 \to \CZ_2$ be once again a map of sectionally laft prestacks.  Additionally, let $\CY$ be a laft-def prestack and suppose that
we have a map
$$ \CY \to \CS(\CZ_1) .$$
Let
$$ \overset{\circ}{\CS}(\CY/\CZ_1/\CZ_2)\strut^{\wedge}_{\on{Ran}^{\on{untl}} } := (\CY / \overset{\circ}{\CS}(\CZ_1/\CZ_2)_{\on{Ran}^{\on{untl}} })_{\on{dR}} $$
denote the formal completion of $\overset{\circ}{\CS}(\CZ_1/\CZ_2)_{\on{Ran}^{\on{untl}} }$ along the composite
$$ \CY \times \Ran^{\on{untl}} \to \CS(\CZ_1) \times \Ran^{\on{untl}} \to 
\overset{\circ}{\CS}(\CZ_1/\CZ_2)_{\on{Ran}^{\on{untl}} } ,$$
and let
$$ \CH^{\on{inf}}(\CY/\CZ_1/\CZ_2)_{\on{untl}} \in \on{Alg}(\on{End}(\IndCoh(\CY))\otimes \Dmod(\on{Ran}^{\on{untl}} (X))) $$
denote the monad corresponding to the nil-isomorphism
$$ \CY \times \Ran^{\on{untl}} \to \overset{\circ}{\CS}(\CY/\CZ_1/\CZ_2)\strut^{\wedge}_{\on{Ran}^{\on{untl}} } .$$

\medskip

As before, we have a map of algebras
\begin{equation}\label{e:enh hecke to dmod}
\CH^{\on{inf}}(\CY/\CZ_1/\CZ_2)_{\on{untl}} \to \on{Diff}_{\CY/\CS(\CZ_2)} \boxtimes \ \omega_{\on{Ran}^{\on{untl}} (X)}
\end{equation}
where $\on{Diff}_{\CY/\CS(\CZ_2)} \in \on{Alg}(\on{End}(\IndCoh(\CY))$ is the monad corresponding to push-pull along the nil-isomorphism
$$ \CY \to (\CY/\CS(\CZ_2))_{\on{dR}}:=\CY_\dr\underset{\CS(\CZ_2)_\dr}\times \CS(\CZ_2) .$$

\sssec{}
We can now state the parameterized version of \thmref{t:rel main}.

\begin{thm}\label{t:enh main}
Let $X$ be a proper connected scheme and $\CZ_1 \to \CZ_2$ a map of sectionally laft prestacks.  Let $\CY \to \CS(\CZ_1)$ 
be a laft-def prestack over $\CS(\CZ_1)$. Then
the map \eqref{e:enh hecke to dmod} induces an isomorphism of monads
$$ \Gamma_{c,\on{Ran}^{\on{untl}} }(\CH^{\on{inf}}(\CY/\CZ_1/\CZ_2)_{\on{untl}}) \simeq \on{Diff}_{\CY/\CS(\CZ_2)}.$$
\end{thm}

Note that \thmref{t:rel main} is an instance of \thmref{t:enh main} in the case that $\CY = \CS(\CZ_1)$.

\sssec{}

As in the absolute case, from \thmref{t:enh main}, we obtain:

\begin{cor}\label{c:unital Ran enh fiber product}
Under the assumptions of \thmref{t:enh main}, the functor
$$ \IndCoh((\CY/\CS(\CZ_2))_{\on{dR}})\to \IndCoh(\CY) \underset{\IndCoh(\CY) \otimes \Dmod(\on{Ran}^{\on{untl}} (X))}{\times} 
\IndCoh(\overset{\circ}{\CS}(\CY/\CZ_1/\CZ_2)\strut^{\wedge}_{\on{Ran}^{\on{untl}} }) $$
is an equivalence of categories.
\end{cor}

\begin{thm} \label{t:main non-unital enh}
Under the assumptions of \thmref{t:enh main}, the functor
$$ \IndCoh((\CY/\CS(\CZ_2))_{\on{dR}})\to \IndCoh(\CY) \underset{\IndCoh(\CY) \otimes \Dmod(\Ran)}{\times} 
\IndCoh(\overset{\circ}{\CS}(\CY/\CZ_1/\CZ_2)\strut^{\wedge}_{\on{Ran}}) $$
is an equivalence of categories.
\end{thm}

\ssec{The push-out picture}

\sssec{}
In this subsection, we will prove the following variant of \thmref{t:main non-unital rel}.

\begin{thm}\label{t:push-out}
Let $X$ be a proper connected scheme and $\CZ_1 \to \CZ_2$ be a map of sectionally laft prestacks.
Then the diagram
$$
\xymatrix{
 \CS(\CZ_1) \times \Ran \ar[r]\ar[d] & \overset{\circ}{\CS}(\CZ_1/\CZ_2)\strut^{\wedge}_{\on{Ran}} \ar[d] \\
\CS(\CZ_1) \ar[r] & (\CS(\CZ_1)/\CS(\CZ_2))_{\on{dR}}
}$$
is a push-out square in $\on{PreStk}_{\on{laft-def}}$.
\end{thm}

\sssec{}
Specializing to the case that $\CZ_2 = X_{\on{dR}}$, we obtain:

\begin{thm}
Let $X$ be a proper connected scheme and $\CZ \to X_{\on{dR}}$ a sectionally laft prestack.
Then the diagram
$$
\xymatrix{
 \CS(\CZ) \times \Ran \ar[r]\ar[d] & \overset{\circ}{\CS}(\CZ)\strut^{\wedge}_{\on{Ran}} \ar[d] \\
\CS(\CZ) \ar[r] & \CS(\CZ)_{\on{dR}}
}$$
is a push-out square in $\on{PreStk}_{\on{laft-def}}$.
\end{thm}

\sssec{}

The rest of this subsection is devoted to the proof of \thmref{t:push-out}. However, before we begin the proof, 
we describe how push-outs in $\on{PreStk}_{\on{laft-def}}$ interact with $\IndCoh$:

\begin{prop}\label{p:indcoh on push-out}
Let
$$\xymatrix{
\CX \ar[r]\ar[d]_f & \CZ \ar[d] \\
\CY \ar[r] & \CW
}$$
be a push-out square in $\on{PreStk}_{\on{laft-def}}$ such that the map $\CX \to \CZ$ is a nil-isomorphism
and $f:\CX \to \CY$ is pseudo-proper. Then the pullback functor
$$ \IndCoh(\CW) \to \IndCoh(\CY) \underset{\IndCoh(\CX)}{\times} \IndCoh(\CZ) $$
is an equivalence.
\end{prop}
\begin{proof}
Recall from \cite[Chapter 9, Sect. 5.1]{Vol2} that given a nil-isomorphism $\CX \to \CZ$, we have a sequence
of prestacks in $\on{PreStk}_{\on{laft-def}}$
$$ \CX = \CX^{(0)} \to \CX^{(1)} \to \ldots \to \CX^{(n)} \to \ldots \to \CZ $$
such that each $\CX^{(n)}$ is obtained from $\CX^{(n-1)}$ by a square-zero extension (in the sense of \cite[Chapter 8, Sect. 5.1]{Vol2}) and the natural map
$$ \underset{i}{\on{colim}}\  \CX^{(i)} \to \CZ $$
is an isomorphism (where the colimit is taken in $\on{PreStk}$, and in particular, $\on{PreStk}_{\on{laft-def}}$).
Thus, we can assume without loss of generality that the map $\CX \to \CZ$ is a square-zero extension, which is classified
by a map
$$ \CM \to T(\CX) \in \IndCoh(\CX), \quad \CM\in \IndCoh(\CX).$$

In this case, by \cite[Chapter 8, Lemma 5.1.3]{Vol2}, for a laft-def prestack $\CZ'$, the space
of maps
$$ \CW \to \CZ' $$
is identified with the space of pairs
$$ \{ (\CY \to \CZ', \text{null-homotopy of the composite } \CM \to T(\CX) \to T(\CZ')|_{\CX}) \}.$$

Since $f$ is pseudo-proper, the functor $f^!$ admits a left adjoint $f_!$,
which is compatible with base change and satisfies the projection formula.

\medskip
 
By adjunction, the space of  null-homotopies of the composition 
$$\CM \to T(\CX) \to T(\CZ')|_{\CX}$$
is isomorphic to the space of null-homotopies of the composition
$$f_!(\CM) \to f_!(T(\CX)) \to T(\CY) \to T(\CZ')|_{\CY}.$$

Therefore, by \cite[Chapter 8, Lemma 5.1.3]{Vol2}, $\CW$ identifies with the square zero extension of $\CY$
classified by the the composite map
$$f_!(\CM) \to f_!(T(\CX)) \to T(\CY) .$$
The desired result now follows from \cite[Chapter 8, Theorem 6.3.3]{Vol2} (using the projection formula for $f_!$).
\end{proof}

\sssec{}
We will deduce \thmref{t:push-out} from \thmref{t:enh main} using the following partial converse of 
\propref{p:indcoh on push-out} (which will be proved in \secref{ss:iso lemma proof}):

\begin{thm}\label{t:indcoh iso}
Let $f:\CX \to \CY$ be a nil-isomorphism of laft inf-schemes such that
$$ f^!: \IndCoh(\CY) \to \IndCoh(\CX) $$
is an equivalence of categories.  Then $f$ is an isomorphism.
\end{thm}

\sssec{Proof of \thmref{t:push-out}, Step 1}
Let $S$ be an affine scheme almost of finite type with a map to $\CS(\CZ_1)$.
We have a commutative square
\begin{equation}\label{e:sch push-out square}
\xymatrix{
S \times \Ran \ar[d]\ar[r] & \overset{\circ}{\CS}(S/\CZ_1/\CZ_2)\strut^{\wedge}_{\on{Ran}} \ar[d] \\
S \ar[r]& (S/\CS(\CZ_2))_{\on{dR}}
}
\end{equation}
in $\on{PreStk}_{\on{laft-def}}$.  We claim that \eqref{e:sch push-out square} is a push-out square in $\on{PreStk}_{\on{laft-def}}$.

\medskip

Indeed, let $\CW$ be the push-out
$$S \underset{S \times \Ran}{\sqcup} \overset{\circ}{\CS}(S/\CZ_1/\CZ_2)\strut^{\wedge}_{\on{Ran}} $$
in $\on{PreStk}_{\on{laft-def}}$, and let
$$ f: \CW \to (S/\CS(\CZ_2))_{\on{dR}} $$
denote the map corresponding to \eqref{e:sch push-out square}.  By \propref{p:indcoh on push-out}, the functor
$$ \IndCoh(\CW) \to \IndCoh(S) \underset{\IndCoh(S \times \Ran)}{\times} \IndCoh(\overset{\circ}{\CS}(S/\CZ_1/\CZ_2)\strut^{\wedge}_{\on{Ran}}) $$
is an equivalence.  Thus, by \thmref{t:main non-unital enh}, the functor
$$ f^!: \IndCoh(\CW) \to \IndCoh( (S/\CS(\CZ_2))_{\on{dR}}) $$
is an equivalence.  

\medskip

Hence, $\CW\to (S/\CS(\CZ_2))_{\on{dR}}$ is an isomorphism by \thmref{t:indcoh iso}.

\sssec{Proof of \thmref{t:push-out}, Step 2}

We have
$$\CS(\CZ_1) = \underset{S \in (\affSch_{\on{aft}})_{/\CS(\CZ_1)}}{\on{colim}} S ,$$
where the colimit is taken in $\on{PreStk}$.
To deduce the theorem, it suffices to show that the natural maps
\begin{enumerate}
\item
$\underset{S \in (\affSch_{\on{aft}})_{/\CS(\CZ_1)}}{\on{colim}} (S \times \Ran) \to \CS(\CZ_1) \times \Ran$
\item
$\underset{S \in (\affSch_{\on{aft}})_{/\CS(\CZ_1)}}{\on{colim}} (S/\CS(\CZ_2))_{\on{dR}} \to (\CS(\CZ_1)/\CS(\CZ_2))_{\on{dR}}$
\item
$\underset{S \in (\affSch_{\on{aft}})_{/\CS(\CZ_1)}}{\on{colim}} \overset{\circ}{\CS}(S/\CZ_1/\CZ_2)\strut^{\wedge}_{\on{Ran}} \to \overset{\circ}{\CS}(\CZ_1/\CZ_2)\strut^{\wedge}_{\on{Ran}}$
\end{enumerate}
are isomorphisms, where the colimit is taken in $\on{PreStk}$ (and therefore also $\on{PreStk}_{\on{laft-def}}$).

\medskip

The assertion that (1) is an isomorphism is \cite[Chapter 1, Theorem 9.1.4]{Vol2}, and (2) and (3) follow from the following general 
fact (see \cite[Lemma 1.1.4]{Crys}): given a prestack $\CX$, the functor
$$ \on{PreStk}_{/\CX} \to \on{PreStk} $$
given by
$$ (\CY \to \CX) \mapsto (\CY/\CX)_{\on{dR}}:=\CY_{\on{dR}} \underset{\CX_{\on{dR}}}{\times} \CX $$
preserves colimits.

\qed

\ssec{Proof of \thmref{t:indcoh iso}}  \label{ss:iso lemma proof}

\sssec{Step 1} \label{sss:iso lemma proof 1}

Consider the fiber square
$$ \xymatrix{
\CX \underset{\CY}{\times} \CX \ar[r]^-{p_1} \ar[d]_-{p_2} & \CX \ar[d]^f \\
\CX \ar[r]^f & \CY
}$$
By assumption, $f^!$ is an equivalence.  Therefore, by base change, we have 
$$ (p_2)^{\IndCoh}_* \circ p_1^! \simeq f^! \circ f^{\IndCoh}_*\simeq \on{Id} \in \on{End}(\IndCoh(\CX)) \simeq \IndCoh(\CX \times \CX) .$$
It follows that
$$ (p_1 \times p_2)^{\IndCoh}_*(\omega_{\CX \underset{\CY}{\times} \CX}) \simeq \Delta^\IndCoh_*(\omega_{\CX}), $$
where $\Delta: X_1 \to X_1 \times X_1$ is the diagonal map. 

\medskip

In the above formula, for a nil-isomorphism $\phi:\CX_1\to \CX_2$, the functor 
$\phi^{\IndCoh}_*:\IndCoh(\CX_1)\to  \IndCoh(\CX_2)$
is the left adjoint to $\phi^!$, by \cite[Chapter 3, Sect. 4]{Vol2}. 

\medskip

Since $\CX$ is an inf-scheme, the pushforward functor
$$ (p_1 \times p_2)^{\IndCoh}_*: \IndCoh(\CX \underset{\CX_\dr}{\times} \CX) \to \IndCoh(\CX \times \CX) $$
is fully faithful.  Since $f:\CX \to \CY$ is a nil-isomorphism, $\CX_\dr \simeq \CY_\dr$.  Therefore, we have a map
$$ p_1 \underset{\CX_\dr}{\times} p_2: \CX \underset{\CY}{\times}\CX \to \CX \underset{\CX_{\dr}}{\times}\CX .$$
Thus,
\begin{equation}\label{e:iso on formal square}
(p_1 \underset{\CX_\dr}{\times} p_2)^{\IndCoh}_*(\omega_{\CX \underset{\CY}{\times} \CX}) \simeq \hat{\Delta}^{\IndCoh}_*(\omega_{\CX}),
\end{equation}
as objects of $\IndCoh(\CX \underset{\CX_{\dr}}{\times}\CX)$, 
where $\hat{\Delta}: \CX \to  \CX \underset{\CX_{\dr}}{\times}\CX$ is the diagonal map.

\sssec{Step 2} 

Now, let $\CZ$ and $\CW$ be defined so that both squares in the diagram below are Cartesian
\begin{equation}\label{e:formal loop spaces}
\xymatrix{
\CW \ar[r]^-h\ar[d] & \CZ \ar[r]^-g\ar[d] & \CX \ar[d]^{\hat{\Delta}} \\
\CX \ar[r] & \CX \underset{\CY}{\times}\CX \ar[r] & \CX \underset{\CX_\dr}{\times} \CX
}
\end{equation}
By base change and \eqref{e:iso on formal square}, the natural map
\begin{equation}\label{e:iso on omegas}
(h \circ g)^{\IndCoh}_*(\omega_\CW) \to g^{\IndCoh}_*(\omega_\CZ)
\end{equation}
is an isomorphism. 

\medskip

By definition, $\CW$ is the infinitesimal inertia group of $\CX$ (see \cite[Chapter 8, Sect. 1.2.3]{Vol2}),
and is, in particular, a formal group over $\CX$.  It is straightforward to see that we have a fiber square
$$ \xymatrix{
\CZ \ar[d]_g \ar[r] & \on{Inert}^{\on{inf}}(\CY) \ar[d] \\
\CX \ar[r]^f & \CY
},$$
where $\on{Inert}^{\on{inf}}(\CY)$ is the infinitesimal inertia group of $\CY$.
Therefore, $g: \CZ \to \CX$ is also a formal group over $\CX$.

\medskip

By \cite[Chapter 7, Corollary 3.2.2]{Vol2}, formal groups are inf-affine.  In particular, the isomorphism
\eqref{e:iso on omegas} implies that the map $h: \CW \to \CZ$ is an isomorphism.

\sssec{A digression}

To proceed, we will make use of the following basic fact:

\begin{prop}\label{p:inf fiber product conservative}
Let
$$
\xymatrix{
\CY_1 \ar[r]^{f_1} \ar[d] & \CX_1 \ar[d]\\
\CY_2 \ar[r] & \CX_2
}
$$
be a Cartesian square of laft-def prestacks such that all maps are nil-isomorphisms.  If the map $\CY_1 \to \CY_2$
is an isomorphism, so is the map $\CX_1 \to \CX_2$.
\end{prop}

\begin{proof}
Since $\CY_1 \to \CY_2$ is an isomorphism, we have 
$$ T(\CY_1 /\CY_2) \simeq f_1^!(T(\CX_1/\CX_2)) = 0 .$$
Since $f_1$ is a nil-isomorphism, the functor $f^!_1$ is conservative.  Therefore, $T(\CX_1/\CX_2)=0$ and since
the map $\CX_1 \to \CX_2$ is also a nil-isomorphism, it is an isomorphism.
\end{proof}

\sssec{Step 3} \label{sss:iso lemma proof 3}

Since $h$ is an isomorphism, we can apply \propref{p:inf fiber product conservative} to the two Cartesian squares
$$
\xymatrix{
\CW \ar[r]^h\ar[d] & \CZ\ar[d] \\
\CX \ar[r] & \CX\underset{\CY}{\times} \CX \ar[d]\ar[r] & \CX\ar[d] \\
 & \CX\ar[r]^f & \CY 
}$$
to obtain that $f$ is an isomorphism, as desired.

\qed[\thmref{t:indcoh iso}]

\section{Proof of \thmref{t:enh main}}\label{s:proof}

\ssec{Recollections on deformation to the normal cone}
In this subsection, we recall from \cite[Chapter 9]{Vol2} the construction and basic properties of deformation to the normal cone, a fundamental construction of derived algebraic geometry.  It will play a key role in the proof of \thmref{t:enh main}.  Specifically,
it will allow us to reduce the theorem to a statement about tangent spaces.

\sssec{}
Suppose we have a diagram
$$
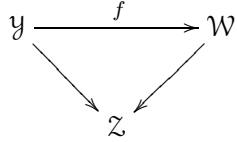
\xymatrix{
\CY \ar[rr]^f \ar[dr] && \CW \ar[dl] \\
 & \CZ
}
$$
of laft prestacks such that $f$ is a nil-isomorphism.  In this case, the adjunction $(f^{\IndCoh}_*,f^!)$ gives a monad
on $\IndCoh(\CY)$.  The symmetric monoidal category $\IndCoh(\CZ)$ acts on both $\IndCoh(\CY)$ and $\IndCoh(\CW)$
and both functors $f^{\IndCoh}_*$ and $f^!$ are canonically functors of $\IndCoh(\CZ)$-linear DG categories.  Therefore,
we have an object
$$ \sM_f \in \on{Alg}(\on{End}_{\IndCoh(\CZ)}(\IndCoh(\CY))) $$
such that the underlying endofunctor of $\sM_f$ is given by
$$f^! \circ f^{\IndCoh}_*: \IndCoh(\CY) \to \IndCoh(\CY).$$

\medskip

The above construction of $\sM_f$ is functorial in the followng sense.  Let
$$ (\on{PreStk}_{\on{laft}})^{\on{nil-isom}}_{\CY//\CZ} \subset (\on{PreStk}_{\on{laft}})_{\CY//\CZ} $$
denote the full subcategory consisting of objects of the form $\CY \to \CW \to \CZ$ such that the map
$\CY \to \CW$ is a nil-isomorphism.  We obtain a functor
\begin{equation}\label{e:pushpull monad}
(\on{PreStk}_{\on{laft}})^{\on{nil-isom}}_{\CY//\CZ} \to \on{Alg}(\on{End}_{\IndCoh(\CZ)}(\IndCoh(\CY)))
\end{equation}
which maps $\CY \overset{f}{\to} \CW \to \CZ$ to $\sM_f$.

\medskip
Forgetting the algebra structure, we obtain a functor
\begin{equation}\label{e:pushpull funct}
(\on{PreStk}_{\on{laft}})^{\on{nil-isom}}_{\CY//\CZ} \to \on{End}_{\IndCoh(\CZ)}(\IndCoh(\CY))
\end{equation}

\sssec{}

The construction of \cite[Chapter 9, Section 2]{Vol2} gives a functor
\begin{equation}
\on{DefNorm}: (\on{PreStk_{\on{laft-def}}})^{\on{nil-isom}}_{\CY//\CZ} \to
((\on{PreStk}_{\on{laft-def}})^{\on{nil-isom}}_{\CY \times \mathbb{A}^1//\CZ \times \mathbb{A}^1})^{\mathbb{A}^1_{\on{left-lax}}}
\end{equation}
(see \cite[Chapter 9, Section 1]{Vol2} for an explanation of what $\mathbb{A}^1_{\on{left-lax}}$-equivariance means; here
we will only use a few formal properties described in \secref{sss:left-lax} below).
The composite
$$ (\on{PreStk_{\on{laft-def}}})^{\on{nil-isom}}_{\CY//\CZ} \to
((\on{PreStk}_{\on{laft-def}})^{\on{nil-isom}}_{\CY \times \mathbb{A}^1//\CZ \times \mathbb{A}^1})^{\mathbb{A}^1_{\on{left-lax}}} 
\xrightarrow{- \underset{\mathbb{A}^1}{\times} \{1\}}
(\on{PreStk}_{\on{laft-def}})^{\on{nil-isom}}_{\CY//\CZ}
$$
is isomorphic to the identity, and the composite
$$ (\on{PreStk_{\on{laft-def}}})^{\on{nil-isom}}_{\CY//\CZ} \to
((\on{PreStk}_{\on{laft-def}})^{\on{nil-isom}}_{\CY \times \mathbb{A}^1//\CZ \times \mathbb{A}^1})^{\mathbb{A}^1_{\on{left-lax}}} 
\xrightarrow{- \underset{\mathbb{A}^1}{\times} \{0\}}
(\on{PreStk}_{\on{laft-def}})^{\on{nil-isom}}_{\CY//\CZ}
$$
is given by
$$ (\CY \to \CW \to \CZ) \mapsto (\CY \to \on{Vect}_{\CY}(T(\CY/\CW)) \to \CZ), $$
where $\on{Vect}_{\CY}(T(\CY/\CW))$ is the ``infinitesimal vector prestack'' associated to the relative tangent complex
$T(\CY/\CW) \in \IndCoh(\CY)$ (see \cite[Chapter 7, Section 1.4]{Vol2}) and the map $\on{Vect}_{\CY}(T(\CY/\CW))\to \CZ$ is
the composite
$$ \on{Vect}_{\CY}(T(\CY/\CW)) \to \CY \to \CZ ,$$
where the first map is the canonical projection.

\sssec{}
Composing the functor $\on{DefNorm}$ with \eqref{e:pushpull monad}, we obtain a functor

\begin{equation}\label{e:filtered monad prelim}
(\on{PreStk_{\on{laft-def}}})^{\on{nil-isom}}_{\CY//\CZ} \to \on{Alg}(\on{End}_{\IndCoh(\CZ \times \mathbb{A}^1)}(\IndCoh(\CY\times \mathbb{A}^1)))^{\mathbb{A}^1_{\on{left-lax}}}
\end{equation}

The category $\IndCoh(\mathbb{A}^1)$ is dualizable, and therefore
$$ \IndCoh(\CY \times \mathbb{A}^1) \simeq \IndCoh(\CY) \otimes \IndCoh(\mathbb{A}^1) \simeq \IndCoh(\CY) \otimes \QCoh(\mathbb{A}^1).$$
Moreover, as in \lemref{l:dualizable endomorphisms}, we have:
$$ \on{End}_{\IndCoh(\CZ) \otimes \QCoh(\mathbb{A}^1)}(\IndCoh(\CY) \otimes \QCoh(\mathbb{A}^1)) \simeq
\on{End}_{\IndCoh(\CZ)}(\IndCoh(\CY)) \otimes \QCoh(\mathbb{A}^1) .$$

\sssec{} \label{sss:left-lax}

By \cite[Chapter 9, Lemma 1.5.2(a)]{Vol2}, for a DG category $\bC$,
$$ (\bC \otimes \QCoh(\mathbb{A}^1))^{\mathbb{A}^1_{\on{left-lax}}} \simeq \bC^{\on{Fil}, \geq 0}:=\on{Funct}(\mathbb{Z}_{\geq 0}, \bC), $$
where $\mathbb{Z}_{\geq 0}$ is the category corresponding to the poset of non-negative integers.  In other words,
$(\bC \otimes \QCoh(\mathbb{A}^1))^{\mathbb{A}^{\on{left-lax}}}$ is the category of non-negatively filtered objects
in $\bC$.
With this identification, the forgetful functor
$$ \oblv_{\on{Fil}}: \bC^{\on{Fil}, \geq 0} \to \bC $$
is given by evaluation at $1 \in \mathbb{A}^1$, and the associated graded functor
\begin{equation}\label{e:gr functor}
 \on{gr}: \bC^{\on{Fil}, \geq 0} \to \bC
\end{equation}
is given by evaluation at $0 \in \mathbb{A}^1$.  Note that because we only consider non-negatively filtered objects, the associated graded functor $\eqref{e:gr functor}$ is conservative.

\medskip

In the case when $\bC$ is a monoidal DG category, we have
$$ \on{Alg}(\bC \otimes \QCoh(\mathbb{A}^1))^{\mathbb{A}^1_{\on{left-lax}}} \simeq \on{Alg}(\bC^{\on{Fil, \geq 0}}).$$
Applying the above equivalences, \eqref{e:filtered monad prelim} gives a functor
\begin{equation}\label{e:filtered monad}
(\on{PreStk_{\on{laft-def}}})^{\on{nil-isom}}_{\CY//\CZ} \to \on{Alg}(\on{End}_{\IndCoh(\CZ)}(\IndCoh(\CY))^{\on{Fil, \geq 0}})
\end{equation}

\sssec{}
Below, we will summarize the basic relevant properties of the functor \eqref{e:filtered monad}.  First, we introduce
an additional bit of notation.

\medskip

Given a monoidal functor $\mathcal{A}\to\mathcal{B}$
of monoidal DG categories, we have a canonically defined monoidal functor
$$ \on{Tens}: \mathcal{B} \simeq \on{End}_{\mathcal{B}}(\mathcal{B}) \to \on{End}_{\mathcal{A}}(\mathcal{B}) .$$
Explicitly, for $b \in \mathcal{B}$, the functor $\on{Tens}(b)$ is given by tensoring with $b$:
$$ \on{Tens}(b) = (-) \otimes_{\mathcal{B}} b .$$

Denote by $\on{Alg}(\on{Tens})$ the induced functor
$$\on{Alg}(\CB) \to \on{Alg}(\on{End}_{\mathcal{A}}(\mathcal{B})).$$ 

\sssec{}
The relevant properties of the functor \eqref{e:filtered monad} established in \cite[Chapter 9]{Vol2} can be summarized as follows:
\begin{thm}\label{t:def normal cone}
Let $\CY \to \CZ$ be a morphism of laft-def prestacks.  We then have the following commutative diagram.
$$
\xymatrix{&  \on{Alg}(\on{End}_{\IndCoh(\CZ)}(\IndCoh(\CY))) \\
(\on{PreStk}_{\on{laft-def}})_{\CY//\CZ}^{\on{nil-isom}} \ar[r]^-{\eqref{e:filtered monad}}
\ar[ur]^{\eqref{e:pushpull monad}} \ar[dr]_{\on{Alg}(\on{Tens}) \circ \on{Sym}\circ T_{\CY/}}
& \on{Alg}(\on{End}_{\IndCoh(\CZ)}(\IndCoh(\CY))^{\on{Fil}, \geq 0}) 
\ar[u]_{\oblv_{\on{Fil}}} \ar[d]^{\on{gr}} \\
& \on{Alg}(\on{End}_{\IndCoh(\CZ)}(\IndCoh(\CY))),}
$$
where the pointing down slanted arrow is the composite
\begin{multline*}
(\on{PreStk}_{\on{laft-def}})_{\CY//\CZ}^{\on{nil-isom}}  \overset{T_{\CY/}}\longrightarrow \IndCoh(\CY) 
\overset{\on{Sym}}\longrightarrow  \on{Alg}(\IndCoh(\CY)) \overset{\on{Alg}(\on{Tens})}\longrightarrow \\
\to \on{Alg}(\on{End}_{\IndCoh(\CZ)}(\IndCoh(\CY))).
\end{multline*} 
\end{thm}

\ssec{Initial Reductions}
We are now ready to begin the proof of \thmref{t:enh main}.  In this subsection, we will reduce the proof of the theorem
to a statement about tangent complexes.

\begin{rem}

The proof of \thmref{t:enh main} will go back-and-forth between the unital and non-unital Ran spaces:

\medskip

On the one hand, the non-unital Ran space is a \emph{prestack}, so is a geometric object of a more 
familiar kind, and is suitable for certain operations (e.g., deformation to the normal cone, discussed
above).

\medskip

On the other hand, the unital Ran space has the feature that the functor $\Gamma_{c,\on{Ran}^{\on{untl}}}$
is (symmetric) monoidal \emph{and} it affords a more convenient description of linear factorization sheaves,
and our proof will make use of both these structures. 

\end{rem} 

\sssec{}

First, we will reduce the theorem to a statement about the non-unital Ran space (though we will later return to the unital Ran space).  

\medskip

By universal homological cofinality of the map $\on{Ran}\to \on{Ran}^{\on{untl}}$ (\thmref{t:Ran cofinality new}),
the functor
$$\Gamma_{c,\on{Ran}^{\on{untl}}}:
 \on{End}(\IndCoh(\CY)) \otimes \Dmod(\on{Ran}^{\on{untl}}) \to \on{End}(\IndCoh(\CY)) $$
is canonically isomorphic to the composite
\begin{multline*}
\on{End}(\IndCoh(\CY)) \otimes \Dmod(\on{Ran}^{\on{untl}}) \overset{\on{Id}\otimes \iota^!}\longrightarrow  \\
\to \on{End}(\IndCoh(\CY)) \otimes \Dmod(\on{Ran}) \overset{\Gamma_{c,\on{Ran}}}\longrightarrow 
\on{End}(\IndCoh(\CY)) .
\end{multline*} 

Recall (see \eqref{e:restr Hecke to Ran}) that
$$(\on{Id}\otimes \iota^!)(\CH^{\on{inf}}(\CY/\CZ_1/\CZ_2)_{\on{Ran}^{\on{untl}}})\simeq  \CH^{\on{inf}}(\CY/\CZ_1/\CZ_2)_{\on{Ran}}.$$

Hence, we need to show that the map of endofunctors
\begin{equation}\label{e:hecke to dmod on Ran int}
\Gamma_{c,\on{Ran}}(\CH^{\on{inf}}(\CY/\CZ_1/\CZ_2)_{\on{Ran}}) \to  \on{Diff}_{\CY/\CS(\CZ_2)}
\end{equation}
arising by adjunction from 
\begin{equation}\label{e:hecke to dmod on Ran}
\CH^{\on{inf}}(\CY/\CZ_1/\CZ_2)_{\on{Ran}} \to \on{Diff}_{\CY/\CS(\CZ_2)} \boxtimes \ \omega_{\on{Ran}}
\end{equation}
is an isomorphism (as objects of $\on{End}(\IndCoh(\CY))$).

\medskip

Note also that since $\Ran$ is universally homologically contractible (\thmref{t:Ran contr}), the map \eqref{e:hecke to dmod on Ran int}
can be obtained from the map \eqref{e:hecke to dmod on Ran} by applying the functor $\Gamma_{c,\on{Ran}}$ to both sides. 

\sssec{}

By definition, the map \eqref{e:hecke to dmod on Ran} arises by applying the functor \eqref{e:pushpull funct} to the map
\begin{equation}\label{e:map of ran prestacks}
\overset{\circ}{\CS}(\CY/\CZ_1/\CZ_2)\strut^{\wedge}_{\on{Ran}} \to (\CY / \CS(\CZ_2))_{\on{dR}} \times \Ran
\end{equation}
in $(\on{PreStk}_{\on{laft-def}})^{\on{nil-isom}}_{\CY \times \Ran // \Ran}$.

\medskip

By \thmref{t:def normal cone}, we have a commutative diagram
$$
\hspace{-2cm}
\xymatrix{
& \on{End}(\IndCoh(\CY)) \otimes \Dmod(\on{Ran})\ar[r]^-{\Gamma_{c,\on{Ran}}} & \on{End}(\IndCoh(\CY)) \\
(\on{PreStk}_{\on{laft-def}})^{\on{nil-isom}}_{\CY \times \Ran// \Ran} \ar[r] \ar@/_2.7pc/[ddr]_(.20){\on{Sym}(T({\CY \times \Ran}/-))} 
\ar@/^1.5pc/[ur]^-{\eqref{e:pushpull funct}} & \on{End}(\IndCoh(\CY)) \otimes \Dmod(\on{Ran})^{\on{Fil}, \geq 0} 
\ar[r]^-{\Gamma_{c,\on{Ran}}}\ar[u]^{\oblv_{\on{Fil}}}\ar[d]_{\on{gr}} & \on{End}(\IndCoh(\CY))^{\on{Fil}, \geq 0} \ar[u]^{\oblv_{\on{Fil}}}\ar[d]_{\on{gr}} \\
& \on{End}(\IndCoh(\CY)) \otimes \Dmod(\on{Ran})\ar[r]^-{\Gamma_{c,\on{Ran}}} & \on{End}(\IndCoh(\CY)) \\
& \IndCoh(\CY) \otimes \Dmod(\on{Ran})\ar[u]^{\on{Tens}}\ar[r]^-{\Gamma_{c,\on{Ran}}} & \IndCoh(\CY) \ar[u]^{\on{Tens}}
}$$
We want to show that the map in $\on{End}(\IndCoh(\CY))$ given by
$$ \Gamma_{c,\on{Ran}} \circ \eqref{e:pushpull funct} (\eqref{e:map of ran prestacks}) $$
is an isomorphism.  Since the functor of associated graded
$$ \on{gr}: \on{End}(\IndCoh(\CY))^{\on{Fil}, \geq 0} \to \on{End}(\IndCoh(\CY)) $$
is conservative, it suffices to show that
$$ \Gamma_{c,\on{Ran}} \circ (\on{Sym}(T(\CY\times \Ran/-))) (\eqref{e:map of ran prestacks}) $$
is an isomorphism in $\IndCoh(\CY)$. 

\medskip

In other words, it suffices to show that $\Gamma_{c,\on{Ran}}$, applied
to the map
\begin{equation}\label{e:map on gr ran}
\on{Sym}(T(\CY \times \Ran/\overset{\circ}{\CS}(\CY/\CZ_1/\CZ_2)\strut^{\wedge}_{\on{Ran}} )) \to 
\on{Sym}(T(\CY \times \Ran/(\CY / \CS(\CZ_2))_{\on{dR}} \times \Ran ))
\end{equation}
in $\IndCoh(\CY)\otimes \Dmod(\on{Ran})$, gives an isomorphism.

\sssec{}

Note that the map \eqref{e:map on gr ran} is the restriction from $\IndCoh(\CY) \otimes \Dmod(\on{Ran}^{\on{untl}})$
of the map

\begin{multline}\label{e:map on gr un-ran}
\on{Sym}(T(\CY \times \Ran^{\on{untl}}/\overset{\circ}{\CS}(\CY/\CZ_1/\CZ_2)\strut^{\wedge}_{\on{Ran}^{\on{untl}} } )) \to \\ 
\to \on{Sym}(T(\CY \times \Ran^{\on{untl}}/(\CY / \CS(\CZ_2))_{\on{dR}} \times \Ran^{\on{untl}} )).
\end{multline}
By \thmref{t:Ran cofinality new}, we have
$$ \Gamma_{c,\on{Ran}}(\eqref{e:map on gr ran}) \simeq \Gamma_{c,\on{Ran}^{\on{untl}} }(\eqref{e:map on gr un-ran}).$$

\medskip

Thus, we wish to show that $\Gamma_{c,\on{Ran}^{\on{untl}}} (\eqref{e:map on gr un-ran})$ is an isomorphism in $\IndCoh(\CY)$.

\sssec{}

Since the functor $\Gamma_{c,\on{Ran}^{\on{untl}} }$ is symmetric monoidal, it suffices to show that the map
\begin{multline}\label{e:map on tangent}
\Gamma_{c,\on{Ran}^{\on{untl}} }(T(\CY \times \Ran^{\on{untl}}/\overset{\circ}{\CS}(\CY/\CZ_1/\CZ_2)\strut^{\wedge}_{\on{Ran}^{\on{untl}} } )) \to \\
\to \Gamma_{c,\on{Ran}^{\on{untl}} }(T(\CY \times \Ran^{\on{untl}}/(\CY / \CS(\CZ_2))_{\on{dR}} \times \Ran^{\on{untl}} ))
\end{multline}
is an isomorphism.  

\medskip

Equivalently, we need to show that the cone of \eqref{e:map on tangent}, i.e., 
\begin{equation}\label{e:relative tangent vanish}
\Gamma_{c,\on{Ran}^{\on{untl}} }\circ 
(\fr_{\on{Ran}^{\on{untl}} })^!(T(\overset{\circ}{\CS}(\CY/\CZ_1/\CZ_2)\strut^{\wedge}_{\on{Ran}^{\on{untl}} }/(\CY/\CS(\CZ_2))_{\on{dR}}\times \Ran^{\on{untl}})) \simeq 0,
\end{equation}
vanishes, where $\fr_{\on{Ran}^{\on{untl}}}$ is the map
$$\CY \times \Ran^{\on{untl}} \to \overset{\circ}{\CS}(\CY/\CZ_1/\CZ_2)\strut^{\wedge}_{\on{Ran}^{\on{untl}} } .$$

\sssec{}
We will use the following key result (which will be proved in \secref{ss:proof of tangent}) to show the desired vanishing of \eqref{e:relative tangent vanish}.

\begin{thm}\label{t:tangent vanishing}
Let $X$ be a proper connected laft scheme and $\CZ \to X_{\on{dR}}$ a sectionally laft prestack.  Then
$$ \Gamma_{c,\on{Ran}^{\on{untl}} }\circ (\fr_{\on{Ran}^{\on{untl}} })^!(T(\overset{\circ}{\CS}(\CZ)\strut^{\wedge}_{\on{Ran}^{\on{untl}} })) \simeq 0 ,$$
where
$$\fr_{\on{Ran}^{\on{untl}} }: \CS(\CZ) \times \Ran^{\on{untl}} \to \overset{\circ}{\CS}(\CZ)\strut^{\wedge}_{\on{Ran}^{\on{untl}} }$$
is the restriction map from sections to punctured sections.
\end{thm}
\begin{rem}
One can regard this theorem as a statement of infinitesimal contractibility of the space of rational maps, in the spirit of \cite{Contr}.
\end{rem}

\sssec{}

We now deduce the desired vanishing in \eqref{e:relative tangent vanish} from \thmref{t:tangent vanishing}.

\medskip

Unwinding the definitions, we obtain that the expression in \eqref{e:relative tangent vanish} is the pullback along $f:\CY\to \CS(\CZ_1)$ 
of 
$$ \Gamma_{c,\on{Ran}^{\on{untl}} }\circ 
(\fr_{\on{Ran}^{\on{untl}} })^! (T(\overset{\circ}{\CS}(\CZ_1)\strut^{\wedge}_{\on{Ran}^{\on{untl}} }/\overset{\circ}{\CS}(\CZ_2)\strut^{\wedge}_{\on{Ran}^{\on{untl}} })),$$
where
$$\fr_{\on{Ran}^{\on{untl}} }: \CS(\CZ_1) \times \Ran^{\on{untl}} \to \overset{\circ}{\CS}(\CZ_1)\strut^{\wedge}_{\on{Ran}^{\on{untl}}}.$$

Now the required vanishing follows from \thmref{t:tangent vanishing} applied to $\CZ_1$ and $\CZ_2$ and the fiber squares
$$
\CD
\CY \times \Ran^{\on{untl}} @>>>  \CS(\CZ_i) \times \Ran^{\on{untl}}  \\
@VVV @VVV \\
\CY @>>>  \CS(\CZ_i) 
\endCD
$$
for $i=1,2$. 

\qed

The rest of this section is devoted to the proof of \thmref{t:tangent vanishing}.

\ssec{Towards the proof of \thmref{t:tangent vanishing}} 

\sssec{}

The proof of \thmref{t:tangent vanishing} will consist of the following two assertions.

\medskip


\begin{prop}\label{p:tangent pushforward to Ran}
Let $X$ be a laft scheme, and $\CZ \to X_{\on{dR}}$ a sectionally laft prestack.
The relative tangent complex
$$ T\left(\CS(\CZ)\times \Ran^{\on{untl}}/\overset{\circ}{\CS}(\CZ)\strut^{\wedge}_{\on{Ran}^{\on{untl}}}\right) 
\in \IndCoh(\CS(\CZ)) \otimes \Dmod(\on{Ran}^{\on{untl}} (X)) $$
lies in the essential image of the fully-faithful functor
$$\on{Id} \otimes i_!: \IndCoh(\CS(\CZ)) \otimes \Dmod(X) \to \IndCoh(\CS(\CZ)) \otimes \Dmod(\on{Ran}^{\on{untl}} (X)), $$
where $i: X \to \on{Ran}^{\on{untl}} (X)$ is the diagonal embedding.
\end{prop}

Recall the object $\Theta(\CZ)\in  \IndCoh(\CS(\CZ)) \otimes \Dmod(X) $, see \eqref{e:Theta}. 

\begin{prop}\label{p:tangent on diag} \hfill
\begin{enumerate}[label={(\alph*)}]
\item
There is a canonical identification
$$(\on{Id} \otimes i^!)\left(T\left(\CS(\CZ)\times \Ran^{\on{untl}}/\overset{\circ}{\CS}(\CZ)\strut^{\wedge}_{\on{Ran}^{\on{untl}}}\right) \right) 
\simeq \Theta(\CZ)$$

\item
Under the identification $T(\CS(\CZ)) \simeq (\on{Id}\otimes (p_X)_{\dr,*})(\Theta(\CZ))$
of \propref{p:sect tangent}, the map
\begin{multline*} 
(\on{Id} \otimes i^!)\left(T\left(\CS(\CZ)\times \Ran^{\on{untl}}/\overset{\circ}{\CS}(\CZ)\strut^{\wedge}_{\on{Ran}^{\on{untl}}}\right)\right) \to \\
\to (\on{Id} \otimes i^!)\left(T\left(\CS(\CZ)\times \Ran^{\on{untl}}\right)\right)\simeq T(\CS(\CZ)) \boxtimes \omega_X
\end{multline*} 
identifies with the canonical map
$$\Theta(\CZ)\to (\on{Id}\otimes (p_X)^!)\circ (\on{Id}\otimes (p_X)_{\dr,*})(\Theta(\CZ))$$
given by the unit of the $((p_X)_{\dr,*},(p_X)^!)$-adjunction. 
\end{enumerate}
\end{prop}

\sssec{Proof of \thmref{t:tangent vanishing}}

We will deduce \thmref{t:tangent vanishing} from a combination of Propositions \ref{p:tangent pushforward to Ran} and \ref{p:tangent on diag}.

\medskip

By definition, we need to show that the map
\begin{multline*} 
(\on{Id} \otimes \Gamma_{c,\Ran^{\on{untl}}})\left(T\left(\CS(\CZ)\times \Ran^{\on{untl}}/\overset{\circ}{\CS}(\CZ)\strut^{\wedge}_{\on{Ran}^{\on{untl}}}\right)\right)\to \\
\to (\on{Id} \otimes \Gamma_{c,\Ran^{\on{untl}}})\left(T\left(\CS(\CZ)\times \Ran^{\on{untl}}\right)\right)
\end{multline*} 
is an isomorphism. 

\medskip

By \propref{p:tangent pushforward to Ran}, the map
\begin{multline*}
(\on{Id} \otimes i)_!\circ (\on{Id} \otimes i^!)
\left(T\left(\CS(\CZ)\times \Ran^{\on{untl}}/\overset{\circ}{\CS}(\CZ)\strut^{\wedge}_{\on{Ran}^{\on{untl}}}\right)\right)\to \\
\to T\left(\CS(\CZ)\times \Ran^{\on{untl}}/\overset{\circ}{\CS}(\CZ)\strut^{\wedge}_{\on{Ran}^{\on{untl}}}\right)
\end{multline*} 
is an isomorphism. Hence, it is enough to show that the composite map
\begin{multline*}
(\on{Id} \otimes \Gamma_{c,X})\circ 
(\on{Id} \otimes i^!) \left(T\left(\CS(\CZ)\times \Ran^{\on{untl}}/\overset{\circ}{\CS}(\CZ)\strut^{\wedge}_{\on{Ran}^{\on{untl}}}\right)\right)\to \\
\to (\on{Id} \otimes \Gamma_{c,X})\circ (\on{Id} \otimes i^!) \left(T\left(\CS(\CZ)\times \Ran^{\on{untl}}\right)\right) \to \\
\to (\on{Id} \otimes \Gamma_{c,\Ran^{\on{untl}}})\left(T\left(\CS(\CZ)\times \Ran^{\on{untl}}\right)\right)
\end{multline*}
is an isomorphism. (Note that $\Gamma_{c,X}\simeq (p_X)_{\dr,*}$.) 

\medskip

By \propref{p:tangent on diag}, the latter map identifies with the composition
\begin{multline*}
(\on{Id} \otimes \Gamma_{c,X})(\Theta(\CZ))\to (\on{Id} \otimes \Gamma_{c,X})\circ 
(\on{Id}\otimes (p_X)^!)\circ (\on{Id}\otimes (p_X)_{\dr,*})(\Theta(\CZ))\to \\
\to (\on{Id} \otimes \Gamma_{c,\Ran^{\on{untl}}})\circ 
(\on{Id}\otimes (p_{\Ran^{\on{untl}}})^!\circ (\on{Id}\otimes (p_X)_{\dr,*})(\Theta(\CZ)).
\end{multline*}

Now, the contractibility of $\Ran^{\on{untl}}$ (\corref{c:Gamma on Ran untl mon}) implies that the counit of
the $(\Gamma_{c,\Ran^{\on{untl}}},(p_{\Ran^{\on{untl}}})^!)$-adjunction is an isomorphism. Hence, it is enough 
to show that the composition of the last map with
$$(\on{Id} \otimes \Gamma_{c,\Ran^{\on{untl}}})\circ 
(\on{Id}\otimes (p_{\Ran^{\on{untl}}})^!)\circ (\on{Id}\otimes (p_X)_{\dr,*})(\Theta(\CZ))\to (\on{Id}\otimes (p_X)_{\dr,*})(\Theta(\CZ))$$
is an isomorphism. 

\medskip

Note the latter composition identifies with
\begin{multline*}
(\on{Id} \otimes (p_X)_{\dr,*})(\Theta(\CZ)) \to 
(\on{Id} \otimes (p_X)_{\dr,*})\circ (\on{Id}\otimes (p_X)^!)\circ (\on{Id}\otimes (p_X)_{\dr,*})(\Theta(\CZ))\to \\
\to (\on{Id}\otimes (p_X)_{\dr,*})(\Theta(\CZ)),
\end{multline*}
where the last arrow is the counit of the $((p_X)_{\dr,*},(p_X)^!)$-adjunction.

\medskip

However, the last composition is the identity map, by the adjunction axioms.

\qed[\thmref{t:tangent vanishing}]

\ssec{Tangent spaces over the Ran space} \label{ss:proof of tangent}

\sssec{}

In order to prove Propositions \ref{p:tangent pushforward to Ran} and \ref{p:tangent on diag}, we will describe explicitly
the objects
\begin{equation} \label{e:tangents to descr}
(\fr_{\on{Ran}^{\on{untl}} })^!(T(\overset{\circ}{\CS}(\CZ)\strut^{\wedge}_{\on{Ran}^{\on{untl}}}))  \text{ and }
T\left(\CS(\CZ)\times \Ran^{\on{untl}}/\overset{\circ}{\CS}(\CZ)\strut^{\wedge}_{\on{Ran}^{\on{untl}}}\right)
\end{equation} 
of $\IndCoh(\CS(\CZ))\otimes \Dmod(X)$, in the spirit of the description of $T(\CS(\CZ))$, given by \propref{p:sect tangent},
namely:
\begin{equation} \label{e:sect tangent}
T(\CS(\CZ)) \simeq (\on{Id}\otimes (p_X)_{\dr,*})(\Theta(\CZ)).
\end{equation} 

\sssec{}

Specifically, we will describe the restrictions of the objects \eqref{e:tangents to descr} along
$$\Delta_I:X^I\to \Ran\to \Ran^{\on{untl}}$$
for a finite set $I$. 

\medskip

Recall the relative scheme $\overset{\circ}{\mathfrak{X}}_I$ over $X^I_\dr$, see \secref{ss:punctured fibration}. Let $j$ denote the open
embedding
$$\overset{\circ}{\mathfrak{X}}_I\hookrightarrow X\times X^I_\dr.$$

Let $p_1$, and $p_2$ denote the projections from $X\times X^I_\dr$ to $X$ and $X^I_\dr$, respectively, and denote
$$\overset{\circ}p_1:=p_1\circ j \text{ and } \overset{\circ}p_2:=p_2\circ j.$$

\medskip


\sssec{}

With these notations, as in \propref{p:sect tangent}, we obtain that the pullback 
$$(\on{Id}\otimes \Delta_I)^!\circ (\fr_{\on{Ran}^{\on{untl}}})^!(T(\overset{\circ}{\CS}(\CZ)\strut^{\wedge}_{\on{Ran}^{\on{untl}}}))\in
\IndCoh(\CS(\CZ)) \otimes \Dmod(X^I)$$
identifies with
$$(\on{Id}\otimes (\overset{\circ}p_2)_{\dr,*})\circ (\on{Id}\times (\overset{\circ}p_1)^!)(\Theta(\CZ)).$$

\medskip

Under this identification and that of \eqref{e:sect tangent}, the map
\begin{multline*}
T(\CS(\CZ))\boxtimes  \omega_{X^I}\simeq 
(\on{Id}\otimes \Delta_I)^!(T(\CS(\CZ))\boxtimes \omega_{\Ran^{\on{untl}}}) \to \\
\to (\on{Id}\otimes \Delta_I)^!\circ (\fr_{\on{Ran}^{\on{untl}}})^!(T(\overset{\circ}{\CS}(\CZ)\strut^{\wedge}_{\on{Ran}^{\on{untl}}}))
\end{multline*}
is the map
\begin{multline*}
(\on{Id}\otimes (p_2)_{\dr,*})\circ (\on{Id}\otimes (p_1)^!)(\Theta(\CZ)) \to \\
\to (\on{Id}\otimes (p_2)_{\dr,*})\circ (\on{Id}\otimes j_{\dr,*})\circ (\on{Id}\otimes j^!)\circ 
(\on{Id}\otimes (p_1)^!)(\Theta(\CZ)) \simeq \\
\simeq (\on{Id}\otimes (\overset{\circ}p_2)_{\dr,*})\circ (\on{Id}\otimes (\overset{\circ}p_1)^!)(\Theta(\CZ)).
\end{multline*}

\sssec{Proof of \propref{p:tangent pushforward to Ran}}

Let $\sH_I$ denote the complement of $\overset{\circ}{\mathfrak{X}}_I$ in $X\times X^I$, and let $\sh_I$ denote its embedding.

\medskip

From the above we obtain that 
$$(\on{Id}\otimes \Delta_I)^!\left(T\left(\CS(\CZ)\times \Ran^{\on{untl}}/\overset{\circ}{\CS}(\CZ)\strut^{\wedge}_{\on{Ran}^{\on{untl}}}\right)\right)$$
identifies with
\begin{equation} \label{e:exp relative}
(\on{Id}\otimes (p_2)_{\dr,*})\circ (\on{Id}\otimes (\sh_I)_{\dr,*}) \circ (\on{Id}\otimes (\sh_I)^! \circ   (\on{Id}\otimes (p_1)^!)(\Theta(\CZ)).
\end{equation} 

Note that for $I=I_1\sqcup I_2$ and the corresponding open subset
$$(X^{I_1}\times X^{I_2})_{\on{disj}}\subset X^{I_1}\times X^{I_2}\simeq X^I,$$
we have a natural isomorphism
\begin{equation} \label{e:factor incidence}
(\sH^{I_1}\times \sH^{I_2})\underset{X^{I_1}\times X^{I_2}}\times (X^{I_1}\times X^{I_2})_{\on{disj}}
\simeq \sH \underset{X^{I_1}\times X^{I_2}}\times (X^{I_1}\times X^{I_2})_{\on{disj}}.
\end{equation} 

This readily implies the assertion of \propref{p:tangent pushforward to Ran} in view of the characterization of the
essential image of $i_!$ given by \thmref{t:image of Dmod in Ran}. 

\qed[\propref{p:tangent pushforward to Ran}]

\sssec{Proof of \propref{p:tangent on diag}}

For the proof of \propref{p:tangent on diag}, we take $I$ to be a singleton. In this case, the pair
$(\sH_I,\sh_I)$ identifies with $(X,\Delta_X)$. 

\medskip

Hence, in this case, the expression in \eqref{e:exp relative} identifies with
$$(\on{Id}\otimes (p_2)_{\dr,*})\circ (\on{Id}\otimes (\Delta_X)_{\dr,*}) \circ (\on{Id}\otimes (\Delta_X)^! \circ   (\on{Id}\otimes (p_1)^!)(\Theta(\CZ))\simeq
\Theta(\CZ),$$
which gives the assertion of \propref{p:tangent on diag}(a).

\medskip

The assertion of \propref{p:tangent on diag}(b) amounts to the following. Let $\CM$ be an object of $\Dmod(X)$.  Then the unit of the adjunction
$$\CM\to (p_X)^!\circ (p_X)_{\dr,*}(\CM)$$
identifies with the map 
$$\CM\simeq (p_2)_{\dr,*}\circ (\Delta_X)_{\dr,*}\circ (\Delta_X)^!\circ (p_1)^!(\CM) \to
(p_2)_{\dr,*}\circ (p_1)^!(\CM) \simeq (p_X)^!\circ (p_X)_{\dr,*}(\CM).$$

\section{First applications and examples} \label{s:examples}

\ssec{The case of a laft target}

In this subsection, we will show that under certain conditions, if $\CZ$ is a laft-def prestack
over $X_\dr$, then it is also sectionally laft. 

\sssec{}

We now return to the situation in \secref{sss:Theta on S} when $\CZ$ itself is laft; in particular it admits a tangent complex
$$T(\CZ)\in \IndCoh(\CZ).$$

Suppose we have an affine scheme $S$ (almost of finite type) and a map
$s:S\to \CS(\CZ)$. Recall that $\on{ev}_s$ denotes the resulting map
$$S\times X_\dr\to \CZ.$$

\medskip

On the one hand we can consider the object
$$\on{ev}_s^!(T(\CZ))\in \IndCoh(S\times X_\dr),$$
and on the other hand, we have the object 
$$(s\times \on{id})^!(\Theta(\CZ))\in \IndCoh(S\times X_\dr),$$
provided that 
$$ \oblv^{\on{fake}}\circ \on{ev}_s^{\#}(T^*(\CZ)) \in \on{Pro}(\QCoh(S \times X_{\on{dR}})^-)_{\on{laft}}.$$

We claim:

\begin{prop} \label{p:target ft}
The object $(s\times \on{id})^!(\Theta(\CZ))$ exists and is related to $\on{ev}_s^!(T(\CZ))$ by the formula
$$(s\times \on{id})^!(\Theta(\CZ)) \simeq \on{ev}_s^!(T(\CZ))\sotimes \BD^{\on{Verdier}}(\omega_X).$$
\end{prop} 

The proposition will be proved in \secref{ss:target ft}. 

\begin{rem}

Note that the object $\BD^{\on{Verdier}}(\omega_X)\in \Dmod(X)$ appearing in the above proposition
is the ``constant sheaf" on $X$. In particular, for $X$ smooth of dimension $d$, we have
$\BD^{\on{Verdier}}(\omega_X) \simeq \omega_X[-2d]$.

\end{rem} 

As an immediate corollary of \propref{p:target ft}, we obtain:

\begin{cor} \label{c:laft => sect laft}
Let $\CZ$ be a laft prestack over $X_\dr$, such that the \emph{classical prestack} underlying
$\CS(\CZ)$ is locally of finite type. Then $\CZ$ is sectionally laft.
\end{cor} 

\sssec{}

We will now consider a particular class of examples when the condition of \corref{c:laft => sect laft}
is satisfied:

\begin{prop} \label{p:sect laft ex}
Assume that $\CZ$ is of the form $\CZ'/\sG$, where $\CZ'\to X_\dr$ is laft schematic, and $\sG$ is an
algebraic group. Then $\CZ$ is sectionally laft.
\end{prop}

Note that a particular example of $\CZ$ as in \secref{p:sect laft ex} is $\CZ=X_\dr\times \on{pt}/\sG$.
The resulting prestack $\CS(\CZ)$ is then $\LocSys_\sG(X)$, the stack of de Rham $\sG$-local systems 
on $X$, studied in more detail in \secref{ss:LocSys}. 

\medskip

Before proving \propref{p:sect laft ex}, we establish the following fact:

\begin{lem}\label{l:maps scheme to laft}
Suppose that $X$ is eventually coconnective.  Let $\CZ$ be a laft prestack satisfying Zariski descent.  Then the
mapping prestack
$$ \underline{\on{Maps}}(X, \CZ) $$
is laft.
\end{lem}
\begin{proof}
Since finite limits of laft prestacks are laft, by Zariski descent, we can assume that $X$ is affine, i.e.
$X \in {}^{\leq m}\affSch$, for some $m \in \mathbb{N}$.

\medskip
By \lemref{l:convergence}, $\on{\underline{Maps}}(X, \CZ)$ is convergent.
Now, suppose that $S_{\alpha}$ is a filtered diagram in ${}^{\leq n}\affSch$ and $S=\on{colim} S_{\alpha}$.  Then $X \times S_\alpha$
is a filtered diagram in ${}^{\leq m+n}\affSch$, and therefore, the map
$$ \on{colim} \underline{\on{Maps}}(X,\CZ)(S_\alpha) \simeq \on{colim} \on{Maps}(X\times S_\alpha, \CZ) \to \on{Maps}(X \times S, \CZ) \simeq \underline{\on{Maps}}(X,\CZ)(S) $$
is an isomorphism, since $\CZ$ is laft.
\end{proof}

\begin{proof}[Proof of \propref{p:sect laft ex}]
By \corref{c:laft => sect laft} we only need to show that the underlying classical prestack of $\CS(\CZ)$ is lft.  We first consider the case when $\CZ=X_\dr\times \on{pt}/\sG$. In this case 
$$\CS(\CZ)=\LocSys_\sG(X),$$
and the assertion that $\LocSys_\sG(X)$ is laft is \cite[Proposition 10.3.8(b)]{AG}. 

\medskip

Thus, it suffices to show that for any affine classical scheme $S$ of finite type and a map
$\sigma: S \to \LocSys_{\sG}(X)$, the fiber product
\begin{equation} \label{e:sect over LocSys}
S\underset{\LocSys_\sG(X)}\times \CS(\CZ)
\end{equation}
is laft.

\medskip

By definition, $\sigma$ corresponds to a map
$$ X_{\dr} \times S \to X_{\dr} \times \on{pt}/G $$
and the fiber product \eqref{e:sect over LocSys} is given by
$$ \on{Res}^{X_{\dr} \times S}_S (\CZ'_{\sigma}), $$
where
$$ \CZ'_{\sigma} =  \CZ'/G \underset{X_{\dr} \times \on{pt}/G}{\times} (X_{\dr} \times S)$$
is the prestack over $X_{\dr} \times S$ obtained by twisting $\CZ'$ using $\sigma$ (which by
assumption is schematic).

\medskip
In what follows, let $\CW:=\CZ'_{\sigma}$.  We need to show that the underlying classical prestack
of
$$ \on{Res}^{X_\dr\times S}_S (\CW) $$
is lft. The assertion only depends on $X_{\dr}$; therefore we can assume that $X$ is classical.

\medskip

Since $\CW$ is laft-def, we have that
$$ \on{Res}^{X_{\dr} \times S}_S(\CW) = \on{Tot}(\on{Res}^{\hat{X}^{\bullet} \times S}_S(\CW_\bullet)) ,$$
where $\hat{X}^\bullet$ is the infinitesimal groupoid of $X$, and
$$ \CW_n = \hat{X}^n \underset{X_{\dr}}{\times} \CW .$$

\medskip
Since $\CW \to X_{\dr} \times S$ is schematic, we have that the classical prestack of ${}^{\on{cl}}\!\on{Res}^{X_\dr \times S}_S(\CW)$
is given by the underlying classical prestack of the equalizer
$$\on{Res}^{X\times S}_S(\CW_1) \rightrightarrows \on{Res}^{\hat{X}^2\times S}_S(\CW_2) .$$

\medskip

Now, let
$$ X = X_0 \subset X_1 \subset \ldots \subset \hat{X}^2 $$
be the (classical) neighborhood filtration of $X \subset \hat{X}^2$.  We thus have that ${}^{\on{cl}}\!\on{Sect}(X_{\dr}, Z)$
is given by the underlying classical prestack of the inverse limit of the equalizers
$$ M_n := \on{eq}(\on{Res}^{X\times S}_S(\CW_1) \rightrightarrows \on{Res}^{X_n\times S}_S(\CW_{2,n})), $$
where $\CW_{2,n} = \CW_2 \underset{\hat{X}^2}{\times} X_n$.

\medskip
Recall that for any finite type schemes $S$ and $T$, where $T$ is separated, the mapping prestack
$$ \underline{\on{Maps}}(S,T) $$
is separated.  In particular, each
$$ \on{Res}^{X_n \times S}_S(\CW_{2,n}) \simeq \underline{\on{Maps}}(X_n, \CW_{2,n}) \underset{\underline{\on{Maps}}(X_n, X_n\times S)}{\times} S $$
is separated.

\medskip
It follows that the maps
$$ M_n \to \on{Res}^{X_n \times S}_S(\CW_{2,n}) $$
are closed immersions.  Therefore,
$$ {}^{\on{cl}}\!\on{Res}^{X_{\dr}\times S}_S(\CW) \simeq \lim_n {}^{\on{cl}}\!M_n \to \on{Res}^{X_n \times S}_S(\CW_{2,n})$$
is also a closed immersion.

\medskip
By \lemref{l:maps scheme to laft}, the prestack
$$ \on{Res}^{X \times S}_S(\CW_1) \simeq \underline{\on{Maps}}(X, \CW_1) \underset{\underline{\on{Maps}}(X,X\times S)}{\times} S $$
is laft.  Therefore, ${}^{\on{cl}}\!\on{Res}^{X_{\dr}\times S}_S(\CW)$ is lft.
\end{proof}

\ssec{Proof of \propref{p:target ft}} \label{ss:target ft}

\sssec{}

Let $\CY$ be a laft prestack, and $\CF$ an object of $\IndCoh(\CY)$. Then by \cite[Chapter 1, Corollary 4.4.2]{Vol2},
to $\CF$ we can attach an object
$$\BD^{\on{Serre}}(\CF)\in \on{Pro}(\QCoh(\CY)^-)^{\on{fake}}).$$ 

From $\BD^{\on{Serre}}(\CF)$ we obtain an object of $\on{Pro}(\QCoh(\CY)^-)$ by applying the functor 
$$\oblv^{\on{fake}}: \on{Pro}(\QCoh(\CY)^-)^{\on{fake}})\to \on{Pro}(\QCoh(\CY)^-).$$

We take $\CY:=S\times X_\dr$, and thus, on the one hand, we obtain a functor
$$\oblv^{\on{fake}}\circ  \BD^{\on{Serre}}:\IndCoh(S\times X_\dr)^{\on{op}}\to \on{Pro}(\QCoh(S\times X_\dr)^-).$$

On the other hand, we have the equivalence
$$\BD^{\on{SV}}: \IndCoh(S\times X_\dr)^{\on{op}}\simeq \on{Pro}(\QCoh(S\times X_\dr)^-)_{\on{laft}},$$
which we can follow by the embedding 
$$\on{Pro}(\QCoh(S\times X_\dr)^-)_{\on{laft}}\hookrightarrow \on{Pro}(\QCoh(S\times X_\dr)^-).$$

The abstract form of \propref{p:target ft} is:

\begin{prop} \label{p:target ft abs}
The above two functors $\IndCoh(S\times X_\dr)^{\on{op}}\to \on{Pro}(\QCoh(S\times X_\dr)^-)$
are related by
\begin{equation} \label{e:SV vs Serre}
\BD^{\on{SV}}(\CF\sotimes \BD^{\on{Verdier}}(\omega_X))\simeq \oblv^{\on{fake}}\circ  \BD^{\on{Serre}}(\CF), \quad \CF\in \IndCoh(S\times X_\dr). 
\end{equation}
In particular, $\oblv^{\on{fake}}$ maps $\on{Pro}(\QCoh(S\times X_{\dr})^-)^{\on{fake}}_{\on{laft}}$ to $\on{Pro}(\QCoh(S\times X_{\dr})^-)_{\on{laft}}$.
\end{prop} 

The rest of this subsection is devoted to the proof of \propref{p:target ft abs}. 


%

\sssec{Preliminaries}

We recall the construction of the object $\oblv^{\on{fake}} \circ \BD^{\on{Serre}}(\CF)$ in general
for $\CF\in \IndCoh(\CY)$, where $\CY$ is a left prestack.  

\medskip

First, the object $\BD^{\on{Serre}}(\CF)\in \on{Pro}(\QCoh(\CY)^-)$ is characterized by the requirement
that for $(T\overset{f}\to \CY) \in (\affSch_{\on{aft}})_{/\CY}$, the object
$$f^\#(\BD^{\on{Serre}}(\CF))\in \on{Pro}(\QCoh(T)^-)$$
belongs to
$$\on{Pro}(\QCoh(T)^-)_{\on{laft}}\subset \on{Pro}(\QCoh(T)^-)$$
and equals $\BD^{\on{Serre}}(f^!(\CF))$.

\medskip

Assume that $\CY$ is laft-def. Then by \cite[Chapter 1, Theorem 9.1.4]{Vol2}, the inclusion
$$(\affSch_{\on{aft}})_{/\CY} \hookrightarrow (\affSch)_{/\CY}$$
is cofinal. 

\medskip

Hence, for $M\in \QCoh(\CY)^-$ and $\CF'\in \on{Pro}(\QCoh(\CY)^-)^{\on{fake}}$, 
we can rewrite
$$\oblv^{\on{fake}}(\CF')(M):=\underset{(T\overset{f}\to \CY) \in (\affSch)_{/\CY}}{\on{lim}}\, 
f^\#(\CF')(f^*(M))$$ 
also as  
$$\underset{(T\overset{f}\to \CY) \in (\affSch_{\on{aft}})_{/\CY}}{\on{lim}}\, 
f^\#(\CF')(f^*(M)).$$ 

\medskip

Applying this to $\CF':=\BD^{\on{Serre}}(\CF)$, we obtain that
\begin{equation} \label{e:calc oblv}
\oblv^{\on{fake}}\circ \BD^{\on{Serre}}(\CF)(M)\simeq 
\underset{(T\overset{f}\to \CY) \in (\affSch_{\on{aft}})_{/\CY}}{\on{lim}}\, 
\langle f^!(\CF),f^*(M)\rangle_T,
\end{equation} 
where $\langle -,- \rangle_T$ denotes the canonical pairing
\begin{equation} \label{e:pairing}
\IndCoh(T)\otimes \QCoh(T)^-\to \Vect,
\end{equation}
corresponding to the identification
$$\IndCoh(T)^{\on{op}}\simeq \on{Pro}(\QCoh(T)^-)_{\on{laft}}.$$

Explicitly, \eqref{e:pairing} maps $\CF_T\in \IndCoh(T)$ and $M_T\in \QCoh(T)^-$ to
$$\underset{n}{\on{lim}}\, \Gamma\left(T,\Psi_T(\CF_T\sotimes \Psi_T^{-1}(\tau^{\geq -n}(M_T))\right),$$
where $\Psi_T$ is as in \secref{sss:Psi and Ups}.

\sssec{Reduction to the bounded case}
By \corref{c:Serre Verdier} and \eqref{e:calc oblv}, both sides of \eqref{e:SV vs Serre}, for $M \in \QCoh(S \times X_{\dr})^-$,
map the inverse family
$$ \tau^{\geq -n} M $$
to a limit in $\on{Vect}$.  Hence, to prove \eqref{e:SV vs Serre}, it suffices to prove that their restrictions to $\on{Pro}(\QCoh(S \times X_{\dr})^b)$ are isomorphic.

\sssec{The easy case}

To start, we will consider the case $\CY=X_\dr$.  We have that the inclusion 
$$({}^{\on{cl}}\!\affSch_{\on{aft}})_{/X_\dr}\hookrightarrow (\affSch_{\on{aft}})_{/X_\dr}$$
is cofinal; hence so is 
$$({}^{<\infty}\!\affSch_{\on{aft}})_{/X_\dr}\hookrightarrow (\affSch_{\on{aft}})_{/X_\dr}.$$

\medskip
Thus, we have
$$\oblv^{\on{fake}}\circ \BD^{\on{Serre}}(\CF)(M)\simeq 
 \underset{(T\overset{f}\to \CY) \in ({}^{<\infty}\!\affSch_{\on{aft}})_{/X_{\dr}}}{\on{lim}}\, \langle f^!(\CF), f^*(M) \rangle_T .$$

\medskip
For $M_1 \in \QCoh(T), M_2 \in \QCoh(T)^-$, we have a canonical isomorphism
\begin{equation}\label{e:restr of pairing}
\Psi_T(\Upsilon_T(M_1) \otimes^! \Psi_T^{-1}(M_2)) \simeq M_1 \otimes M_2 .
\end{equation}
Therefore, the composite
$$\QCoh(T)\otimes \QCoh(T)^- \overset{\Upsilon_T\otimes \on{Id}}\longrightarrow \IndCoh(T)\otimes \QCoh(T)^-
\overset{\langle -,- \rangle_T}\longrightarrow \Vect$$
is the functor
$$M_1,M_2\mapsto \lim_{n} \Gamma(T,M_1\otimes \tau^{\geq -n} (M_2)).$$

\medskip
 Given $$\CF\in \IndCoh(X_\dr)=\Dmod(X),$$
write it as $\Upsilon_{X_\dr}(\CF^l)$ for $\CF^l\in  \QCoh(X_\dr)=\Dmod^l(X)$, where $\Upsilon$
is as in \secref{sss:Psi and Ups}.
\medskip

Hence for $T\overset{f}\to X_\dr$, we have
$$f^!(\CF)\simeq \Upsilon_T\circ f^*(\CF^l).$$

\medskip
Note that for $T\overset{f}\to X_\dr$ with $T\in {}^{<\infty}\!\affSch_{\on{aft}}$, the functor
$$f^*:\QCoh(X_\dr)\to \QCoh(T)$$
has a bounded cohomological amplitude; i.e., $f^*$ preserves bounded objects.  Therefore, we obtain that for $\CF \in \IndCoh(X_{\dr})$ and $M \in \QCoh(X_\dr)^b$,
$$\oblv^{\on{fake}}\circ \BD^{\on{Serre}}(\CF)(M)\simeq 
 \underset{(T\overset{f}\to \CY) \in ({}^{<\infty}\!\affSch_{\on{aft}})_{/X_{\dr}}}{\on{lim}}\,  \Gamma(T,f^*(\CF^l\otimes M)),$$
where $\CF^l = \Upsilon_{X_\dr}^{-1}(\CF) \in \QCoh(X_\dr)$ as above.
The latter limit identifies with
$$\Gamma(X_\dr,\CF^l\otimes M)\simeq
\CHom_{\Dmod^l(X)}(\CO_X,\CF^l\otimes M) \simeq
\CHom_{\Dmod(X)}(\omega_X,\CF\sotimes \Upsilon_{X_\dr}(M)),$$
and finally with
\begin{equation} \label{e:limit exp}
\Gamma_\dr\left(X,(\CF\sotimes \BD^{\on{Verdier}}(\omega_X)) \sotimes \Upsilon_{X_\dr}(M)\right).
\end{equation}

By \corref{c:Serre Verdier}, we obtain that
$$\left(\BD^{\on{SV}}(\CF\sotimes \BD^{\on{Verdier}}(\omega_X))\right)(M)$$ 
is also given by \eqref{e:limit exp}, thus establishing \eqref{e:SV vs Serre}. 

\sssec{The general case}

We now consider the general case of $\CY=S\times X_\dr$. Note that the functor
$$(\affSch_{\on{aft}})_{/X_\dr}\to (\affSch_{\on{aft}})_{/S\times X_\dr},\quad (T,f)\mapsto (S\times T,\on{id}\times f)$$
is cofinal.  Moreover, as before, 
$$({}^{<\infty}\!\affSch_{\on{aft}})_{/X_\dr}\hookrightarrow (\affSch_{\on{aft}})_{/X_\dr}$$
is also cofinal.

\medskip

Next, recall that for $T\overset{f}\to X_\dr$ with $T\in {}^{<\infty}\!\affSch_{\on{aft}}$, the functor
$$f^*:\QCoh(X_\dr)\to \QCoh(T)$$
has a bounded cohomological amplitude. Hence, so do the functors
$$(\on{id}\times f)^*:\QCoh(S\times X_\dr)\to \QCoh(S\times T).$$
In particular, for $M\in \QCoh(S\times X_\dr)^b$, we have
$$(\on{id}\times f)^*(M)\in \QCoh(S\times T)^b.$$

\medskip

Hence, the terms in the limit \eqref{e:calc oblv} can be identified with
$$\Gamma\left(S\times T,\Psi_{S\times T}\left((\on{id}\times f)^!(\CF)\sotimes \Psi_{S\times T}^{-1}((\on{id}\times f)^*(M))\right)\right).$$

\medskip

We claim:

\begin{lem} \label{l:Xi Upsilon}
For a given $\CF\in \IndCoh(S\times X_\dr)$, there is a canonical isomorphism of functors
$\QCoh(S\times X_\dr)^b\to \QCoh(S\times T)$: 
\begin{multline*}
\Psi_{S\times T}\left((\on{id}\times f)^!(\CF)\sotimes \Psi_{S\times T}^{-1}((\on{id}\times f)^*(M))\right)
\simeq \\
\simeq (\on{id}\times f)^*\left((\Psi_S\otimes \Upsilon^{-1}_{X_\dr})\left(\CF\sotimes (\Psi_S\otimes \Upsilon^{-1}_{X_\dr})^{-1}(M)\right)\right).
\end{multline*}
\end{lem}

The lemma will be proved below. 

\sssec{End of the proof}

Applying the lemma, we obtain that $\oblv^{\on{fake}}\circ  \BD^{\on{Serre}}(\CF)(M)$ is given as the limit over $(T,f)\in ({}^{<\infty}\!\affSch_{\on{aft}})_{/X_\dr}$ of 
$$\Gamma\left(S\times T, (\on{id}\times f)^*\left((\Psi_S\otimes \Upsilon^{-1}_{X_\dr})\left(\CF\sotimes (\Psi_S\otimes \Upsilon^{-1}_{X_\dr})^{-1}(M)\right)\right)\right),$$
i.e.,
$$\CHom_{\QCoh(S\times T)}\left(\CO_{S\times T},
(\on{id}\times f)^*\left((\Psi_S\otimes \Upsilon^{-1}_{X_\dr})\left(\CF\sotimes (\Psi_S\otimes \Upsilon^{-1}_{X_\dr})^{-1}(M)\right)\right)\right).$$

By the cofinality above, this limit identifies with
$$\CHom_{\QCoh(S\times X_\dr)}\left(\CO_{S\times X_\dr},
(\Psi_S\otimes \Upsilon^{-1}_{X_\dr})\left(\CF\sotimes (\Psi_S\otimes \Upsilon^{-1}_{X_\dr})^{-1}(M)\right)\right),$$
and further with
$$\CHom_{\QCoh(S)\otimes \IndCoh(X_\dr)}\left(\CO_{S}\boxtimes \omega_X,
(\Psi_S\otimes \on{Id})\left(\CF\sotimes (\Psi_S\otimes \Upsilon^{-1}_{X_\dr})^{-1}(M)\right)\right).$$

Finally, we rewrite the latter expression as
\begin{multline*}
(\Gamma^\IndCoh_S \otimes \Gamma_{X,\dr})
\left(\CF\sotimes \BD^{\on{Verdier}}(\omega_X)\sotimes (\Psi_S\otimes \Upsilon^{-1}_{X_\dr})^{-1}(M)\right) \simeq \\
\simeq
\Gamma^\IndCoh\left(S\times X_\dr,(\CF\sotimes \BD^{\on{Verdier}}(\omega_X)) 
\sotimes (\Psi_S\otimes \Upsilon^{-1}_{X_\dr})^{-1}(M)\right).
\end{multline*}

By \corref{c:Serre Verdier}, this establishes \eqref{e:SV vs Serre}. 

\qed[\propref{p:target ft abs}]

\sssec{Proof of \lemref{l:Xi Upsilon}} 

Both sides in the lemma map uniformly bounded colimits in $M$ (see \secref{sss:unif bdd} for what this means)
to colimits. Hence, by \propref{p:Psi Ups}(a), we can restrict both functors to the coherent subcategory of 
$\QCoh(S\times X_\dr)$. Further, by \propref{p:Psi Ups}(c), we can precompose the two functors in question with
$$(\Psi_S\otimes \Upsilon^{-1}_{X_\dr}): \IndCoh(S\times X_\dr)^b\simeq \QCoh(S\times X_\dr)^b.$$

Thus, it suffices to establish an isomorphism of functors
\begin{multline} \label{e:Xi Upsilon}
\Psi_{S\times T}\left((\on{id}\times f)^!(\CF)\sotimes \Psi_{S\times T}^{-1}((\on{id}\times f)^*\circ (\Psi_S\otimes \Upsilon^{-1}_{X_\dr}) (M'))\right)
\simeq \\
\simeq (\on{id}\times f)^*\left((\Psi_S\otimes \Upsilon^{-1}_{X_\dr})(\CF\sotimes M')\right)
\end{multline}
as $M'$ ranges over the category of coherent objects in $\IndCoh(S\times X_\dr)$. 

\medskip

Now, by \thmref{t:finiteness product}, the category of coherent objects in $\IndCoh(S\times X_\dr)$ equals that of compacts. 
Hence, since 
$$\IndCoh(S\times X_\dr)\simeq \IndCoh(S)\otimes \IndCoh(X_\dr),$$
it suffices to establish \eqref{e:Xi Upsilon} as an isomorphism of bilinear functors
$$\Coh(S)\times \IndCoh(X_\dr)^c\to \QCoh(S\times T),$$
i.e., we can take $M'$ to be of the form
$$M'_S\boxtimes M'_X, \quad M'_S\in \Coh(S),\,\, M'_X\in \IndCoh(X_\dr)^c.$$

\medskip

Similarly, both sides in \eqref{e:Xi Upsilon}, viewed as functors of $\CF\in \IndCoh(S\times X_\dr)$, commute with colimits.
Hence, by the same reasoning, we can take $\CF$ to be of the form
$$\CF_S\boxtimes \CF_X, \quad \CF_S\in \IndCoh(S),\,\, \CF_X\in \IndCoh(X_\dr).$$

Then both sides of \eqref{e:Xi Upsilon} split as tensor products. The $S$-factor is tautologically
$$\Psi_S(\CF_S\sotimes M'_S)$$ 
in both the left-hand side and the right hand side.

\medskip

The $X$-factor gives
$$\Psi_T\left(f^!(\CF_X) \sotimes (\Psi_T^{-1}\circ f^*\circ \Upsilon^{-1}_{X_\dr}(M'_X))\right) \simeq \Psi_T\left(\Upsilon_T \circ f^* \circ \Upsilon_{X_\dr}^{-1}(\CF_X) \sotimes (\Psi_T^{-1}\circ f^*\circ \Upsilon^{-1}_{X_\dr}(M'_X))\right)$$
in the left-hand side, and 
$$
f^*\circ \Upsilon^{-1}_{X_\dr}(\CF_X\sotimes M'_X) \simeq f^* \circ \Upsilon_{X_\dr}^{-1}(\CF_X) \otimes f^* \circ \Upsilon_{X_\dr}^{-1}(M'_X)
$$
in the right-hand side.  The two are isomorphic by \eqref{e:restr of pairing}.

\qed[\lemref{l:Xi Upsilon}]

\ssec{Infinitesimal groupoid on the stack of local systems} \label{ss:LocSys}

\sssec{}

Let $\sG_1\to \sG_2$ be a homomorphism of algebraic groups, and take $\CZ_i:=(\on{pt}/\sG_i)\times X_\dr$.
Consider the resulting map
$$(\on{pt}/\sG_1)\times X_\dr\to (\on{pt}/\sG_2)\times X_\dr.$$

As remarked above, the prestack $\CS(\CZ_i)$ is in this case the (derived) algebraic stack $\LocSys_{\sG_i}(X)$
of $\sG_i$-local systems $X$. So, Theorems \ref{t:rel main} and \ref{t:main non-unital rel}, and \corref{c:loc rel}
are the different ways to give a description of the category
$$\IndCoh((\LocSys_{\sG_1}(X)/\LocSys_{\sG_2}(X))_\dr$$
and the monad
\begin{multline*}
(p_{\LocSys_{\sG_1}(X),\dr/\LocSys_{\sG_1}(X)})^\IndCoh_*:\IndCoh(\LocSys_{\sG_1}(X))\rightleftarrows \\
\rightleftarrows \IndCoh((\LocSys_{\sG_1}(X)/\LocSys_{\sG_2}(X))_\dr): (p_{\LocSys_{\sG_1}(X),\dr/\LocSys_{\sG_1}(X)})^!
\end{multline*}
in terms of the infinitesimal Hecke groupoid. 

\sssec{}

Denote the resulting stack $\overset{\circ}{\CS}(\CZ_1/\CZ_2)\strut^{\wedge}_{\on{Ran}}$ by
$$\LocSys_{\sG_1/\sG_2}(\ofX_{\on{Ran}}).$$

\medskip

Note that the infinitesimal Hecke groupoid
$$\widehat{\on{Hecke}}(\LocSys_{\sG_1}(X)/\LocSys_{\sG_2}(X))_{\on{Ran}}:=
\LocSys_{\sG_1}(X)\underset{\LocSys_{\sG_1/\sG_2}(\ofX)}\times \LocSys_{\sG_1}(X)$$
can be explicitly described as follows:

\medskip

It attaches to a test affine scheme $S$ a finite subset in $I\subset \Maps(S,X_\dr)$, a pair $(\sigma',\sigma'')$ of
$\sG_1$-torsors on $S\times X_\dr$, and compatible identifications:
\begin{itemize}

\item $\on{Ind}^{\sG_2}_{\sG_1}(\sigma') \simeq \on{Ind}^{\sG_2}_{\sG_1}(\sigma'')$;

\item $\sigma'|_{^{\on{red}}\!S\times X_\dr}\simeq \sigma''|_{^{\on{red}}\!S\times X_\dr}$;

\item $\sigma'|_{\ofX_{I,\dr}}\simeq \sigma''|_{\ofX_{I,\dr}}$, where $\ofX_{I,\dr}=S\times X_\dr-\on{Graph}(I)$.

\end{itemize}

\sssec{}

We note the following particular cases of the above situation.

\medskip

One is when $\sG_1=\sG$ and $\sG_2$ is trivial. In this case, $\widehat{\on{Hecke}}(\LocSys_{\sG}(X))_{\on{Ran}}$
is indeed what we would call the infinitesimal Hecke groupoid. It classifies the data of $I\subset \Maps(S,X_\dr)$,
$\sigma',\sigma''\in \Maps(S,\LocSys_\sG(X))$ and a compatible pair of isomorphisms 

\begin{itemize}

\item $\sigma'|_{^{\on{red}}\!S\times X_\dr}\simeq \sigma''|_{^{\on{red}}\!S\times X_\dr}$;

\item $\sigma'|_{\ofX_{I,\dr}}\simeq \sigma''|_{\ofX_{I,\dr}}$.

\end{itemize}

\medskip

The other case is when $\sG_1=\sG$ and $\sG_2=\sG\times \sG$, so that $\sG_1\to \sG_2$ is the diagonal
embedding. In this case, the prestack $\LocSys_{\sG/\sG\times \sG}(\ofX_{\on{Ran}})$ identifies with the infinitesimal 
Hecke groupoid $\widehat{\on{Hecke}}(\LocSys_{\sG}(X))_{\on{Ran}}$ described above. 

\sssec{}

We now restate our main theorems in this specific situation:

\begin{thm} \label{t:main LocSys}
For an object $\CF\in \IndCoh(\LocSys_{\sG_1}(X))$ the tautological maps between the following spaces of data on $\CF$ are isomorphisms:
\begin{enumerate}[label={(\roman*)}]
\item
The data of descent of $\CF$ to an object of $\IndCoh((\LocSys_{\sG_1}(X)/\LocSys_{\sG_2}(X))_\dr)$;

\item
The data of descent of the pullback of $\CF$ to $\LocSys_{\sG_1}(X)\times \Ran^{\on{untl}}$ to an object of
$\IndCoh(\LocSys_{\sG_1/\sG_2}(\ofX_{\on{Ran}^{\on{untl}}}))$;

\item[(ii')]
The data of descent of the pullback of $\CF$ to $\LocSys_{\sG_1}(X)\times \Ran$ to an object of
$\IndCoh(\LocSys_{\sG_1/\sG_2}(\ofX_{\on{Ran}}))$. 
\end{enumerate}
\end{thm}

\begin{cor}
The localization functor
$$\on{pre-Loc}: \IndCoh(\LocSys_{\sG_1/\sG_2}(\ofX_{\on{Ran}^{\on{untl}}}))\to \IndCoh(\LocSys_{\sG_1}(X)),$$ defined as  
\begin{multline*}
\IndCoh(\LocSys_{\sG_1/\sG_2}(\ofX_{\on{Ran}^{\on{untl}}})) \overset{\fr^!_{\on{Ran}^{\on{untl}}}}\longrightarrow
\IndCoh(\LocSys_{\sG_1}(X)\otimes \Dmod(\on{Ran}^{\on{untl}}) \to \\
\overset{\on{Id}\otimes \Gamma_{c,\Ran^{\on{untl}}}}\longrightarrow \IndCoh(\LocSys_{\sG_1}(X))
\end{multline*}
factors canonically as
$$(p_{\LocSys_{\sG_1}(X),\dr/\LocSys_{\sG_1}(X)})^! \circ \on{Loc},$$
where $\on{Loc}$ is the functor
$$(\on{Id}\otimes \Gamma_{c,\on{Ran}^{\on{untl}}}) \circ \fq^\IndCoh_*,$$
where
$$\fq: \LocSys_{\sG_1/\sG_2}(\ofX_{\on{Ran}^{\on{untl}}}) \to (\LocSys_{\sG_1}(X)/\LocSys_{\sG_2}(X))_\dr \times \on{Ran}^{\on{untl}}.$$
\end{cor}

\ssec{The jet construction} \label{ss:jets}

\sssec{}

We will now consider another family of examples of targets $\CZ$ that are sectionally laft. Namely, let $Z$ be a laft-def prestack over $X$
and set
$$\on{Jets}(Z):=\on{Res}^X_{X_\dr}(Z).$$

\begin{rem}
Typically, $\on{Jets}(Z)$ will \emph{not} be laft over $X_\dr$. For instance if $X$ is a smooth curve, and $Y$ is a smooth
scheme, the fibers of the map $\on{Jets}(X \times Y)$ are given by (twists of) the formal arc space $Y[[t]]$.
\end{rem}

We are going to prove:

\begin{prop} \label{p:jets sect laft}
Suppose $X$ is eventually coconnective and $Z$ satisfies Zariski descent. Then $\on{Jets}(Z)$ is sectionally laft.
\end{prop}

The rest of this subsection is devoted to the proof of this proposition.

\sssec{}

By adjunction, $\CS(\on{Jets}(Z))$ identifies with the prestack of sections
$$\on{Sect}(X,Z) = \underline{\on{Maps}}(X,Z) \underset{\underline{\on{Maps}}(X,X)}{\times} \{\on{id}\},$$
which is laft by \lemref{l:maps scheme to laft}.

\begin{rem} \label{r:punctured jets} 
Note that the same argument, combined with Corollaries \ref{c:Weil def} and \ref{c:fibration laft-def}, 
shows that for $Z$ as above, the categorical prestack
$\overset{\circ}{\CS}(\on{Jets}(Z))_{\Ran^{\on{untl}}}$ is laft. 
\end{rem}

\sssec{}

It remains to show that for a test laft affine scheme and a map $s:S\to \CS(\on{Jets}(Z))$, we have
$$\oblv^{\on{fake}}\circ \on{ev}_s^{\#}(T^*(\on{Jets}(Z))) \in \on{Pro}(\QCoh(S \times X_{\on{dR}})^-)_{\on{laft}}.$$

In fact, we will describe the Serre-Verdier dual of this object, i.e., 
$$(s\times \on{id})^!(\Theta(\on{Jets}(Z))\in \IndCoh(S)\otimes \Dmod(X)$$
explicitly. 

\begin{prop} \label{p:tangent Jets}
The object $\oblv^{\on{fake}}\circ \on{ev}_s^{\#}(T^*(\on{Jets}(Z)))$ identifies canonically with
$\BD^{\on{SV}}((s\times \on{id})^!(\Theta(\on{Jets}(Z)))$ for
$$(s\times \on{id})^!(\Theta(\on{Jets}(Z)):=(\on{Id}\otimes \ind)\circ (\on{ev}'_s)^!(T(Z))\in \IndCoh(S)\otimes \Dmod(X),$$
where $\on{ev}'_s$ is the map $S\times X\to Z$ corresponding to $s$ by adjunction.
\end{prop}

\begin{proof}

Consider the diagram
$$
\CD
S \times X @>>> \on{Jets}(Z)\underset{X_\dr}\times X   @>>>  Z  \\
@V{\on{id}\times p_{X,\dr}}VV @VVV \\
S \times X_\dr @>{\on{ev}_s}>> \on{Jets}(Z),
\endCD
$$
in which the composite horizontal arrow is $\on{ev}'_s$ and the square is Cartesian. 

\medskip

By \corref{c:Weil def} and \lemref{l:pro base change}, we obtain that 
$$\on{ev}_s^{\#}(T^*(\on{Jets}(Z)))\simeq (\on{id}\times p_{X,\dr})_{\#}\circ (\on{ev}'_s)^{\#}(T^*(Z)).$$

\medskip

Hence, \propref{p:tangent Jets} is a particular case of the following general statement
for $\CF:=(\on{ev}'_s)^{\#}(T^*(Z))$:

\begin{prop} \label{p:tangent Jets abs}
For any $\CF\in \IndCoh(S \times X)$, 
\begin{equation} \label{e:fake dir im}
\oblv^{\on{fake}}\circ (\on{id}\times p_{X,\dr})_{\#}\circ \BD^{\on{Serre}}(\CF) \simeq 
\BD^{\on{SV}} \circ (\on{id}\times p_{X,\dr})^\IndCoh_*(\CF).
\end{equation}
\end{prop}

\end{proof}

\sssec{Proof of \propref{p:tangent Jets abs}}
We evaluate both sides of \eqref{e:fake dir im} on an object $M\in \QCoh(S\times X_{\dr})^-$.

\medskip

By \lemref{l:pro oblv functors}(a), the left-hand side gives
$$\langle \CF,(\on{id}\times p_{X,\dr})^*(M)\rangle_{S\times X},$$
where $\langle-,-\rangle_{S\times X}$ is as in \eqref{e:pairing}. 

\medskip

In particular, we obtain that the value of the left-hand side in \eqref{e:fake dir im} on $M$ maps isomorphically to the 
limit of its values on the truncations $\tau^{\geq -n}(M)$. Hence, we can assume that $M\in \QCoh(S\times X_\dr)^b$.

\medskip

In this case, arguing as in the proof of \lemref{l:Xi Upsilon}, we can assume that
$$\CF\simeq \CF_S\boxtimes \CF_X, \quad \CF_S\in \Coh(S),\,\, \CF_X\in \Coh(X)$$
and
$$M\simeq M_S\boxtimes M_X, \quad M_S\in \Coh(S),\,\, M_X \in \Dmod^l(X)^c.$$

Then the evaluation of the left-hand side in \eqref{e:fake dir im} gives
$$\langle \CF_S,M_S\rangle_S \otimes \langle \CF_X,\oblv^l(M_X)\rangle_X,$$
and the evaluation of the right-hand side in \eqref{e:fake dir im} gives
$$\Gamma(S,\Psi_S(\CF_S \sotimes \Psi_S^{-1}(M_S)))\otimes 
\Gamma_\dr(X,\CF_X\sotimes \Upsilon_{X_\dr}(M_X)).$$

The resulting expressions match on the nose.

\qed[\propref{p:tangent Jets abs}]

\ssec{Infinitesimal groupoid on the stack of principal bundles}

In this subsection we will apply \propref{p:jets sect laft} in a key case of interest.

\sssec{} \label{sss:Hecke groupoid}

We take $Z$ to be $X\times \on{pt}/G$, where $G$ is an algebraic group. In this case, the prestack
$$\CS(\on{Jets}(\on{pt}/G))\simeq \underline\Maps(X,\on{pt}/G)$$
is the stack $\Bun_G(X)$ of principal $G$-bundles on $X$. 

\medskip

The resulting groupoid 
$$\widehat{\on{Hecke}}(\Bun_G(X))_{\on{Ran}^{\on{untl}}}:=
\widehat{\on{Hecke}}(\on{Jets}(\on{pt}/G))_{\on{Ran}^{\on{untl}}}$$
can be described explicitly as follows.

\medskip

Let $\on{Hecke}(\on{Jets}(\Bun_G(X))_{\on{Ran}^{\on{untl}}}$ denote the \emph{usual} Hecke groupoid; i.e., 
for an affine test-scheme $S$, an $S$-point of $\on{Hecke}(\Bun_G(X))_{\on{Ran}^{\on{untl}}}$ is given by:

\begin{itemize}
\item
An $S$-point of $\Ran^{\on{untl}}$; let $\on{Graph} \subset S \times X$ denote the (set-theoretic) 
union of the graphs of the points corresponding to this $S$-point of the Ran space.

\item
A pair of $G$-bundles $\CP_1$ and $\CP_2$ on $S\times X$;

\item
An isomorphism
$\alpha: \CP_1|_{S \times X- \on{Graph}} \simeq \CP_2|_{S \times X - \on{Graph}}$.

\end{itemize}

Then $\widehat{\on{Hecke}}(\Bun_G(X))_{\on{Ran}^{\on{untl}}}$ is the completion of $\on{Hecke}(\Bun_G(X))_{\on{Ran}^{\on{untl}}}$
along the unit section
$$\Bun_G(X)\times \on{Ran}^{\on{untl}}\to \on{Hecke}(\Bun_G(X))_{\on{Ran}^{\on{untl}}}.$$

\sssec{} \label{sss:bundles on punctured}

Note that according to Remark \ref{r:punctured jets}, we can also consider the laft-def prestack 
$$\Bun_G(\ofX_{\on{Ran}^{\on{untl}}}):=\overset{\circ}{\CS}(\on{Jets}(\on{pt}/G))_{\Ran^{\on{untl}}};$$
this is the prestack of $G$-bundles ``defined away from a finite collection of points". 

\medskip

By definition, as $S$-point of $\Bun_G(\ofX_{\on{Ran}^{\on{untl}}})$ is given by:

\begin{itemize}

\item
An $S$-point of $\Ran^{\on{untl}}$;

\item 
A $G$-bundle $\CP$ on $S\times X-\on{Graph}$.

\end{itemize} 

Note that
$$ \on{Hecke}(\Bun_G(X))_{\on{Ran}^{\on{untl}}} \simeq \Bun_G(X)\times \on{Ran}^{\on{untl}} \underset{\Bun_G(\ofX_{\on{Ran}^{\on{untl}}})}{\times} \Bun_G(X)\times \on{Ran}^{\on{untl}}.$$
In particular, $\on{Hecke}(\Bun_G(X))_{\on{Ran}^{\on{untl}}}$ is laft-def.

\sssec{}

Let 
$$\Bun_G(\ofX_{\on{Ran}^{\on{untl}}})\strut^{\wedge}:=\overset{\circ}{\CS}(\on{Jets}(\on{pt}/G))\strut^{\wedge}_{\on{Ran}^{\on{untl}}}$$
be the formal completion of $\Bun_G(\ofX_{\on{Ran}^{\on{untl}}})$ along the restriction map
$$\fR:\Bun_G(X)\times \on{Ran}^{\on{untl}}\to \Bun_G(\ofX_{\on{Ran}^{\on{untl}}}).$$

\medskip 

By definition, as $S$-point of $\Bun_G(\ofX_{\on{Ran}})\strut^{\wedge}$ is given by:

\begin{itemize}

\item
An $S$-point of $\Ran^{\on{untl}}$; 

\item 
A $G$-bundle $\CP$ on $S\times X-\on{Graph}$;
 
\item 
A $G$-bundle $\CP_{\on{red}}$ on $^{\on{red}}\!S\times X$;

\item 
An identification 
$$\CP|_{^{\on{red}}\!S\times X-\on{Graph}}\simeq \CP_{\on{red}}|_{^{\on{red}}\!S\times X-\on{Graph}}.$$

\end{itemize} 

\sssec{}

Applying our main theorems to this situation, we obtain:

\begin{thm} \label{t:main Bun_G}
For an object $\CF\in \IndCoh(\Bun_G(X))$ the tautological maps between the following spaces of data on $\CF$ are isomorphisms:
\begin{enumerate}[label={(\roman*)}]
\item
The data of descent of $\CF$ to an object of $\IndCoh((\Bun_G(X))_\dr)$, i.e., the data of exhibiting $\CF$
as $\oblv$ of an object in $\Dmod(\Bun_G(X))$.

\item
The data of descent of the pullback of $\CF$ to $\Bun_G\times \Ran^{\on{untl}}$ to an object of the category
$\IndCoh(\Bun_G(\ofX_{\on{Ran}^{\on{untl}}})\strut^{\wedge})$;

\item[(ii')]
The data of descent of the pullback of $\CF$ to $\Bun_G(X)\times \Ran$ to an object of
$\IndCoh(\Bun_G(\ofX_{\on{Ran}})\strut^{\wedge})$. 
\end{enumerate}
\end{thm}

\begin{cor} \label{c:Loc BunG}
The localization functor
$$\on{pre-Loc}: \IndCoh(\Bun_G(\ofX_{\on{Ran}^{\on{untl}}})\strut^{\wedge})\to \IndCoh(\Bun_G(X)),$$ defined as 
\begin{multline*}
\IndCoh(\Bun_G(\ofX_{\on{Ran}^{\on{untl}}})\strut^{\wedge}) \overset{\fr^!_{\on{Ran}^{\on{untl}}}}\longrightarrow
\IndCoh(\Bun_G(X)\otimes \Dmod(\on{Ran}^{\on{untl}}) \to \\
\overset{\on{Id}\otimes \Gamma_{c,\Ran^{\on{untl}}}}\longrightarrow \IndCoh(\Bun_G(X))
\end{multline*}
factors canonically as
$$(p_{\Bun_G,\dr})^! \circ \on{Loc},$$
where $\on{Loc}:\IndCoh(\Bun_G(\ofX_{\on{Ran}^{\on{untl}}})\strut^{\wedge})\to \IndCoh(\Bun_G(X)_\dr)$ is the functor
$$(\on{Id}\otimes \Gamma_{c,\on{Ran}^{\on{untl}}}) \circ \fq^\IndCoh_*,\quad 
\fq: \Bun_G(\ofX_{\on{Ran}^{\on{untl}}})\strut^{\wedge} \to \Bun_G(X)_\dr \times \on{Ran}^{\on{untl}}.$$
\end{cor}

\section{Uniformization and WZW conformal blocks} \label{s:WZW}

In this section we will let $X$ be a smooth projective curve and $G$ a reductive group. 
We will write $\Bun_G$ instead of $\Bun_G(X)$. 

\ssec{Hecke equivariant sheaves}

\sssec{}

Recall the Hecke groupoid $\on{Hecke}(\Bun_G)_{\on{Ran}^{\on{untl}}}$ acting on 
$\Bun_G\times \on{Ran}^{\on{untl}}$.

\medskip

Pull-push along the two projections
$$\partial_s,\partial_t:\on{Hecke}(\Bun_G)_{\on{Ran}^{\on{untl}}}\to \Bun_G\times \Ran^{\on{untl}}$$
defines a monad, to be denoted $\CH(\Bun_G)_{\Ran^{\on{untl}}}$ acting on 
$$\IndCoh(\Bun_G\times \Ran^{\on{untl}}) \simeq \IndCoh(\Bun_G)\otimes \Dmod(\Ran^{\on{untl}}).$$

By construction, the category
$$\CH(\Bun_G)_{\Ran^{\on{untl}}}\mod(\IndCoh(\Bun_G\times \Ran^{\on{untl}}))$$
consists of objects in $\IndCoh(\Bun_G\times \Ran^{\on{untl}})$ equipped with a data of equivariance with respect to
groupoid $\on{Hecke}(\Bun_G)_{\on{Ran}^{\on{untl}}}$, which we will also denote by
$$\IndCoh(\Bun_G\times \Ran^{\on{untl}})^{\on{Hecke}_{\Ran^{\on{untl}}}}.$$

\sssec{} \label{sss:bundles on punctured again}

Recall the laft-def categorical prestack $\Bun_G(\ofX_{\on{Ran}^{\on{untl}}})$, see \secref{sss:bundles on punctured}. Note that
since $G$ is reductive, the projection 
$$\fR:\Bun_G \times \Ran^{\on{untl}}\to \Bun_G(\ofX_{\on{Ran}^{\on{untl}}})$$
is ind-proper (see \cite[Theorem A.3.7]{GLys}). 

\medskip

Hence, we can consider the monad $\fR^!\circ \fR^{\IndCoh}_*$ acting on $\IndCoh(\Bun_G\times \Ran^{\on{untl}})$. By base change,
this monad identifies with $\CH(\Bun_G)_{\Ran^{\on{untl}}}$.

\begin{rem}
The map $\fR$ is surjective at the level of $k$-points. This implies that the functor $\fR^!$ is conservative, i.e., conditions
of the Barr-Beck-Lurie theorem apply, and we obtain that the functor $\fR^!$ induces an \emph{equivalence}
$$\IndCoh(\Bun_G(\ofX_{\on{Ran}^{\on{untl}}}))\to \CH(\Bun_G)_{\Ran^{\on{untl}}}\mod(\IndCoh(\Bun_G\times \Ran^{\on{untl}})).$$
\end{rem} 

\sssec{}

As in \secref{ss:pushforward monad}, from the monad $\CH(\Bun_G)_{\Ran^{\on{untl}}}$ we can create a monad
$\Gamma_{c,\Ran^{\on{untl}}}(\CH(\Bun_G)_{\Ran^{\on{untl}}})$ acting on $\IndCoh(\Bun_G)$. 

\medskip

The underlying endofunctor of $\IndCoh(\Bun_G)$ is given by 
\begin{equation} \label{e:monad downstairs BunG}
(\on{Id}\otimes \Gamma_{c,\Ran^{\on{untl}}})\circ \CH(\Bun_G)_{\Ran^{\on{untl}}} \circ (\on{Id}\otimes p^!_{\Ran^{\on{untl}}}),
\end{equation} 
see \eqref{e:monad downstairs}. 

The structure of monad on the endofunctor \eqref{e:monad downstairs BunG} is given by the structure of monad
on $\CH(\Bun_G)_{\Ran^{\on{untl}}}$, and the monoidal structure on $\Gamma_{c,\Ran^{\on{untl}}}$.

\begin{rem} \label{r:full Hecke Ran}
Note that by \thmref{t:Ran cofinality new}, the endofunctor of $\IndCoh(\Bun_G)$ underlying the monad 
$\Gamma_{c,\Ran^{\on{untl}}}(\CH(\Bun_G)_{\Ran^{\on{untl}}})$ identifies with 
$$(\on{Id}\otimes \Gamma_{c,\Ran})\circ \CH(\Bun_G)_{\Ran} \circ (\on{Id}\otimes p^!_{\Ran}).$$
\end{rem}

\sssec{}

As in \secref{sss:modules for pushforward monad}, the category
$$\Gamma_{c,\Ran^{\on{untl}}}(\CH(\Bun_G)_{\Ran^{\on{untl}}})\mod(\IndCoh(\Bun_G))$$
identifies with
$$\IndCoh(\Bun_G) \underset{\IndCoh(\Bun_G\times \Ran^{\on{untl}})}\times \IndCoh(\Bun_G\times \Ran^{\on{untl}})^{\on{Hecke}_{\Ran^{\on{untl}}}}.$$

By \secref{ss:non-unital}, we can equivalently rewrite it as
$$\IndCoh(\Bun_G) \underset{\IndCoh(\Bun_G\times \Ran)}\times \IndCoh(\Bun_G\times \Ran^{\on{untl}})^{\on{Hecke}_{\Ran}}.$$

We will use the notation:
$$\Gamma_{c,\Ran^{\on{untl}}}(\CH(\Bun_G)_{\Ran^{\on{untl}}})\mod(\IndCoh(\Bun_G))=:\IndCoh(\Bun_G)^{\on{Hecke}}.$$

We will denote by $\oblv^{\on{Hecke}}$ the forgetful functor
$$\IndCoh(\Bun_G)^{\on{Hecke}}\to \IndCoh(\Bun_G).$$

\sssec{}

Note that the pullback functor
$$p_{\Bun_G}^!:\Vect=\IndCoh(\on{pt})\to \IndCoh(\Bun_G)$$ naturally lifts to a functor
\begin{equation} \label{e:full Hecke equiv}
\Vect \to \IndCoh(\Bun_G)^{\on{Hecke}}.
\end{equation} 

The main result of this subsection is:

\begin{thm} \label{t:Hecke equiv}
The functor \eqref{e:full Hecke equiv} is an equivalence.
\end{thm}

Here is a key corollary of \thmref{t:Hecke equiv}:

\begin{cor} \label{c:Hecke equiv}
For any $k$-point $y\in \Bun_G$, the composite functor
$$\IndCoh(\Bun_G)^{\on{Hecke}}\overset{\oblv^{\on{Hecke}}}\to \IndCoh(\Bun_G) \overset{i_y^!}\to \Vect$$
is an equivalence, where $i_y$ is the map $\on{pt}\to \Bun_G$, corresponding to $y$.
\end{cor} 

\begin{proof}[Proof of \corref{c:Hecke equiv}]
Indeed, the composite functor
$$\Vect \overset{p_{\Bun_G}^!}\to \IndCoh(\Bun_G)^{\on{Hecke}}\overset{\oblv^{\on{Hecke}}}\to \IndCoh(\Bun_G) \overset{i_y^!}\to \Vect$$
is the identity functor.
\end{proof} 

\ssec{Proof of \thmref{t:Hecke equiv}}

\sssec{}

We can use the groupoid $\on{Hecke}(\Bun_G(X))_{\on{Ran}^{\on{untl}}}$ to obtain a monad acting on the category
$$\Dmod(\Bun_G\times \Ran^{\on{untl}}),$$
to be denoted $\CH((\Bun_G)_\dr)_{\Ran^{\on{untl}}}$, and subsequently a monad, denoted $$\Gamma_{c,\Ran^{\on{untl}}}(\CH((\Bun_G)_\dr)_{\Ran^{\on{untl}}}),$$
acting on $\Dmod(\Bun_G)$. 

\medskip

Consider the corresponding categories of modules, denoted 
$$\Dmod(\Bun_G\times \Ran^{\on{untl}})^{\on{Hecke}_{\Ran^{\on{untl}}}} \text{ and } \Dmod(\Bun_G)^{\on{Hecke}},$$
respectively, where the latter identifies with 
$$\Dmod(\Bun_G)\underset{\Dmod(\Bun_G\times \Ran^{\on{untl}})}\times
 \Dmod(\Bun_G\times \Ran^{\on{untl}})^{\on{Hecke}_{\Ran^{\on{untl}}}},$$
which can also be written as 
$$\Dmod(\Bun_G)\underset{\Dmod(\Bun_G\times \Ran)}\times \Dmod(\Bun_G\times \Ran)^{\on{Hecke}_\Ran}.$$

As in \eqref{e:full Hecke equiv}, we have a tautologically defined functor
\begin{equation} \label{e:full Hecke equiv dr}
\Vect \to \Dmod(\Bun_G)^{\on{Hecke}}.
\end{equation} 

We have the following result of \cite[Theorem 4.4.9]{Contr}: 

\begin{thm}[Gaitsgory] \label{t:Hecke equiv dr}
The functor \eqref{e:full Hecke equiv dr} is an equivalence.
\end{thm}

%

\begin{rem}
We will deduce \thmref{t:Hecke equiv} from \thmref{t:Hecke equiv dr} and \corref{c:Loc BunG}.

Informally, the idea is that
$\on{Hecke}(\Bun_G)_{\on{Ran}^{\on{untl}}}$ contains $\widehat{\on{Hecke}}(\Bun_G)_{\on{Ran}^{\on{untl}}}$ as a normal subgroupoid,
and the quotient gives the de Rham space of $\on{Hecke}(\Bun_G(X))_{\on{Ran}^{\on{untl}}}$. 
\end{rem}

\sssec{}\label{sss:qcoh gamma-c}

Recall from \cite[Corollary 4.3.2]{Contr} that the functor
$$p^!_{\Bun_G,\dr}:\Vect\to \Dmod(\Bun_G),$$
admits a left adjoint; to be denoted $\Gamma_{c,\Bun_G,\dr}$. 

\medskip

The fact that the left adjoint $\Gamma_{c,\Bun_G,\dr}$ to $p^!_{\Bun_G,\dr}$ is well-defined implies that the left adjoint, to be denoted 
$\Gamma_{c,\Bun_G}$, to the functor
$$p^!_{\Bun_G}:\Vect\to \IndCoh(\Bun_G)$$
is well-defined. 

\medskip

Namely, 
$$\Gamma_{c,\Bun_G}\simeq \Gamma_{c,\Bun_G,\dr}\circ \ind,$$
where $\ind$ is the left adjoint to
$$\oblv:\Dmod(\Bun_G)\to \IndCoh(\Bun_G).$$

\sssec{}

The functor \eqref{e:full Hecke equiv} can be interpreted as a map of monads
\begin{equation} \label{e:map of monads}
\Gamma_{c,\Ran^{\on{untl}}}(\CH(\Bun_G)_{\Ran^{\on{untl}}}) \to p^!_{\Bun_G}\circ \Gamma_{c,\Bun_G},
\end{equation} 
and the assertion of \thmref{t:Hecke equiv} is equivalent to the statement that the map \eqref{e:map of monads}
is an isomorphism (at the level of the underlying endofunctors). 

\medskip

Using the interpretation of $\CH(\Bun_G)_{\Ran^{\on{untl}}}$ as $\fR^!\circ \fR^{\IndCoh}_*$ (see \secref{sss:bundles on punctured again}),
we can write the map \eqref{e:map of monads} explicitly as
\begin{multline} \label{e:map of monads again}
(\on{Id}\otimes \Gamma_{c,\Ran^{\on{untl}}})\circ \fR^!\circ \fR^{\IndCoh}_* \circ (\on{Id}\otimes p^!_{\Ran^{\on{untl}}}) \to \\
\to (\on{Id}\otimes \Gamma_{c,\Ran^{\on{untl}}})\circ 
(p_{\Bun_G}^!\otimes p_{\Ran^{\on{untl}}}^!) \circ (\Gamma_{c,\Bun_G}\otimes \Gamma_{c,\Ran^{\on{untl}}}) \circ (\on{Id}\otimes p^!_{\Ran^{\on{untl}}}) \simeq \\
\simeq p^!_{\Bun_G}\circ \Gamma_{c,\Bun_G},
\end{multline}
where the first arrow comes from factoring the projection
$$\Bun_G\times \Ran^{\on{untl}}\to \on{pt}$$ as
$$\Bun_G\times \Ran^{\on{untl}} \overset{\fR}\to \Bun_G(\ofX_{\on{Ran}^{\on{untl}}})\to \on{pt},$$
and the second isomorphism is the counit of the $(\Gamma_{c,\Ran^{\on{untl}}},p^!_{\Ran^{\on{untl}}})$-adjunction,
which is in fact an isomorphism by \propref{p:Ran untl contr}(a). 

\sssec{}

The functor \eqref{e:full Hecke equiv dr} can also be interpreted as a map of monads
\begin{equation} \label{e:map of dr monads}
\Gamma_{c,\Ran^{\on{untl}}}(\CH((\Bun_G)_\dr)_{\Ran^{\on{untl}}})\to p^!_{\Bun_G,\dr}\circ \Gamma_{c,\Bun_G,\dr},
\end{equation} 
and the assertion of \thmref{t:Hecke equiv dr} is equivalent to the statement that \eqref{e:map of dr monads} is an isomorphism
(at the level of the underlying endofunctors). 

\medskip

As in \secref{sss:bundles on punctured again}, we can consider the laft-def categorical prestack $\Bun_G(\ofX_{\on{Ran}^{\on{untl}}})$, and
the adjoint pair of functors
$$\fR_{!,\dr}:\Dmod(\Bun_G\times \Ran^{\on{untl}}) \rightleftarrows \Dmod(\Bun_G(\ofX_{\on{Ran}^{\on{untl}}})):\fR^!_\dr,$$
and the resulting monad 
$\fR^!_\dr\circ \fR_{!,\dr}$ identifies canonically with $\CH((\Bun_G)_\dr)_{\Ran^{\on{untl}}}$. 

\medskip

In terms of this identification, the map \eqref{e:map of dr monads} can be written as 
\begin{multline} \label{e:map of monads dr again}
(\on{Id}\otimes \Gamma_{c,\Ran^{\on{untl}}})\circ \fR^!_\dr\circ \fR_{!,\dr} \circ (\on{Id}\otimes p^!_{\Ran^{\on{untl}}}) \to \\
\to (\on{Id}\otimes \Gamma_{c,\Ran^{\on{untl}}})\circ 
(p_{\Bun_G,\dr}^!\otimes p_{\Ran^{\on{untl}}}^!) \circ (\Gamma_{c,\Bun_G,\dr}\otimes \Gamma_{c,\Ran^{\on{untl}}}) \circ (\on{Id}\otimes p^!_{\Ran^{\on{untl}}}) \simeq \\
\simeq p^!_{\Bun_G,\dr}\circ \Gamma_{c,\Bun_G}. 
\end{multline}

\sssec{}

We write the left-hand side in \eqref{e:map of monads again} as push-pull along the following diagram, where we apply !-pullback along the horizontal arrows and
$\IndCoh$-pushforward along the vertical arrows:
$$
\CD
& & \Bun_G \times \Ran^{\on{untl}} @>>> \Bun_G \\
& & @VV{\fR}V  \\
\Bun_G \times \Ran^{\on{untl}} @>{\fR}>> \Bun_G(\ofX_{\on{Ran}^{\on{untl}}}) \\
@VVV \\
\Bun_G,
\endCD
$$
which we can expand as 
$$
\CD
& & & & \Bun_G \times \Ran^{\on{untl}} @>>> \Bun_G \\
& & & & @VV{\fr}V  \\
& & & & \Bun_G(\ofX_{\on{Ran}^{\on{untl}}})\strut^{\wedge} \\
& & & & @VVV  \\
\Bun_G \times \Ran^{\on{untl}} @>{\fr}>>  \Bun_G(\ofX_{\on{Ran}^{\on{untl}}})\strut^{\wedge} @>>>  \Bun_G(\ofX_{\on{Ran}^{\on{untl}}}) \\
@VVV \\
\Bun_G.
\endCD
$$

We now recall that by \corref{c:Loc BunG}, the push-pull along the latter diagram maps isomorphically to one along
\begin{equation} \label{e:ult diag}
\CD
& & & & \Bun_G \times \Ran^{\on{untl}} @>>> \Bun_G \\
& & & & @VV{\fr}V  \\
& & & & \Bun_G(\ofX_{\on{Ran}^{\on{untl}}})\strut^{\wedge} \\
& & & & @VVV  \\
& & \Bun_G(\ofX_{\on{Ran}^{\on{untl}}})\strut^{\wedge} @>>>  \Bun_G(\ofX_{\on{Ran}^{\on{untl}}}) \\
& & @VVV \\
& & (\Bun_G)_\dr\times \on{Ran}^{\on{untl}} \\
& & @VVV \\
\Bun_G @>>> (\Bun_G)_\dr
\endCD
\end{equation}

\medskip

Using the Cartesian diagram
$$
\CD
\Bun_G(\ofX_{\on{Ran}^{\on{untl}}})\strut^{\wedge} @>>>  \Bun_G(\ofX_{\on{Ran}^{\on{untl}}})  \\
@VVV @VVV \\
(\Bun_G)_\dr\times \on{Ran}^{\on{untl}} @>>> \Bun_G(\ofX_{\on{Ran}^{\on{untl}}})_\dr
\endCD
$$
we replace diagram \eqref{e:ult diag} by 

$$
\CD
& & & & \Bun_G \times \Ran^{\on{untl}} @>>> \Bun_G \\
& & & & @VV{\fr}V  \\
& & & & \Bun_G(\ofX_{\on{Ran}^{\on{untl}}})\strut^{\wedge} \\
& & & & @VVV  \\
& &  & &  \Bun_G(\ofX_{\on{Ran}^{\on{untl}}}) \\
& & & & @VVV \\
& & (\Bun_G)_\dr\times \on{Ran}^{\on{untl}} @>>> \Bun_G(\ofX_{\on{Ran}^{\on{untl}}})_\dr \\
& & @VVV \\
\Bun_G @>>> (\Bun_G)_\dr.
\endCD
$$

\medskip

We contract the latter diagram to:

$$
\CD
& & & & \Bun_G \times \Ran^{\on{untl}} @>>> \Bun_G \\
& & & & @VVV \\
& & & & (\Bun_G)_\dr\times \on{Ran}^{\on{untl}} \\
& & & & @VVV \\
& & (\Bun_G)_\dr\times \on{Ran}^{\on{untl}} @>>> \Bun_G(\ofX_{\on{Ran}^{\on{untl}}})_\dr \\
& & @VVV \\
\Bun_G @>>> (\Bun_G)_\dr.
\endCD
$$

Applying base change, we further replace the latter diagram by:
\begin{equation} \label{e:ult diag next next}
\CD
& & & & & & \Bun_G \\
& & & & & & @VVV \\
& & & & (\Bun_G)_\dr\times \on{Ran}^{\on{untl}} @>>> (\Bun_G)_\dr \\
& & & & @VVV \\
& & (\Bun_G)_\dr\times \on{Ran}^{\on{untl}} @>>> \Bun_G(\ofX_{\on{Ran}^{\on{untl}}})_\dr \\
& & @VVV \\
\Bun_G @>>> (\Bun_G)_\dr.
\endCD
\end{equation}

\sssec{}

We now use the fact that the map \eqref{e:map of monads dr again} is an isomorphism. Hence, we can replace diagram
\eqref{e:ult diag next next} by
$$
\CD
& & & & \Bun_G \\
& & & & @VVV \\
& & & & (\Bun_G)_\dr \\
& & & & @VVV \\
\Bun_G @>>> (\Bun_G)_\dr @>>> \on{pt} \\
\endCD
$$

The latter diagram is the same as
$$
\CD
& & \Bun_G \\
& & @VVV \\
\Bun_G @>>> \on{pt}, \\
\endCD
$$
i.e., the right-hand side in \eqref{e:map of monads again}. 

\sssec{}

It is a straightforward verification that the map from the left-hand side of \eqref{e:map of monads again} to
the right-hand side of \eqref{e:map of monads again} we have constructed by the above diagram chase agrees
with the one in \eqref{e:map of monads again}. 

\qed[\thmref{t:Hecke equiv}]

\ssec{A uniformization result} \label{ss:O contr}

\sssec{}

%
%
%
%
%
%
%
%
%

Let $\Gr_{G,\Ran}$ be the Ran version of the affine Grassmannian of $G$;
i.e., for a test affine scheme $S$, and $S$-point of $\Gr_{G,\Ran}$ is
the data of:

\begin{itemize}

\item 
An $S$-point of $\Ran^{\on{untl}}$; 

\item
A $G$-bundles $\CP$ on $S\times X$;

\item
An isomorphism
$\alpha: \CP|_{S \times X- \on{Graph}} \simeq \CP^0|_{S \times X - \on{Graph}}$,
where $\CP^0$ is the trivial $G$-bundle.  

\end{itemize}

Let $\pi$ denote the forgetful map
$$\Gr_{G,\Ran}\to \Bun_G,$$
and let $\pi$ denote its restriction to $\Gr_{G,\Ran}$.

\sssec{}

From \thmref{t:Hecke equiv} we will now deduce the following uniformization result: 

\begin{thm} \label{t:uniformization}
For $\CF\in \IndCoh(\Bun_G)$ and $V\in \Vect$, the map
$$\CHom_{\IndCoh(\Bun_G)}(\CF,V\otimes \omega_{\Bun_G})\to
\CHom_{\IndCoh(\Gr_{G,\Ran})}(\pi^!(\CF),V\otimes \omega_{\Gr_{G,\Ran}})$$
is an isomorphism.
\end{thm}

Recall that for any laft prestack $\CY$, the functor $\Upsilon_\CY:\QCoh(\CY)\to \IndCoh(\CY)$
is fully faithful when restricted to the full subcategory $\on{Perf}(\CY)\subset \QCoh(\CY)$
of monoidally dualizable (a.k.a. perfect) objects. Hence, from \thmref{t:uniformization}
we obtain:

\begin{cor} \label{c:uniformization}
The functor $\pi^*: \QCoh(\Bun_G)\to \QCoh(\Gr_{G,\Ran})$ is fully faithful when restricted to
$\on{Perf}(\Bun_G)\subset \QCoh(\Bun_G)$.
\end{cor} 

\begin{rem} 
One can ask more ambitiously, whether for a pair of objects $\CF_1,\CF_2\in \IndCoh(\Bun_G)$, the map
\begin{equation} \label{e:contr}
\CHom_{\IndCoh(\Bun_G)}(\CF_1,\CF_2)\to
\CHom_{\IndCoh(\Gr_{G,\Ran})}(\pi^!(\CF_1),\pi^!(\CF_2))
\end{equation} 
is an isomorphism. We do not currently know the answer.

\medskip

This question is equivalent to the following one. Let
$\underline\Maps^{\on{gen}}(X,G)_{\Ran}$ be the prestack that assigns to a test affine scheme $S$
the groupoid of

\begin{itemize}

\item 
An $S$-point of $\Ran^{\on{untl}}$;

\item 
A map $(S \times X- \on{Graph})\to G$.

\end{itemize}

Is it true that the functor
$$\Vect\to \IndCoh(\underline\Maps^{\on{gen}}(X,G)_{\Ran}), \quad k\mapsto \omega_{\underline\Maps^{\on{gen}}(X,G)_{\Ran}}$$
is fully faithful?

\medskip

Since the prestack $\underline\Maps^{\on{gen}}(X,G)_{\Ran}$ is laft and formally smooth, by \cite[Theorem 10.1.1]{IndSch}, the functor 
$$\Upsilon:\QCoh(\underline\Maps^{\on{gen}}(X,G)_{\Ran})\to \IndCoh(\underline\Maps^{\on{gen}}(X,G)_{\Ran})$$
is an equivalence. So, the above question is equivalent to the following one:

\medskip

Is the functor
$$\Vect\to \QCoh(\underline\Maps^{\on{gen}}(X,G)_{\Ran}), \quad k\mapsto \CO_{\underline\Maps^{\on{gen}}(X,G)_{\Ran}}$$
is fully faithful?

\end{rem} 

\sssec{Proof of \thmref{t:uniformization}}

The forgetful functor
$$\IndCoh(\Bun_G)^{\on{Hecke}}\to \IndCoh(\Bun_G)$$
admits a left adjoint; we denote it by $\ind^{\on{Hecke}}$. 

\medskip

Thus, we can rewrite 
$\CHom_{\IndCoh(\Bun_G)}(\CF,V\otimes \omega_{\Bun_G})$ as
$$\CHom_{\IndCoh(\Bun_G)^{\on{Hecke}}}\left(\ind^{\on{Hecke}}(\CF),V\otimes \omega_{\Bun_G}\right).$$

Applying \corref{c:Hecke equiv}, for the map $\on{pt}\to \Bun_G$ corresponding to the trivial $G$-bundle. 
we map the latter expression isomorphically to 
\begin{equation} \label{e:Hom on Gr prel}
\CHom_{\Vect}\left(i_0^!\circ \oblv^{\on{Hecke}}\circ \ind^{\on{Hecke}}(\CF),V\right).
\end{equation} 

Recall that by Remark \ref{r:full Hecke Ran}, the endofunctor $\oblv^{\on{Hecke}}\circ \ind^{\on{Hecke}}$
is isomorphic to 
$$(\on{Id}\otimes \Gamma_{c,\Ran})\circ (\partial_t)^{\IndCoh}_*\circ (\partial_s)^!\circ (\on{Id}\otimes p^!_{\Ran}),$$
where
$$\partial_s,\partial_t:\on{Hecke}(\Bun_G)_{\on{Ran}}\to \Bun_G\times \Ran$$
are the two projections.

\medskip

So, we can further rewrite \eqref{e:Hom on Gr prel} as 
\begin{equation} \label{e:Hom on Gr}
\CHom_{\Vect}\left(i_0^!\circ (\on{Id}\otimes \Gamma_{c,\Ran})\circ (\partial_t)^{\IndCoh}_*\circ (\partial_s)^!\circ (\on{Id}\otimes p^!_{\Ran})(\CF),V\right).
\end{equation} 

Consider the Cartesian diagram
$$
\CD
\Gr_{G,\Ran} @>>> \on{Hecke}(\Bun_G)_{\on{Ran}} @>{\partial_s}>> \Bun_G \times \Ran @>{\on{Id}\otimes p^!_{\Ran}}>> \Bun_G \\
@V{q}VV @VV{\partial_t}V \\
\Ran @>{i_0\times \on{id}}>> \Bun_G \times \Ran \\
@VVV @VVV \\
\on{pt} @>{i_0}>> \Bun_G,
\endCD
$$
where the composite horizontal arrow is $\pi$. Hence, by base change,
$$i_0^!\circ (\on{Id}\otimes \Gamma_{c,\Ran})\circ (\partial_t)^{\IndCoh}_*\circ (\partial_s)^!\circ (\on{Id}\otimes p^!_{\Ran})\simeq
\Gamma_{c,\Ran} \circ q^{\IndCoh}_* \circ \pi^!.$$

Hence, we can map the expression \eqref{e:Hom on Gr} isomorphically to
$$\CHom_{\Vect}(\Gamma_{c,\Ran} \circ q^{\IndCoh}_* \circ \pi^!(\CF),V),$$
which by adjunction is the same as 
$$\CHom_{\IndCoh(\Gr_{G,\Ran})}(\pi^!(\CF),V\otimes \omega_{\Gr_{G,\Ran}}).$$

It is a straightforward verification that the resulting map
$$\CHom_{\IndCoh(\Bun_G)}(\CF,V\otimes \omega_{\Bun_G})\to
\CHom_{\IndCoh(\Gr_{G,\Ran})}(\pi^!(\CF),V\otimes \omega_{\Gr_{G,\Ran}})$$
is given by applying the functor $\pi^!$. 

\qed[\thmref{t:uniformization}]

\ssec{WZW conformal blocks}

\sssec{}\label{sss:fact line bundles}

Let $\on{FactPic}(\Gr_{G,\Ran})$ be the Picard category of factorizable line bundle on $\Gr_{G,\Ran}$, see \cite[Sect. 1.2.5]{K2}. 
According to \cite[Theorem 0.1]{TZ}, we have a fiber sequence of Picard categories 
$$0\to \ul\Hom(\pi_{1,\on{alg}}(G),\on{Pic}(X)) \to \on{FactPic}(\Gr_{G,\Ran})\to \on{Quad}(\Lambda,\BZ)^W\to 0,$$
where:

\begin{itemize}

\item $\Lambda$ is the coweight lattice of $G$ and $W$ is the Weyl group;

\item $\on{Quad}(\Lambda,\BZ)^W$ is the discrete Picard groupoid (i.e., abelian group) of
$W$-invariant integer-valued quadratic forms on $\Lambda$;

\item $\pi_{1,\on{alg}}(G)$ is the algebraic fundamental group of $G$, i.e., $\Lambda/\Lambda_{\on{sc}}$,
where $\Lambda_{\on{sc}}$ is the coweight lattice of the simply-connected cover of the derived group of $G$;

\item $\ul\Hom(-,-)$ is the space of maps of the corresponding Picard categories
(i.e., $1$-truncated commutative groups in $\Spc$), viewed
itself as a Picard category. 

\end{itemize}

For an object $\CL\in \on{FactPic}(\Gr_{G,\Ran})$, the corresponding quadratic form 
$$\kappa\in \on{Quad}(\Lambda,\BZ)^W$$
is called the \emph{level} of $\CL$.

\sssec{}

Let $\CL$ be an object of $\on{FactPic}(\Gr_{G,\Ran})$. To it we can associate a \emph{factorization algebra}
$\CA_\CL$ in $\Dmod(\Ran)$ by setting
$$\CA_\CL:=q^{\IndCoh}_*(\CL\otimes \omega_{\Gr_{G,\Ran}}).$$

We have the following fundamental result, which is an affine version of the Borel-Weil-Bott theorem, 
see \cite[Theorem 8.3.11]{Ku} (we include another proof by Gaitsgory in
\secref{s:BBW}): 

\begin{thm} \label{t:affine BBW} \hfill
\begin{enumerate}[label={(\alph*)}]
\item
Suppose that the restriction of the level $\kappa$ to $\Lambda_{\on{sc}}$
is non-positive definite. Then the restriction of $\CA_\CL$ to $X\to \Ran$,
viewed as a D-module on $X$, lives in single cohomological degree $-1$.

\item
Suppose that the restriction of the level $\kappa$ to one of the simple factors in $\Lambda_{\on{sc}}$
is positive-definite. Then $\CA_\CL=0$.

\item
Let $G$ be semi-simple and simply-connected and $\kappa=0$.
Then the unit section $\Ran\to \Gr_{G,\Ran}$ gives rise to an isomorphism
$$\omega_{\Ran}\to \CA_\CL.$$
\end{enumerate}
\end{thm}

The above theorem implies that for $\kappa$ satisfying the non-positive definiteness
assumption of the theorem, the object $\CA_\CL$ can be considered as a \emph{classical}
factorization algebra; i.e., $\CA_\CL|_X[-1]$ is a classical chiral algebra in the sense of
\cite{BD}.

\begin{rem} \label{r:int}
The proof of \thmref{t:affine BBW}(a) in \secref{s:BBW} shows that when $G$ is semi-simple and simply connected
(in which case, $\CL$ can be uniquely recovered from that of $\kappa$), and $\kappa$
non-positive, the chiral algebra $$\CA_{\CL,X}:=\CA_\CL|_X[-1]$$ is the 
integrable quotient of the Kac-Moody chiral algebra corresponding to $\fg$ and level $-\kappa$.

\medskip

At the other extreme, when $G=T$ is a torus, $\CA_{\CL,X}$ is the lattice chiral algebra
of \cite[Sect. 3.10]{BD}. 

\end{rem}

\sssec{}

Consider the vector space
$$\Gamma_{c,\Ran}(\CA_\CL).$$ 

By construction
\begin{equation} \label{e:pre WZW}
\Gamma_{c,\Ran}(\CA_\CL)\simeq \Gamma^\IndCoh(\Gr_{G,\Ran},\CL\otimes \omega_{\Gr_{G,\Ran}}).
\end{equation} 

By definition, $\Gamma_{c,\Ran}(\CA_\CL)$ is the factorization (a.k.a. chiral) homology of the (DG) chiral algebra  $\CA_{\CL,X}:=\CA_\CL|_X[-1]$, to be denoted also
$$\on{C}^{\on{Fact}}_\cdot(X,\CA_{\CL,X}).$$

\begin{rem} 
In light of Remark \ref{r:int}, when $G$ is semi-simple and simply-connected and $\kappa$
non-positive definite, the object $\on{C}^{\on{Fact}}_\cdot(X,\CA_{\CL,X})\in \Vect$ is the (vacuum) chiral homology
of the integrable quotient of the Kac-Moody chiral algebra corresponding at level $-\kappa$.

\medskip

By definition, its $0$th cohomology, denoted
$$H^{\on{Fact}}_0(X,\CA_{\CL,X}),$$
is the space of global (vacuum) coinvariants of $\CA_{\CL,X}$. Thus, its linear dual 
$$\left(H^{\on{Fact}}_0(X,\CA_{\CL,X})\right)^\vee$$
is the space of (vacuum) conformal blocks of the WZW model at level $\kappa$.
\end{rem} 

\ssec{(Higher) WZW conformal blocks as (derived) global sections}

\sssec{}

Let $\CL$ be as above. According to \cite[Sect. 2.4]{K2}, there exists a canonically defined line bundle, to be denoted
$\CL_{\on{glob}}$ on $\Bun_G$ such that
\begin{equation} \label{e:pullback global}
\CL\simeq \pi^*(\CL_{\on{glob}}).
\end{equation} 

\begin{rem}
By \corref{c:uniformization}, the data of $\CL_{\on{glob}}$ equipped with an isomorphism \eqref{e:pullback global} is unique. 
\end{rem}

\sssec{}

Thus, from \thmref{t:uniformization}, we obtain:

\begin{thm} \label{t:WZW}
The map 
$$\Gamma^\IndCoh(\Gr_{G,\Ran},\CL\otimes \omega_{\Gr_{G,\Ran}}) \to 
\Gamma_c(\Bun_G,\CL_{\on{glob}} \otimes \omega_{\Bun_G}),$$
induced by the functor $\pi^!$, is an isomorphism. 
\end{thm} 

Combining with \eqref{e:pre WZW}, we can reformulate \thmref{t:WZW} as follows:

\begin{thm}  \label{t:WZW'}
There exists a canonical isomorphism
$$\on{C}^{\on{Fact}}_\cdot(X,\CA_{\CL,X})\to \Gamma_c(\Bun_G,\CL_{\on{glob}}\otimes \omega_{\Bun_G}).$$
\end{thm}

Passing to the duals, and noting that
\begin{multline*}
\Gamma_c(\Bun_G,\CL_{\on{glob}}\otimes \omega_{\Bun_G})^\vee \simeq 
\CHom_{\IndCoh(\Bun_G)}(\CL\otimes \omega_{\Bun_G},\omega_{\Bun_G}) \simeq \\
\simeq  \CHom_{\QCoh(\Bun_G)}(\CL,\CO_{\Bun_G})\simeq \Gamma(\Bun_G,\CL^{-1}),
\end{multline*}
we obtain:

\begin{cor} \label{c:WZW}
There exists a canonical isomorphism
$$\Gamma(\Bun_G,\CL^{-1})\simeq \on{C}^{\on{Fact}}_\cdot(X,\CA_{\CL,X})^\vee.$$
\end{cor}

\begin{rem}
Suppose that $G$ is semi-simple and simply-connected. Then at the level of $0$-th
cohomology, the isomorphism
$$H^0(\Bun_G,\CL^{-1}) \simeq \left(H^{\on{Fact}}_0(X,\CA_{\CL,X})\right)^\vee$$
is well-known.

\medskip

The point of \corref{c:WZW} is that it shows that this isomorphism extends (at the derived level)
to the full chiral homology. 

\medskip

Note that for $G=T$ a torus, the assertion of \corref{c:WZW} coincides with that of \cite[Theorem 4.9.3]{BD}.

\end{rem} 

\sssec{}

Let's consider some particular cases of \corref{c:WZW}.
First, by combining \corref{c:WZW} with \thmref{t:affine BBW}(b,c), we obtain: 

\begin{cor}
\hfill
\begin{enumerate}[label={(\alph*)}]
\item
Suppose that the restriction of the level $\kappa$ to one of the simple factors in $\Lambda_{\on{sc}}$ is 
negative definite. Then $\Gamma(\Bun_G,\CL)=0$.
\item
Suppose that $G$ is semi-simple and simply-connected. Then
$k\to \Gamma(\Bun_G,\CO_{\Bun_G})$
is an isomorphism. 
\end{enumerate}
\end{cor}

Next, we consider the case the case of tori. 

\begin{cor} Let $G=T$ be a torus.
\begin{enumerate}[label={(\alph*)}]

\item
If $\CL$ is the trivial factorization line bundle, then $\on{C}^{\on{Fact}}_\cdot(X,\CA_{\CL,X})$
is isomorphic to $\Lambda$-many copies of (the exterior algebra) $\Sym(H^0(X,\omega_X)[1]\otimes \ft)$,
where $\ft:=\Lie(T)$. 

\item
If $\kappa=0$ but $\CL$ is non-trivial, then $\on{C}^{\on{Fact}}_\cdot(X,\CA_{\CL,X})=0$. 

\item
If $\kappa$ is non-degenerate, then $\on{C}^{\on{Fact}}_\cdot(X,\CA_{\CL,X})$ is finite-dimensional.

\item
If $\kappa$ is negative-definite, then $\on{C}^{\on{Fact}}_\cdot(X,\CA_{\CL,X})$ is concentrated in cohomological
degree $0$.

\item
If $\kappa$ is positive-definite, then $\on{C}^{\on{Fact}}_\cdot(X,\CA_{\CL,X})$ is concentrated in cohomological
degree $-\dim(T)\cdot g$, where $g$ is the genus of $X$.
\end{enumerate}
\end{cor}

\begin{proof}

We use \corref{c:WZW} to calculate $\on{C}^{\on{Fact}}_\cdot(X,\CA_{\CL,X})$. 

\medskip

In case (a), we are calculating 
the predual of the space of
(derived) global sections of the structure sheaf of $\Bun_T$, whence the result.

\medskip

In case (b), we are calculating the space of global sections of a line bundle of degree $0$; hence we obtain $0$,
unless the line bundle in question is trivial.

\medskip

In case (c), it is easy to see that the group $T$, which acts by $2$-automorphisms of $\Bun_T$, 
acts by character 
$\kappa(\lambda,-)$ on the restriction of $\CL$ to the connected component $\Bun_T^\lambda$
of $\Bun_T$. Hence, for $\lambda\neq 0$, we obtain 
$$\Gamma(\Bun^\lambda_T,\CL)=0.$$

Hence,
\begin{equation} \label{e:Pic 0}
\Gamma(\Bun^\lambda_T,\CL)\simeq \Gamma(\Bun^0_T,\CL),
\end{equation} 
while $\Bun_T^0$ is the quotient of the proper scheme 
$$\on{Jac}(X)\underset{\BZ}\otimes \Lambda$$
by the trivial action of $T$. Hence, (derived) global sections
of a coherent sheaf on $\Bun^0_T$ are finite-dimensional. 

\medskip

In cases (d) and (e), still using \eqref{e:Pic 0}, it suffices to show that $\CL|_{\on{Jac}(X)\underset{\BZ}\otimes \Lambda}$ is ample/anti-ample,
which is well-known. 

\end{proof}

\ssec{Vanishing of higher conformal blocks}

\sssec{} 

Now, suppose that $G$ is semi-simple and $\kappa$ is non-negative definite. In this case we have
the following fundamental result from \cite{Te}:

\begin{thm}[Teleman] \label{t:Teleman}
The vector space $H^0(\Bun_G,\CL_{\on{glob}})$ is finite-dimensional, and $H^i(\Bun_G,\CL_{\on{glob}})=0$
for $i>0$.
\end{thm}

Combining with \corref{c:WZW}, we obtain:

\begin{cor} \label{c:WZW pos}
Let $G$ be semi-simple and $\kappa$ non-positive definite. Then 
$\on{C}^{\on{Fact}}_\cdot(X,\CA_{\CL,X})$ is finite-dimensional and is concentrated in 
cohomological degree $0$. 
\end{cor} 

\begin{cor} \label{c:WZW pos pos}
Let $G$ be semi-simple and simply-connected and $\kappa$ a positive level. Then the chiral homology of the
integrable quotient of the Kac-Moody chiral algebra at level $\kappa$ is concentrated in
cohomological degree $0$. 
\end{cor} 

\begin{rem} 
The assertion of \corref{c:WZW pos pos} gives a positive answer in the case of the WZW model to the following
general question.  Let $\CB$ be a rational VOA.  For any curve $X$, we have the corresponding chiral algebra $\CB_X$.
Is the factorization homology $\on{C}^{\on{Fact}}_\cdot(X,\CB_X)$ is concentrated in cohomological degree $0$?
\end{rem} 

\appendix

\section{Finiteness conditions} \label{s:finiteness}

The purpose of this section is to provide a framework for the notion of \emph{sectional left-ness} introduced
in \secref{ss:sect laft}. 

\medskip

We begin with a discussion of coherence in a DG category with a t-structure, 
which will be used to formulate the appropriate finiteness condition on the relative deformation theory.

\ssec{Coherence and renormalization}

\sssec{} \label{sss:unif bdd}

Let $\bC$ be a DG category with a t-structure\footnote{Recall that all DG categories, unless stated otherwise, are assumed to be presentable.  
Similarly, t-structures on a DG category are assumed to be accessible (see \cite[Definition 1.4.4.12]{HA})}.  We will denote by $\bC^b \subset \bC$, 
$\bC^- \subset \bC$, and $\bC^+ \subset \bC$ the
full subcategories of bounded,  bounded above (i.e., eventually connective), and bounded below (i.e., eventually coconnective) objects, respectively.  

\medskip

Let $I$ be an index category.  We will say that a diagram $F: I \to \bC$ is \emph{uniformly bounded}
if there exist $m,n \in \mathbb{Z}$ such that
$$ F(i) \in \bC^{\leq m, \geq n} $$
for all $i\in I$.

\medskip

We will say that the t-structure on $\bC$ is \emph{weakly compatible with filtered colimits} if for every
uniformly bounded filtered diagram $F: I \to \bC$, we have
$$\underset{i\in I}{\on{colim}}\  F(i) \in \bC^b .$$

\begin{rem}

Recall that the t-structure on $\bC$ is \emph{compatible with filtered colimits} if the subcategory
$\bC^{\geq 0}$ is preserved by filtered colimits. This formally implies that all the subcategories
$\bC^{\leq n,\geq m}$ are preserved by filtered colimits. 

\end{rem}

\sssec{}

We will say that $x\in \bC$ is \emph{coherent} if it is bounded and for every uniformly bounded filtered diagram 
$F: I \to \bC$, the natural map
$$\underset{i\in I}{\on{colim}}\  \on{Maps}(x, F(i)) \to \on{Maps}(x, \underset{i\in I}{\on{colim}}\  F(i)) $$
is an isomorphism.

\medskip

Note that if the t-structure on $\bC$ is compatible with filtered colimits, an object $x\in \bC$ is coherent if and only if it is 
bounded and for every $n \in \mathbb{N}$, $\tau^{\geq -n}(x)$ is compact as an object of $\bC^{\geq -n}$.

\sssec{}

We will denote by
$$\on{Coh}(\bC) \subset \bC $$
the full subcategory of coherent objects; this is a stable subcategory of $\bC$.  Consider the category 
$$\on{Ind}(\on{Coh}(\bC)).$$

\medskip

Ind-extension of the inclusion $\on{Coh}(\bC) \to \bC$ defines a functor
$$ \Psi_{\bC}: \Ind(\Coh(\bC)) \to \bC $$
The DG category $\Ind(\Coh(\bC))$ admits a t-structure such that $M \in \Ind(\Coh(\bC))$ is connective iff $\Psi_{\bC}(M)$ is connective.  In particular,
$\Psi_{\bC}$ is right t-exact.

\begin{lem} \label{l:coh t}
Let $\bD$ be another DG category with a t-structure, and let $F: \bC \to \bD$ be a t-bounded functor with a continuous right adjoint.  
Then $F$ preserves coherent objects.  In particular, we have a commutative square
$$ \xymatrix{
\Ind(\Coh(\bC)) \ar[r]^{\tilde{F}}\ar[d]_{\Psi_\bC} & \Ind(\Coh(\bD)) \ar[d]^{\Psi_{\bD}}\\
\bC \ar[r]^F & \bD
}$$
where $\tilde{F}: \Ind(\Coh(\bC)) \to \Ind(\Coh(\bD))$ is the ind-extension of the restriction of $F$ to $\on{Coh}(\bC)$.
\end{lem}

\sssec{}

Let $\bC$ be a DG category with a t-structure.  We will say that $\bC$ is \emph{has enough coherent objects} if
every object $x \in \bC^b$ can be expressed as a uniformly bounded filtered colimit of coherent objects.


\sssec{}

Let $\bC$ be a DG category with a t-structure. We will say that $\bC$ is \emph{coherent} if:
\begin{itemize}
\item
The t-structure is right complete and compatible with filtered colimits.
\item
$\bC$ has enough coherent objects.
\item
The subcategory $\on{Coh}(\bC) \subset \bC$ is closed with respect to truncations; i.e., if $x \in \on{Coh}(\bC)$
so is $\tau^{\leq n}(x)$ for all $n\in \mathbb{Z}$.
\end{itemize}

\sssec{}

%

For our purposes, the main consequence of coherence is the following (see \cite[Chapter 4, Proposition 1.2.2]{Vol1}):

\begin{prop} \label{p:coh coh t}
Let $\bC$ be a coherent DG category with t-structure.  Then:
\begin{enumerate}[label={(\alph*)}]
\item
An object $x \in \bC$ is coherent if and only if it is bounded and for every $n\in \mathbb{Z}$, the object
$$ H^n(x) \in \bC^{\heartsuit} $$
is compact \emph{as an object of} $\bC^{\heartsuit}$.
\item
The functor
$$ \Psi_{\bC}: \Ind(\Coh(\bC)) \to \bC $$
is t-exact and induces an equivalence $\Ind(\Coh(\bC))^+ \simeq \bC^+$.  
\end{enumerate}
\end{prop}

\sssec{Locally Noetherian categories}

Sometimes it is more convenient to work with an even stronger notion: that of locally Noetherian categories.  

\medskip

We will say that that a DG category $\bC$ with a t-structure is locally Noetherian (cf. \cite[Definition C.6.9.1]{SAG}) if
\begin{itemize}
\item
The t-structure on $\bC$ is right complete and compatible with filtered colimits;
\item
The abelian category $\bC^{\heartsuit}$ is locally Noetherian; 
\item
Every compact object $x\in \bC^{\heartsuit}$ is coherent as an object in $\bC$.
\end{itemize}

\medskip

%

\sssec{}

The following is not difficult: 

\begin{lem}\label{l:Noeth coherent}
A locally Noetherian DG category with a t-structure is coherent.
\end{lem}

\sssec{An example} \label{sss:IndCoh sch}
Let $X$ be a laft scheme. Then the category $\QCoh(X)$, equipped with its t-structure
is locally Noetherian, and hence coherent. 

\medskip

We will denote the
category of coherent objects in $\QCoh(X)$ simply by $\Coh(X)$, and
the resulting category $\Ind(\Coh(X))$ by $\IndCoh(X)$. 

\medskip 

The resulting functor $\Psi_{\QCoh(X)}$ is the functor
$$\Psi_X:\IndCoh(X)\to \QCoh(X)$$
of \cite[Chapter 4, Sect. 1.1.2]{Vol1}.

%

\ssec{The \emph{laft} subcategory} \label{ss:laft}

A construction introduced in this subsection plays a key technical role in the main body of the paper.

\sssec{}\label{sss:pro laft cat}

Let $\bD$ be a DG category (not necessarily assumed cocomplete). Recall that the pro-completion of $\bD$,
denoted $\on{Pro}(\bD)$, can be thought of as the category \emph{opposite} to that of (accessible) exact functors
$\bD\to \Vect$. 

\medskip

For $\CF\in \on{Pro}(\bD)$ and $M\in \bD$, we will write $\CF(M)\in \Vect$ for the value of the corresponding functor.

\medskip

The Yoneda embedding $Y_\bD:\bD\to  \on{Pro}(\bD)$ maps $M'\in \bD$ to the functor
$$M\mapsto \CHom_\bD(M',M).$$

\sssec{} \label{sss:pro-functor}

Let $\Phi:\bD\to \bD'$ be a functor. Restriction along $\Phi$ defines a functor
$$\Res_\Phi:\on{Pro}(\bD')\to \on{Pro}(\bD).$$

The functor $\Res_\Phi$ admits a \emph{right} adjoint, given by \emph{left} Kan extension of functors with values in $\Vect$;
we will denote it by 
$$\on{Pro}(\Phi):\on{Pro}(\bD)\to \on{Pro}(\bD'),$$
so that we have a commutative diagram 
$$
\CD 
\bD @>{\Phi}>> \bD' \\
@V{Y_\bD}VV @VV{Y_{\bD'}}V \\
\on{Pro}(\bD) @>{\on{Pro}(\Phi)}>> \on{Pro}(\bD'). 
\endCD
$$

Alternatively, $\on{Pro}(\Phi)$ is the \emph{right} Kan extension of the composite $Y_{\bD'}\circ \Phi$ along the functor $Y_{\bD}$. 

\medskip

If $\Phi$ admits a right adjoint, then we can we can calculate $\on{Pro}(\Phi)$ explicitly by
$$\left(\on{Pro}(\Phi)(\CF)\right)(M')\simeq \CF(\Phi^R(M')), \quad M'\in \bD',\,\, \CF\in \on{Pro}(\bD).$$

\sssec{}

Let $\bC$ be a cocomplete DG category, equipped with a t-structure. 
Suppose that the t-structure on $\bC$ is weakly compatible with filtered colimits. In this case, we introduce the full subcategory
$$\on{Pro}(\bC^{-})_{\on{laft}} \subset \on{Pro}(\bC^{-}) $$
to consist of functors $F: \bC^{-} \to \Vect$ such that:
\begin{enumerate}
\item
$F$ is \emph{convergent}; i.e., for every $x \in \bC^{-}$, the natural map
$$ F(\underset{n}{\varprojlim}\  \tau^{\geq -n} (x)) \to \underset{n}{\varprojlim}\  F(\tau^{\geq -n} (x)) $$
is an isomorphism.  Equivalently, $F$ is the right Kan extension of a functor
$$ F_{| \bC^{b}}: \bC^b \to \Vect .$$
\item
$F$ commutes with uniformly bounded filtered colimits; i.e., if $I\ni i \mapsto x_i$ is a uniformly bounded
filtered diagram in $\bC$, the natural map
$$ \underset{i\in I}{\on{colim}}\  F(x_i) \to F(\underset{i\in I}{\on{colim}}\  x_i) $$
is an isomorphism. 
\end{enumerate}

\sssec{}

The next proposition shows that the subcategory $\on{Pro}(\bC^{-})_{\on{laft}}$ can be expressed as the pro-completion 
of the category of coherent objects of $\bC$, provided that $\bC$ has enough coherent objects: 

\begin{prop}\label{p:pro-coh}
Suppose that $\bC$ is a DG category with a t-structure, which is weakly compatible with filtered colimits.  
If $\bC$ has enough coherent objects, then the restriction functor
$$ \on{Pro}(\bC^{-})_{\on{laft}} \to \on{Pro}(\on{Coh}(\bC)) $$
is an equivalence.
\end{prop}
\begin{proof}
Consider the composite functor
$$ \on{Pro}(\on{Coh}(\bC)) \overset{\on{LKE}}{\longrightarrow} \on{Pro}(\bC^b) \overset{\on{RKE}}{\longrightarrow} \on{Pro}(\bC^-) .$$
It is fully faithful and by the hypotheses on $\bC$, the essential image is exactly given by $\on{Pro}(\bC^-)_{\on{laft}}$.
\end{proof}

\sssec{}

Applying \propref{p:pro-coh} to the
example in \secref{sss:IndCoh sch}, we obtain:

\begin{cor} \label{c:extended Serre}
Let $S$ be an affine scheme almost of finite type. Then the Serre duality anti self-equivalence
$$\BD^{\on{Serre}}:(\Coh(S))^{\on{op}}\simeq \Coh(S)$$
gives rise to an equivalence
$$\IndCoh(S)^{\on{op}} \simeq \on{Pro}(\QCoh(S)^{-})_{\on{laft}}.$$
\end{cor}  

\ssec{Coherence in D-module categories}

\sssec{Recollections on D-modules}\label{sss:dmod recall}

Let $X$ be a scheme almost of finite type, and consider its de Rham prestack $X_\dr$. The category of
(right) D-modules on $X$, is defined by 
$$\Dmod(X):=\IndCoh(X_{\on{dR}}).$$

\medskip

The natural map $p_{X,\dr}: X \to X_{\on{dR}}$ gives the induction functor
$$ \on{ind} := (p_{X,\dr})_*^{\on{IndCoh}}: \IndCoh(X) \to \IndCoh(X_{\on{dR}})$$
which has a conservative right adjoint $\on{oblv}:= (p_{X,\dr})^!$.  
In particular, since $\IndCoh(X)$ is compactly generated by $\on{Coh}(X)$, the category $\IndCoh(X_{\on{dR}})$ is compactly 
generated by objects $\on{ind}(M)$ for $M \in \on{Coh}(X)$.

\medskip

Being the IndCoh category of an inf-scheme, the category $\Dmod(X)$ carries a canonically defined t-structure.
It has the feature that the (a priori, right t-exact) functor $\on{ind}$ is t-exact. In particular, compact objects of
$\Dmod(X)$ are bounded, and hence, coherent. 

\medskip

The above t-structure has particularly favorable finiteness properties. For one thing, the category $\Dmod(X)$ is locally Noetherian. 
But in fact, something much stronger is true: objects of $\Dmod(X)^\heartsuit$ have uniformly bounded cohomological dimension
(in fact, by $\dim(X)$), which implies that the inclusion 
$$\Dmod(X)^c\subset \Coh(\Dmod(X))$$
is an equality. 

\sssec{}

For our purposes, we will need a relative variant of the above properties. Specifically, let $S$ be an affine scheme almost
of finite type. Our object of interest is the category
$$\IndCoh(S\times X_\dr).$$

\sssec{}

First, recall that for a pair of inf-schemes $\CX_1$ and $\CX_2$, external tensor product gives rise to a t-exact equivalence
$$\IndCoh(\CX_1)\otimes \IndCoh(\CX_2)\simeq \IndCoh(\CX_1\times \CX_2).$$

In particular, we obtain an equivalence
$$\IndCoh(S)\otimes \Dmod(X)\simeq \IndCoh(S \times X_\dr).$$

Recall (see \cite[Proposition D.5.6.1]{SAG}) that the property of being locally Noetherian is preserved by tensor
products. From here we obtain:

\begin{cor}
The DG category $\IndCoh(S \times X_\dr)$, equipped with its t-structure, is locally Noetherian, and
in particular, coherent. 
\end{cor} 

Since the functor $\on{ind}$ is t-exact, the same is true for the functor
$$(\on{Id}\otimes \on{ind}): \IndCoh(S)\otimes \IndCoh(X)\to \IndCoh(S)\otimes \Dmod(X).$$

From here we obtain that the compact objects of $\IndCoh(S \times X_\dr)$ are bounded, and hence
coherent. 

\sssec{}

We now state the main result of this subsection:  

\begin{thm} \label{t:finiteness product}
The inclusion $\IndCoh(S \times X_\dr)^c$ into the category of coherent objects of $\IndCoh(S \times X_\dr)$
is an equality.
\end{thm}


\ssec{Proof of \thmref{t:finiteness product}} \label{ss:proof finiteness product}

\sssec{}

The proof of \thmref{t:finiteness product} will use the following assertion:

\medskip

Let $F: \bC \to \bD$ be a t-bounded functor between coherent DG categories which admits a continuous right adjoint $F^R$.  
Then, by \lemref{l:coh t}(b), we have a functor $\tilde{F}: \Ind(\Coh(\bC)) \to \Ind(\Coh(\bD))$ which preserves compact objects.  
In particular, $\tilde{F}$ admits a continuous right adjoint $\tilde{F}^R$.

\medskip

\begin{prop}\label{p:coherent right adjoint}
In the situation above, we have a commutative diagram
$$ \xymatrix{
\Ind(\Coh(\bD))\ar[r]^{\tilde{F}^R}\ar[d]_{\Psi_{\bD}} & \Ind(\Coh(\bC)) \ar[d]^{\Psi_{\bC}} \\
\bD \ar[r]^{F^R} & \bC
}$$
obtained by adjunction from \lemref{l:coh t}(b).
Moreover, if $F^R$ is fully faithful, so is $\tilde{F}^R$.
\end{prop}

\begin{proof}
Since $\tilde{F}^R$ and $\tilde{F}^R$ are both continuous, it suffices to prove that the natural map obtained by adjunction
$$ \Psi_{\bC} \circ \tilde{F}^R (x) \to F^R \circ \Psi_{\bD}(x)$$
is an isomorphism for $x \in \on{Coh}(\bD) \subset \Ind(\Coh(\bD))^+$.
By composing $F$ with a cohomological shift, we can assume without loss of generality that $F$ is right t-exact.
In particular,
$$F_{|\on{Coh}}: \on{Coh}(\bC) \to \on{Coh}(\bD)$$
is t-bounded and right t-exact; therefore, so is
$$ \tilde{F} : \Ind(\Coh(\bC)) \to \Ind(\Coh(\bD)). $$
By adjunction, the functors $F^R$ and $\tilde{F}^R$ are left t-exact.  Thus, the adjunctions $(F, F^R)$ and $(\tilde{F}, \tilde{F}^R)$ 
restrict to adjunctions between the subcategories of eventually coconnective objects.  
By \propref{p:coh coh t}(b), $\Psi_{\bC}$ and $\Psi_{\bD}$ restrict to equivalences on the subcategories of eventually coconnective objects.  
Therefore, for $x\in \Ind(\Coh(\bD))^+$, we have
$$ \tilde{F}^R(x) \simeq ((\Psi_{\bC})_{| \bC^{+}})^{-1}\circ F^R \circ \Psi_{\bD}(x) \in \Ind(\Coh(\bC))^+ ,$$
as desired.

\medskip

Now, suppose that $F^R$ is fully faithful.  Since $\tilde{F}^R$ commutes with colimits, to show that $\tilde{F}^R$ is fully faithful, it suffices to show that the counit map
$$(\tilde{F} \circ \tilde{F}^R)(x) \to x $$
is an isomorphism for $x \in \on{Coh}(\bD)$.  Since both sides are eventually coconnective, it suffices to show that this map is an isomorphism after applying $\Psi_\bD$; i.e., that
$$ \Psi_{\bD}(\tilde{F} \circ \tilde{F}^R)(x) \simeq (F \circ F^R)(\Psi_D (x)) \to \Psi_D (x) $$
is an isomorphism.  This is indeed the case because $F^R$ is fully faithful.

\end{proof}

\sssec{}

We now begin the proof of  \thmref{t:finiteness product}, by making a series of reduction steps.

\medskip

We need to show that every coherent object of $\IndCoh(S \times X_{\on{dR}})$ is compact.  Clearly, the assertion is Zariski local on $X$ and 
so we can assume without loss of generality that $X$ is an affine scheme.

\medskip

Now, let $i: X \to Y$ be a closed embedding of $X$ into a smooth affine scheme $Y$.  By Kashiwara's lemma, the functor 
$$ i^\IndCoh_{*}: \IndCoh(S \times X_{\on{dR}}) \to \IndCoh(S \times Y_{\on{dR}}) $$
is fully-faithful t-exact with a continuous right adjoint.  

\medskip

In particular, by \lemref{l:coh t}(b), $M \in \IndCoh(S \times X_{\on{dR}})$ 
is coherent if and only if $i^\IndCoh_*(M)$ is coherent, and similarly for compact.  Therefore, replacing $X$ by $Y$, 
we can assume without loss of generality that $X$ is a smooth affine scheme.

\sssec{}

In what follows, denote
$$ \bD_{S}:= \IndCoh(S\times X_{\on{dR}}) .$$
By the above $\bD_S$ is a locally Noetherian DG category with t-structure which is compactly generated by bounded objects.  We wish to show that the fully-faithful functor
$$ \Xi_S: \bD_S \to \Ind(\Coh(\bD_S)) $$
given by ind-extension of the inclusion of compact objects, and left adjoint to the functor $\Psi_S:=\Psi_{\bD_S}$, 
is an equivalence.

\medskip

We will prove this by Noetherian induction on $S$.  The base case is that $S$ is a smooth (and, in particular, classical and regular) scheme.  In this case
$$ \bD_S \simeq (\mathcal{O}_S \otimes \mathcal{D}_X)\on{-mod},$$
where $\mathcal{D}_X$ is the ring of differential operators on $X$. 

\medskip

The (ordinary) associative ring $\mathcal{O}_S \otimes \mathcal{D}_X$ has a finite cohomological
dimension. This implies that every compact object (i.e., finitely generated module) $M \in (\mathcal{O}_S 
\otimes \mathcal{D}_X)^{\heartsuit}$ is a compact object of $(\mathcal{O}_S \otimes 
\mathcal{D}_X)\on{-mod}$.  Therefore, by \propref{p:coh coh t}(a), every coherent object is compact, as desired.

\sssec{}

Now suppose that we have a closed subscheme $T \overset{i}{\hookrightarrow} S$ with open affine complement $U \overset{j}{\hookrightarrow} S$
such that the functors
$$ \Xi_T: \bD_T \to \Ind(\Coh(\bD_T))  \mbox{ and } \Xi_U: \bD_U \to \Ind(\Coh(\bD_U)) $$
are equivalences.  We wish to show that $\Xi_S: \bD_S \to \Ind(\Coh(\bD_S)) $ is also an equivalence.  Since $\Xi_S$ is fully-faithful, it will suffice to show that the image of $\Xi_S$ generates $\Ind(\Coh(\bD_S)) $.

\medskip
We have the functor $j^!: \bD_S \to \bD_U$ which is t-exact and admits a fully-faithful continuous right adjoint
$j_*: \bD_U \to \bD_S$.  Therefore, by \lemref{l:coh t}(b), $j^!$ preserves both compact and coherent objects, and we have a commutative square 
$$ \xymatrix{
\bD_S \ar[r]^{j^!} \ar[d]_{\Xi_S} & \bD_U \ar[d]^{\Xi_U} \\
\Ind(\Coh(\bD_S)) \ar[r]^{\tilde{j}^!} & \Ind(\Coh(\bD_U)) 
} $$
By \propref{p:coherent right adjoint}, the functor $\tilde{j}^!$ admits a continuous fully-faithful right adjoint $\tilde{j}_*$.
By the inductive hypothesis, $\Xi_U$ is an equivalence.  Clearly,
$\Ind(\Coh(\bD_S)) $ is generated by $\on{ker}(\tilde{j}^!)$ and $\on{Im}(\tilde{j}_*)$.  

\medskip

For $M \in \Ind(\Coh(\bD_U)) $, we have 
$$ \on{cone}(\Xi_S\circ j_* \circ (\Xi_U)^{-1}(M) \to \tilde{j}_*(M)) \in \on{ker}(\tilde{j}^!) .$$
Therefore, it suffices to show that
$$\on{ker}(\tilde{j}^!) \subset \on{Im}(\Xi_S) .$$

\sssec{}

Now, the functor $i_{*,\dr}: \bD_T \to \bD_S$ is t-exact and admits a continuous right adjoint. Therefore, as before, 
it preserves coherent and compact objects and we have a commutative square
$$
\xymatrix{
\bD_T \ar[r]^{i_{*,\dr}}\ar[d]_{\Xi_T} & \bD_S \ar[d]^{\Xi_S}\\
\Ind(\Coh(\bD_T))  \ar[r]^{\tilde{i}_{*,\dr}} & \Ind(\Coh(\bD_S))
}
$$
Clearly, we have $\on{Im}(\tilde{i}_{*,\dr})\subset \on{ker}(\tilde{j}^!)$.
By the inductive hypothesis, $\Xi_T$ is an equivalence.  Therefore, $\on{Im}(\tilde{i}_{*,\dr}) \subset \on{Im}(\Xi_S)$.  
Thus, it suffices to show that $\on{Im}(\tilde{i}_{*,\dr})$ generates $\on{ker}(\tilde{j}^!)$.

\sssec{}

Since $\tilde{j}^!$ is t-exact, the subcategory $\on{ker}(\tilde{j}^!) \subset \Ind(\Coh(\bD_S)) $ inherits a t-structure. Clearly, $\on{ker}(\tilde{j}^!)$ is 
generated by objects of the form
$$ \on{cone}(M \to \tilde{j}_* \circ \tilde{j}^!(M)), \quad M \in \on{Coh}(\bD_S).$$

Since the above objects belong to 
$$\on{ker}(\tilde{j}^!)^+:= \Ind(\Coh(\bD_S)) ^+ \cap \on{ker}(\tilde{j}^!),$$
we obtain that $\on{ker}(\tilde{j}^!)$ is generated by $\on{ker}(\tilde{j}^!)^+$. Furthermore, since the t-structure 
on $\on{ker}(\tilde{j}^!)$ is right-complete and $\tilde{j}^!$ is t-exact, we obtain that $\on{ker}(\tilde{j}^!)$ is generated by
$$\on{ker}(\tilde{j}^!)^b:= \Ind(\Coh(\bD_S))^b \cap \on{ker}(\tilde{j}^!).$$
Similarly, we make a reduction to
$$\on{ker}(\tilde{j}^!)^\heartsuit:= \Ind(\Coh(\bD_S))^b \cap \on{ker}(\tilde{j}^!).$$

\medskip

Thus, it suffices to show that every object of
$\on{ker}(\tilde{j}^!)^{\heartsuit}$
admits an exhaustive filtration with associated graded in the image of $\tilde{i}_{*,\dr}: \Ind(\Coh(\bD_T)) ^{\heartsuit} \to \Ind(\Coh(\bD_S))^{\heartsuit}$.  

\medskip

This follows from the commutativity of the diagram
$$\xymatrix{
\Ind(\Coh(\bD_T)) ^{\heartsuit} \ar[r]^{\tilde{i}_{*,\dr}}\ar[d]_{\Psi_T}^{\sim} & \Ind(\Coh(\bD_S)) ^{\heartsuit}\ar[d]^{\Psi_S}_{\sim} \ar[r]^{\tilde{j}^!} &
\Ind(\Coh(\bD_U)) ^{\heartsuit}\ar[d]^{\Psi_U}_{\sim} \\
\bD_T^{\heartsuit} \ar[r]^{i_{*,\dr}} & \bD_S^{\heartsuit} \ar[r]^{j^!} & \bD_U^{\heartsuit}
}$$
and the corresponding statement about $\on{ker}(j^!)^{\heartsuit}$.

\qed[\thmref{t:finiteness product} ]

\begin{rem}
As the reader could notice, the proof of \thmref{t:finiteness product} given above applies more generally: instead of
the category $\Dmod(X)$, we could take any DG category $\bC$ with a t-structure, which is locally Noetherian and
such that objects from $\bC^\heartsuit$ have a bounded cohomological dimension.
\end{rem} 

\ssec{Left vs right D-modules}

\sssec{} \label{sss:Psi and Ups}

Recall that for $X\in \Sch_{\on{aft}}$ we have canonically defined functors
$$\Psi_X:\IndCoh(X)\to \QCoh(X) \text{ and } \Upsilon_X:\QCoh(X)\to \IndCoh(X).$$

The functor $\Psi_X$ has already appeared in \secref{sss:IndCoh sch}; it is the ind-extension
of the embedding $\Coh(X)\hookrightarrow \QCoh(X)$.

\medskip

The functor $\Psi_X$ is explicitly given by
$$\CF\mapsto \CF\otimes \omega_X,$$
where $\omega_X\in \IndCoh(X)$ is the dualizing object. 

\begin{rem}
The functors $\Psi_X$ and $\Upsilon_X$ are mutually dual with respect to the 
self-dualities
$$\QCoh(X)\simeq \QCoh(X)^\vee \text{ and } \IndCoh(X)\simeq \IndCoh(X)^\vee$$
that are given on compact objects by
$$\CE\mapsto \CE^\vee \text{ and } \CF\mapsto \BD^{\on{Serre}}(\CF),$$
where $\CE^\vee$ is the monoidal dual of $\CE\in \QCoh(X)^c=\on{Perf}(X)$ and
$\BD^{\on{Serre}}$ is the Serre duality equivalence $\Coh(X)^{\on{op}}\to \Coh(X)$. 
\end{rem}

\sssec{}

For a map $f:X\to Y$ we have commutative diagrams
$$
\CD
\IndCoh(X) @>{\Psi_X}>> \QCoh(X) \\
@V{f^\IndCoh_*}VV @VV{f_*}V \\
\IndCoh(Y) @>{\Psi_Y}>> \QCoh(Y) 
\endCD
$$
and 
$$
\CD
\IndCoh(X) @<{\Upsilon_X}<< \QCoh(X) \\
@A{f^!}AA @ AA{f^*}A \\
\IndCoh(Y) @<{\Upsilon_Y}<< \QCoh(Y).
\endCD
$$

The second of these diagrams allows to define a functor $\Upsilon_\CX$ for $\CX\in \on{PreStk}_{\on{laft}}$, so
that for a morphism $f:\CX\to \CY$, the diagram
$$
\CD
\IndCoh(\CX) @<{\Upsilon_\CX}<< \QCoh(\CX) \\
@A{f^!}AA  @AA{f^*}A \\
\IndCoh(\CY) @<{\Upsilon_\CY}<< \QCoh(\CY)
\endCD
$$
commutes. 

\sssec{} 

Let $X$ be a scheme of finite type. Consider the category of \emph{left} D-modules on $X$:
$$\Dmod^l(X) := \QCoh(X_{\on{dR}}).$$

However, as shown in \cite[Proposition 2.4.4]{Crys}, the functor
$$ \Upsilon_{X_{\on{dR}}}: \Dmod^l(X) = \QCoh(X_{\on{dR}}) \to \IndCoh(X_{\on{dR}}) = \Dmod(X) $$
given by tensoring by the dualizing complex $\omega_{X_{\on{dR}}}$ is an equivalence of categories.

\sssec{}

Recall that if $\CX$ is any prestack, the category $\QCoh(\CX)$ carries a canonically defined t-structure:
an object of $\QCoh(\CX)$ is connective if its pullback to any affine scheme is connective. We remark
that, in general, this t-structure may be quite ill-behaved. 

\medskip

We obtain that the category $\Dmod^l(X)$ also carries a t-structure. The equivalence $\Upsilon_{X_{\on{dR}}}$ is \emph{not} t-exact, 
but by \cite[Proposition 4.4.4]{Crys}, it is t-bounded (if $X$ is smooth of dimension $d$, then this functor is t-exact up to $[d]$). 
In particular, the functor $\Upsilon_{X_{\on{dR}}}$ induces an equivalence
$$\Upsilon_{X_{\on{dR}}}:\Coh(\Dmod^l(X))\to \Coh(\Dmod(X)).$$

\sssec{}

Let $S$ be an affine scheme and $\CX$ an arbitrary prestack. By \cite[Chapter 3, Proposition 3.1.7]{Vol1}, the external tensor product functor 
\begin{equation} \label{e:ext ten prod}
\QCoh(S) \otimes \QCoh(\CX) \to \QCoh(S \times \CX)
\end{equation} 
is an equivalence. We observe: 

\begin{lem}\label{l:tensor t-equiv}
The equivalence \eqref{e:ext ten prod} is t-exact (where t-structure on the left hand side is the tensor product t-structure).
\end{lem}

\begin{proof}
We have a commutative triangle
$$ \xymatrix{
\QCoh(S) \otimes \QCoh(\CX) \ar[rd] \ar[rr] && \QCoh(S \times \CX)\ar[ld] \\
& \QCoh(\CX) 
}$$
with vertical arrows given by pushforward from $S$ to point.  The desired t-exactness follows from the fact that both vertical arrows are conservative and t-exact.
\end{proof}

\sssec{}

Thus, we obtain that external tensor product defines a t-exact equivalence
$$\QCoh(S)\otimes \QCoh(X_{\on{dR}}) \to \QCoh(S \times X_{\on{dR}}).$$

Now, consider the functor
\begin{multline} \label{e: Psi Ups}
\IndCoh(S \times X_{\on{dR}}) \simeq \IndCoh(S)\otimes \IndCoh(X_{\on{dR}}) \overset{\Psi_S \otimes \Upsilon^{-1}_{X_\dr}}{\longrightarrow} \\
\to \QCoh(S) \otimes \QCoh(X_{\on{dR}}) \simeq \QCoh(S \times X_{\on{dR}}) 
\end{multline}

We claim:

\begin{prop} \label{p:Psi Ups} \hfill
\begin{enumerate}[label={(\alph*)}]
\item
The functor \eqref{e: Psi Ups} induces an equivalence on the eventually coconnective (=bounded below)
subcategories. 

\item
The category $\QCoh(S \times X_{\on{dR}})$ is coherent.

\item The functor \eqref{e: Psi Ups} defines an equivalence from the category of
coherent objects in $\IndCoh(S \times X_{\on{dR}})$ to the category of coherent objects in 
$\QCoh(S \times X_{\on{dR}})$. 
\end{enumerate}
\end{prop}

\begin{proof}

Since the functor $\Upsilon_{X_\dr}$ has a bounded cohomological amplitude, it suffices to prove the 
proposition for the functor 
$$\IndCoh(S)\otimes \IndCoh(X_{\on{dR}}) \overset{\Psi_S\otimes \on{Id}}\longrightarrow 
\QCoh(S)\otimes \IndCoh(X_{\on{dR}}).$$

Now, since the functor $\Psi_S:\IndCoh(S)\to \QCoh(S)$ induces an equivalence of the left completions, by 
\cite[Proposition C.4.6.1]{SAG}, the same is true for the functor $\Psi_S\otimes \on{Id}$. In particular, we obtain
that the latter functor induces a t-exact equivalence on eventually coconnective (=bounded below) subcategories.

\medskip

This implies all three parts of the proposition. 

\end{proof} 

\ssec{Serre-Verdier duality}

\sssec{}

Being the IndCoh category of an inf-scheme, the category $\IndCoh(S\times X_{\on{dR}})$ is canonically self-dual. In terms
of the identification
$$\IndCoh(S\times X_{\on{dR}})\simeq \IndCoh(S)\otimes \Dmod(X),$$
this self-duality is the tensor product of the Serre self-duality of $\IndCoh(S)$ and the Verdier 
self-duality of $\Dmod(X)$.

\medskip

In particular, we obtain a canonical anti-equivalence of the subcategory of compact objects
$$ \mathbb{D}^{\on{SV}}: \IndCoh(S\times X_{\on{dR}})^c \simeq (\IndCoh(S \times X_{\on{dR}})^c)^{\on{op}} , $$
which is given by Serre duality on $S$ and Verdier duality for D-modules on $X$; i.e., for $M \in \on{Coh}(S)$ and $N \in \Dmod(X)^c$, we have
$$ \mathbb{D}^{\on{SV}}(M \boxtimes N) = \mathbb{D}^{\on{Serre}}(M) \boxtimes \mathbb{D}^{\on{Verdier}}(N).$$

\sssec{}

Combining \propref{p:Psi Ups} and \thmref{t:finiteness product}, we obtain an equivalence between 
$(\IndCoh(S \times X_{\on{dR}})^c)^{\on{op}}$ and the category of coherent objects in 
$\QCoh(S \times X_{\on{dR}})$.

\medskip

Combining further with \propref{p:pro-coh}, we obtain:

\begin{cor} \label{c:Serre Verdier}
We have a canonical equivalence
$$\BD^{\on{SV}}:\IndCoh(S \times X_{\on{dR}})^{\on{op}}\simeq \on{Pro}(\QCoh(S\times X_{\on{dR}})^-)_{\on{laft}},$$
explicitly given as follows: for $\CF\in \IndCoh(S \times X_{\on{dR}})$, the resulting functor 
$$\QCoh(S\times X_{\on{dR}})^-\to \Vect$$
maps $M\in \QCoh(S\times X_{\on{dR}})^-$ to
$$\underset{n}{\on{lim}}\, \Gamma^{\IndCoh}\left(S\times X_{\on{dR}},\CF\sotimes\left((\Psi_S\otimes \Upsilon^{-1}_{X_\dr})^{-1}(\tau^{\geq -n}(M))\right)\right),$$
where $(\Psi_S\otimes \Upsilon^{-1}_{X_\dr})^{-1}$ is the inverse of the equivalence
$$\left(\IndCoh(S)\otimes \IndCoh(X_{\on{dR}})\right)^+ \overset{\Psi_S\otimes \Upsilon^{-1}_{X_\dr}}\longrightarrow 
\left(\QCoh(S)\otimes \QCoh(X_{\on{dR}})\right)^+\simeq \QCoh(S\times X_{\on{dR}})^+.$$
\end{cor} 

%
%
%

\section{Deformation theory and mapping prestacks} \label{s:fake}

In this section we will review, mostly based on \cite[Chapter 1]{Vol2}, several notions from deformation
theory, which are also used in the definition of the notion of sectional laft-ness in the main body of the
paper. 

\ssec{The ``fake" pro-category} \label{ss:fake}

\sssec{} \label{sss:fake}

Let $\CY$ be a prestack. Denote
\begin{equation} \label{e:pro fake}
\on{Pro}(\QCoh(\CY)^-)^{\on{fake}}:= \underset{S \in \affSch_{/\CY}}{\on{lim}}\, \on{Pro}(\QCoh(S)^-).
\end{equation} 

In the above formula, in the formation of the limit, for a map of affine schemes $$g: S \to S'$$ the transition functor 
$\on{Pro}(\QCoh(S')^-)\to \on{Pro}(\QCoh(S)^-)$ is 
$g^\#:=\on{Pro}(g^*)$, see \secref{sss:pro-functor}. 

%

\medskip

For a given $(f: S \to \CY') \in \affSch_{/\CY'}$, denote by
$f^{\#}$ the evaluation functor 
$$\on{Pro}(\QCoh(\CY)^-)^{\on{fake}}\to \on{Pro}(\QCoh(S)^-).$$

\sssec{} \label{sss:oblv fake}

Given an object $\CF=\{\CF_S\}$ in the limit \eqref{e:pro fake}, we can evaluate it an on object 
$M\in \QCoh(\CY)^-$. Indeed, representing $M$ by 
$$\{M_S\}\in \underset{S \in \affSch_{/\CY}}{\on{lim}}\, \QCoh(S)^-,$$
we set
$$\CF(\CM):=\underset{S \in \affSch_{/\CY}}{\on{lim}}\, \CF_S(\CM_S).$$

\medskip

This defines a functor
\begin{equation}
\oblv^{\on{fake}}: \on{Pro}(\QCoh(\CY)^{-})^{\on{fake}} \to \on{Pro}(\QCoh(\CY)^-) ,
\end{equation}
which is in general \emph{not} an equivalence.  However, by \cite[Chapter 1, Lemma 4.6.2]{Vol2}, $\oblv^{\on{fake}}$ is an
equivalence whenever $\CY$ is a quasi-compact quasi-separated scheme.

\sssec{}

Given a map of prestacks $f: \CY' \to \CY$, we have a tautologically defined functor
$$ f^{\#}: \on{Pro}(\QCoh(\CY)^-)^{\on{fake}} \to \on{Pro}(\QCoh(\CY')^-)^{\on{fake}}.$$

\medskip

The functor $f^{\#}$ admits a \emph{left} adjoint (cf. \secref{sss:pro-functor})
$$ f_{\#}: \on{Pro}(\QCoh(\CY')^-)^{\on{fake}} \to \on{Pro}(\QCoh(\CY)^-)^{\on{fake}}, $$
which can be described as follows.  Let $\mathcal{F} \in \on{Pro}(\QCoh(\CY')^-)^{\on{fake}}$ and
let $(g: S \to \CY) \in \affSch_{/\CY}$. Consider the pullback square
$$
\xymatrix{
\CZ \ar[r]^{\tilde{f}} \ar[d]_{h} & S \ar[d]^g \\
\CY' \ar[r]^f & \CY
}
$$
Then $g^{\#}(f_{\#}(\mathcal{F})) \in \on{Pro}(\QCoh(S)^-)$ is given by
$$ M \mapsto \oblv^{\on{fake}}(h^{\#}(\mathcal{F}))(\tilde{f}^*(M)),$$
for $M \in \QCoh(S)^-$.

\sssec{}

The functors $f^{\#}$ and $f_{\#}$ tautologically satisfy base change in full generality:

\begin{lem}\label{l:pro base change}
Suppose that we have a pullback square of prestacks
$$
\xymatrix{
\CZ' \ar[d]_{f'}\ar[r]^h & \CZ \ar[d]^f \\
\CY' \ar[r]^g & \CY
}
$$
Then the natural transformation (given by adjunction)
$$ f'_{\#} \circ h^{\#} \to g^{\#}\circ f_{\#} : \on{Pro}(\QCoh(\CZ)^{-})^{\on{fake}} \to \on{Pro}(\QCoh(\CY')^-)^{\on{fake}} $$
is a natural isomorphism.
\end{lem}

\sssec{}

For the sequel, we record the following properties of $\on{Pro}(\QCoh(\CY)^-)^{\on{fake}}$ (which easily follow from the definitions):

\begin{lem}\label{l:pro oblv functors}
Let $f:\CY' \to \CY$ be a map of prestacks.  Then
\begin{enumerate}[label={(\alph*)}]
\item
We have a commutative diagram
$$ \xymatrix{
\on{Pro}(\QCoh(\CY')^-)^{\on{fake}} \ar[r]^{f_{\#}} \ar[d]_{\oblv^{\on{fake}}} & \on{Pro}(\QCoh(\CY)^-)^{\on{fake}} \ar[d]^{\oblv^{\on{fake}}} \\
\on{Pro}(\QCoh(\CY')^-) \ar[r]^{\Res_{f^*}} & \on{Pro}(\QCoh(\CY)^-)
},$$
where the bottom horizontal functor is given by precomposition with $f^*: \QCoh(\CY)^- \to \QCoh(\CY')^{-}$.

\item
If $f$ is schematic, and is quasi-compact quasi-separated, we have a commutative diagram
$$ \xymatrix{
\on{Pro}(\QCoh(\CY)^-)^{\on{fake}} \ar[r]^{f^{\#}} \ar[d]_{\oblv^{\on{fake}}} & \on{Pro}(\QCoh(\CY')^-)^{\on{fake}} \ar[d]^{\oblv^{\on{fake}}} \\
\on{Pro}(\QCoh(\CY)^-) \ar[r] & \on{Pro}(\QCoh(\CY')^-)
},$$
where the bottom horizontal functor is given by precomposition with $$f_*: \QCoh(\CY')^- \to \QCoh(\CY)^-$$
(which preserves the eventually connective subcategories by the assumption on $f$).
\end{enumerate}
\end{lem}

\ssec{Recollections on deformation theory} \label{ss:def}

Recall from \cite[Chapter 1, Corollary 7.2.6]{Vol2} that a prestack $\CY$ \emph{admits deformation theory} if:
\begin{itemize}
\item
$\CY$ is \emph{convergent}; i.e., for every $A \in \on{ComAlg}^{\leq 0}$, the map
$$ \CY(\Spec(A)) \to \underset{n}{\varprojlim}\ \CY(\Spec(\tau^{\geq -n} (A))) $$
is an isomorphism.
\item
For every push-out square
\begin{equation}\label{e:sq zero push-out}
 \xymatrix{
S_1 \ar[r]\ar[d] & S_2 \ar[d] \\
S_1' \ar[r] & S_2'
}
\end{equation}
of affine schemes such that the map $S_1 \to S_1'$ has the structure of a square zero extension (see e.g.,
\cite[Chapter 1, Section 5.1]{Vol2}), we have a pullback square
$$ \xymatrix{
\CY(S_2') \ar[r]\ar[d] & \CY(S_2) \ar[d] \\
\CY(S_1') \ar[r] & \CY(S_1)
}$$
\end{itemize}

\sssec{}\label{sss:cotangent complex}
Suppose that $\CY$ is a prestack with deformation theory.  In this case, $\CY$ admits a cotangent complex,
which is an object
$$ T^*(\CY) \in \underset{S \in \affSch_{/\CY}}{\on{lim}}\ \on{Pro}(\QCoh(S)^-)=: \on{Pro}(\QCoh(\CY)^-)^{\on{fake}}.$$ 

\medskip

For each map $f:S \to \CY$ from an affine scheme, the object 
$$ f^{\#} (T^*(\CY)) \in \on{Pro}(\QCoh(S)^-) $$
is such that for $M \in \QCoh(S)^{\leq 0}$,
$$ f^{\#} (T^*(\CY))(M) \simeq \on{Maps}_{S/}(S_M, \CY),$$
where $S\to S_M$ is the split square zero extension of $S$ by $M$.

\sssec{}\label{sss:prestack square zero}

Suppose we have a map of prestacks $f: \CY' \to \CY$ such that $\CY$ admits deformation theory.
Given $M \in \QCoh(\CY')^{\leq 0}$, we can consider the corresponding split square zero extension (see \cite[Chapter 1, Section 10.3.1]{Vol2})
$$ \CY' \to \CY'_M .$$
Explicitly,
$$ \CY'_M = \underset{(f: S\to \CY') \in \affSch_{/\CY'}}{\on{colim}}\ S_{f^*(M)} .$$
Tautologically, we have
$$ \oblv^{\on{fake}}(f^{\#}(T^*(\CY))) (M) \simeq \on{Maps}_{\CY'/}(\CY'_M, \CY).$$

\sssec{}
In what follows, it will be convenient to work in the relative setting.  Namely, suppose that we have a map
of prestacks $\CY \to \CX$.  Then $\CY$ \emph{admits deformation theory relative to $\CX$} if for every affine
scheme $S \to \CX$ over $\CX$, the pullback $\CY \underset{\CX}{\times} S$ admits deformation theory.

\medskip
Unraveling the definitions, we have:

\begin{prop}\label{p:relative def}
Let $\CY \to \CX$ be a map of prestacks.  Then $\CY$ admits defomration theory relative to $\CX$ if and only if

\begin{enumerate}[label={(\roman*)}]
\item
For every affine scheme $S \to \CX$ over $\CX$, the fiber product $\CY \underset{\CX}{\times} S$ is convergent.
\item
For every push-out square \eqref{e:sq zero push-out} mapping to $\CX$, we have a pullback square
$$ \xymatrix{
\on{Maps}_{/\CX}(S_2', \CY) \ar[r]\ar[d] & \on{Maps}_{/\CX}(S_2, \CY) \ar[d] \\
\on{Maps}_{/\CX}(S_1', \CY) \ar[r] & \on{Maps}_{/\CX}(S_1', \CY).
}$$
\end{enumerate}
\end{prop}

\sssec{}
From this we immediately obtain:

\begin{cor}
Suppose $f: \CY \to \CX$ is a map of prestacks such that $\CX$ admits deformation theory.  Then $\CY$ admits deformation theory if and only if it admits deformation theory relative to $\CX$.
\end{cor}

\sssec{}
The discussion of \secref{sss:cotangent complex} immediately generalizes to the relative setting using \propref{p:relative def}.  Namely, in this case, we have the relative cotangent complex
$$ T^*(\CY/\CX) \in \on{Pro}(\QCoh(\CY)^-)^{\on{fake}} .$$
Moreover, in the case that $\CX$ admits deformation theory, we have a cofiber sequence in $\on{Pro}(\QCoh(\CY)^{-})^{\on{fake}}$
\begin{equation}
f^{\#}(T^*(\CX)) \to T^*(\CY) \to T^*(\CY/\CX) .
\end{equation}

\ssec{Relative mapping prestacks}

In the main body of the paper, we study the deformation theory of relative mapping stacks and Weil restrictions.

\sssec{}

Let $\CX$ be a prestack and let $\CY \to \CX$ and $\CZ \to \CX$ be prestacks over $\CX$.  We define the relative mapping prestack
$$\underline{\on{Maps}}_{/\CX}(\CY, \CZ)\in \on{PreStk}_{/\CX},$$
whose value on an affine scheme $S\in  \affSch_{/\CX}$ is 
$$\underline{\on{Maps}}_{/\CX}(\CY, \CZ)(S):=\on{Maps}_{/S}(\CY \underset{\CX}{\times} S, \CZ \underset{\CX}{\times} S).$$

\begin{prop}\label{p:mapping has defs}
Suppose that $\CY \to \CX$ and $\CZ \to \CX$ are prestacks over $\CX$ and that $\CZ$ admits deformation theory relative to $\CX$.  Then the relative 
mapping prestack $\underline{\on{Maps}}_{/\CX}(\CY, \CZ)$ admits deformation theory relative to $\CX$.
\end{prop}
\begin{proof}
By definition, we need to show that for every $U \in \affSch_{/\CX}$, the fiber product
$$\underline{\on{Maps}}_{/\CX}(\CY, \CZ) \underset{\CX}{\times} U \simeq \underline{\on{Maps}}_{/U}(\CY \underset{\CX}{\times} U, \CZ \underset{\CX}{\times} U) $$
admits deformtion theory.  Thus, we can assume that $\CX=U$ is an affine scheme.  Moreover, the condition of
admitting relative deformation theory is manifestly stable with respect to taking limits of prestacks.
Therefore, we can assume that $\CY = T \in \affSch_{/U}$.

\medskip

Thus,  suppose that we have a push-out square \eqref{e:sq zero push-out} mapping to $U$.  In this case,
by \cite[Chapter 1, Proposition 5.3.2]{Vol2}, the map
$$ S_1 \underset{U}{\times} T \to S_1' \underset{U}{\times} T $$
admits a canonical structure of a square-zero extension.  Moreover, we have a push-out square of affine schemes over $U$:
$$
 \xymatrix{
S_1 \underset{U}{\times} T \ar[r]\ar[d] & S_2 \underset{U}{\times} T \ar[d] \\
S_1' \underset{U}{\times} T \ar[r] & S_2' \underset{U}{\times} T
}
$$
Therefore, since $\CZ$ admits relative deformation theory over $U$, by \propref{p:relative def}, we obtain a pullback square
$$ \xymatrix{
\underline{\on{Maps}}_{/U}(T, \CZ)(S_2') = \on{Maps}_{/U}(S_2' \underset{U}{\times} T, \CZ) \ar[r]\ar[d] & \on{Maps}_{/U}(S_2 \underset{U}{\times} T, \CZ)=\underline{\on{Maps}}_{/U}(T, \CZ)(S_2) \ar[d] \\
\underline{\on{Maps}}_{/U}(T, \CZ)(S_1') = \on{Maps}_{/U}(S_1' \underset{U}{\times} T, \CZ) \ar[r] & 
\on{Maps}_{/U}(S_1 \underset{U}{\times} T, \CZ) = \underline{\on{Maps}}_{/U}(T, \CZ)(S_1),
}$$
as required.

\medskip

For convergence, it suffices to show that for every $\on{Spec}(A) \to U$, the natural map
$$ \CZ(A \underset{\mathcal{O}_U}\otimes \mathcal{O}_T) \to \underset{n}{\varprojlim}\ \CZ( (\tau^{\geq -n} A) \underset{\mathcal{O}_U}{\otimes} \mathcal{O}_T) $$
is an isomorphism.  This follows from the convergence of $\CZ$ and the following assertion, applied to
$$ \underset{n}{\varprojlim}\ ((\tau^{\geq -n} A) \underset{\mathcal{O}_U}{\otimes} \mathcal{O}_T) .$$
\end{proof}

\begin{lem}\label{l:convergence}
Suppose that $\CZ$ is a convergent prestack.
Let
$$A = \underset{n}{\varprojlim}\ A_n \in \on{ComAlg}^{\leq 0}$$
be an inverse limit such that for every $m\geq 0$, there
exists $N$ such that for $n>N$, the map $\tau^{\geq -m}(A) \to \tau^{\geq -m}(A_n)$ is an isomorphism.  Then
the map
$$ \CZ(A) \to \varprojlim \CZ(A_n) $$
is an isomorphism.
\end{lem}

\begin{proof}
By hypothesis, the natural map
$$ \CZ(\tau^{\geq -m}(A)) \to \underset{n}{\varprojlim}\ \CZ(\tau^{\geq -m}(A_n)) $$
is an isomorphism.
Now, since $\CZ$ is convergent, we have
$$ \CZ(A) \simeq \underset{m}{\varprojlim}\ \CZ(\tau^{\geq -m}(A)) \simeq
\underset{m}{\varprojlim}\  \underset{n}{\varprojlim}\ \CZ(\tau^{\geq -m}(A_n)) \simeq
\underset{n}{\varprojlim}\  \underset{m}{\varprojlim}\ \CZ(\tau^{\geq -m}(A_n)) \simeq
\underset{n}{\varprojlim}\ \CZ(A_n),
$$
as desired.
\end{proof}

\sssec{}
We can describe the relative cotangent complex of relative mapping prestacks as follows.  Suppose, as before,
we have $\CY, \CZ \in \on{PreStk}_{/\CX}$ and that $\CZ$ admits deformation theory relative to $\CX$.

\medskip

Consider the correspondence
$$ \xymatrix{
 \underline{\on{Maps}}_{/\CX}(\CY, \CZ) \underset{\CX}{\times} \CY \ar[r]^-{\on{ev}}\ar[d]_{p_\CY} & \CZ  \\
\underline{\on{Maps}}_{/\CX}(\CY, \CZ)
}$$

\begin{prop}\label{p:mapping cotangent}
In the situation above,
$$ T^*(\underline{\on{Maps}}_{/\CX}(\CY, \CZ)/\CX) \simeq (p_\CY)_{\#} \circ \on{ev}^{\#}(T^*(\CZ/\CX)).$$
\end{prop}

\begin{proof}
By definition, given an affine scheme $S$ over $\CX$ with a map $s: S \to \underline{\on{Maps}}_{/\CX}(\CY, \CZ)$, 
for $M \in \QCoh(S)^{\leq 0}$,
$$s^{\#}(T^*(\underline{\on{Maps}}_{/\CX}(\CY, \CZ)/\CX))(M) \simeq \on{Maps}_{S//\CX}(S_M, \underline{\on{Maps}}_{/\CX}(\CY, \CZ)) .$$
By definition,
$$ \on{Maps}_{S//\CX}(S_M, \underline{\on{Maps}}_{/\CX}(\CY, \CZ)) \simeq \on{Maps}_{S \underset{\CX}{\times} \CY//\CX}(S_M \underset{\CX}{\times} \CY, \CZ). $$
We have a commutative diagram
$$\xymatrix{
S \underset{\CX}{\times} \CY \ar[r]^-{s \times \on{id}}\ar[d]_{p_\CY} &  \underline{\on{Maps}}_{/\CX}(\CY, \CZ)\underset{\CX}{\times} \CY \ar[r]^-{\on{ev}}\ar[d]^{p_\CY} & \CZ \\
S \ar[r]^-s & \underline{\on{Maps}}_{/\CX}(\CY, \CZ)
}
$$
where the square is Cartesian. Now, the prestack $S_M \underset{\CX}{\times} \CY$ is a split square zero extension of $S \underset{\CX}{\times} \CY$ 
(see \secref{sss:prestack square zero}) by $p_\CY^*(M)$. Therefore, 

$$ \on{Maps}_{S \underset{\CX}{\times} \CY//\CX}(S_M \underset{\CX}{\times} \CY, \CZ) \simeq
\oblv^{\on{fake}}((s\times \on{id})^{\#} \circ  \on{ev}^{\#}(T^*(\CZ/\CX)))(p_\CY^*(M)) \simeq $$
$$\simeq (p_\CY)_{\#}\circ (s\times \on{id})^{\#} \circ  \on{ev}^{\#}(T^*(\CZ/\CX))(M) 
\simeq s^{\#}\circ (p_\CY)_{\#} \circ \on{ev}^{\#}(T^*(\CZ/\CX))(M)$$
(the last isomorphism is the base change from \lemref{l:pro base change}), as desired.
\end{proof}

\sssec{Weil restriction} \label{sss:Weil restr}
A slight variant of relative mapping stacks is Weil restriction.  Namely, suppose that we have a map of
prestacks
$$ f: \CX \to \CX' $$
and a prestack $\CZ \in \on{PreStk}_{/\CX}$ over $\CX$.  We can then define the Weil restriction of $\CZ$
along $f$, to be denoted
$$\Res^{\CX}_{\CX'}(\CZ) \text{ or } \Res_f(\CZ),$$
and defined by
$$ \on{Maps}_{/\CX'}(S, \Res^{\CX}_{\CX'}(\CZ)):= \on{Maps}_{/\CX}(S \underset{\CX'}{\times} \CX, \CZ), $$
for $S \in \affSch_{/\CX'}$.

\medskip

In other words,
$$\Res^{\CX}_{\CX'}(\CZ)\simeq  \underline{\on{Maps}}_{/\CX'}(\CX, \CZ) \underset{ \underline{\on{Maps}}_{/\CX'}(\CX, \CX)}\times \CX',$$
where $\underline{\on{Maps}}_{/\CX'}(\CX, \CZ)\to  \underline{\on{Maps}}_{/\CX'}(\CX, \CX)$
is induced by the map $\CZ\to \CX$, and 
$$\CX'\to  \underline{\on{Maps}}_{/\CX'}(\CX, \CX)$$
corresponds to the identity map on $\CX$. 

\medskip

Hence, from \propref{p:mapping cotangent}, we obtain: 

\begin{cor}\label{c:Weil def}
Suppose that $f:\CX \to \CX'$ is a map of prestacks and $\CZ \to \CX$ is a prestack over $\CX$ that admits deformation theory relative to $\CX$.  Then the Weil restriction $f_*(\CZ)$ admits deformation theory
relative to $\CX'$ and
$$ T^*(\Res^{\CX}_{\CX'}(\CZ)/\CX') = (\on{id}\times f)_{\#} \circ \on{ev}^{\#}(T^*(\CZ/\CX)), $$
where
$$\on{ev}: \Res^{\CX}_{\CX'}(\CZ)\underset{\CX'}{\times}\CX \to \CZ $$
is the tautological evaluation map. 
\end{cor}

\section{Categorical prestacks} \label{s:categ prestacks}

In this section, we describe what we need from the theory of categorical prestacks; i.e., prestacks valued in categories.

\ssec{Basic notions}

\sssec{}
A categorical prestack is a(n accessible) functor
$$ (\affSch)^{\on{op}}  \to \on{Cat} $$
Let $\on{CatPreStk}:=\on{Funct}((\affSch)^{\on{op}} , \on{Cat})$ denote the (2-)category of categorical prestacks.

\medskip

The inclusion $\Spc \to \on{Cat}$ gives an inclusion $\on{PreStk} \hookrightarrow \on{CatPreStk}$.

\sssec{}

Let $f:\CX\to \CY$ be a map of categorical prestacks. We shall say that $f$ is a level-wise
Cartsian/coCartesian fibration (resp., fibration in groups) if for every $S\in (\affSch)^{\on{op}}$,
the corresponding functor
$$\CX(S)\to \CY(S)$$
has this property. 

\sssec{}
We will say that a categorical prestack $\CX \in \on{CatPreStk}$ is locally almost of finite type (laft) if the corresponding functor
$$ \CX: (\affSch)^{\on{op}}  \to \on{Cat} $$
satisfies:

\begin{enumerate}[label=(\roman*)]

\item It is convergent, i.e., for $A \in \on{ComAlg}^{\leq 0}$
the natural map
$$\CX(\Spec(A)) \to \underset{n}{\varprojlim}\ \CX(\Spec(\tau^{\geq -n} (A))) $$
is an equivalence;

\item For each $n\geq 0$, the functor
$$A \mapsto \CX(\Spec(A)), \quad A\in \on{ComAlg}^{\leq 0, \geq -n}$$
commutes with filtered colimits.

\end{enumerate}


\sssec{}

For the most part, we will only be concerned with laft categorical prestacks.  However, the following recognition 
result will be useful:
%
%
%

\begin{lem}\label{l:laft cat fibration}
Let $f: \CX \to \CY$ be a map in $\on{CatPreStk}$.
\begin{enumerate}[label={(\alph*)}]

\item[(a)]
Suppose that $\CY$ is convergent and for any $S\to \CY$ 
(resp., $[1]\times S\to \CY$), the fiber product $\CX \underset{\CY}{\times} S$
(resp., $\CX \underset{\CY}{\times} ([1]\times S)$) is convergent. Then $\CX$
is convergent. 

\item[(a')]
Suppose that $f$ is a level-wise Cartesian or coCartesian fibration.
Then in (a) it suffices to require that the fiber products $\CX \underset{\CY}{\times} S$
be convergent. 

\item[(b)]
Suppose that $\CY$ is laft, and every map 
$S \to \CY$ (resp., $[1]\times S\to \CY$), with $S\in {}^{<\infty}\!\affSch_{\on{aft}}$
the fiber product $\CX \underset{\CY}{\times} S$
(resp., $\CX \underset{\CY}{\times} ([1]\times S)$) is laft. Then $\CX$
is laft. 

\item[(b')] Suppose that $f$ is a level-wise Cartesian or coCartesian fibration.
Then in (b) it suffices to require that the fiber products $\CX \underset{\CY}{\times} S$
be laft.  
\end{enumerate}
\end{lem}

\sssec{} \label{sss:formal compl categ}

Let $\CX \in \on{CatPreStk}$ be a categorical prestack and let the corresponding deRham categorical prestack $\CX_{\dr} \in \on{CatPreStk}$ be defined as
$$ \CX_{\dr}(S) := \CX(S^{\on{red}}) ,$$
as in the case of ordinary prestacks.

\medskip 

Given a map of categorical prestacks $f: \CX \to \CY$, we define the formal completion
$$ \CY^{\wedge}_{\CX}:= \CY \underset{\CY_{\on{dR}}}{\times} \CX_{\on{dR}} ,$$
as in the case of ordinary prestacks.

\medskip

We record the following:

\begin{lem}\label{l:form compl fibration}
Let $\CX \to \CY$ be a level-wise (co-)Cartesian fibration in groupoids of categorical prestacks.  Then the corresponding maps
$$ \CX \to \CY^{\wedge}_{\CX} \to \CY $$
are also level-wise (co-)Cartesian fibrations in groupoids.
\end{lem}

\sssec{Weil restriction for categorical prestacks} \label{sss:Weil categ}

Let $\gamma:\CX'\to \CX$ be a level-wise \emph{coCartesian} fibration in groupoids, and let $\CZ'\to \CX'$ be 
a level-wise \emph{Cartesian} fibration in groupoids. In this case, we can define the categorical prestack
$$\Res_\gamma(\CZ')$$
over $\CX$, which is the right adjoint to the functor
$$ \on{CatPreStk}_{/\CX} \to \on{CatPreStk}_{/\CX'} $$
given by $(-) \underset{\CX}{\times}\CX'$.

\medskip

By definition
$$(\Res_\gamma(\CZ'))^{\on{grpd}}:=\Res_\gamma((\CZ')^{\on{grpd}}),$$
where in the right-hand side $\gamma$ stands for the map $(\CX')^{\on{grpd}}\to \CX^{\on{grpd}}$.

\medskip

Unlike the case of prestacks, the functor $\Res$ does not always exist for an arbitrary map of categorical prestacks.  However, it is easy to see that that the conditions on $\gamma$ and $\CZ'\to \CX'$ guarantee that 
$\Res_\gamma(\CZ')$ is well-defined and is a level-wise coCartesian fibration in groupoids over $\CX$. 

\ssec{Pseudo-proper morphisms} \label{ss:pseudo-proper}

In this subsection, we review the theory of pseudo-proper morphisms, following \cite[Sect. 1.5]{AB}.

\sssec{}

Given an affine scheme $S\in \affSch_{\on{aft}}$, we define the category of pseudo-proper prestacks over $S$ to be the smallest full subcategory
$$(\on{PreStk}_{\on{aft}})_{/S}^{\on{ps-pr}} \subset (\on{PreStk}_{\on{aft}})_{/S} $$
closed under colimits that contains all schemes $X\to S$, for which the map $X\to S$ is proper.  By \cite[Chapter 2, Cor. 4.1.4]{Vol2}, this includes all ind-inf-schemes $X \to S$ ind-proper over $S$.

\medskip

For a general laft prestack $\CX$, we define the category of pseudo-proper prestacks over $\CX$
$$ \on{PreStk}_{/\CX}^{\on{ps-pr}} \subset \on{PreStk}_{/\CX} $$
to consist of those laft prestacks $f: \CY \to \CX$ such that for any map $S \to \CX$ from an affine scheme $S\in \affSch_{\on{aft}}$,
the fiber product
$$ \CY \underset{\CX}{\times} S \to S $$
is pseudo-proper.

\sssec{}

We record the following variant of \cite[Prop. 7.4.2]{AB} (which is proved identically):

\begin{prop}\label{p:ps-pr colimit}
Let $F: I \to \on{infSch}$ be a diagram of inf-schemes such that for each $i \to j \in I$, the corresponding map
$F(i) \to F(j)$ is proper.  Then for each $i \in I$, the map of prestacks
$$ F(i) \to \on{colim} F $$
(where the colimit it taken in $\on{PreStk}$) is pseudo-proper.
\end{prop}

\sssec{}

We have the following basic fact about pseudo-proper morphisms of prestacks.

\begin{lem}\label{l:adj ps-pr prestacks}
Let $f:\CX \to \CY$ be a pseudo-proper map of laft prestacks.  Then the symmetric monoidal functor
$$ f^!: \IndCoh(\CY) \to \IndCoh(\CX) $$
admits a left adjoint $f_!$ such that:

\begin{enumerate}[label={(\alph*)}]
\item
$f_!$ is compatible with base change; i.e., for any $g: S \to \CY$, $S \in \affSch_{\on{aft}}$ and $\CX':=\CX \underset{\CY}{\times} S$, the natural transformation
$$ f'_! \circ g'^! \to g^! \circ f_! ,$$
arising from the Cartesian diagram
$$ \xymatrix{
\CX' \ar[r]^{g'}\ar[d]_{f'} & \CX \ar[d]^f\\
S \ar[r]^g & \CY
}$$
is an isomorphism.

\item
$f_!$ is a \emph{strict} functor of $\IndCoh(\CY)$-module categories; 
in other words, $f_!$ satisfies the projection formula: for $M \in \IndCoh(\CX)$ and $N \in \IndCoh(\CY)$, the natural map
$$ f_!(M \otimes f^!(N)) \to f_!(M) \otimes N $$
is an isomorphism.
\end{enumerate}
\end{lem}

\sssec{}\label{sss:pseudo proper def}

We will say that a map of laft categorical prestacks $f: \CX \to \CY$ is pseudo-proper if:
\begin{itemize}
\item
$f$ is a level-wise coCartesian fibration in groupoids;
\item
For every $S\to \CY$ with $S\in \affSch_{\on{aft}}$, the fiber product
$$ \CX \underset{\CY}{\times} S \to S$$
is a pseudo-proper map.
\end{itemize}

\sssec{}
The class of pseudo-proper maps of categorical prestacks enjoys a number favorable of persistence properties.  Namely, we have:

\begin{lem}\label{l:pseudo-proper ops} \hfill 

\begin{enumerate}[label={(\alph*)}]

\item
Pseudo-proper morphisms are stable under base change; namely, if we have a Cartesian square
$$ \xymatrix{
\CY' \ar[r]\ar[d] & \CY \ar[d] \\
\CX' \ar[r] & \CX
}$$
such that the morphism $\CY \to \CX$ is pseudo-proper then the morphism $\CY'\to \CX'$ is also pseudo-proper.

\item
If $\CZ \to \CY$ and $\CY\to \CX$ are pseudo-proper morphisms then so is the composite $\CZ \to \CX$.

\item
If $\CY \to \CX$ and $\CY' \to \CX'$ are pseudo-proper morphisms, then so is the product
$$ \CY \times \CY' \to \CX \times \CX' .$$

\item
If $\CY \to \CX$ is pseudo-proper and $\CX'$ has pseudo-proper diagonal, then for any morphism $\CY \to \CX'$, the corresponding morphism
$$ \CY \to \CX \times \CX' $$
is pseudo-proper.
\end{enumerate}
\end{lem}

\ssec{Sheaves on categorical prestacks}

\sssec{}
Let $\IndCoh, \Dmod \in \on{CatPreStk}_{\on{laft}}$ denote the laft categorical prestacks
$$ S \mapsto \IndCoh(S) \mbox{ and } S \mapsto \Dmod(S), $$
respectively, for $S \in \affSch_{\on{aft}}$.  Note that the functoriality is given by !-pullback.

\medskip

For a laft categorical prestack $\CX \in \on{CatPreStk}_{\on{laft}}$, set
$$ \IndCoh(\CX):= \on{Funct}_{\on{CatPreStk}_{\on{laft}}}(\CX, \IndCoh)$$
and
$$\Dmod(\CX):= \on{Funct}_{\on{CatPreStk}_{\on{laft}}}(\CX, \Dmod) .$$

\sssec{}

Concretely, the data of $\CF\in \IndCoh(\CX)$ assigns to an object $x\in \CX(S)$ with $S\in \affSch$ an
object $x^!(\CF)\in \IndCoh(S)$ and for any $\alpha:x_s\to x_t$ in $\CX(S)$ a map
$$x_s^!(\CF)\to x_t^!(\CF)$$
in $\IndCoh(S)$.

\medskip

These data must satisfy a homotopy-coherent system of compatibilities with respect to pullbacks
along maps $S'\to S$. 

\sssec{}

As with the case of ordinary prestacks, D-modules on categorical prestacks can be described in terms of ind-coherent sheaves on the deRham stack: 

\begin{lem}
Let $\CX \in \on{CatPreStk}_{\on{laft}}$.  Then
$$ \Dmod(\CX) \simeq \IndCoh(\CX_{\dr}).$$
\end{lem}

\sssec{}

Tautologically, for a map $f:\CX\to \CY$, we have the pullback functor
$$f^!:\IndCoh(\CY)\to \IndCoh(\CX).$$

In general, pushforward of sheaves on categorical prestacks is a very delicate opertaton.  
However, in the case of pseudo-proper maps, we have a well-behaved left adjoint to the pullback functor:

\begin{lem}\label{l:pseudo-proper pushforward}
Let $f:\CX \to \CY$ be a pseudo-proper map of laft categorical prestacks.  The pullback functor
$$ f^!: \IndCoh(\CY) \to \IndCoh(\CX) $$
admits a left adjoint $f_!$ such that:
\begin{enumerate}[label={(\alph*)}]
\item
$f_!$ is compatible with base change; i.e., for any $g: \CY' \to \CY$, $\CY' \in \on{CatPreStk}_{\on{laft}}$ and $\CX':=\CX \underset{\CY}{\times} S$, 
the natural transformation
$$ f'_! \circ g'^! \to g^! \circ f_! ,$$
arising from the Cartesian diagram
$$ \xymatrix{
\CX' \ar[r]^{g'}\ar[d]_{f'} & \CX \ar[d]^f\\
\CY' \ar[r]^g & \CY
}$$
is an isomorphism.

\item
$f_!$ is a \emph{strict} functor of $\IndCoh(\CY)$-module categories; in other words, $f_!$ satisfies the 
projection formula: for $M \in \IndCoh(\CX)$ and $N \in \IndCoh(\CY)$, the natural map
$$ f_!(M \otimes f^!(N)) \to f_!(M) \otimes N $$
is an isomorphism.
\end{enumerate}

\end{lem}



\sssec{}

As in the case of prestacks, we can give a general condition so that ind-coherent sheaves on the product are the tensor product of ind-coherent sheaves:

\begin{lem}\label{l:product dualizable}
Let $\CX, \CY \in \on{CatPreStk}_{\on{laft}}$.  Suppose that $\IndCoh(\CX)$ is dualizable.  Then
$$ \IndCoh(\CX \times \CY) \simeq \IndCoh(\CX) \otimes \IndCoh(\CY) .$$
\end{lem}


%

\ssec{Deformation theory for categorical prestacks}

\sssec{}

Let $\CX$ be a categorical prestack. We shall say that $\CX$ \emph{admits deformation theory} if the following 
conditions hold:

\begin{enumerate}

\item[(A)]
$\CX^{\on{grpd}}$ admits deformation theory (see \secref{ss:def} for a reminder of what this means); 

\item[(B)]
The prestack $\on{Funct}([1],\CX)^{\on{grpd}}$ admits deformation theory;

\item[(C)] The restriction map $\partial_s:\on{Funct}([1],\CX)^{\on{grpd}}\to \CX^{\on{grpd}}$
corresponding to $\{0\}\hookrightarrow [1]$ induces an \emph{isomorphism} 
$$\on{Pro}(\partial_s^*)(T^*(\CX^{\on{grpd}}))\to T^*(\on{Funct}([1],\CX)^{\on{grpd}}).$$

\end{enumerate}

\begin{rem} 

Note that in the above definition, we can replace conditions (B) and (C) by any of the following equivalent ones: 

\begin{enumerate}

\item[(BC)] $\CX$ is convergent and for any square-zero embedding of affine schemes $S\hookrightarrow S'$,
the map $\CX(S')\to \CX(S)$ is a level-wise coCartesian fibration in groupoids;

\item[(BC')] $\CX$ is convergent and for any nilpotent embedding of affine schemes $S\hookrightarrow S'$,
the map $\CX(S')\to \CX(S)$ is a level-wise coCartesian fibration in groupoids;

\item[(BC'')]  the map
$$ \CX \to \CX_{\on{dR}} $$
is a level-wise coCartesian fibration in groupoids.
\end{enumerate}

\end{rem}

\sssec{} 

Let $\CX$ admit deformation theory, and fix a morphism $\alpha:x_s\to x_t$ in the category $\CX(S)$ for
some $S\in \affSch$, i.e., an object in $\on{Funct}([1],\CX)(S)$. We obtain that $\alpha$ induces a map
$$T^*_{x_t}(\CX)\to T^*_{x_s}(\CX)$$
in $\on{Pro}(\QCoh(S)^-)$.

\medskip
More functorially, in this case, the prestack $\CX$ admits a cotangent complex
$$ T^*(\CX) \in \on{Funct}_{\on{CatPreStk}}(\CX^{op}, \on{Pro}(\QCoh^-)). $$
It is characterized by the following: for $S\in \affSch$, the corresponding functor
$$ T^*(\CX)_S: \CX(S)^{op} \to \on{Pro}(\QCoh(S)^-) $$
is given by
$$T^*(\CX)_S(f)(M) = \CX(S_M) \underset{\CX(S)}{\times} \{f\} $$
for $f\in \CX(S)$ and $M \in \QCoh(S)^{\leq 0}$, where $S_M$ is the split square zero extension of $S$ by $M$.

\sssec{}

We will say that a categorical prestack is laft-def if it is laft and admits deformation theory. 

\medskip

In particular, if $\CX$ is a laft-def categorical prestack, then $\CX^{\on{grpd}}$ is a laft-def
prestack. Then for $\alpha:x_s\to x_t$ as above with $S\in \affSch_{\on{aft}}$, 
we obtain a well-defined map
$$T_{x_s}(\CX)\to T_{x_t}(\CX)$$
in $\IndCoh(S)$. 

\medskip

This defines the tangent complex of $\CX$:
$$ T(\CX) \in \IndCoh(\CX) = \on{Funct}_{\on{CatPreStk}_{\on{laft}}}(\CX, \IndCoh).$$
Specifically, for each $S \in \affSch_{\on{aft}}$, the tangent complex
$$ T(\CX)_S: \CX(S) \to \IndCoh(S) $$
is given by
$$ T(\CX)(f)(M) = \CX(\on{RealSplitSqZ}(M)) \underset{\CX(S)}{\times} \{f\} ,$$
for $f\in \CX(S)$ and $M \in \IndCoh(S)$, where $\on{RealSplitSqZ}(M)$ is the $\IndCoh$ split square
zero extension functor defined in \cite[Chapter 7, Sect. 3.7]{Vol2}.

%
%
%

\sssec{}

Finally, we have:

\begin{lem} \label{l:fibration def theory}
Let $f:\CX \to \CY$ be a level-wise coCartesian fibration in groupoids of categorical prestacks. Suppose that
$\CY$ admits deformation theory and for any $S \in \affSch_{\on{aft}}$, the prestack 
$S \underset{\CY}{\times} \CX$ admits deformation theory. Then:
\begin{enumerate}[label={(\alph*)}]
\item  $\CX$ admits deformation theory;
\item The restriction map
$$T^*(\CX/\CY)|_{\CX^{\on{grpd}}}\to T^*(\CX^{\on{grpd}}/\CY^{\on{grpd}})$$
is an isomorphism, where
$$T^*(\CX/\CY):=\on{Cone}\left(\on{Pro}(f^*)(T^*(\CY))\to T^*(\CX)\right).$$
\end{enumerate}
\end{lem}

\sssec{}

From here, using \lemref{l:laft cat fibration}, we obtain:

\begin{cor}\label{c:fibration laft-def}
Let $\CX \to \CY$ be a level-wise coCartesian fibration in groupoids of categorical prestacks.  
Suppose that $\CY$ is laft-def and for every $S \in \affSch_{\on{aft}}$, the prestack $S \underset{\CY}{\times} \CX$ is laft-def.
Then:
\begin{enumerate}[label={(\alph*)}]
\item $\CX$ is laft-def;
\item The restriction map
$$T(\CX^{\on{grpd}}/\CY^{\on{grpd}})\to T(\CX/\CY)|_{\CX^{\on{grpd}}}$$
is an isomorphism in $\IndCoh(\CX^{\on{grpd}})$, where
$$T(\CX/\CY):=\on{Fib}\left(T(\CX)\to f^!(T(\CY))\right).$$
\end{enumerate}

\end{cor}

\section{Proof of the affine Borel-Weil-Bott theorem} \label{s:BBW}
\centerline{By Dennis Gaitsgory}

\bigskip

This Appendix is devoted to a representation-theoretic proof
of the affine Borel-Weil-Bott Theorem (\thmref{t:affine BBW}). The   
main tool is the Kashiwara-Tanisaki localization theorem at
the positive level onto the ``thick" affine flag scheme. 

\ssec{Reduction to a representation-theoretic statement}

\sssec{}

It is easy to see that the statement of theorem is \'etale-local with respect to the curve $X$. 
Hence, we can assume that $X=\BA^1$. 

\medskip

In this case of $\BA^1$, the D-module $\CA_{\CL,X}:=\CA_\CL|_X[-1]$ is translation-invariant. So, it is enough to show
that its !-fiber 
$$\CA_{\CL,x}:=\CA_{\CL,X}|_x[1]\simeq \Gamma^\IndCoh(\Gr_{G,x},\CL\otimes \omega_{\Gr_{G,x}}),$$
at some/any point $x\in X$ lives in cohomological degree $0$, and
vanishes if the restriction of $\kappa$ to one of the simple factors is positive-definite. 

\begin{rem}

Note that when $\kappa$ is negative-definite on each simple factor, the line bundle $\CL$
is \emph{anti-ample} on $\Gr_{G,x}$. The pro-finite dimensional vector space dual to
$\Gamma^\IndCoh(\Gr_{G,x},\CL\otimes \omega_{\Gr_{G,x}})$ is
$$\Gamma(\Gr_{G,x},\CL^{-1}),$$
where $\CL^{-1}$ is ample. 

\medskip

Thus, point (a) of \thmref{t:affine BBW} can be viewed as saying that the line bundle $\CL^{-1}$
does not have nigher cohomology groups on $\Gr_{G,x}$ as long as $\kappa$ is non-positive 
definite.

\medskip

Suppose for a moment that $G$ is simple, and suppose that $\kappa$ is positiver-definite, so that
point (b) of \thmref{t:affine BBW} says that $\Gamma^\IndCoh(\Gr_{G,x},\CL\otimes \omega_{\Gr_{G,x}})=0$.

\medskip

Let us observe that this is natural (but still not automatic) from the above perspective: in this case the line bundle 
$\CL$ is ample, while the object $\omega_{\Gr_{G,x}}$ is ``infinitely connective", i.e., it belongs to
$\IndCoh(\Gr_{G,x})^{\leq -n}$ for any $n$.

\end{rem}

\sssec{}

Let $\fg(\CK_x)_{-\kappa}$ be the category of Kac-Moody modules at level $-\kappa$, and let 
$$\KL(G)_{-\kappa}:=\fg(\CK_x)_{-\kappa}\mod^{G(\CO_x)}$$
be the Kazhdan-Lusztig category.  We refer the reader to \cite[Appendix]{FG} for a discussion of the notion of groups
acting on categories.

\medskip

The category $\fg(\CK_x)_{-\kappa}$ carries a (strong) action of the loop group $G(\CK_x)$, which induces an action of 
the spherical Hecke category 
$$\Sph_{G,x}:=\Dmod(G(\CO_x)\backslash G(\CK_x)/G(\CO_x)$$
on $\KL(G)_{-\kappa}$; we will denote it by
$$\CF,\CM\mapsto \CF\star \CM.$$

\medskip

We have a $G(\CK_x)$-equivariant functor\footnote{In the formula below, the inversion of the sign of the level is due to the 
fact that we are considering \emph{right} twisted D-modules.}
$$\Gamma^\IndCoh(\Gr_{G,x},\CL\otimes -):\Dmod(\Gr_{G,x})\to \fg(\CK_x)_{-\kappa},$$
which induces a functor
$$\Dmod(\Gr_{G,x})^{G(\CO_x)}\to \KL(G)_{-\kappa},$$
compatible with the $\Sph_{G,x}$-actions.

\sssec{} \label{sss:int rep}

Consider the object 
$$\omega_{\Gr_{G,x}}\in \Dmod(\Gr_{G,x})^{G(\CO_x)}\simeq \Sph_{G,x}.$$
It has a natural structure of algebra in the monoidal category $\Sph_{G,x}$. In particular, the endofunctor of $\KL(G)_{-\kappa}$ given by 
\begin{equation} \label{e:conv omega}
\CM\mapsto \omega_{\Gr_{G,x}}\star \CM
\end{equation}
has a structure of monad. 

\medskip

Define
$$\KL(G)_{-\kappa}^{\on{Hecke}}:=\omega_{\Gr_{G,x}}\mod(\KL(G)_{-\kappa}).$$

It is easy to see that we can equivalently define $\KL(G)_{-\kappa}^{\on{Hecke}}$
as the category $$\fg(\CK_x)_{-\kappa}\mod^{G(\CK_x)}$$ of (strongly) $G(\CK_x)$-equivariant
objects in $\fg(\CK_x)_{-\kappa}$. In light of this, we will use one more notation for this category,
namely, $G(\CK_x)\mod_{-\kappa}$. This is the (derived) definition of the category of 
integrable loop group representations at level $\kappa$. 

\medskip

We have a pair of adjoint functors
\begin{equation} \label{e:Hecke adj}
\ind^{\on{Hecke}}:\KL(G)_{-\kappa}\rightleftarrows \KL(G)_{-\kappa}^{\on{Hecke}}:\oblv^{\on{Hecke}}.
\end{equation} 

\medskip

Since the endofunctor of $\KL(G)_{-\kappa}$ underlying the monad \eqref{e:conv omega} is
right t-exact with respect to the natural t-structure, the category $\KL(G)_{-\kappa}^{\on{Hecke}}$
acquires a t-structure, for which the forgetful functor $\oblv^{\on{Hecke}}$ is t-exact.

\medskip

The abelian category 
$$(\KL(G)_{-\kappa}^{\on{Hecke}})^\heartsuit=(G(\CK_x)\mod_{-\kappa})^\heartsuit$$
is the usual abelian category of integrable loop group representations at level $\kappa$. 

\sssec{}

We have
$$\Gamma^\IndCoh(\Gr_{G,x},\CL\otimes \delta_{1,\Gr})\simeq \BV_{-\kappa},$$
as objects of $\KL(G)_{-\kappa}$, where $\BV_{-\kappa}$ is the vacuum module
and $\delta_{1,\Gr}\in \Dmod(\Gr_{G,x})$ is the dellta-function at the unit point
 of $\Gr_{G,x}$. 

\medskip

Denote
$$\BV_{-\kappa}^{\on{int}}:=\omega_{\Gr_{G,x}}\star \BV_{-\kappa}.$$

By construction, the object $\CA_{\CL,x}$ is the image of $\BV_{-\kappa}^{\on{int}}$ under the forgetful functor
$$\KL(G)_{-\kappa}\to \Vect.$$

By \secref{sss:int rep}, the object $\BV_{-\kappa}^{\on{int}}$ naturally upgrades to an object 
of the category $\KL(G)_{-\kappa}^{\on{Hecke}}$. 

\sssec{} \label{sss:max int}
 
We have
$$\BV_{-\kappa}^{\on{int}} \in (G(\CK_x)\mod_{-\kappa})^{\leq 0}.$$

Furthermore,  it follows from \eqref{e:Hecke adj} that the object
\begin{equation} \label{e:max int}
H^0(\BV_{-\kappa}^{\on{int}})\in (G(\CK_x)\mod_{-\kappa})^\heartsuit
\end{equation} 
is universal (within $(G(\CK_x)\mod_{-\kappa})^\heartsuit$) 
with respect to the property of receiving a map from $\BV_{-\kappa}$.

\medskip

In particular, it follows that when $G$ is semi-simple and
simply-connected, the object \eqref{e:max int}
is the classical maximal integrable quotient of $\BV_{-\kappa}$. 

\sssec{}

At this stage we can already prove point (b) of \thmref{t:affine BBW}. Namely, we claim that if $\kappa$
is positive-definite on one of the simple factors, then all the cohomologies of $\BV_{-\kappa}^{\on{int}}$ are
zero.

\medskip

Indeed, each of these cohomologies naturally upgrades to an object of the abelian category $(G(\CK_x)\mod_{-\kappa})^\heartsuit$,
while it is known that when the level is negative on one of the simple factors, the above category is zero (it violates the condition
of dominance of highest weights). 

\begin{rem}
We have just shown that if $\kappa$ is positive-definite on one of the simple factors, then the object 
$\BV_{-\kappa}^{\on{int}}$, regarded either as an object of $G(\CK_x)\mod_{-\kappa}$ or
$\KL(G)_{-\kappa}$, is such that all of its cohomologies (with respect to the natural
t-structure) are zero, or, equivalently, that its image under the forgetful functor $G(\CK_x)\mod_{-\kappa}\to \KL(G)_{-\kappa}\to \Vect$
vanishes. 

\medskip

However, the object $\BV_{-\kappa}^{\on{int}}$ itself is \emph{non-zero}. For example, one can show that
$$\CHom_{\KL(G)_{-\kappa}}(\BV_{-\kappa},\BV_{-\kappa}^{\on{int}})\simeq k.$$

This does not cause a contradiction because the t-structure on $\KL(G)_{-\kappa}$ is \emph{not} left-separated 
(equivalently, the forgetful functor $\KL(G)_{-\kappa}\to \Vect$ is \emph{not conservative}). 

\end{rem}

\sssec{}

We proceed to the proof of points (a) and (c) of  \thmref{t:affine BBW}. Thus, we will assume that $\kappa$
is non-positive definite on each simple factor. We will show that in this case the object 
$\BV_{-\kappa}^{\on{int}}$, viewed as an object of $\KL(G)_{-\kappa}$, belongs to the heart of the t-structure.

\medskip

In particular, this will show that when $G$ is semi-simple and simply-connected, the object 
$\BV_{-\kappa}^{\on{int}}$ is the maximal integrable quotient of $\BV_{-\kappa}$ (see \secref{sss:max int}).
In particular, for $\kappa=0$, we will obtain that $\BV_{-\kappa}^{\on{int}}$ is the trivial representation of
$G(\CK_x)$, as desired. 

\ssec{Kashiwara-Tanisaki localization}

Our proof that $\BV_{-\kappa}^{\on{int}}$ lies in the heart of the t-structure will rely 
on the localization theory of Kashiwara-Tanisaki at the \emph{positive level}. 

\sssec{}

Consider the \emph{thick} base affine space:
$$\widetilde{\Fl}^{\on{thick}}_G:=\Bun_G(\BP^1)\underset{\on{pt}/G(\CO_x) \times \on{pt}/G}\times (\on{pt} \times \on{pt}/N),$$
where the map 
$$\Bun_G(\BP^1)\to \on{pt}/G(\CO_x) \times \on{pt}/G$$
corresponds to restricting a $G$-bundle to the formal neighborhood of $x\in \BA^1\subset \BP^1$ and the point 
$\infty\in \BP^1$. 

\sssec{}

We have a $G(\CK_x)$-equivariant localization functor 
$$\Delta_{\CL^{-1}}:\fg(\CK_x)_{-\kappa}\to \Dmod^l(\widetilde\Fl^{\on{thick}}_G),$$
and its right adjoint
\begin{equation} \label{e:sects thick}
\Dmod^l(\widetilde\Fl^{\on{thick}}_G)\to \fg(\CK_x)_{-\kappa}, \quad 
\CM \mapsto \Gamma(\widetilde{\Fl}^{\on{thick}}_G,\CL^{-1}\otimes \CM),
\end{equation} 
where $\CL$ denotes the pullback of the same-named line bundle along the forgetful map
$$\widetilde{\Fl}^{\on{thick}}_G\to \Bun_G.$$

\sssec{}

Consider the categories 
$$\Dmod^l(\widetilde\Fl^{\on{thick}}_G)^I \text{ and } \fg(\CK_x)_{-\kappa}^I,$$
where $I\subset G(\CO_x)$ is the Iwahori subgroup. The functor \eqref{e:sects thick} induces a functor
\begin{equation} \label{e:Iwahori localization}
\Dmod^l(\widetilde\Fl^{\on{thick}}_G)^I\to \fg(\CK_x)_{-\kappa}^I.
\end{equation} 



\medskip

We are now ready to quote the version of the localization theorem that we will need
(\cite[Theorem 5.2.1]{Ka}): 

\begin{thm} \label{t:Ka-Ta}
Suppose that $\kappa$ is non-positive definite on every simple factor. Then the functor 
\eqref{e:Iwahori localization} is t-exact and maps standard objects to standard objects. 
\end{thm} 

\sssec{}


\medskip

The functor \eqref{e:sects thick} also induces a functor 
\begin{equation} \label{e:KL localization}
\Dmod^l(\widetilde\Fl^{\on{thick}}_G)^{G(\CO_x)}\to \KL(G)_{-\kappa}.
\end{equation}
Since the functor \eqref{e:sects thick} was $G(\CK_x)$-equivariant, the functor \eqref{e:Iwahori localization}
intertwines the actions of $\Sph_{G,x}$ on both sides. 

\medskip

From \thmref{t:Ka-Ta} we formally obtain:

\begin{cor} \label{c:Ka-Ta}
Suppose that $\kappa$ is non-positive definite on every simple factor. Then 
the functor \eqref{e:KL localization} is t-exact and maps standard objects to standard objects. 
\end{cor} 

\ssec{End of the proof}

Recall that our goal is to show that the object $\BV_{-\kappa}^{\on{int}}$ lies in the heart of the t-structure.

\sssec{}

Note that the stack $\widetilde\Fl^{\on{thick}}_G/G(\CO_x)$ identifies with 
$$\Bun_G(\BP^1)\underset{\on{pt}/G}\times \on{pt}/N,$$
where the map $\Bun_G(\BP^1)\to \on{pt}/G$ is given by taking the fiber
of a $G$-bundle at $\infty\in \BP^1$.

\medskip

Let us denote by $\on{pr}$ the projection 
$$\widetilde\Fl^{\on{thick}}_G\to \Bun_G(\BP^1)\underset{\on{pt}/G}\times \on{pt}/N.$$

\medskip

We have a t-exact the identification 
$$\Dmod^l(\widetilde\Fl^{\on{thick}}_G)^{G(\CO_x)}\simeq \Dmod(\Bun_G(\BP^1)\underset{\on{pt}/G}\times \on{pt}/N),$$
so that the action of
$\Sph_{G,x}$ on $\Dmod^l(\widetilde\Fl^{\on{thick}}_G)^{G(\CO_x)}$
corresponds to the Hecke action on $\Dmod(\Bun_G(\BP^1)\underset{\on{pt}/G}\times \on{pt}/N)$
at $x$.

\sssec{}

Consider the open embedding
$$\on{pt}/N \simeq \on{pt}/G\underset{\on{pt}/G}\times \on{pt}/N
\overset{j}\hookrightarrow \Bun_G(\BP^1)\underset{\on{pt}/G}\times \on{pt}/N,$$
and consider the corresponding standard object
$$\on{pr}^*(j_!(\underline{k}_{\on{pt}/N}))[-\dim(N)]\in \Dmod^l(\widetilde\Fl^{\on{thick}}_G)^{G(\CO_x)}.$$

By \corref{c:Ka-Ta}, we have
$$\BV_{-\kappa} \simeq \Gamma\left(\widetilde\Fl^{\on{thick}}_G,
\CL^{-1}\otimes \on{pr}^*(j_!(\underline{k}_{\on{pt}/N}))[-\dim(N)]\right).$$

Hence,
$$\BV_{-\kappa}^{\on{int}}\simeq
\Gamma\left(\widetilde\Fl^{\on{thick}}_G,
\CL^{-1}\otimes \on{pr}^*(\omega_{\Gr_{G,x}}\star  j_!(\underline{k}_{\on{pt}/N}))[-\dim(N)]\right).$$

Thus, applying \corref{c:Ka-Ta}, we conclude that it suffices to show that the object 
\begin{equation} \label{e:loc int}
\omega_{\Gr_{G,x}}\star  j_!(\underline{k}_{\on{pt}/N})[-\dim(N)] \in \Dmod(\Bun_G(\BP^1)\underset{\on{pt}/G}\times \on{pt}/N)
\end{equation}
lies in the heart of the t-structure.

\sssec{}

We can rewrite the object $j_!(\underline{k}_{\on{pt}/N})\in \Dmod(\Bun_G(\BP^1)\underset{\on{pt}/G}\times \on{pt}/N)$
also as
$$\wt{j}_!(k)[2\dim(N)],$$
where $\wt{j}$ is the map
$$\on{pt}\to \on{pt}/N \overset{j}\to \Bun_G(\BP^1)\underset{\on{pt}/G}\times \on{pt}/N.$$

\medskip

Thus, we have to show that the object
$$\omega_{\Gr_{G,x}}\star \wt{j}_!(k)[\dim(N)]\in \Dmod(\Bun_G(\BP^1)\underset{\on{pt}/G}\times \on{pt}/N)$$
lies in the heart of the t-structure.

\sssec{}

Let $\pi$ denote the natural projection
$$\Gr_{G,x}\to \Bun_G(\BP^1)\underset{\on{pt}/G}\times \on{pt}/N.$$

The corresponding functor
$$\pi_!:\Dmod(\Gr_{G,x})\to \Dmod(\Bun_G(\BP^1)\underset{\on{pt}/G}\times \on{pt}/N)$$
intertwines the $\Sph_{G,x}$-action on $\Dmod(\Gr_{G,x})$ by right convolutions and the
Hecke action at $x$ on $\Dmod(\Bun_G(\BP^1)\underset{\on{pt}/G}\times \on{pt}/N)$. 

\medskip

Hence, 
$$\omega_{\Gr_{G,x}}\star \wt{j}_!(k)\simeq \pi_!(\omega_{\Gr_{G,x}}).$$

\sssec{}

Note now that the natural map 
$$\pi_!(\omega_{\Gr_{G,x}})\to 
\omega_{\Bun_G(\BP^1)\underset{\on{pt}/G}\times \on{pt}/N}$$
is an isomorphism. 

\medskip

Indeed, the map $\pi$ is a locally trivial fibration with typical fiber
\begin{equation} \label{e:rat maps}
\underline{\Maps}(\BP^1-x,G)\underset{G}\times N,
\end{equation} 
which is contractible: indeed, there exists a $\BG_m$-action that contracts the 
ind-scheme \eqref{e:rat maps} to a point. 

\sssec{}

Hence, we obtain that 
$$\omega_{\Gr_{G,x}}\star \wt{j}_!(k)\simeq \omega_{\Bun_G(\BP^1)\underset{\on{pt}/G}\times \on{pt}/N}[\dim(N)].$$

\medskip

This gives the desired result since the object
$\omega_{\Bun_G(\BP^1)\underset{\on{pt}/G}\times \on{pt}/N}[\dim(N)]$
indeed lies in the heart of the t-structure: the stack $\Bun_G(\BP^1)\underset{\on{pt}/G}\times \on{pt}/N$
is smooth of dimension $-\dim(N)$. 

\qed[\thmref{t:affine BBW}]

\begin{rem} 

One can use a similar argument to show that the category $G(\CK_x)\mod_{-\kappa}$ is semi-simple. 
(The fact that the abelian category $(G(\CK_x)\mod_{-\kappa})^\heartsuit$ is semi-simple is known;
so the content of the assertion is that $G(\CK_x)\mod_{-\kappa}$ is the derived category of its heart.) 

\medskip

For example, let us show that $\BV_{-\kappa}^{\on{int}}\in G(\CK_x)\mod_{-\kappa}$
(which we already know lies in the heart of the t-structure) does not have 
higher self Exts. 

\medskip

By adjunction,
$$\CHom_{G(\CK_x)\mod_{-\kappa}}(\BV_{-\kappa}^{\on{int}},\BV_{-\kappa}^{\on{int}})
\simeq \CHom_{\KL(G)_{-\kappa}}(\BV_{-\kappa},\BV_{-\kappa}^{\on{int}}).$$

Now, one can show that \cite[Theorem 5.2.1]{Ka} implies that the functor 
\eqref{e:Iwahori localization}, and hence, \eqref{e:KL localization} is fully faithful. 

\medskip

Hence, we can rewrite the latter expression as
\begin{multline*} 
\CHom_{\Dmod(\Bun_G(\BP^1)\underset{\on{pt}/G}\times \on{pt}/N)}\left(\wt{j}_!(k)[\dim(N)],
\omega_{\Bun_G(\BP^1)\underset{\on{pt}/G}\times \on{pt}/N})[\dim(N)]\right) \simeq \\
\simeq \CHom_{\Dmod(\on{pt})}\left(k,\wt{j}^!(\omega_{\Bun_G(\BP^1)\underset{\on{pt}/G}\times \on{pt}/N})\right)
\simeq k,
\end{multline*} 
as desired. 

\end{rem}

\end{document}